\documentclass[a4paper,draft]{book}
\usepackage[psamsfonts]{amssymb}
\usepackage{amsmath}
\usepackage{amsthm}
\usepackage{longtable}
\newcommand{\R}{\mathbb R}

\newcommand{\C}{\mathbb C}
\newtheorem{teore}{Theorem}[chapter]
\newtheorem{lemma}{Lemma}[chapter]
\newtheorem{defin}{Definition}[chapter]
\newtheorem{conse}{Corollary}[chapter]
\newtheorem{remarka}{Remark}[chapter]

\allowdisplaybreaks[2]
\numberwithin{equation}{chapter}

\newsavebox{\vla}
\savebox{\vla}(3,18){
  \put(1,18){\thicklines \line(0,-1){18} \thinlines}
  \put(1,9){\circle*{3}}
}

\newsavebox{\vlb}
\savebox{\vlb}(3,3){
  \put(1,3){\thicklines \line(0,-1){3} \thinlines}
}

\newsavebox{\vlc}
\savebox{\vlc}(3,21){
  \put(1,9){\thicklines \line(0,-1){9} \thinlines}
  \put(1,9){\thicklines \line(1,2){6} \thinlines}
  \put(1,9){\thicklines \line(-1,2){6} \thinlines}
  \put(1,9){\circle*{3}}
}

\newsavebox{\vld}
\savebox{\vld}(3,20){
  \put(1,9){\thicklines \line(0,-1){9} \thinlines}
  \put(1,9){\thicklines \line(1,1){11} \thinlines}
  \put(1,9){\thicklines \line(-1,1){11} \thinlines}
  \put(1,9){\circle*{3}}
}

\newsavebox{\vle}
\savebox{\vle}(3,21){
  \put(1,21){\thicklines \line(0,-1){21} \thinlines}
  \put(1,9){\thicklines \line(2,1){24} \thinlines}
  \put(1,9){\thicklines \line(-2,1){24} \thinlines}
  \put(1,9){\circle*{3}}
}

\newsavebox{\vlf}
\savebox{\vlf}(3,20){
  \put(1,9){\thicklines \line(0,-1){9} \thinlines}
  \put(1,9){\thicklines \line(1,1){11} \thinlines}
  \put(1,9){\thicklines \line(-1,1){11} \thinlines}
  \put(1,9){\thicklines \line(4,1){44} \thinlines}
  \put(1,9){\thicklines \line(-4,1){44} \thinlines}
  \put(1,9){\circle*{3}}
}

\newsavebox{\vlg}
\savebox{\vlg}(3,20){
  \put(1,20){\thicklines \line(0,-1){20} \thinlines}
  \put(1,9){\thicklines \line(2,1){22} \thinlines}
  \put(1,9){\thicklines \line(-2,1){22} \thinlines}
  \put(1,9){\thicklines \line(5,1){55} \thinlines}
  \put(1,9){\thicklines \line(-5,1){55} \thinlines}
  \put(1,9){\circle*{3}}
}

\newsavebox{\vlh}
\savebox{\vlh}(3,20){
  \put(1,9){\thicklines \line(0,-1){9} \thinlines}
  \put(1,9){\thicklines \line(1,1){11} \thinlines}
  \put(1,9){\thicklines \line(-1,1){11} \thinlines}
  \put(1,9){\thicklines \line(3,1){33} \thinlines}
  \put(1,9){\thicklines \line(-3,1){33} \thinlines}
  \put(1,9){\thicklines \line(6,1){66} \thinlines}
  \put(1,9){\thicklines \line(-6,1){66} \thinlines}
  \put(1,9){\circle*{3}}
}

\begin{document}
\frontmatter

\title{Symplectic Cobordism in Small Dimensions and 
a Series of Elements of  Order Four}
\author{Aleksandr L. Anisimov \and 
Vladimir V.~Vershinin}
\date{}
\maketitle
{\bf Abstract.} 
We present the structure of symplectic cobordism ring $MSp_{*}$ 
in dimensions up to 51 and give a construction  of an infinite
series of elements $\Gamma_i$, $ i=1, 3,4, ...$, of order four
in this ring, where 
$\operatorname{dim} \, \Gamma_i=8i+95$. The key
element of the series is $\Gamma_1$ in dimension 103.

\vfill

\footnoterule

{\footnotesize{ 2000 
\hfill\break}}

{\footnotesize{{\it Key words and phrases}.
Symplectic cobordism ring, Adams-Novikov
spectral sequence, Massey product.\hfill\break }}

\tableofcontents

\mainmatter

\chapter*{Introduction}
\addcontentsline{toc}{chapter}{Introduction}
\markboth{Introduction}{Introduction}

Traditionally it is considered that roots of cobordism are contained in the work
of H.~Poincar\'e \cite{Poi} where he introduced homology based on the notion of
(what we would say now) smooth manifold. Only some time later
as an answer to the critics of Heegaar
Poincar\'e turned to definition of homology using simplicial subdivision.
On the contemporary level cobordism theory started in the works of 
L.~S.~Pontriagin \cite{Pon1}, \cite{Pon2} and  V.~A.~Rohlin \cite{Ro}.

The principal possibility to consider cobordism of manifolds
with a stable normal quaternionic bundle appeared in the work of
Rene Thom \cite{T} when he first defined cobordism for manifolds
with an action of a subgroup  of the orthogonal group $G$ in the
stable normal bundle and proved the variant of 
Pontriagin-Thom theorem which establishes connections between cobordism 
classes of such manifolds 
and homotopy groups of Thom  complexes $MG$ which he also constructed. 
The introduction of cobordism theory associated  with such group $G$
gave a possibility to unify the known cases of non-orientable ($O(n)$),
orientable ($SO(n)$) and framed (trivial group $e$) cobordisms.
It also raised a question about the other classical groups.
The theory associated to the unitary group was first studied by 
John Milnor  \cite{Mi} and S.~P.~Novikov \cite{N1}, \cite{N2}. It actually became 
one of the main objets of Algebraic Topology. Cobordism theories associated 
with other groups, for example $Spin(n)$, also found their applications 
in Mathematics \cite{St2}. The question about the symplectic cobordism,
which corresponds to the group $G=Sp(n)$ is very natural, in the same sense
that we have real numbers $\R$, complex numbers $\C$ and quaternions $\mathbb H$,
 so what constructions and facts that are valid for the first two objects
have any meaning for the third one. 
Unfortunately the case of symplectic cobordism is very difficult.
The present work is one of the illustrations to this fact.
Formally it was John Milnor who asked first the question 
about symplectic cobordism in \cite{Mi} and the initial results on this
cobordism were obtained in the work of S.~P.~Novikov \cite{N1},
\cite{N2}. The main method in the works of Milnor and Novikov was the
Adams spectral sequence. S.~P.~Novikov proved that for $p>2$ this
spectral sequence is trivial, so the symplectic cobordism 
ring $MSp_*$
is such that $MSp_*\otimes Z[\frac12]$ is the polynomial  algebra
over $Z[\frac12]$  with one $4k$-dimensional generator for any
natural number $k$.

The next step after Novikov in the study of symplectic cobordism was made by 
A.~Liulevichius \cite{Liu}. He also applied the Adams spectral sequence
and calculated the symplectic cobordism ring in dimensions up to 6. 
Then R.~Stong \cite{St1} gave a construction of symplectic manifolds, in particular,
 $\text{mod}\,p$ generators for odd primes $p$. Later R.~Nadiradze
constructed free involutions on Stong manifolds, whose orbit spaces provide 
examples of symplectic manifolds \cite{Na1}.  
 
The subsequent calculations in the Adams spectral sequence  were continued by
D.~M.~Segal \cite{Se}. This spectral sequence is very complicated, even the
calculation of its initial term is a nontrivial algebraic problem, it is
studied in the work of L.~N.~Ivanovskii \cite{Iv}. 

In the beginning of 1970-ties Nigel Ray \cite{R3} applied the Atiah-Hirzebruch
spectral sequence to the study of symplectic cobordism. He also 
constructed in \cite{R1} an infinite  series of elements of order two in the 
symplectic
cobordism ring $MSp_{*}$. These elements $\theta_{1} \in MSp_{1}$,
$\Phi_{i} \in MSp_{8i-3}$, $i= 1,  2, \dots, $ are multiplicatively 
indecomposable and close
under the action of operations $S_{\omega }$ from the Landweber-Novikov algebra
${\mathbb A}^{MSp}$ \cite{N3}, \cite{La}. They will play the key role in our 
investigations.

The fundamental works of S.~O.~Kochman \cite{KI}, \cite{KII}, \cite{KIII}
summarize the application of 
the Adams spectral sequence to the study of symplectic cobordism. 
In his works Kochman computed $E_2$ and $E_3$-terms, calculated the image
of $MSp_*$ in the unoriented cobordism ring $MO_*$, found an element of order four
in degree 111  and also computed the first 100 stems.

The other methods of Algebraic Topology were also applied to the study of 
symplectic cobordisms. After the work of D.~Quillen \cite{Qui} the formal 
group techniques entered into the Algebraic Topology. Later V.~M.~Buchstaber 
constructed the theory of multi-valued formal groups and applied it to the study 
of symplectic cobordism \cite{Bu1}, \cite{Bu2}. This direction was developed 
further in the works of R.~Nadiradze and M.~Bakuradze \cite{BN}, \cite{Bak1},
\cite{Bak2}, \cite{Bak3}, \cite{Bak4}.

Having in mind the very complicated work in the classical Adams spectral sequence
the second author turned to the Adams-Novikov spectral sequence \cite{V1}, 
\cite{V2}. The arguments are 
evident: this spectral sequence is based on the complex cobordism, which is
``closer" to the symplectic cobordism than the ordinary homology theory 
used in the classical Adams spectral sequence.
This is not justified completely: still the the Adams-Novikov spectral sequence
for the spectrum $MSp$ is very complicated. May be this depicts the
nature of symplectic cobordism itself.
Anyway, the Adams-Novikov spectral sequence is the principal tool of 
calculations of the present work.
In \cite{V1} the second author studied the algebraic Novikov spectral sequence
that converges to the initial term of the Adams-Novikov spectral sequence 
and in \cite{V2} calculated the symplectic cobordism ring up to dimension
31 where he found a counter-example to Ray's conjecture that ``$KO$-theory
decides symplectic cobordism" \cite{R2}.

The techniques of cobordism with singularities developed by D.~Sullivan 
\cite{Su} and N.~A.~Baas \cite{Ba} was applied by the second author to the 
study of the symplectic cobordism \cite{V3}.
The main result is that if we take as a sequence of permitted 
``singularities types" the following subsequence of Ray's elements
\begin{equation*}
\Sigma= (\theta_1, \Phi_1, \dots, \Phi_{2^i}, \dots ),
\label{eq:ray_sin}
\end{equation*}
then the corresponding cobordism theory will be without torsion: that is
similar to the complex cobordism. This result shows that the elements of
the  sequence
$(\Sigma)$ serve as building blocks for construction of the 
whole torsion of symplectic cobordism. 

The classical cobordism graded rings  consist  of
finitely generated  abelian  groups  in  each  dimension.  The
complex cobordism ring have no elements of finite
order  and in the rings of the unoriented,  oriented, special
unitary and $Spin$ cobordism all the elements of finite order
have order two \cite{St2}. So the natural question about the existence 
of elements of order four in symplectic cobordism arises.
It is known \cite{V2} that in small dimensions 
the ideal of the elements  of  finite
order $\operatorname{Tors} MSp_*$ contains
only elements of order two. 

The main purpose of this work is to present the structure of  $MSp_{*}$ 
in dimensions up to 51 (see Table~18 at the end of the work) and
a construction  of an infinite
series of elements $\Gamma_i,  i=1, 3,4, ...$, of order four
in the symplectic cobordism ring, where 
$\operatorname{dim} \, \Gamma_i=8i+95$. The key
element of the series is $\Gamma_1$ in dimension 103.
 So, we are proving (in Chapter 4)  the following fact.

\noindent {\bf Main Theorem } {\it (i) There exists an indecomposable element
$\Omega_1 \in MSp_{49}$ of order two in the symplectic cobordism ring, such that
the product  $\theta_1 \Phi_{6+i} \Omega_1 \neq 0$.

(ii) Let $\Gamma_i \in <\Phi_{6+i},2,\Omega_1>$, for $i= 1, 3, 4, ... $.
Then the elements $\Gamma_i$ have order four and
$2\Gamma_i=\theta_1 \Phi_{6+i} \Omega_1 \neq 0$.
}

Relations between the Ray's elements and the free generators in low dimensions 
are also given in Table~18. These relations were studied from the point 
of view of characteristic classes by M.~Bakuradze, M.~Jibladze and the 
second author in \cite{BJV}.

The present work was finished in the middle of 90-ties and the main result
was announced in \cite{VA}. 
About the same time there appeared the works of B.~I.~Botvinnik and 
S.~O.~Kochman \cite{BK1}, \cite{BK2} asserting the existence of higher 
torsion in $MSp_*$.
There is no intersection between these papers and the present work.
Several 
years were spent for verification of our calculations. At that time the 
interest to symplectic cobordisms largely diminished, even the term
``symplectic cobordism" is using now in a different sense: as cobordism
of  manifolds with symplectic structure, i.e. closed non-degenerate 
differential 2-form \cite{Gi}. 
On the other hand there exists an opinion that our claim of the series of
order four is not justified by calculations. So finally we decided to submit this
text of our work. Otherwise a lot of labor would be lost.
  
The work is organized as follows. In the first Chapter we prove the 
necessary facts about the action of Landweber-Novikov operations on 
the Ray's elements. This action is the essential tool in our calculations.
Results are places in Table~9 at section Tables of the work. 

In Chapter~2 we study the so-called modified algebraic spectral sequence
(MASS) which converges to the initial term of the Adams-Novikov spectral sequence. 
First of all we precise the projections of Ray's elements in terms of the
generators of MASS. Together with the results of the previous chapter this 
gives the possibility to fix the action of Landweber-Novikov operations on 
the generators of MASS. This action is presented in Table~10. Then we study matrix Massey 
products in MASS and describe the following cells of the MASS: 
$E_2^{0,1,t}$, $E_2^{2,0,t}$ and $E_2^{1,1,t}$ for $t<108$, results are 
given in 
Tables 11, 12 and 13. In the last section
of Chapter~2 we prove a theorem that in these dimensions there is no higher 
differentials and $E_2=E_\infty$; relations for $E_\infty$ are presented in
Table~14.

 In Chapter~3 we start with algebraic 
description of the initial term of the Adams-Novikov spectral sequence for 
$MSp$ for the topological dimensions up to 56. The graphical description of 
this term takes a lot of volume in Table~15. In Table~16 multiplicative
relations are given. The second part of this Chapter is devoted to the
calculations of the differential $d_3$. The main tool is the action of 
Landweber-Novikov operations. Results are presented in Table~17. After that we 
reconstruct the term $E_4$ and see that it coincides with $E_\infty$ what
gives the possibility to determine symplectic cobordism ring in dimensions
up to 51, Table~18 depicts the result. 

In Chapter~4 we prove the Main Theorem based on the given calculations.

The authors are very grateful  to Malkhaz Bakuradze for his valuable remarks.

\chapter[Landweber-Novikov operations]{Action of Landweber-Novikov
operations on the Ray's elements}

As we had already mentioned in the Introduction N.~Ray constructed in \cite{R1} 
a series of elements in the symplectic
cobordism ring $\theta_{1} \in MSp_{1}$,
$\Phi_{i} \in MSp_{8i-3}$, $i= 1,  2, \dots, $  which are multiplicatively 
indecomposable and close
under the action of operations $S_{\omega }$ from the
Landweber-Novikov algebra
${\mathbb A}^{MSp}$. For our purposes we need exact values of these operations
on the elements $\Phi _{i}$. Stanly Kochman proved  in \cite{K}
a formula giving a possibility to calculate the action of an arbitrary
Landweber-Novikov operation $S_{\omega }$ on any Ray's element $\Phi _{i}$.
Let us remind Kochman's construction. Denote by
$h :\pi _{*}(MSp) \rightarrow MSp_{*}(MSp)$ the generalized Hurewicz
homomorphism, $\{1,b_{0},\ldots ,b_{n},\ldots \}$ is the canonical
$\pi _{*}(MSp)$-basis for $MSp_{*}(MSp)$; $\deg  b_{n}=4n$, then the
equality $h(x)= \sum^{}_{E} x_{E}b_{E}$ is equivalent to the following
assertion: the action of the Landweber-Novikov operation $S_{E}$ on
$x$ is equal to $S_{E}(x)=x_{E}$. Here we have $ E =(e_{1},\ldots ,e_{n})$,
$e_i\in \mathbb{N}$, $b_{E}=b^{e_{1}}_{1}\ldots b^{e_{n}}_{n}$ and
$x,x_{E} \in \pi _{*}(MSp)$.
The following formula \cite{K} is valid for $n=2m$:

\begin{multline}
h(\Phi_{m})=\sum^{m}_{i=1}[b_{2m-2i}+\sum^{m-i-1}_
{h=0}{b_{2h}\chi(B)^{2h+2i}_{2m-2h-2i}}]\Phi_{i} \\
+ \sum^{m-1}_{k=0}b_{2k}\chi (B)^{2k+1}_{2m-2k-1} \Phi _{0} .
\label{fkoch}\end{multline}

Here $B=1+b_{1}+\ldots+b_{t}+\ldots$; $ B^{k}_{n-k}$ denotes the component
of $B^{k}$ of degree $4n-4k$;
$\chi $ is the conjugation in the Hopf algebra $MSp_{*}(MSp)$:
\[
\chi (B)^{t}_{n-t} =\sum^{}_{r\ge 0} \sum^{}_{n>q_{r}>\ldots
>q_{1}>t}{(-1)^{r+1}B^{q_{r}}_{n-q_{r}}B^{q_{r-1}}_{q_{r}-q_{r-1}}\ldots
B^{q_{1}}_{q_{2}-q_{1}}B^{t}_{q_{1}-t}}
\]
Let us calculate the value of some operations $S_{\omega}$ on the
elements $\Phi _{n}$, using the formula  (\ref{fkoch}).
\begin{lemma} a) Let $mk$ be even and $mk<2n-1$. Then the coefficient
$\alpha ^{n}_{m;k}$ before $b^{m}_{k}$ in $\chi (B)^{2n-km}_{km}$
is equal to the expression:
\begin{multline*}
\alpha ^{n}_{m;k} =\sum^{}_{r\ge 0} \sum^{}_{m>i_{r}>\ldots
>i_{1}>0} (-1)^{r+1} C^{m-i_{r}}_{2n-(m-i_{r})k}
C^{i_{r}-i_{r-1}}_{2n-(m-i_{r-1})k} \\
\ldots C^{i_{2}-i_{1}}_{2n-(m-i_{1})k} C^{i_{1}}_{2n-mk}.
\end{multline*}
(Here $C^{m}_{n}$ denote as usual the binomial coefficient, if
$m>n$, then $C^{m}_{n}=0$).

$b)$ If $mk=2n-1$, then the coefficient $\gamma ^{n}_{m;k}$ before
$b^{m}_{k}$ in $\chi (B)_{2k-1}$ is given by the formula:

\[
\gamma ^{n}_{m;k} =\sum^{}_{r\ge 0} \sum^{}_{m>i_{r}>\ldots
>i_{2}>1}(-1)^{r+1} C^{m-i_{r}}_{ki_{r}+1}
C^{i_{r}-i_{r-1}}_{ki_{r-1}+1}\ldots
C^{i_{3}-i_{2}}_{ki_{2}+1} C^{i_{2}-1}_{k+1}.
\]
\end{lemma}
\begin{proof}  Consider the case a). It is evident that the monomial 
$b^{m}_{k}$  appears in the decomposition of
 $\chi (B)^{2n-km}_{km}$ only in the
following expression:
\[
\sum^{}_{r\ge 0} \sum^{}_{m>i_{r}>\ldots
>i_{1}>0} (-1)^{r+1} B^{2n-k(m-i_{r})}_{k(m-i_{r})}
B^{2n-k(m-i_{r-1})}_{k(i_{r}-i_{r-1})}
\ldots
 B^{2n-k(m-i_{1})}_{k(i_{2}-i_{1})} B^{2n-km}_{ki_{1}}.
\]
Analogous expression appears in the case b). \end{proof}

\begin{lemma} The result of the action of the operation $ S_{k,\ldots ,k}$
($k$ is taken $m$ times) on the element $\Phi _{n}$ is equal to:

$a) \alpha ^{n}_{m;k} \Phi _{n-mk/2}$, if $k$ is even, $m$ is even and
$mk<2n-1$;

$b) (\alpha ^{n}_{m;k}+ \alpha ^{n}_{m-1;k})\Phi _{n-mk/2}$, if
$k$ is even and $mk<2n-1$;

$c) \gamma ^{n}_{m;k} \theta _{1}$, if $mk=2n-1.$
\end{lemma}
Proof is evident.

\begin{remarka} \label{C2n-1}It follows from the equality
$C^{2n-1}_{2m} = \frac{2m}{2(m-n)-1} C^{2n-1}_{2m-1}$  that
$$C^{2n-1}_{2m}\equiv  0 \mod{2}.$$
Consider an arbitrary summand in the
decomposition $\alpha ^{n}_{m;k}$, when $m=2s-1, k=2q$ :
\[
(-1)^{r+1} C^{m-i_{r}}_{2n-k(m-i_{r})}
C^{i_{r}-i_{r-1}}_{2n-k(m-i_{r-1})}\ldots
C^{i_{2}-i_{1}}_{2n-k(m-i_{1})} C^{i_{1}}_{2n-km} .
\]
For this summand to be non-equal to zero $\operatorname{mod}$ 2 
the fulfillment of the following conditions is necessary:
\[
i_{1}\equiv 0 \mod{2},
\ \ i_{2} \equiv 0 \mod{2}, \ \ldots \ , \ \
m \equiv 0 \mod{2}.
\]
But it is not true, so $\alpha ^{n}_{2s-1;2q}\equiv 0 \mod{2}$.
\end{remarka}

\begin{conse} If $n>k$ then we have:
\[
S_{k,k}\Phi_{n} \equiv (n-k)\Phi_{n-k} \mod{2}.
\]
\end{conse}
\begin{proof} For any $k$ we have
$\alpha^{n}_{2;k}=-C^{2}_{2(n-k)}+C^{1}_{2n-k} C^{1}_{2(n-k)} \\
\equiv (n-k) \mod{2}$.  If $k=2s$, then
$\alpha ^{n}_{1;k} =-C^{1}_{2(n-s)} \equiv 0 \mod{2}$. \end{proof}

\begin{conse} The is the formula :
\[
S_{k,k,k}\Phi_{n} {\equiv } \left\{
\begin{array}{rl}
(n-k)\Phi_{n-3s} \mod{2}, & \text{if } k=2s, 3s<n, \\
(n-k)\theta_{1} \mod{2},  & \text{if } 2n+1=3k ,   \\
0,                 & \text{in the other cases}.
\end{array} \right.
\]
\end{conse}
\begin{proof} Let $k=2s$, $3s<n$ then result of the action of the operation
$S_{k,k,k}$ on the element $\Phi _{n}$ is equal to
$(\alpha ^{n}_{3;k}+\alpha ^{n}_{2;k})\Phi _{n-3s}$ . From the
Corollary~1.1 we have
$\alpha ^{n}_{2;k} \equiv (n-k)$ $\operatorname{mod}~2$, but
 $\alpha ^{n}_{3,2s}\equiv  0$ $\operatorname{mod}~2$ according to the
Remark~\ref{C2n-1}. Let $3k=2n+1$, then
$\gamma ^{3}_{k} = C^{2}_{2(n-k)}- C^{1}_{2n-k}C^{1}_{2(n-k)}
\equiv (n-k)$ $\operatorname{mod}~2$. \end{proof}

\begin{defin} Let us define the following function of the  integer
arguments $\mu(n;k)$ by the formula:
\[
\mu(n;k) = \left\{
\begin{array}{rl}
1,  & \mbox{if } n \mbox{ is odd}, \ k   \mbox{ is even},\\
0,  & \mbox{in the other cases}.
\end{array} \right.
\]
\end{defin}

\begin{conse} \label{2k<n} Let $2k<n$, then we have the formula:
\[
S_{k,k,k,k}\Phi_{n}\equiv (C^{4}_{2(n-2k)}+\mu(n,k))\Phi_{n-2k} \mod{2}.
\]
\end{conse}
\begin{proof} According to Remark~\ref{C2n-1}
$\alpha ^{n}_{3;k} \equiv  0 \mod{2}$, if $k$ is even. In the decomposition
of the number $\alpha ^{n}_{4;k}$ all the summands except $C^{4}_{2(n-2k)}$
and $C^{2}_{2(n-k)}C^{2}_{2(n-2k)}=(n-k)(n-2k)$ are even. We have
 $(n-k)(n-2k) \equiv 1 \mod{2}$ only if $n$ is odd and $k$ is even.
\end{proof}

\begin{conse} a) If $k=2s$ and $5k<2n-1$, then we have the formula:

\[
S_{k,k,k,k,k}\Phi_{n} \equiv (C^{4}_{2(n-2k)}+ n)\Phi_{n-5s} \mod{2}.
\]
b) If $5k=2n-1,$ then we have the identity:
\[
S_{k,k,k,k,k}\Phi_{n} \equiv  C^{4}_{k+1}\theta_{1} \mod{2}.
\]
\end{conse}
\begin{proof} a) According to Remark~\ref{C2n-1}
$\alpha ^{n}_{5,2s}\equiv  0 \mod{2}$.
We have,
$S_{k,k,k,k,k}\Phi _{n} \equiv  \alpha ^{n}_{4;k} \Phi _{n-5s} \mod{2}$.
According to Corollary~\ref{2k<n} $\alpha^{n}_{4,k}\equiv (C^{4}_{2(n-2k)}
+ (n-k)(n-2k)) \mod{2}$. However for $k=2s$ we have
$(n-2s)(n-4s) \equiv n \mod{2}$.

b) In the decomposition of the number $\gamma^{n}_{5;k}$ the majority of
summands contain factors $C^{1}_{k+1}$ or $C^{1}_{3k+1}$, $C^{3}_{k+1}$,
and because $k$ is even, these factors are even.
So we have $\gamma^{n}_{5;k}$ $\equiv$ $(C^{4}_{k+1}$
$- C^{2}_{3k+1}C^{2}_{k+1})$ $\mod{2}$.
Let $k=2s-1$, then
$C^{2}_{3k+1}C^{2}_{k+1}$ $ = $ $C^{2}_{2(3s-1)}C^{2}_{2s}$
$\equiv$ $(3s-1)s $
$\equiv  0 \mod{2}$.
\end{proof}

\begin{conse}\label{3k<n} Let $3k < n$ then we have the formula:
\[
S_{k,k,k,k,k,k} \Phi_{n} \equiv (\frac{(n+1)}{2} \mu(n;k) +
\frac{(n-2)}{2} \mu(k;n) )\Phi_{n-3k} \mod{2}.
\]
\end{conse}
\begin{proof} If $k$ is even, then according to Remark~\ref{C2n-1}
we have $\alpha ^{n}_{5;k} \equiv 0 \mod{2}$, and there is 
the following comparison
$\operatorname{mod}~2$:
\begin{multline*}
\alpha ^{n}_{6;k} \equiv - C^{6}_{2(n-3k)}+
C^{4}_{2(n-2k)}C^{2}_{2(n-3k)}+C^{2}_{2(n-k)}C^{4}_{2(n-3k)}\\
-C^{2}_{2(n-k)}C^{2}_{2(n-2k)}C^{2}_{2(n-3k)}
\equiv -\frac{(n-3k)(n-3k-1)(n-3k-2)}{ 2} \\
+ \frac{(n-2k)(n-2k-1)(n-3k)}{ 2} + \frac{(n-3k)(n-3k-1)(n-k)}{ 2} \\
-(n-k)(n-2k)(n-3k) \equiv \frac{n-3k}{ 2}(-(n-3k-1)(n-3k-2)\\
+(n-2k)(n-2k-1) +(n-3k-1)(n-k))-\mu(n;k) \\
\equiv \frac{n-3k}{ 2} ((n-3k-1)(k+1)2+(n-2k-1)(n-2))-\mu (n;k) \\
\equiv  1/2(n-3k)(n-2k-1)(n-2k) - \mu(n;k) \mod{2}.
\end{multline*}
There is also such a comparison
$\operatorname{mod}~2$:
\begin{multline*}
(1/2)(n-3k)(n-2k)(n-2k-1) \\
\equiv \left\{
\begin{array}{rl}
(n-2k)/2, & \mbox{if $n$ is even and $k$ is odd,} \\
(n-2k-1)/2, & \mbox{if $k$ is even and $n$ is odd,} \\
0, & \mbox{in the other cases.}
\end{array} \right.
\end{multline*}
Hence we can write
$\alpha^{n}_{6;k} \equiv  \mu(n;k)((n-2k-1)/2)
+\mu(k;n)((n-2k)/2)-\mu(n;k) \equiv  \mu(n;k)((n+1)/2)-\mu(k;n)((n-2)/2)
\mod{2}$.
In all our calculations we supposed that
 $3k\le n-3$. Let now $3k=n-2$, then we have
$\alpha^{n}_{6;k} \equiv  2(k+2)-3(k+1)(k+2) \equiv  0 \mod{2}$.
It follows from the expression $3k=n-2$ that $k \equiv  n$
$\operatorname{mod}$ 2. Let $3k=n-1$, then:
$\alpha ^{n}_{6;k} \equiv  -C^{4}_{2(k+1)}+(k+1)$ $\operatorname{mod}$ 2.
Consider separately various cases. Let $k=1$, then:
$n=4, \alpha ^{4}_{6;1} \equiv  1$ $\operatorname{mod}$ 2. This coincides
with the value given by the formula. Let now $k>1$. We have:
$\alpha^{k}_{6;k} \equiv - (k+1)k/2 +k+1 \equiv
(k+1)(k-2) \mod{2}$. If $k$ is even (then $n$ is odd), this coincides
with $(k-2)/2 \equiv  (n-3)/2 \equiv (n+1)/2 \mod{2}$. If $k$ is odd
($n$ is even), this is comparable with
$(k+1)/2 \equiv  3(k+1)/2 \equiv (n-2)/2 \mod{2}$.\end{proof}

\begin{conse} Let $k=2s$, $7k < 2n-1$, then there is the formula:
$$
S_{k,k,k,k,k,k,k}\Phi_{n} \equiv n[(n+1)/2]\Phi_{n-7s} \mod{2},
$$
($[x]$ denotes the integer part of the number $x$).
\end{conse}
\begin{proof} We have $\alpha ^{n}_{7;2s}\equiv  0$ $\operatorname{mod}$~2,
and the formula for $\alpha ^{n}_{6;2s}$ from Corollary~\ref{3k<n} becomes
(because of relations $\mu (n;2s)\equiv n \mod{2},
\mu (2s;n) \equiv  0 \mod{2}$), the following comparison
$\alpha ^{n}_{6;2s} \equiv  n[(n+1)/2] \mod{2}$. \end{proof}

\begin{conse} Let $4k <n$, then we have the formula:
\[
S_{k,k,k,k,k,k,k,k}\Phi_{n}\equiv
(\mu(n;k)+(n/2)\mu(n+1;k)+C^{8}_{2(n-4k)}) \Phi _{n-4k}\mod{2}.
\]
\end{conse}
\begin{proof} Let us calculate the value of $\alpha^{n}_{8;k}$.
Having in mind Remark~\ref{C2n-1} we have the comparison $\mod~2$
(if $4k \le  n-4$):
\begin{multline*}
\alpha ^{n}_{8;k}\equiv -C^{8}_{2(n-4k)}+
C^{6}_{2(n-3k)}C^{2}_{2(n-4k)} + C^{4}_{2(n-2k)}C^{4}_{2(n-4k)} \\
+C^{2}_{2(n-k)}C^{6}_{2(n-4k)}-
C^{4}_{2(n-2k)}C^{2}_{2(n-3k)}C^{2}_{2(n-4k)} \\
- C^{2}_{2(n-k)}C^{4}_{2(n-3k)}C^{2}_{2(n-4k)} -
C^{2}_{2(n-k)}C^{2}_{2(n-2k)}C^{4}_{2(n-4k)} \\
+ C^{2}_{2(n-k)}C^{2}_{2(n-2k)}C^{2}_{2(n-3k)}C^{2}_{2(n-4k)}
\equiv C^{8}_{2(n-4k)} \\
+ \mu(n,k) + (n-2k)(n-2k-1)(n-2k+1)(n-4k)/4.
\end{multline*}
If $n = 2s+1$, then
$\frac{(n-2k)(n-2k-1)(n-2k+1)(n-4k)}{ 4} \equiv  s(s+1) \equiv  0 \mod{2}$.
If $n=2s$, then
\begin{multline*}
\frac{(n-2k)(n-2k-1)(n-2k+1)(n-4k)}{ 4} \equiv (n/2-k)(n/2-2k) \\
\equiv \left\{
\begin{array}{rl}
n/2, & \mbox{if } k \mbox{ is even,} \\
0,  & \mbox{if } k \mbox{ is odd.}
\end{array} \right.
\end{multline*}
So,
${(n-2k)(n-2k-1)(n-2k+1)(n-4k)/4} \equiv  \mu(n+1;k)(n/2).$
Let $n-4k=3$, we have
\begin{multline*}
\alpha ^{4k+3}_{8;k} \equiv  C^{6}_{2(k+3)} + C^{4}_{2(2k+3)} +
C^{2}_{2(3k+3)} \\
- C^{4}_{2(2k+3)}C^{2}_{2(k+3)} -C^{4}_{2(k+3)}C^{2}_{2(3k+3)}
-C^{2}_{2(2k+3)}C^{2}_{2(3k+3)}\\
+ C^{2}_{2(2k+3)}C^{2}_{2(3k+3)}C^{2}_{2(k+3)} \equiv  k+1
\equiv  \mu(4k+3;k)\mod{2}.
\end{multline*}
Let $n-4k=2$, then
\begin{multline*}
\alpha^{4k+2}_{8;k} \equiv  C^{4}_{2(2k+2)} -
C^{2}_{2(2k+2)}C^{2}_{2(3k+2)} \\
\equiv  (k+1)(2k+1) \equiv  \mu (n+1;k)(n/2)\mod{2}.
\end{multline*}
Finally, let $n-4k=1$, then we get
\[
\alpha ^{4k+1}_{8;k} \equiv  k+1 \equiv  \mu (n;k)\mod{2} .
\]
\end{proof}

\begin{conse} If $k=2s$ and $9k<2n$, then
\[
S_{k,k,k,k,k,k,k,k,k} \Phi_{n} \equiv
(n+(n+1)[n/2]+C^{8}_{2(n-4k)})\Phi _{n-9s}\mod{2}.
\]
\end{conse}
The proof follows directly form the previous formula because of
the condition $k=2s$.

\begin{conse} Let $5k<n$, then we have:
\[
S_{k,k,k,k,k,k,k,k,k,k} \Phi_{n} \equiv
(\mu(n;k) + (n-k)C^{8}_{2(n-4k)})\Phi _{n-5k}\mod{2}.
\]
\end{conse}
\begin{proof} If $5k<n-5$, then we have
\begin{multline*}
\alpha ^{n}_{10;k}\equiv -C^{10}_{2(n-5k)} +
C^{8}_{2(n-4k)}C^{2}_{2(n-5k)} +C^{6}_{2(n-3k)}C^{4}_{2(n-5k)} \\
+C^{4}_{2(n-2k)}C^{6}_{2(n-5k)} + C^{2}_{2(n-k)}C^{8}_{2(n-5k)} -
C^{6}_{2(n-3k)}C^{2}_{2(n-4k)}C^{2}_{2(n-5k)} \\
- C^{4}_{2(n-2k)}C^{4}_{2(n-4k)}C^{2}_{2(n-5k)} -
C^{4}_{2(n-2k)}C^{2}_{2(n-3k)}C^{2}_{2(n-4k)}C^{2}_{2(n-5k)} \\
- C^{2}_{2(n-k)}C^{6}_{2(n-4k)}C^{2}_{2(n-5k)} +
C^{2}_{2(n-k)}C^{4}_{2(n-3k)}C^{2}_{2(n-4k)}C^{2}_{2(n-5k)} \\
- C^{4}_{2(n-2k)}C^{2}_{2(n-3k)}C^{4}_{2(n-5k)} +
C^{2}_{2(n-k)}C^{2}_{2(n-2k)}C^{4}_{2(n-4k)}C^{2}_{2(n-5k)}  \\
- C^{2}_{2(n-k)}C^{4}_{2(n-3k)}C^{4}_{2(n-5k)} +
C^{2}_{2(n-k)}C^{2}_{2(n-2k)}C^{2}_{2(n-3k)}C^{4}_{2(n-5k)} \\
- C^{2}_{2(n-k)}C^{2}_{2(n-2k)}C^{2}_{2(n-3k)}
C^{2}_{2(n-4k)}C^{2}_{2(n-5k)} \\
- C^{2}_{2(n-k)}C^{2}_{2(n-2k)}C^{6}_{2(n-5k)}
\equiv (n-k)C^{8}_{2(n-4k)} \\
-\mu(n;k) +(n-5k)(1/4)(n-2k)(n-2k+1)(n-2k-1)(n-4k) \mod{2}
\end{multline*}
We have:
\begin{multline*}
\frac{(n-5k)(n-2k)(n-2k+1)(n-2k-1)(n-4k)}{ 4} \\ \equiv \left\{
\begin{array}{rl}
({n/2} - k){n/2} \equiv  0, & \mbox{if $n$ is even and $k$ is odd,} \\
({(n+1)/2})({(n-1)/2}) \equiv  0, &
\mbox{if $k$ is even and $n$ is odd,} \\
0, & \mbox{in the rest cases.}
\end{array} \right.
\end{multline*}
Cases when $n-5k<5$ are considered by the direct substitution.
 \end{proof}

\begin{conse} Let $k=2s$ and $11s<n$, then we have the formula:
\[
S_{k,k,k,k,k,k,k,k,k,k,k} \Phi_{n} \equiv
(n + nC^{8}_{2(n-4k)})\Phi_{n-11s}\mod{2}.
\]
\end{conse}
\begin{proof} If $k=2s$, then $\alpha^{n}_{10;k}\equiv
( nC^{8}_{2(n-4k)} + n) \mod{2}$,
and $\alpha^{n}_{11,k} \equiv  0 \mod{2}$.
\end{proof}

\begin{conse} Let $6k<n$, then we have the formula:
\begin{multline*}
S_{k,k,k,k,k,k,k,k,k,k,k,k} \Phi_{n} \equiv
(\mu(n;k)(n^{2}+3)/4 + [n/2]([n/2]-k) \\
+ \frac{1}{ 2}(n-2k)C^{5}_{n-4k+1} )\Phi_{n-6s}\mod{2}.
\end{multline*}
\end{conse}
\begin{proof} Let us calculate the value $\alpha ^{n}_{12;k}$,
supposing that $6k\le n-6$. Cases with $6k>n-6$ are considered
by the direct substitution.
\begin{multline*}
\alpha ^{n}_{12;k}\equiv
-C^{12}_{2(n-6k)} +
C^{2}_{2(n-k)}C^{2}_{2(n-2k)}C^{2}_{2(n-3k)}
C^{2}_{2(n-4k)}C^{2}_{2(n-5k)}C^{2}_{2(n-6k)} \\
+C^{10}_{2(n-5k)}C^{2}_{2(n-6k)} +
C^{6}_{2(n-3k)}C^{2}_{2(n-4k)}C^{2}_{2(n-5k)}C^{2}_{2(n-6k)} \\
+C^{8}_{2(n-4k)}C^{4}_{2(n-6k)} +
C^{4}_{2(n-2k)}C^{4}_{2(n-4k)}C^{2}_{2(n-5k)}C^{2}_{2(n-6k)} \\
+ C^{6}_{2(n-3k)}C^{6}_{2(n-6k)}+
C^{2}_{2(n-k)}C^{6}_{2(n-4k)}C^{2}_{2(n-5k)}C^{2}_{2(n-6k)} \\
+C^{4}_{2(n-2k)}C^{8}_{2(n-6k)} +
C^{4}_{2(n-2k)}C^{2}_{2(n-3k)}C^{4}_{2(n-5k)}C^{2}_{2(n-6k)} \\
+C^{2}_{2(n-k)}C^{10}_{2(n-6k)} +
C^{2}_{2(n-k)}C^{4}_{2(n-3k)}C^{4}_{2(n-5k)}C^{2}_{2(n-6k)} \\
- C^{8}_{2(n-4k)}C^{2}_{2(n-5k)}C^{2}_{2(n-6k)} -
C^{6}_{2(n-3k)}C^{4}_{2(n-5k)}C^{2}_{2(n-6k)}  \\
- C^{4}_{2(n-2k)}C^{6}_{2(n-5k)}C^{2}_{2(n-6k)} -
C^{2}_{2(n-k)}C^{8}_{2(n-5k)}C^{2}_{2(n-6k)} \\
- C^{6}_{2(n-3k)}C^{2}_{2(n-4k)}C^{4}_{2(n-6k)} -
C^{4}_{2(n-2k)}C^{4}_{2(n-4k)}C^{4}_{2(n-6k)} \\
- C^{2}_{2(n-k)}C^{6}_{2(n-4k)}C^{4}_{2(n-6k)} -
C^{4}_{2(n-2k)}C^{2}_{2(n-3k)}C^{6}_{2(n-6k)} \\
- C^{2}_{2(n-k)}C^{6}_{2(n-6k)}C^{4}_{2(n-3k)} -
C^{2}_{2(n-k)}C^{2}_{2(n-2k)}C^{8}_{2(n-6k)} \\
+ C^{2}_{2(n-k)}C^{2}_{2(n-2k)}C^{6}_{2(n-5k)}C^{2}_{2(n-6k)} \\
+ C^{4}_{2(n-2k)}C^{2}_{2(n-3k)}C^{2}_{2(n-4k)}C^{4}_{2(n-6k)} \\
+ C^{2}_{2(n-k)}C^{4}_{2(n-3k)}C^{2}_{2(n-4k)}C^{4}_{2(n-6k)} \\
+ C^{2}_{2(n-k)}C^{2}_{2(n-2k)}C^{4}_{2(n-4k)}C^{4}_{2(n-6k)} \\
+ C^{2}_{2(n-k)} C^{2}_{2(n-2k)}C^{2}_{2(n-3k)}C^{6}_{2(n-6k)} \\
- C^{4}_{2(n-2k)}C^{2}_{2(n-3k)}C^{2}_{2(n-4k)}
C^{2}_{2(n-5k)}C^{2}_{2(n-6k)} \\
- C^{2}_{2(n-k)}C^{4}_{2(n-3k)}C^{2}_{2(n-4k)}
C^{2}_{2(n-5k)}C^{2}_{2(n-6k)} \\
- C^{2}_{2(n-k)}C^{2}_{2(n-2k)}C^{4}_{2(n-4k)}
C^{2}_{2(n-5k)}C^{2}_{2(n-6k)} \\
- C^{2}_{2(n-k)}C^{2}_{2(n-2k)}C^{2}_{2(n-3k)}
C^{4}_{2(n-5k)}C^{2}_{2(n-6k)} \\
- C^{2}_{2(n-k)}C^{2}_{2(n-2k)}C^{2}_{2(n-3k)}
C^{2}_{2(n-4k)}C^{4}_{2(n-6k)} \\
\equiv  \mu(n;k) + C^{4}_{2(n-2k)}C^{4}_{2(n-4k)}C^{4}_{2(n-6k)} \\
+ \frac{1}{ 4}(n-2k+1)(n-2k)(n-2k+1)(n-4k)(n-5k)(n-6k) \\
- \frac{1}{16}(n-4k)(n-4k-1)(n-4k+1)(n-4k-2)(n-4k-3)(n-6k) \mod{2}.
\end{multline*}
We have the following comparisons $\operatorname{mod}~2$:
\begin{multline*}
\frac{1}{4}(n-2k+1)(n-2k)(n-2k+1)(n-4k)(n-5k)(n-6k) \\
\equiv  (\frac{n+1}{ 2})(\frac{n-1}{2})\mu (n;k),
\end{multline*}
\begin{multline*}
\frac{1}{16}(n-4k)(n-4k-1)(n-4k+1)(n-4k-2)(n-4k-3) \\
\equiv  (\frac{n-2k}{2})C^{5}_{n-4k+1},
\end{multline*}
\[
C^{4}_{2(n-2k)}C^{4}_{2(n-4k)}C^{4}_{2(n-6k)} \equiv
[\frac{n}{2}]([\frac{n}{2}] -k).
\]
\end{proof}

\begin{conse} If $k=2s$ and $13s<n$, then:
\begin{multline*}
S_{k,k,k,k,k,k,k,k,k,k,k,k,k}\Phi_{n} \\
\equiv  (n[(n^{2}+3)/4] +(n/2)C^{5}_{n+1}+[n/2])\Phi _{n-13s}\mod{2}.
\end{multline*}
\end{conse}
\begin{proof} In the case $k=2s$ we have the following comparisons:
\[
\frac{n-2k}{2} \equiv \frac{n}{2}\mod{2},\mbox{ } C^{5}_{n-4k+1}
\equiv  C^{5}_{n+1}\mod{2},\mbox{ } \mu(n;k) \equiv n\mod{2}.
\]
\end{proof}

We put the obtained results in the Table~9 where only nonzero
values are depicted.

\begin{center}
\chapter{Modified algebraic spectral sequence}
\section{Preliminary information}
\end{center}
\medskip

For the calculation of the initial term  $E^{*,*}_{2}$ of the
Adams-Novikov spectral sequence converging to $MSp_{*}$ the second author
had built in the works \cite{V1, V2} the so called {\it modified
algebraic spectral sequence} (MASS) which converges to the initial term
$E^{*,*}_{2}$ of the Adams-Novikov spectral sequence. He had also proved
that its initial term $E^{*,*,*}_{1}$ in this case
has especially simple structure: it is a three-graded algebra
isomorphic to the polynomial algebra over the field ${\mathbb Z}/{2}$:
\[
E^{q,s,t}_{1}\cong{\mathbb Z}/{2}\bigl[h_{0},h_{1},\ldots ,h_{i},
\ldots ;u_{1},\ldots u_{j},\ldots ;c_{2},c_{4},\ldots ,c_{n},
\ldots \bigr].
\]
where $i = 0,1,2,\ldots $ ; $j = 1,2,\ldots $; $n = 2,4,5,
\ldots$, $n \neq  2^{k}-1$, $\deg  h_{0}= (2,0,0)$,
$\deg  h_{i}= (1,0,2(2^{i}-1))$, $\deg  u_{j}= (0,1,2(2^{j}-1))$,
$\deg  c_{n}= (0,0,4n).$

For the calculation of the symplectic cobordism ring $MSp_{*}$,
the second author in the papers \cite{V1, V2, V3} used the following
scheme. At the beginning the modified algebraic spectral sequence,
MASS, for the spectrum $MSp$ (or algebraic spectral sequence, ASS)
is calculated. The term $E^{*,*,*}_{\infty }$ (as well as the
corresponding term
of ASS) is associated to the initial term of the
Adams-Novikov spectral sequence $E^{*,*}_{2}$. This spectral sequence
in its turn converges to the ring  $MSp_{*}$. Using this scheme
the ring $MSp_{*}$ was calculated up to dimension 31 \cite{V2}.
In this chapter
we make the calculations of the term $E^{q,s,t}_{\infty }$ of MASS
for $t < 108$.

Let $d_{1}$ denote the first differential of the MASS and $\varphi _{n}$
(respectively $\theta _{1}$) denote the projection of the Ray's element
$\Phi _{n}$ (respectively $\Phi _{0}$) in the term $E^{0,1,*}_{1}$ of
MASS. Let us introduce the following notations. If $n=2m-1$ and
$m=2^{i_{1}-2}+\ldots + 2^{i_{q}-2}$ is the digital decomposition
of the number $m$ $( 2\le i_{1}<\ldots <i_{q},q\ge 3)$, then the
generator
$c_{n}$ we denote by $c_{i_{1},\ldots ,i_{q}}$, and the element
$\varphi _{m}$ we denote by $\varphi _{i_{1},\ldots ,i_{q}}$.
The element $c_{2^{i-1}}$ we denote by $c_{1,i}$.

The following facts were proved in the works \cite{V1, V2, V3, V4}.
The analogous facts for the classical Adams spectral sequence were proved 
by S.~Kochman \cite{KI}.

1) On the generators $c_{n}$ the differential $d_{1}$ acts by the
following way:

a) if $n$ is even and not a power of 2, then it is possible to
choose the element $c_{n}$ in such a way that it becomes an
infinite cycle in MASS;

b) \begin{equation}\label{eq:cij}
d_{1}(c_{i,j}) = u_{i}h_{j}+u_{j}h_{i};
\end{equation}

c) \begin{equation}\label{eq:ciq}
d_{1}(c_{i_{1},\ldots ,i_{q}}) = \sum^{}_{1\le s\le t\le q}
( u_{i_{t}}h_{i_{s}}+ u_{i_{s}}h_{i_{t}})c_{1,i_{1}}\cdots
\hat{c}_{1,i_{s}}\cdots \hat{c}_{1,i_{t}}\cdots c_{1,i_{q}};
\end{equation}

2) The generators $h_{0},u_{j}$ are infinite cycles and the following
equalities hold in the algebra $E^{*,*,*}_{1}$:
\[
u_{1}=\theta _{1} , u_{i} = \varphi _{2^{i-2}} , i=2,3,\ldots
\]
\begin{equation}\label{eq:phiiq}
\varphi _{i_{1},\ldots ,i_{q}} = u_{1}c_{i_{1},\ldots ,i_{q}} +
\sum^{q}_{i=1} u_{i_{t}}c_{1,i_{1}}\cdots \hat{c}_{1,i_{t}}\cdots
c_{1,i_{q}} + \sum^{}_{J} \varphi _{q}c_{J_{q}},
\end{equation}
where the elements $c_{J_{q}}\in  E^{0,0,*}_{2}$ are infinite cycles,
$\varphi _{q}$ are the projections of the elements $\Phi _{q}$ for
$q < m$.

3) For all $j = 1,2,\ldots $, there is the formula:
\[
d_{1}(h_{j}) = h_{0}u_{j} .
\]
Let us introduce the following notation according with the formula
(\ref{eq:phiiq}):
\[
\varphi _{i_{1},\ldots ,i_{q}} = \tilde{\varphi }_{i_{1},\ldots ,
i_{q}} + \sum^{}_{J_{q}} \varphi _{q}c_{J_{q}}.
\]
Then we can choose elements $\tilde{\varphi }_{n}$ or elements
$\varphi _{n}$ as the generators in $E^{0,1,t}_{r}$, $r > 1$.
The differentials $d_{r}, r > 1$, conserve the grading $t$, increase
the grading $s$ by 1 and increase the grading $q$ by $r$.

Initial term of the Adams-Novikov spectral sequence has the following
description:
\[
E^{*,*}_{2} \cong  {\rm Ext}_{A^{BP}}( BP^{*}(MSP), BP^{*} ).
\]
There exists a filtration $F^{q}E^{s,t}_{2}$, such that we have an
isomorphism
\[
F^{q}E^{s,t}_{2} / F^{q+1}E^{s,t}_{2} \cong  E^{q,s,t}_{\infty },
\]
connecting the subfactors of term $E^{s,t}_{2}$ of the Adams-Novikov
spectral sequence and the term $E^{q,s,t}_{\infty }$ of MASS.

\medskip
\section[Projections of Ray's elements]{Projections of Ray's
elements into the \\ MASS and action of
Landweber-Novikov \\  operations on the elements ${\bf c}_{{\bf i}}$}

\medskip

It is proved in the work of the second author [6] that the projection
of the Ray's element $\Phi_{3}$ can be chosen in the form:
\begin{equation} \label{eq:phi3}
\varphi_{3} = u_{1}c_{5} + u_{2}c_{4} + u_{3}c_{2} .
\end{equation}

\begin{lemma} It is possible to choose the element $c_{2m-2}$ in
such a way that for the decomposition of the element $\varphi_{m}$ from
(\ref{eq:phiiq}) the following condition holds:
\[
\sum^{}_{J_{q}} \varphi_{q}c_{J_{q}} = u_{2}c_{2m-2} + u_{3}(c_{2m-4}+\text{decomposable }) + \sum^{}_{J_{q^{\prime}}}\varphi_{q^{\prime}}c_{J_{q^{\prime}}}
\]
(where $q^{\prime}\ge 3$).
\begin{proof} From the condition $S_{2m-2}\varphi_{m} = u_{2}$
($S_{2m-4}\varphi_{m} = u_{3}$ respectively) it follows that in the
decomposition $\varphi_{m}$ present the following summands
$u_{2}c_{2m-2}$ ($u_{3}c_{2m-4}$ respectively). If we take
if necessary instead of $c_{2m-2}$ the element
$c^{\prime}_{2m-2} = c_{2m-2} + (\text{decomposable elements of power} \
2m-2)$,
we get the necessary equality.
\end{proof}
\end{lemma}

Let us precise the form of projections of some Ray's elements
$\Phi _{i}$ in MASS and also calculate the action of Landweber-Novikov
operations on the elements $c_{k}$. In all the calculations given below
in the subsections from 1 till 9 everything is made $\mod 2$ and hence
all the equalities are relations $\mod 2$.

1. Applying the operation $S_{1}$ to the decomposition (\ref{eq:phi3})
we get that
$S_{1}\varphi_{3}$ $=$ $0$ $=$ $u_{1}(S_{1}c_{5}$ $+$ $c_{4})$. Hence
$S_{1}c_{5}$ $=$ $c_{4}$.

Let us apply the operation $S_{2}$ to the decomposition (\ref{eq:phi3}).
We get that
$S_{2}\varphi_{3} = u_{3} = u_{2}(S_{2}c_{4}+ c_{2})+ u_{3}S_{2}c_{2}$.
Hence, $S_{2}c_{4}= c_{2}$, $S_{2}c_{2} = 1$.

Using the operation $S_{3}$ we get that
$S_{3}\varphi_{3} = 0 = u_{1}(S_{3}c_{5} + c_{2})$, or $S_{3}c_{5}= c_{2}$.

Applying the operation $S_{2,2}$ we have $S_{2,2}c_{4}= 0$.

And finally using $S_{5}$ we get an equality $S_{5}c_{5}= 1$.

2. The projection $\varphi _{5}$ of the element $\Phi _{5}$ has the form:
\[
\varphi_{5} = u_{1}c_{4}+ u_{2}c_{5}+ u_{3}c_{6}+
u_{4}c_{2}+ \beta\varphi_{3}c^{2}_{2}
\]
Let us apply the operation $S_{2,2}$ to this expression, we get
\[
S_{2,2}\varphi _{5} = \varphi _{3} = u_{1}S_{2,2}c_{9}+
u_{2}S_{2,2}c_{8}+ u_{2}S_{2}c_{6}+ \varphi _{3}+
\beta \varphi _{3}+ u_{2}\beta c^{2}_{2}.
\]
Hence, $S_{2,2}c_{9}= 0$, $S_{2,2}c_{8}= S_{2}c_{6}$, $\beta = 0$ .

Finally the form of $\varphi _{5}$ is the following:
\[
\varphi _{5}= u_{1}c_{9}+ u_{2}c_{8}+ u_{3}c_{6}+ u_{4}c_{2}.
\]
Using the obtained decomposition let us calculate the values of the
operations $S_{\omega }$ on $c_{9}$, $c_{8}$ and $c_{6}$.

3. The projection of the element $\Phi_{6}$ has the form:
\[
\varphi _{6} = u_{1}c_{11}+u_{2}c_{10} +u_{3}c_{8} +
u_{4}c_{4} +u_{3}(\beta_{1}c^{2}_{4} +\beta_{2}c^{4}_{2}) +
\beta_{3}\varphi_{3}c_{6} +\beta_{4}u_{4}c^{2}_{2}.
\]
Applying the operation $S_{4,4}$ we get:
$S_{4,4}\varphi_{6} = 0 = \beta_{1}u_{3}$.
Hence, $\beta_{1}=0$. Applying the operation $S_{6}$we get:
\[
S_{6}\varphi _{6} = \varphi _{3}= u_{1}S_{6}c_{11}+
u_{2}S_{6}c_{10}+ u_{3}c_{2}+ u_{2}c_{4}+ \beta _{3}\varphi _{3}+
\beta _{4}u_{2}c^{2}_{2} .
\]
Hence $S_{6}c_{11}=c_{5}$, $\beta _{3}= 0$,
$S_{6}c_{10}= \beta _{4}c^{2}_{2}$.

Let us apply the operation $S_{2,2}$, we get:
\[
S_{2,2}\varphi _{6} =0=u_{1}S_{2,2}c_{11} +u_{2}S_{2,2}c_{10} +
u_{2}c_{2}c_{4} +\varphi _{3}c_{2}+\beta _{4}u_{4} +u_{3}c^{2}_{2} +
u_{2}c_{6}.
\]
Hence, $S_{2,2}c_{11}= c_{2}c_{5}$, $\beta _{4}=0$, $S_{2,2}c_{10}= c_{6}$.

Let us apply the operation $S_{2,2,2,2}$, we have:
$S_{2,2,2,2}\varphi_{6} = u_{3} = \beta _{2}u_{3}$.
Hence, $\beta _{2}=1$ .

The final form of the projection $\varphi_{6}$ of the element $\Phi_{6}$
is the following :
\[
\varphi _{6} = u_{1}c_{11}+ u_{2}c_{10}+ u_{3}c_{8}+ u_{3}c^{4}_{2} + u_{4}c_{4} .
\]
Using the obtained decomposition let us calculate the results of the
action of the operation $S_{\omega }$ on the elements $c_{11}$, $c_{10}$.

4. The projection $\varphi_{7}$ of the element $\Phi_{7}$ has the form:
\begin{multline*}
\varphi_{7}= u_{1}c_{13} +u_{2}c_{4}c_{8} +u_{3}c_{2}c_{8} +
u_{4}c_{2}c_{4} +u_{2}c_{12}\\
+ u_{3}(c_{10}+\beta _{1}c^{2}_{5}+\beta _{2}c^{2}_{2}c_{6})+
\varphi _{3}(\beta _{3}c^{2}_{4} +\beta _{4}c^{4}_{2})+ \beta _{5}u_{4}c_{6}+\beta _{6}\varphi _{5}c^{2}_{2} .
\end{multline*}
Applying the operation $S_{5,5}$ we have:
$S_{5,5}\varphi _{7}=0=u_{3}(1+\beta _{1})$, hence, $\beta _{1}=1$.
Applying the operation $S_{2,2}$ we get:
\begin{multline*}
S_{2,2}\varphi_{7}=\varphi_{5}=u_{1}S_{2,2}c_{13}+ u_{2}(c_{8}+
c^{2}_{2}c_{4}+ S_{2,2}c_{12}+ c^{2}_{4}(1+ \beta _{3})\\
+ c^{4}_{2}(1+\beta _{2}+\beta _{4}))+ u_{3}(c_{2}c_{4}+
\beta _{2}c_{6}+ c^{3}_{2} )+ \varphi _{3}(c_{4} \\
+ c^{2}_{2}(1+\beta _{3}+\beta _{5}+\beta _{6})) + u_{4}c_{2}+
\beta _{6}\varphi_{5}.
\end{multline*}
It follows from this equality that $\beta _{6}=0$,
$\beta _{3}=\beta _{5}$,
$S_{2,2}c_{13}= c_{4}c_{5}+ c^{2}_{2}c_{5}+ c_{9}$ ,
$\beta_{2}=1$, $S_{2,2}c_{12}= \beta _{3}c^{2}_{4}+ \beta _{4}c^{4}_{2}$.

Applying the operation $S_{6}$ we have:
\[
S_{6}\varphi _{7}= u_{4} = u_{1}S_{6}c_{13}+ u_{2}S_{6}c_{12}+
u_{2}\beta _{5}c_{6}+ \beta _{5}u_{4}.
\]
Hence, $\beta _{5}=1$, $S_{6}c_{13}=0$, $S_{6}c_{12}=c_{6}$,
$\beta _{3}=1$.

Let us apply the operation $S_{2,2,2,2}$, then we have:
\[
S_{2,2,2,2}\varphi_{7} = 0 = u_{1}S_{2,2,2,2}c_{13} +
u_{2}(S_{2,2,2,2}c_{12} + c^{2}_{2} ) +\varphi_{3}(1+\beta_{4} ).
\]
Hence, $\beta _{4}=1$ , $S_{2,2,2,2}c_{13}=0$,
$S_{2,2,2,2}c_{12}=c^{2}_{2}$ .

The final form of the projection of $\Phi_{7}$ will be the following:
\begin{multline*}
\varphi_{7}= u_{1}c_{13}+ u_{2}c_{4}c_{8}+ u_{3}c_{2}c_{8}+
u_{4}c_{2}c_{4}+ u_{2}c_{12} \\
+ u_{3}(c_{10} +c^{2}_{5} + c^{2}_{2}c_{6}) + \varphi _{3}(c^{2}_{4}+
c^{4}_{2} )+ u_{4}c_{6}.
\end{multline*}
Using this decomposition let us calculate the result of the action
of the operation $S_{\omega }$ on the elements $c_{13}$ and $c_{12}$ .

5. The projection $\varphi_{9}$ of the element $\Phi_{9}$ has the form:
\begin{multline*}
\varphi_{9} = u_{1}c_{17}+ u_{2}c_{16} + u_{3}c_{14} +
u_{3}(\beta _{12}c^{2}_{2}c_{10}+ \beta _{13}c^{2}_{2}c^{2}_{5} +
\beta _{14}c^{4}_{2}c_{6} \\
+ \beta _{15}c^{2}_{4}c_{6}) + \varphi _{3}(\beta _{1}c_{12} +
\beta _{2}c^{2}_{6} + \beta _{3}c^{6}_{2} +
\beta _{4}c^{2}_{2}c^{2}_{4} ) + u_{4}( \beta _{5}c^{2}_{5}+
\beta _{6}c^{2}_{2}c_{6} \\
+ \beta _{7}c_{10})+ \varphi _{5}(\beta _{8}c^{2}_{4}+
\beta _{9}c^{4}_{2} )+ \beta _{10}\varphi _{6}c_{6}+
\beta _{11}\varphi _{7}c^{2}_{2}+ u_{5}c_{2}.
\end{multline*}
Choose the element $c_{14}$ in such a way that $\beta _{12}$,
$\beta _{13}$, $\beta _{14}$, $\beta _{15}= 0$.

Applying the operation $S_{10}$we get:
\begin{multline*}
S_{10}\varphi _{9}\\
=u_{4}=u_{1}S_{10}c_{17}+u_{2}(S_{10}c_{16}+\beta _{10}c_{6})+
\varphi _{3}c_{2}+u_{3}(S_{2}c_{14}+\beta _{11}c^{2}_{2} )+
\beta _{7}u_{4}.
\end{multline*}
It follows from here that
$\beta _{7}=1$, $S_{10}c_{17}=c_{2}c_{5}$, $S_{10}c_{16}=c_{2}c_{4}+
\beta _{10}c_{6}$,
$S_{10}c_{14}=c^{2}_{2}(1+\beta _{11})$.

Let us apply the operation $S_{12}$, we have:
\[
S_{12}\varphi _{9}= \varphi _{3}= u_{1}S_{12}c_{17}+
u_{2}(S_{12}c_{16}+ \beta _{11}c^{2}_{2} )+ u_{3}c_{2}+
\beta _{1}\varphi _{3}S_{12}c_{12} .
\]
It follows that $\beta _{1}=0 $, $S_{12}c_{17}=c_{5}$,
$S_{12}c_{16}= c_{4} + \beta _{11}c^{2}_{2}$.

Let us apply the operation $S_{6,6}$, then we have:
$S_{6,6}\varphi _{9}= \varphi _{3}= u_{1}S_{6,6}c_{17}+
u_{2}(S_{6,6}c_{16} + c^{2}_{2}(\beta _{6}+ \beta _{11})) +
\varphi _{3}(\beta _{2}+\beta _{10})$, that gives:
$ \beta _{2}+\beta _{9}=1$, $S_{6,6}c_{17}=0$,
$S_{6,6}c_{16}=c^{2}_{2}(\beta _{6}+ \beta _{11})$.

Let us apply the operation $S_{2}$, we get:
\begin{multline*}
S_{2}\varphi _{9}= u_{5} \\
= u_{1}S_{2}c_{17}+ u_{2}S_{2}c_{16}+ u_{2}c_{14}+
u_{3}S_{2}c_{14}+\varphi _{3}c_{10}+ u_{4}c^{2}_{4}+
u_{4}c^{4}_{2} \\
+ \varphi _{7}c_{2}+ u_{5}+ u_{3}(\beta _{2}c^{2}_{6}+
\beta _{3}c^{6}_{2} + \beta _{4}c^{2}_{2}c^{2}_{4} )+
\varphi _{3}(\beta _{5}c^{2}_{5}+ \beta _{6}c^{2}_{2}c_{6} ) \\
+ u_{4}\beta _{6}c^{4}_{2}+ u_{4}(\beta _{8}c^{2}_{4}+
\beta _{9}c^{4}_{2} )+ \varphi _{5}\beta _{10}c_{6}+
\varphi _{6}\beta _{11}c^{2}_{2} = u_{1}[ S_{2}c_{17} \\
+ c_{5}c_{10}+ c_{2}c_{13}+ c^{5}_{2}c_{5}+ \beta _{5}c^{2}_{5}+
\beta _{6}c^{2}_{2}c_{5}c_{6} + c_{2}c^{2}_{4}c_{5}+
\beta _{10}c_{6}c_{9} \\
+ c^{2}_{2}c_{11}(\beta _{10}+\beta _{11})] +
u_{2}[ S_{2}c_{16}+c_{14}+ c_{2}c_{12}+ c_{4}c_{10}+
c_{2}c_{4}c_{8} \\
+ c^{5}_{2}c_{4}+c_{2}c^{3}_{4}+ \beta _{5}c_{4}c^{2}_{5}+
\beta _{6}c^{2}_{2}c_{4}c_{6}+ \beta _{10}c_{6}c_{8}+
c^{2}_{2}c_{10}(\beta _{10}+\beta _{11})] \\
+ u_{3}[ S_{2}c_{14}+ c_{2}c^{2}_{5}(1+ \beta _{5})+
c^{2}_{2}c_{8}(1+\beta _{10}+ \beta _{11})+
c^{3}_{2}c_{6}(1+ \beta _{6}) \\
+ c^{2}_{6}(\beta _{2}+ \beta _{10})+ c^{6}_{2}(1+ \beta _{3}+
\beta _{10}+ \beta _{11})+ c^{2}_{2}c^{2}_{4}(1+ \beta _{4})] \\
+ u_{4}[ c^{2}_{2}c_{4}(1+ \beta _{10}+ \beta _{11})+
c^{2}_{4}(1+ \beta _{8})+ c^{4}_{2}(1+ \beta _{6}+ \beta _{9}) \\
+ c_{2}c_{6}(1+ \beta _{10})]+ u_{5}.
\end{multline*}
We conclude from this that
\begin{multline*}
\beta _{8}= 1, \ \beta _{10}= 1, \ \beta _{11}= 0, \ \beta _{5}= 1,
\ \beta _{6}= 1, \ \beta _{9}= 0, \ \beta _{2}= 0; \\
S_{2}c_{16}
= c_{14}+c_{2}c_{12}+c_{4}c_{10}+c_{2}c_{4}c_{8}+c^{5}_{2}c_{4}+
c_{4}c^{2}_{5}+c^{2}_{2}c_{4}c_{6}+c_{6}c_{8}+c^{2}_{2}c_{10}+
c_{2}c^{3}_{4}; \\
S_{2}c_{17}=c_{2}c_{13}+c_{5}c_{10}+c^{5}_{2}c_{6}+c^{3}_{5}+
c^{2}_{2}c_{5}c_{6}+c_{6}c_{9}+c^{2}_{2}c_{11}+c_{2}c^{2}_{4}c_{5}; \\
S_{2}c_{14}=c^{2}_{6}+\beta _{3}c^{6}_{2}+
(1+\beta _{4})c^{2}_{2}c^{2}_{4}.
\end{multline*}
So, $\varphi_{9}$ has the following form:
\begin{multline*}
\varphi _{9}= u_{1}c_{17}+ u_{2}c_{16}+ u_{3}c_{14}+
\varphi _{3}(\beta _{3}c^{6}_{2}+ \beta _{4}c^{2}_{2}c^{2}_{4})+
u_{4}(c_{10}+c^{2}_{5}+ c^{2}_{2}c_{6}) \\
+ \varphi _{5}c^{2}_{4} + \varphi _{6}c_{6} + u_{5}c_{2}.
\end{multline*}
To determine the coefficient $\beta_{4}$ let us apply the operation
$S_{2,2}$:
\begin{multline*}
S_{2,2}\varphi_{9} = \varphi_{7} = u_{1}S_{2,2}c_{17}+
u_{2}(S_{2,2}c_{16}+ c^{2}_{6}+ c^{2}_{2}c^{2}_{4}) +
u_{3}S_{2,2}c_{14} \\
+ \varphi _{3}((\beta _{3}+ \beta _{4})c^{4}_{2}+
\beta _{4}c^{2}_{4})+ \varphi _{7}.
\end{multline*}
We get $S_{2,2}c_{17}=(\beta _{3}+\beta _{4})c^{4}_{2}c_{5}+
\beta _{4}c^{2}_{4}c_{5}$, $S_{2,2}c_{16}= c^{2}_{6}+
c^{2}_{2}c^{2}_{4}+ \beta _{4}c^{3}_{4}
+ (\beta _{3}+ \beta _{4})c^{2}_{2}c^{2}_{4}$,
$S_{2,2}c_{14}= (\beta _{3}+\beta _{4})c^{5}_{2} +
\beta _{4}c_{2}c^{2}_{4}$; and because of $d_{1}(c_{14})
= 0$, we have $\beta _{3},\beta _{4}= 0$.

The final form of $\varphi_{9}$ will be the following:
\[
\varphi _{9}= u_{1}c_{17}+ u_{2}c_{16}+ u_{3}c_{14}+
u_{4}(c_{10}+c^{2}_{5}+ c^{2}_{2}c_{6})+ \varphi _{5}c^{2}_{4}+
\varphi _{6}c_{6}+u_{5}c_{2}.
\]
Let us calculate the result of the action of the operation
$S_{\omega }$ on the elements $c_{17}$, $c_{16}$, $c_{14}$.

6. The projection $\varphi_{10}$ of the element $\Phi_{10}$ has the form:
\begin{multline*}
\varphi _{10} = u_{1}c_{19}+ u_{2}c_{18}+ u_{3}c_{16}+ u_{5}c_{4}+
u_{3}(\beta _{2}c^{2}_{8}+ \beta _{3}c^{4}_{4}+ \beta _{4}c^{8}_{2}+
\beta _{5}c_{6}c_{10} \\
+ \beta _{6}c^{4}_{2}c^{2}_{4}+ \beta _{7}c^{2}_{2}c^{2}_{6}+
\beta _{8}c^{2}_{2}c_{12}+ \beta _{9}c^{2}_{5}c_{6}) +
\varphi _{3}(\beta _{10}c_{14}+ \beta _{11}c^{2}_{2}c_{10} \\
+ \beta _{12}c^{2}_{2}c^{2}_{5}+ \beta _{13}c^{2}_{4}c_{6}+
\beta _{14}c^{4}_{2}c_{6})+ u_{4}(\beta _{15}c_{12}+
\beta _{16}c^{2}_{6}+ \beta _{17}c^{2}_{2}c^{2}_{4}) \\
+ \varphi _{5}(\beta _{18}c_{10}+ \beta _{19}c^{2}_{5}+
\beta _{20}c^{2}_{2}c_{6})+ \varphi _{6}(\beta _{21}c^{2}_{4}+
\beta _{22}c^{4}_{2}) + \beta _{23}\varphi _{7}c_{6} \\
+ \beta _{24}u_{5}c^{2}_{2} + \beta _{1}u_{4}c^{6}_{2}.
\end{multline*}
Applying the operation $S_{8,8}$, we get
$S_{8,8}\varphi_{10}=u_{3}(1+\beta _{2})$, hence, $\beta _{2} = 0$.

Let us apply the operation $S_{4,4,4,4}$, we get
$S_{4,4,4,4}\varphi_{10}=u_{3}=u_{3}(1 +\beta _{3}+\beta _{15})$.
Hence, $\beta _{3}=\beta _{15}$.

Let us apply the operation $S_{14}$, then:
\[
S_{14}\varphi _{10} = \varphi _{3}=u_{1}S_{14}c_{19} +
u_{2}(S_{14}c_{18} + c_{4} + \beta _{24}c^{2}_{2}) +u_{3}c_{2} +
\beta _{10}\varphi _{3}.
\]
It follows from this that $\beta_{10}=0$, $S_{14}c_{19}=c_{5}$,
$S_{14}c_{18}=\beta _{24}c^{2}_{2}$.

Let us apply the operation $S_{12}$, then we get:
\[
S_{12}\varphi _{10}= u_{4}= u_{1}S_{12}c_{19}+u_{2}(S_{12}c_{18}+
\beta _{23}c_{6}) + u_{3}c^{2}_{2}(\beta _{8}+\beta _{24}) +
\beta _{15}u_{4}.
\]
Hence we have $\beta_{15} = 1$, $\beta_{3}= 1$,
$\beta _{8}= \beta _{24}$, $S_{12}c_{19}= 0$,
$S_{12}c_{18}= \beta _{23}c_{6}$.

Let us apply the operation $S_{6,6}$, this gives:
\begin{multline*}
S_{6,6}\varphi _{10}= 0 = u_{1}S_{6,6}c_{19} + u_{2}(S_{6,6}c_{18}+
c_{6}(\beta _{15}+ \beta _{23})) \\
+u_{3}c^{2}_{2}(\beta _{7} + \beta _{20}) + u_{4}(\beta _{16}+
\beta _{23}).
\end{multline*}
Hence, $\beta_{16}=\beta_{23}$, $\beta _{7}=\beta _{20}$,
$S_{6,6}c_{18}=(\beta _{15}+\beta _{23})c_{6}$, $S_{6,6}c_{19}= 0$.

Applying the operation $S_{4,4,4}$, we have:
\begin{multline*}
S_{4,4,4}\varphi _{10} = 0 = u_{1}S_{4,4,4}c_{19} +
u_{2}(S_{4,4,4}c_{18}+ (\beta _{13}+ \beta _{23})c_{6}) \\
+ u_{3}c^{2}_{2}(\beta _{8}+ \beta _{17}) + u_{4}(\beta _{15}+
\beta _{21}).
\end{multline*}
Hence $\beta_{8} = \beta_{17}$, $\beta_{15}=\beta_{21}$,
$S_{4,4,4}c_{19}=0$,
$S_{4,4,4}c_{18}= (\beta_{13}$ + $\beta _{23})c_{6}$.

Let us apply the operation $S_{10}$, we have:
\begin{multline*}
S_{10}\varphi _{10} = \varphi _{5} = u_{1}S_{10}c_{19} +
u_{2}(S_{10}c_{18}+ \beta _{21}c^{2}_{4}+ \beta _{22}c^{4}_{2}) \\
+ u_{3}(c_{2}c_{4} + c_{6}(1+ \beta _{5}+ \beta _{23})) +
\varphi _{3}(c_{4}+ c^{2}_{2}(\beta _{11}+ \beta _{24})) +
\beta _{18}\varphi _{5}.
\end{multline*}
Hence, $\beta_{11}=\beta _{24}$, $\beta_{18}= 1$,
$\beta _{5}+\beta _{23}= 1$, $S_{10}c_{19}=c_{4}c_{5}$,
$S_{10}c_{18}=(1+\beta _{21})c^{2}_{4} + \beta _{22}c^{4}_{2}$.

Applying the operation $S_{5,5}$ we have:
\begin{multline*}
S_{5,5}\varphi _{10}= \varphi _{5}= u_{1}S_{5,5}c_{19}+
u_{2}(S_{5,5}c_{18}+ \beta _{21}c^{2}_{4}+ \beta _{22}c^{4}_{2}) \\
+ u_{3}(c_{2}c_{4} + c_{6}(1+\beta _{5}+\beta _{9})) +
\varphi _{3}(c_{4}+ c^{2}_{2}(\beta _{11}+\beta _{12}+
\beta _{24})) + \varphi _{5}(1+\beta _{19}).
\end{multline*}
We get from this: $\beta_{19}=0$, $\beta _{12}=0$, $\beta _{5}+\beta _{9}=1$,
$S_{5,5}c_{19}=c_{4}c_{5}$, $S_{5,5}c_{18}=(1+\beta _{21})c^{2}_{4} +
\beta _{22}c^{4}_{2}$.

Let us apply the operation $S_{8}$, then:
\begin{multline*}
S_{8}\varphi _{10}= \varphi _{6}= u_{1}S_{8}c_{19} +
u_{2}(S_{8}c_{18} +\beta _{18}c_{10} +\beta _{20}c^{2}_{2}c_{6}) \\
+u_{3}(c_{8} +c_{4}(1+\beta _{21})+ \beta _{22}c^{4}_{2}) +
\beta _{23}\varphi _{3}c_{6} + u_{4}(c_{4}+ \beta _{24}c^{2}_{2}).
\end{multline*}
We get: $S_{8}c_{19}=c_{11} +\beta _{23}c_{5}c_{6}$,
$S_{8}c_{18}= (1+\beta _{18})c_{10} +\beta _{20}c^{2}_{2}c_{6}$,
$\beta _{23}=0$, $\beta _{24}=0$, $\beta _{21}=1$, $\beta _{22}= 1$,
$\beta _{5}=1$, $\beta _{8}=0$, $\beta _{17}=0$, $\beta _{9}=0$,
$\beta _{16}=0$.

Applying the operation $S_{4,4}$, we see:
\begin{multline*}
S_{4,4}\varphi _{10}= 0 = u_{1}S_{4,4}c_{19} + u_{2}(S_{4,4}c_{18}+
c_{10}+ \beta _{20}c^{2}_{2}c_{6}) \\
+ u_{3}\beta _{6}c^{4}_{2} + \varphi _{3}\beta _{13}c_{6}+
u_{3}c^{4}_{2}.
\end{multline*}
We conclude that $\beta _{13}=0$, $\beta _{6}=1$, $S_{4,4}c_{19}=0$,
$S_{4,4}c_{18}= c_{10}+\beta _{20}c^{2}_{2}c_{6}$.

Using the operation $S_{2,2,2,2}$ gives the following:
\begin{multline*}
S_{2,2,2,2}\varphi _{10}= \varphi _{6}= u_{1}S_{2,2,2,2}c_{19} +
u_{2}(S_{2,2,2,2}c_{18}+ \beta _{20}c^{2}_{2}c_{6}) \\
+ u_{3}c^{4}_{2}(\beta _{7}+\beta _{14}+\beta _{20}) +
\varphi _{3}(\beta _{14}+\beta _{20}) + \beta _{1}u_{4}c^{2}_{2} +
\varphi _{6}.
\end{multline*}
We conclude from this that $\beta _{14}= \beta _{20}= 0$,
$\beta _{7}= 0$, $\beta _{1}= 0$, $S_{2,2,2,2}c_{18}= 0$.

For the definition of $\beta_{4}$ let us use the operation
$S_{2,2,2,2,2,2,2,2}$:
\[
S_{2,2,2,2,2,2,2,2}\varphi_{10} = u_{3} = u_{3}(1+\beta_{4}).
\]
We get that $\beta _{4}=0$.

The final form of $\varphi_{10}$ is the following:
\begin{multline*}
\varphi _{10}= u_{1}c_{19} + u_{2}c_{18} + u_{3}c_{16}+
u_{3}(c^{2}_{8}+ c^{4}_{4}+ c_{6}c_{10}+ c^{4}_{2}c^{2}_{4}) \\
+u_{4}c_{12} + \varphi _{5}c_{10}+ \varphi _{6}(c^{4}_{2}+
c^{2}_{4}) + u_{5}c_{4}.
\end{multline*}
Using the obtained decomposition let us calculate the results of
the action of the operations $S_{\omega }$ on $c_{19}$ and $c_{18}$.

7. Let us precise the form of the projection $\varphi_{11}$ of the
element $\Phi_{11} $:
\begin{multline*}
\varphi _{11} = u_{1}c_{21}+ u_{2}c_{4}c_{16}+ u_{3}c_{2}c_{16}+
u_{5}c_{2}c_{4}+ u_{2}c_{20} + u_{3}(c_{18}+ \beta _{1}c^{2}_{9} \\
+ \beta _{2}c^{2}_{2}c_{14}+ \beta _{3}c_{6}c_{12}+
\beta _{4}c^{4}_{2}c_{10}+ \beta _{5}c^{4}_{2}c^{2}_{5}+
\beta _{6}c^{2}_{2}c^{2}_{4}c_{6}+ \beta _{7}c^{6}_{2}c_{6} \\
+ \beta _{8}c^{3}_{6}+ \beta _{9}c^{2}_{4}c_{10}+
\beta _{10}c^{2}_{4}c^{2}_{5})+ \varphi _{3}(\beta _{11}c^{2}_{8}+
\beta _{12}c^{4}_{4}+ \beta _{13}c^{2}_{2}c_{12}+
\beta _{14}c^{8}_{2} \\
+ \beta _{15}c_{6}c_{10}+ \beta _{16}c^{2}_{2}c^{2}_{6}+
\beta _{17}c_{6}c^{2}_{5}+ \beta _{18}c^{4}_{2}c^{2}_{4})+
u_{4}(\beta _{19}c_{14}+ \beta _{20}c^{2}_{2}c_{10} \\
+ \beta _{21}c^{2}_{2}c^{2}_{5}+ \beta _{22}c^{2}_{4}c_{6}+
\beta _{23}c^{4}_{2}c_{6})+ \varphi _{5}(\beta _{24}c_{12}+
\beta _{25}c^{2}_{6}+ \beta _{26}c^{2}_{2}c^{2}_{4} \\
+ \beta _{27}c^{6}_{2})+ \varphi _{6}(\beta _{28}c_{10}+
\beta _{29}c^{2}_{5}+ \beta _{30}c^{2}_{2}c_{6})+
\varphi _{7}(\beta _{31}c^{2}_{4}+ \beta _{32}c^{4}_{2}) \\
+ u_{5}\beta _{33}c_{6}+ \varphi _{9}\beta _{34}c^{2}_{2}.
\end{multline*}
Applying the operation $S_{9,9}$ we get
$S_{9,9}\varphi _{11}= 0 = u_{3}(\beta _{1}+ 1)$,
hence: $\beta _{1}= 1$.

Applying the operation $S_{16}$, we arrive to the equality:
\[
S_{16}\varphi_{11}= \varphi_{3}= u_{1}S_{16}c_{21}+ u_{2}c_{4}+
u_{3}c_{2}+ u_{2}(S_{16}c_{20}+ \beta _{34}c^{2}_{2}),
\]
which gives $S_{16}c_{21}= c_{5}$, $S_{16}c_{20}= \beta _{34}c^{2}_{2}$.

Using the operation $S_{8,8}$, we obtain:
\[
S_{8,8}\varphi _{11}= \varphi _{3}= u_{1}S_{8,8}c_{21}+
u_{2}c_{4}+ u_{3}c_{2}+ u_{2}(S_{8,8}c_{20}+
\beta _{34}c^{2}_{2})+ \beta _{11}\varphi _{3},
\]
so $\beta_{11}= 0$, $S_{8,8}c_{21}= c_{5}$,
$S_{8,8}c_{20}= \beta _{34}c^{2}_{2}$.

Using the operation $S_{4,4,4,4}$, we have:
\begin{multline*}
S_{4,4,4,4}\varphi _{11}=0= u_{1}S_{4,4,4,4}c_{21}+
u_{2}S_{4,4,4,4}c_{20}+ \varphi _{3}(\beta _{12}+ \beta _{24}+
\beta _{31})+\\
+ u_{2}(\beta _{13}+ \beta _{26}+ \beta _{34})c^{2}_{2},
\end{multline*}
Hence, $\beta _{12}+\beta _{24}+\beta _{31}=0$,
$S_{4,4,4,4}c_{20}= (\beta _{13}+\beta _{26}+\beta _{34})c^{2}_{2}$,
$S_{4,4,4,4}c_{21}= 0.$

Using the operation $S_{2,2,2,2,2,2,2,2}$, we get
\begin{multline*}
S_{2,2,2,2,2,2,2,2}\varphi _{11} = \varphi _{3}=
u_{1}S_{2,2,2,2,2,2,2,2}c_{21} + u_{2}c^{2}_{2}(\beta _{3}+
\beta _{4}+\beta_{6}+\beta _{8}+ \\
+ \beta _{13}+ \beta _{15}+ \beta _{16}+ \beta _{18}+
\beta _{24}+ \beta _{26}+ \beta _{28}+ \beta _{34}) +
\varphi _{3}(\beta _{14}+ \beta _{16}+ \\
+ \beta _{23}+ \beta _{25}+ \beta _{27}+ \beta _{34})+
u_{2}S_{2,2,2,2,2,2,2,2}c_{20}.
\end{multline*}
This means that
\begin{multline*}
S_{2,2,2,2,2,2,2,2}c_{20}= c^{2}_{2}(\beta _{3}+ \beta _{4}+
\beta _{6}+ \beta _{8}+ \beta _{13}+ \beta _{15}+ \\
+ \beta _{16}+ \beta _{18}+ \beta _{24}+ \beta _{26}+ \beta _{28}+
\beta _{34}),
\end{multline*}
\[
\beta _{14}+ \beta _{16}+ \beta _{23}+ \beta _{25}+ \beta _{27}+
\beta _{34} = 1,
\]
\[
S_{2,2,2,2,2,2,2,2}c_{21}= 0.
\]
Applying the operation $S_{14}$, we get:
\[
S_{14}\varphi _{11}= u_{4}= u_{1}S_{14}c_{21}+ u_{2}(S_{14}c_{20}+
\beta _{33}c_{6})+u_{3}c^{2}_{2}(1+ \beta _{2}+\beta _{34})+
u_{4}\beta _{19},
\]
Hence, $\beta _{19}= 1$, $S_{14}c_{21}= 0$, $S_{14}c_{20}=
\beta _{33}c_{6}$, $\beta _{2}+ \beta _{34}= 1.$

Let us apply the operation $S_{12}$, we get:
\begin{multline*}
S_{12}\varphi _{11}= \varphi _{5}= u_{1}S_{12}c_{21}+
u_{2}(S_{12}c_{20}+ c^{2}_{4}(1+ \beta _{31})+
\beta _{32}c^{4}_{2})+ \varphi _{5}\beta _{24}+ \\
+ u_{3}(\beta _{3}+ \beta _{33})c_{6}+ \varphi _{3}(\beta _{13}+
\beta _{34})c^{2}_{2}.
\end{multline*}
So, $\beta _{24}= 1$, $\beta _{3}= \beta _{33}$, $\beta _{13}=
\beta _{34}$, $S_{12}c_{20}= c^{2}_{4}(1+ \beta _{31})+
\beta _{32}c^{4}_{2}.$

Let us act by the operation $S_{6,6}$, we get a relation:
\begin{multline*}
S_{6,6}\varphi _{11}= \varphi _{5}= u_{1}S_{6,6}c_{21}+
u_{2}c^{2}_{2}c_{4}+ u_{3}c^{3}_{2}+ u_{2}(S_{6,6}c_{20}+
c^{2}_{4}(\beta _{19}+ \beta _{22}+\\
+ \beta _{31})+ c^{4}_{2}(\beta _{23}+ \beta _{32}))+
u_{3}c_{6}(1+ \beta _{3}+ \beta _{8}+ \beta _{24})+
\varphi _{5}(\beta _{25}+ \beta _{33})+ \\
+ \varphi _{3}c^{2}_{2}(\beta _{16}+ \beta _{30}+ \beta _{34}).
\end{multline*}
Hence: $S_{6,6}c_{21}= c_{5}c^{2}_{2}$, $\beta _{16}+ \beta _{30}+
\beta _{34}= 1$, $\beta _{3}+ \beta _{8}+ \beta _{24}=1$,
$S_{6,6}c_{20}= c^{2}_{4}(\beta _{19}+ \beta _{22}+ \beta _{31})+
c^{4}_{2}(\beta _{23}+ \beta _{32})$, $\beta _{25}+ \beta _{33}= 1.$

Let us apply the operation $S_{4,4,4}$, we get:
\begin{multline*}
S_{4,4,4}\varphi _{11}= \varphi _{5}= u_{1}S_{4,4,4}c_{21}+
u_{2}(S_{4,4,4}c_{20}+ c^{2}_{4}(1+ \beta _{31})+
c^{4}_{2}(\beta _{24}+ \beta _{32}))+ \\
+u_{3}c_{6}(\beta _{3}+ \beta _{22})+
\varphi _{3}c^{2}_{2}(\beta _{13}+ \beta _{26}+ \beta _{34})+ \varphi _{5}(\beta _{24}+ \beta _{31}).
\end{multline*}
Hence, $S_{4,4,4}c_{21}= 0$, $\beta _{24}+ \beta _{31}= 1$,
$\beta _{3}+ \beta _{22}= 0$, $\beta _{31}= 0$, $\beta _{12}= 1$,
$\beta _{13}+ \beta _{26}+ \beta _{34}= 0$, $S_{4,4,4}c_{20}=
c^{2}_{4}(1+ \beta _{31})+ c^{4}_{2}(\beta _{24}+ \beta _{32}).$

Applying the operation $S_{10}$, we arrive to the following:
\begin{multline*}
S_{10}\varphi _{11}= \varphi _{6}= u_{1}S_{10}c_{21}+
u_{2}c_{2}c^{2}_{4}+ u_{2}c_{4}c_{6}+ u_{3}c^{2}_{2}c_{4}+
u_{3}c_{2}c_{6}+ \varphi _{3}c_{2}c_{4}+ \\
+ u_{2}(S_{10}c_{20}+ \beta _{28}c_{10}+ \beta _{29}c^{2}_{5}+
\beta _{30}c^{2}_{2}c_{6})+ u_{3}(c^{4}_{2}(1+ \beta _{2}+
\beta _{4}+ \beta _{32})+ \\
+ c^{2}_{4}(\beta _{9}+ \beta _{31}))+
\varphi _{3}c_{6}(\beta _{15}+ \beta _{33})+
u_{4}c^{2}_{2}(1+ \beta _{20}+ \beta _{34})+ \varphi _{6}\beta _{28}.
\end{multline*}
It follows from this: $\beta _{28}= 1$, $\beta _{9}= 0$, $ \beta _{2}+
\beta _{4}+ \beta _{32}=1$, $\beta _{15}+ \beta _{33}= 1$,  $\beta _{20}+
\beta _{34}=1$,
$S_{10}c_{21}= c_{5}c_{6}+ c_{2}c_{4}c_{6}$, $S_{10}c_{20}= c_{10}+
\beta _{29}c^{2}_{5}+ \beta _{30}c^{2}_{2}c_{6}.$

Let us act by the operation $S_{5,5}$, we get:
\begin{multline*}
S_{5,5}\varphi _{11}= 0= u_{1}S_{5,5}c_{21}+ u_{2}c_{2}c^{2}_{4}+
u_{2}c_{4}c_{6}+ u_{3}c^{2}_{2}c_{4}+ u_{3}c_{2}c_{6}+
\varphi _{3}c_{2}c_{4}+  \\
+ u_{2}(S_{5,5}c_{20}+ \beta _{28}c_{10}+ \beta _{29}c^{2}_{5}+
\beta _{30}c^{2}_{2}c_{6})+ u_{3}(c^{4}_{2}(\beta _{2}+
\beta _{4}+\beta _{5}+ 1)+ \\
+ c^{2}_{4}(\beta _{9}+ \beta _{10})+ \varphi _{3}c_{6}(\beta _{15}+
\beta _{17}+ \beta _{33})+ u_{4}c^{2}_{2}(\beta _{20}+ \beta _{21}+ 1))+ \\
+\varphi _{6}(1+ \beta _{29}).
\end{multline*}
We get from this
$\beta _{29}= 1$, $\beta _{2}+ \beta _{4}+ \beta _{5}= 1$,
$\beta _{9}= \beta _{10}= 0$, $\beta _{15}+ \beta _{17}
+ \beta _{33}= 1$, hence, $\beta _{17}= 0$,
$\beta _{20}+ \beta _{21}= 1$, $S_{5,5}c_{21}= c_{2}c_{4}c_{5}+
+ c_{5}c_{6}$, $S_{5,5}c_{20}= c_{10}+ c^{2}_{5}+
\beta _{30}c^{2}_{2}c_{6}.$

Let us apply the operation $S_{8}$, we get the following:
\begin{multline*}
S_{8}\varphi _{11}= \varphi _{7}= u_{1}S_{8}c_{21}+
u_{2}c_{4}c_{8}+ u_{3}c_{2}c_{8}+ u_{2}c^{3}_{4}+
u_{3}c_{2}c^{2}_{4}+ u_{4}c_{2}c_{4}+\\
+ u_{2}(S_{8}c_{20}+ c_{12}+ \beta _{25}c^{2}_{6}+
\beta _{26}c^{2}_{2}c^{2}_{4}+ \beta _{27}c^{6}_{2})+
u_{3}(c_{10}+ c^{2}_{5}+ \\
+ \beta _{30}c^{2}_{2}c_{6})+ \varphi _{3}c^{4}_{2}\beta _{32}+
u_{4}c_{6}\beta _{33}+ \varphi _{5}c^{2}_{2}\beta _{34}.
\end{multline*}
Hence, $\beta _{34}= 0$, $\beta _{32}= 1$, $\beta _{33}= 1$,
$\beta _{30}= 1$, $S_{8}c_{21}= c_{13}+c_{5}c^{2}_{4}$,
$S_{8}c_{20}= \beta _{25}c^{2}_{6}+ \beta _{26}c^{2}_{2}c^{2}_{4}+
\beta _{27}c^{6}_{2}$, using the previous relations we get from this:
$\beta _{5}, \beta _{3}$, $ \beta _{22}$, $ \beta _{2}$, $\beta _{4},
\beta _{20},\beta _{8}$ are all equal to 1, and $\beta _{13},
\beta _{15}, \beta _{16}, \beta _{25},
\beta _{21}, \beta _{26}$ all are equal to 0.

Let us apply the operation $S_{6}$, then we have:
\begin{multline*}
S_{6}\varphi _{11}= u_{5}= u_{1}S_{6}c_{21}+ u_{2}c_{4}(c_{4}c_{6}+
c_{2}c_{8})+ u_{3}c_{2}(c_{4}c_{6}+ c_{2}c_{8})+
\varphi _{5}c_{2}c_{4}+\\
+ u_{2}(S_{6}c_{20}+ c_{14}+ c^{2}_{2}c_{10}+ c^{2}_{4}c_{6}+
\beta _{23}c^{4}_{2}c_{6})+ u_{3}(c^{2}_{2}c^{2}_{4}(1+ \beta _{6})+ \\
+ c^{6}_{2}(\beta _{7}+ \beta _{27}))+
u_{4}c^{4}_{2}(1+ \beta _{23})+ \varphi _{6}c^{2}_{2}.
\end{multline*}
We get that $\beta _{6}= 1$, $\beta _{7}+ \beta _{27}= 1$,
$\beta _{23}= 1$, $S_{6}c_{21}= c_{5}(c_{10}+ c^{2}_{5}
+ c^{2}_{2}c_{6})+ c_{2}c_{4}c_{9}+ c^{2}_{2}c_{11}$,
$S_{6}c_{20}= c_{14}+ c^{4}_{2}c_{6}.$

Let us calculate the action of the operation $S_{4}$:
\begin{multline*}
S_{4}\varphi _{11}= \varphi _{9}= u_{1}S_{4}c_{21}+ u_{2}c_{16}+
u_{2}c_{4}(c_{12}+ c_{2}c_{10}+ c_{4}c_{8}+ c_{4}c^{4}_{2})+ \\
+u_{5}c_{2}+ u_{3}c_{2}(c_{12}+ c_{2}c_{10}+ c_{4}c_{8}+
c_{4}c^{4}_{2})+ \varphi _{6}c_{2}c_{4}+ \varphi _{6}c_{6}+
u_{3}(c_{14}+\\
+c^{2}_{2}c_{10}+ c^{2}_{4}c_{6})+ u_{2}(S_{4}c_{20}+ c^{4}_{4}+
\beta _{14}c^{8}_{2}+ \beta _{18}c^{4}_{2}c^{2}_{4})+
\varphi _{3}(c_{12}+ c^{6}_{2}\beta _{27})+\\
+ u_{4}(c_{10}+ c^{2}_{5}+ c^{2}_{2}c_{6}).
\end{multline*}
It follows from this that $\beta _{27}= 0$, $\beta _{7}= 1$,
$\beta _{14}= 0$, $S_{4}c_{20}= (1+\beta _{18})c^{2}_{4}c^{4}_{2}+
c^{4}_{4}$, $S_{4}c_{21}= c_{17}+ c_{2}c_{4}c_{11}+
c_{9}c^{2}_{4}+ c_{5}c_{12}.$

To determine the coefficient $\beta _{18}$, let us use the operation
$S_{4,4}$:
\begin{multline*}
S_{4,4}\varphi _{11}= \varphi _{7}=u_{1}S_{4,4}c_{21}+
u_{2}(S_{4,4}c_{20}+ c_{12}+ c_{2}c_{10}+ c_{4}c_{8}+
c_{4}c^{4}_{2}+c^{3}_{4})+\\
+ u_{3}(c_{2}c^{2}_{4}+ c_{10}+ c^{2}_{5}+ c^{2}_{2}c_{6})+
\varphi _{3}c^{4}_{2}\beta _{18}+ u_{4}c_{6}+ \varphi _{6}c_{2}.
\end{multline*}
Hence, $\beta _{18}= 0$, $S_{4,4}c_{21}= c_{13}+ c_{5}(c^{2}_{4}+
c^{4}_{2})+ c_{2}c_{11}$, $S_{4,4}c_{20}= c_{12}.$

The final form of the projection $\varphi _{11}$, will be the following:
\begin{multline*}
\varphi _{11}= u_{1}c_{21}+ u_{2}c_{4}c_{16}+ u_{3}c_{2}c_{16}+
u_{5}c_{2}c_{4}+ u_{2}c_{20}+ u_{3}(c_{18}+c^{2}_{9}+
c^{4}_{2}(c_{10}+\\
+ c^{2}_{5}+ c^{2}_{2}c_{6})+ c^{2}_{2}(c_{14}+ c^{2}_{4}c_{6})+
c_{6}(c_{12}+ c^{2}_{6}))+ \varphi _{3}c^{4}_{4}+ u_{4}(c_{14}+
c^{2}_{2}c_{10}+\\
+ c^{2}_{4}c_{6}+ c^{4}_{2}c_{6})+ \varphi _{5}c_{12}+
\varphi _{6}(c_{10}+ c^{2}_{5}+ c^{2}_{2}c_{6})+
\varphi _{7}c^{4}_{2}+ u_{5}c_{6}.
\end{multline*}
Using the obtained decomposition let us calculate the result of the
action of the operations $S_{\omega }$ on the elements $c_{21}, c_{20}$.

8. Let us precise the form of projection in MASS $\varphi _{12}$ of the
element $\Phi _{12}$:
\begin{multline*}
\varphi _{12}= u_{1}c_{23}+ u_{2}c_{22}+ u_{4}c_{16}+ u_{5}c_{8}+
u_{3}(c_{20}+ \beta _{1}c^{2}_{10}+ \beta _{2}c^{4}_{5}+
\beta _{3}c^{2}_{5}c_{10}+\\
+ \beta _{4}c^{2}_{2}c^{2}_{8}+ \beta _{5}c^{2}_{2}c^{4}_{4}+
\beta _{6}c^{10}_{2}+ \beta _{7}c^{2}_{2}c_{6}c_{10}+
\beta _{8}c^{4}_{2}c^{2}_{6}+ \beta _{9}c^{4}_{2}c_{12}+
\beta _{10}c^{2}_{2}c^{2}_{5}c_{6}+\\
+\beta _{11}c^{6}_{2}c^{2}_{4} + \beta _{12}c_{6}c_{14}+
\beta _{13}c^{2}_{4}c^{2}_{6}+ \beta _{14}c^{2}_{4}c_{12})+
\varphi _{3}(\beta _{15}c_{18} + \beta _{16}c^{2}_{9}+ \\
+\beta _{17}c^{2}_{2}c_{14}+ \beta _{18}c_{6}c_{12}+
\beta _{19}c^{4}_{2}c_{10}+ \beta _{20}c^{4}_{2}c^{2}_{5}+
\beta _{21}c^{2}_{4}c^{2}_{2}c_{6}+ \\
+\beta _{22}c^{6}_{2}c_{6}+ \beta _{23}c^{3}_{6}+
\beta _{24}c^{2}_{4}c_{10}+\beta _{25}c^{2}_{4}c^{2}_{5})+
u_{4}(\beta _{26}c^{2}_{8}+ \beta _{27}c^{4}_{4}+ \\
+\beta _{28}c^{8}_{2}+ \beta _{29}c_{6}c_{10}+
\beta _{30}c^{2}_{2}c^{2}_{6}+ \beta _{31}c^{2}_{2}c_{12}+
\beta _{32}c^{4}_{2}c^{2}_{4} + \beta _{33}c^{2}_{5}c_{6})+\\
+\varphi _{5}(\beta _{34}c_{14}+ \beta _{35}c^{2}_{2}c_{10}+
\beta _{36}c^{2}_{2}c^{2}_{5}+ \beta _{37}c^{2}_{4}c_{6}+
\beta _{38}c^{4}_{2}c_{6})+ \\
+ \varphi _{6}(\beta _{39}c_{12}+ \beta _{40}c^{2}_{6}+
\beta _{41}c^{6}_{2}+ \beta _{42}c^{2}_{2}c^{2}_{4})+
\varphi _{7}(\beta _{43}c_{10}+ \beta _{44}c^{2}_{5}+\\
+ \beta _{45}c^{2}_{2}c_{6})+ u_{5}(\beta _{46}c^{4}_{2}+
\beta _{47}c^{2}_{4})+ \varphi _{9}\beta _{48}c_{6}+
\varphi _{10}\beta _{49}c^{2}_{2}.
\end{multline*}
It follows from the action of the operation $S_{10,10}$ that
\begin{equation}\label{eq:c2p8f1}
S_{10,10}\varphi _{12}= 0= u_{3}(\beta _{1}+\beta _{43}), \
\beta _{1}= \beta _{43}.
\end{equation}
We get from the action of the operation $S_{5,5,5,5}$:
\begin{equation}\label{eq:c2p8f2}
S_{5,5,5,5}\varphi _{12}= u_{3}= u_{3}(\beta _{1}+ \beta _{2}+
\beta _{3}), \ \beta _{1}+ \beta _{2}+ \beta _{3}= 1.
\end{equation}
Applying the operation $S_{18}$, we get the relation:
\[
S_{18}\varphi _{12}= \varphi _{3}= u_{1}S_{18}c_{23}+
u_{2}S_{18}c_{22}+ \beta _{15}\varphi _{3}+
u_{3}c^{2}_{2}\beta _{49},
\]
so, $S_{18}c_{23}= 0$, $S_{18}c_{22}= \beta _{49}c^{2}_{2}$,
$\beta _{15}= 1.$

Let us act by the operation $S_{9,9}$, we get the equality:
\[
S_{9,9}\varphi _{12}= \varphi _{3}= u_{1}S_{9,9}c_{23}+
u_{2}S_{9,9}c_{22}+ u_{2}\beta _{49}c^{2}_{2}+
\varphi _{3}(1+ \beta _{16}).
\]
So, $S_{9,9}c_{23}= 0$, $S_{9,9}c_{22}= \beta _{49}c^{2}_{2}$,
$\beta _{16}= 0.$

Using the operation $S_{16}$, we get:
\[
S_{16}\varphi _{12}= u_{4}= u_{1}S_{16}c_{23}+
u_{2}(S_{16}c_{22}+ \beta _{48}c_{6})+ u_{4}+
u_{3}\beta _{49}c^{2}_{2}.
\]
So, $S_{16}c_{23}= 0$,  $S_{16}c_{22}= \beta _{48}c_{6}$,
$\beta _{49}= 0.$

From the action of the operation $S_{8,8}$, we conclude that:
\[
S_{8,8}\varphi _{12}= 0= u_{1}S_{8,8}c_{23}+ u_{2}(S_{8,8}c_{22}+
\beta _{48}c_{6})+ u_{3}\beta _{4}c^{2}_{2}+ u_{4}\beta _{26}.
\]
So, $S_{8,8}c_{23}= 0, \ S_{8,8}c_{22}= \beta _{48}c_{6}, \
\beta _{4}= 0, \ \beta _{26}= 0.$

Using the operation $S_{4,4,4,4}$, we arrive to the equality:
\begin{multline*}
S_{4,4,4,4}\varphi _{12}= 0= u_{1}S_{4,4,4,4}c_{23}+
u_{2}(S_{4,4,4,4}c_{22}+ (\beta _{37}+\beta _{48})c_{6})+ \\
+ u_{3}c^{2}_{2}(\beta _{5}+ \beta _{31})+ u_{4}(\beta _{27}+
\beta _{39}).
\end{multline*}
It follows from this that:
\begin{equation} \label{eq:c2p8f3}
S_{4,4,4,4}c_{23}= 0, \ S_{4,4,4,4}c_{22}= (\beta _{37}+
\beta _{48})c_{6}, \ \beta _{5}= \beta _{31}, \ \beta _{27}=
\beta _{39}.
\end{equation}
Applying the operation $S_{14}$, we conclude that:
\begin{multline*}
S_{14}\varphi _{12}= \varphi _{5}= u_{1}S_{14}c_{23}+ u_{2}c_{8}+
u_{3}c_{6}(1+ \beta _{12}+ \beta _{48})+ u_{4}c_{2}+
\varphi _{3}\beta _{17}c^{2}_{2}+\\
+ u_{2}(S_{14}c_{22}+ \beta _{46}c^{4}_{2}+ \beta _{47}c^{2}_{4})+
\varphi _{5}\beta _{34}.
\end{multline*}
Hence,
\begin{equation} \label{eq:c2p8f4}
S_{14}c_{23}= c_{5}, \ S_{14}c_{22}= \beta _{46}c^{4}_{2}+
\beta _{47}c^{2}_{4}, \ \beta _{17}= 0, \ \beta _{34}=0, \
\beta _{12}= \beta _{48}.
\end{equation}
From the action of the operation $S_{12}$, it follows that:
\begin{multline*}
S_{12}\varphi _{12}= \varphi _{6}= u_{1}S_{12}c_{23}+ u_{3}c_{8}+
u_{4}c_{4}+ u_{2}(S_{12}c_{22}+ \beta _{43}c_{10}+
\beta _{44}c^{2}_{5}+\\
+ \beta _{45}c^{2}_{2}c_{6})+ u_{3}(c^{2}_{4}(1+ \beta _{14}+
\beta _{47})+ c^{4}_{2}(1+ \beta _{9}+ \beta _{46}))+
\varphi _{3}c_{6}(\beta _{18}+\\
+ \beta _{48})+ u_{4}\beta _{31}c^{2}_{2}+ \varphi _{6}\beta _{39}.
\end{multline*}
We get from this
\begin{multline} \label{eq:c2p8f5}
S_{12}c_{23}= c_{11}, \  S_{12}c_{22}= (\beta _{43}+ 1)c_{10}+
\beta _{44}c^{2}_{5}+ \beta _{45}c^{2}_{2}c_{6}, \\
\beta _{31}=0, \ \beta _{39}= 0, \ \beta _{14}+ \beta _{47}= 1, \
\beta _{9}= \beta _{46}, \ \beta _{18}= \beta _{48}.
\end{multline}
From the formulas (\ref{eq:c2p8f3}), (\ref{eq:c2p8f5}) it follows that
$\beta _{5}= 0, \beta_{27}= 0$.

Let us apply the operation $S_{6,6}$, then we have:
\begin{multline*}
S_{6,6}\varphi _{12}= 0= u_{1}S_{6,6}c_{23}+ u_{2}(c_{2}c_{8}+
c_{4}c_{6})+ \varphi _{5}c_{2}+ \varphi _{3}c_{6}(1+ \beta _{18}+
\beta _{23}+\\
+\beta _{48})+ u_{3}(c^{2}_{4}(\beta _{12}+ \beta _{13}+
\beta _{37})+ c^{4}_{2}(\beta _{8}+ \beta _{38}))+
u_{4}c^{2}_{2}(1+ \beta _{30}+ \beta _{45})+\\
+\varphi _{6}(\beta _{40}+ \beta _{48})+ u_{2}(S_{6,6}c_{22}+
(1+ \beta _{29}+ \beta _{43})c_{10}+ (1+ \beta _{33}+
\beta _{44})c^{2}_{5}+ \\
+ (1+ \beta _{31}+ \beta _{45})c^{2}_{2}c_{6}).
\end{multline*}
Hence,
\begin{multline} \label{eq:c2p8f6}
S_{6,6}c_{23}= c_{2}c_{9}+ c_{5}c_{6},S_{6,6}c_{22}=
(1+ \beta _{29}+ \beta _{43})c_{10}+
(1+ \beta _{33}+ \beta _{44})c^{2}_{5}+\\
+(1+ \beta _{31}+ \beta _{45})c^{2}_{2}c_{6}, \ \beta _{40}=
\beta _{48}, \ \beta _{18}= \beta _{48}+ \beta _{23}, \
\beta _{30}= \beta _{45}, \\
\beta _{12}+ \beta _{13}= \beta _{37}, \ \beta _{8}= \beta _{38}.
\end{multline}
Using theses relations and the formula (\ref{eq:c2p8f5}) we have
$\beta _{23}= 0.$

Let us apply the operation $S_{4,4,4}$, then we have:
\begin{multline*}
S_{4,4,4}\varphi _{12}= 0= u_{1}S_{4,4,4}c_{23}+
u_{3}(c^{2}_{4}\beta _{14}+ c^{4}_{2}(\beta _{9}+
\beta _{32}))+ u_{4}c^{2}_{2}\beta _{42}+\\
+ \varphi _{3}c_{6}(\beta _{18}+ \beta _{37}+ \beta _{48})+
\varphi _{6}(1+ \beta _{47})+ u_{2}(S_{4,4,4}c_{22}+\\
+(1+ \beta _{24}+ \beta _{43})c_{10}+ (\beta _{25}+
\beta _{44})c^{2}_{5}+ (\beta _{21}+ \beta _{45})c^{2}_{2}c_{6}).
\end{multline*}
Hence,
\begin{multline} \label{eq:c2p8f7}
S_{4,4,4}c_{23}= 0, \ S_{4,4,4}c_{22}=
(1+ \beta _{24}+ \beta _{43})c_{10}+
(\beta _{25}+ \beta _{44})c^{2}_{5}+
(\beta _{21}+\beta _{45})c^{2}_{2}c_{6},\\
\beta _{47}= 1, \ \beta _{14}= 0, \ \beta _{42}= 0, \
\beta _{9}= \beta _{32}, \ \beta _{18}= \beta _{48}+ \beta _{37}.
\end{multline}
From these relations and from (\ref{eq:c2p8f5}) it follows that
$\beta_{37}= 0$.

Let us apply the operation $S_{3,3,3,3}$, then it will be:
\begin{multline*}
S_{3,3,3,3}\varphi _{12}= \varphi _{6}= u_{1}S_{3,3,3,3}c_{23}+
u_{2}c_{10}+ u_{3}c_{8}+ u_{4}c_{2}+ u_{2}(c_{2}c_{8}+\\
+ c_{4}c_{6})+ u_{4}c^{2}_{2}+ \varphi _{3}c_{6}(1+ \beta _{48})+
\varphi _{5}c_{2}+ u_{2}(S_{3,3,3,3}c_{22}+\\
+ c^{2}_{5}+ c^{2}_{2}c_{6}(1+ \beta _{33}))+
u_{3}(\beta _{2}+ \beta _{46})c^{4}_{2}.
\end{multline*}
Hence,
\begin{multline}\label{eq:c2p8f8}
S_{3,3,3,3}c_{23}= c_{11}+ c_{2}c_{9}+ c_{5}c_{6}, \ S_{3,3,3,3}c_{22}=
c^{2}_{5}+ c^{2}_{2}c_{6}(1+ \beta _{33}), \\
\beta _{48}= 0, \ \beta _{2}= \beta _{46}+ 1.
\end{multline}
From here and from (\ref{eq:c2p8f5}) it follows that $\beta _{18}= 0$,
from (\ref{eq:c2p8f6})
it follows that $\beta _{40}= 0$, from (\ref{eq:c2p8f4}) it follows that
$\beta _{12}= 0$, from (\ref{eq:c2p8f6}) if follows that $\beta _{13}= 0$.

Applying the operation $S_{10}$, we get:
\begin{multline*}
S_{10}\varphi _{12}= \varphi _{7}= u_{1}S_{10}c_{23}+
\varphi _{3}c_{8}+ u_{4}c_{2}c_{4}+
\varphi _{3}(c^{4}_{2}+ c^{2}_{4})+ u_{3}(c_{10}+ c^{2}_{5}+\\
+ c^{2}_{2}c_{6})+ u_{4}c_{6}+ u_{2}(S_{10}c_{22}+
\beta _{41}c^{6}_{2})+ u_{3}(c_{10}\beta _{43}+
c^{2}_{5}(\beta _{3}+\beta _{44})+\\
+ c^{2}_{2}c_{6}(\beta _{7}+ \beta _{45}))+
\varphi _{3}(c^{2}_{4}\beta _{24}+
c^{4}_{2}(\beta _{19}+ \beta _{46}))+ u_{4}c_{6}\beta _{29}+\\
+ \varphi _{5}c^{2}_{2}\beta _{35}+ \varphi _{7}\beta _{43}.
\end{multline*}
Hence,
\begin{multline} \label{eq:c2p8f9}
S_{10}c_{23}= c_{13}+ c_{5}c_{8}, \ S_{10}c_{22}= c_{12}+
\beta _{41}c^{6}_{2}, \ \beta _{43}= 0, \ \beta _{24}= 0, \\
\beta _{29}= 0, \ \beta _{35}= 0, \ \beta _{3}= \beta _{44},
\ \beta _{7}= \beta _{45}, \ \beta _{19}= \beta _{46}.
\end{multline}
From these relations and from the formula (\ref{eq:c2p8f1}), we have: $\beta _{1}= 0.$
If we apply the operation $S_{5,5}$, we get:
\begin{multline*}
S_{5,5}\varphi _{12}= \varphi _{7}= u_{1}S_{5,5}c_{23}+
u_{4}c_{2}c_{4}+ \varphi _{3}c_{8}+ u_{3}(c_{10}+ c^{2}_{5}+
c^{2}_{2}c_{6})+u_{4}c_{6}+\\
+ \varphi _{3}(c^{2}_{4}+ c^{4}_{2})+ u_{2}(S_{5,5}c_{22}+
\beta _{41}c^{6}_{2})+ u_{3}(c_{10}\beta _{3}+
c^{2}_{5}\beta _{3}+\\
+ c^{2}_{2}c_{6}(\beta _{7}+ \beta _{10}))+
\varphi _{3}(c^{2}_{4}\beta _{25}+ c^{4}_{2}(\beta _{19}+
\beta _{20}+ \beta _{46}))+\\
+ u_{4}c_{6}\beta _{33}+ \varphi _{5}c^{2}_{2}\beta _{36}+
\varphi _{7}\beta _{44}.
\end{multline*}
So,
\begin{multline} \label{eq:c2p8f10}
S_{5,5}c_{23}= c_{13}+ c_{5}c_{8}, S_{5,5}c_{22}= c_{12}+
c^{6}_{2}\beta _{41}, \ \beta _{3}= 0, \ \beta _{25}=0, \
\beta _{33}= 0,\\
\beta _{36}= 0, \ \beta _{44}= 0, \ \beta _{7}= \beta _{10},
\ \beta _{19}= \beta _{20}+ \beta _{46}.
\end{multline}
From these relations and from (\ref{eq:c2p8f9}) it follows that $\beta _{20}= 0$,
from (\ref{eq:c2p8f2}) it follows that $\beta _{2}= 1$, from (\ref{eq:c2p8f8})
 it follows that
$\beta _{46}= 0$, from (\ref{eq:c2p8f9}) it follows that $\beta _{19}= 0$,
from (\ref{eq:c2p8f5}) it follows that $\beta _{9}= 0,$ from (\ref{eq:c2p8f7})
 it follows that
$\beta _{32}= 0$.

Let us apply the operation $S_{8}$, this leads to the relation:
\begin{multline*}
S_{8}\varphi _{12}= u_{5}= u_{1}S_{8}c_{23}+ u_{5}+
u_{2}(S_{8}c_{22}+ \beta _{38}c^{4}_{2}c_{6})+
u_{3}\beta _{41}c^{6}_{2}+\varphi _{3}c^{2}_{2}c_{6}\beta _{45}.
\end{multline*}
This gives us: $S_{8}c_{23}= 0$, $S_{8}c_{22}=
\beta _{38}c^{4}_{2}c_{6}$, $\beta _{41}= 0$, $\beta _{45}= 0.$
From here and from (\ref{eq:c2p8f9}) it follows that $\beta_{7} = 0$,
from (\ref{eq:c2p8f6})
it follows that $\beta _{30}= 0$, from (\ref{eq:c2p8f10}) it follows that
$\beta _{10}= 0$.

If we use the operation $S_{4,4}$, we get:
\begin{multline*}
S_{4,4}\varphi _{12}= 0= u_{1}S_{4,4}c_{23}+ u_{3}(c_{2}c_{10}+
c_{4}c_{8}+ c_{4}c^{4}_{2})+ u_{4}c^{2}_{4}+ \varphi _{3}c_{10}+\\
+ \varphi _{6}c_{4}+ u_{2}(S_{4,4}c_{22}+
\beta _{38}c^{4}_{2}c_{6})+ u_{3}c^{6}_{2}\beta _{11}+
\varphi _{3}c^{2}_{2}c_{6}\beta _{21}.
\end{multline*}
Hence, $S_{4,4}c_{23}= c_{4}c_{11}+ c_{5}c_{10}$, $S_{4,4}c_{22}=
\beta _{38}c^{4}_{2}c_{6}$, $\beta _{11}=0$, $\beta _{21}= 0$.

Using the operation $S_{6}$, we get:
\begin{multline*}
S_{6}\varphi _{12}= \varphi _{9}= u_{1}S_{6}c_{23}+ u_{2}c_{16}+
u_{3}c_{14}+ u_{4}(c_{10}+ c^{2}_{5}+ c^{2}_{2}c_{6})+
\varphi _{5}c^{2}_{4}+\\
+ u_{5}c_{2}+ u_{4}(c_{4}c_{6}+ c_{2}c_{8})+ \varphi _{5}c_{8}+
u_{3}c^{4}_{2}c_{6}(1+ \beta _{38})+ \varphi _{3}c^{6}_{2}\beta _{22}+\\
+ \varphi _{5}c^{4}_{2}\beta _{38}+ u_{2}(S_{6}c_{22}+
\beta _{28}c^{8}_{2}).
\end{multline*}
Hence, $S_{6}c_{23}= c_{17}+ c_{9}c_{8}+ c_{6}c_{11}$, $S_{6}c_{22}=
c^{2}_{8}+ c_{6}c_{10}+ \beta _{28}c^{8}_{2}$,
$\beta _{22}= 0, \beta _{38}= 0$.
From here and from (\ref{eq:c2p8f6}) it follows that $\beta_{8}= 0$.

Finally we use the operation $S_{4}$.
\begin{multline*}
S_{4}\varphi _{12}= \varphi _{10}= u_{1}S_{4}c_{23}+ u_{3}c_{16}+
u_{4}(c_{12}+ c_{2}c_{10}+ c_{4}c_{8}+ c_{4}c^{4}_{2})+
\varphi _{6}c^{2}_{4}+\\
+ \varphi _{6}c_{8}+ u_{5}c_{4}+ u_{2}S_{4}c_{22}+
u_{3}(c^{4}_{2}c^{2}_{4}+ c^{4}_{4})+ u_{2}c_{18}+
u_{3}c^{8}_{2}\beta _{28}.
\end{multline*}
Hence, $\beta _{28}= 1$, $S_{4}c_{23}= c_{19}+ c_{8}c_{11}+
c^{4}_{2}c_{11}+ c_{9}c_{10}$,
$S_{4}c_{22}= c^{4}_{2}c_{10}$.

Finally for the determination of the coefficient $\beta_{6}$, we apply
the operation $ S_{2,2,2,2,2,2,2,2,2,2}$:
\[
S_{2,2,2,2,2,2,2,2,2,2}\varphi _{12}= 0= u_{3}\beta _{6}.
\]
So, $\beta _{6}= 0$, and the projection of the element $\Phi _{12}$
has the form:
\[
\varphi _{12}= u_{1}c_{23}+ u_{4}c_{16}+ u_{5}c_{8}+ u_{2}c_{22}+
u_{3}(c_{20}+ c^{4}_{5})+ \varphi _{3}c_{18}+ u_{4}c^{8}_{2}+
u_{5}c^{2}_{4}.
\]
Using this decomposition let us calculate the action of the operation
$S_{\omega }$ on the elements $c_{23}, c_{22}$.

9. Let us precise the form of the projection in MASS of the element
$\Phi _{13}$:
\begin{multline*}
\varphi _{13}= u_{1}c_{25}+ u_{2}c_{8}c_{16}+ u_{4}c_{2}c_{16}+
u_{5}c_{2}c_{16}+ u_{2}c_{24}+ u_{3}(c_{22}+ \beta _{1}c^{2}_{11} +\\
+ \beta _{2}c^{2}_{2}c_{18}+ \beta _{3}c^{2}_{2}c^{2}_{9}+
\beta _{4}c^{4}_{2}c_{14}+ \beta _{5}c^{2}_{2}c_{6}c_{12}+
\beta _{6}c^{6}_{2}c_{10}+ \beta _{7}c^{6}_{2}c^{2}_{5}+\\
+ \beta _{8}c^{4}_{2}c^{2}_{4}c_{6}+ \beta _{9}c^{8}_{2}c_{6}+
\beta _{10}c^{2}_{2}c^{3}_{6}+ \beta _{11}c^{2}_{2}c^{2}_{4}c_{10}+
\beta _{12}c^{2}_{2}c^{2}_{5}c^{2}_{4}+\beta _{13}c_{6}c^{2}_{8}+\\
+ \beta _{14}c_{6}c^{4}_{4}+ \beta _{15}c^{2}_{6}c_{10}+
\beta _{16}c^{2}_{6}c^{2}_{5}+ \beta _{17}c^{2}_{4}c_{14}+
\beta _{18}c_{10}c_{12}+\beta _{19}c^{2}_{5}c_{12})+\\
+ \varphi _{3}(\beta _{20}c_{20}+ \beta _{21}c^{2}_{10}+
\beta _{22}c^{2}_{2}c^{2}_{8}+ \beta _{23}c^{2}_{2}c^{4}_{4}+
\beta _{24}c^{10}_{2}+ \beta _{25}c^{2}_{2}c_{6}c_{10} +\\
+ \beta _{26}c^{4}_{2}c^{2}_{6}+ \beta _{27}c^{4}_{2}c_{12}+
\beta _{28}c^{2}_{2}c^{2}_{5}c_{6}+ \beta _{29}c^{6}_{2}c^{2}_{4}+
\beta _{30}c_{6}c_{14}+ \beta _{31}c^{2}_{4}c^{2}_{6}+\\
+ \beta _{32}c^{2}_{4}c_{12}+ \beta _{33}c^{2}_{5}c_{10}+
\beta _{34}c^{4}_{5})+ u_{4}(\beta _{35}c_{18}+
\beta _{36}c^{2}_{9}+ \beta _{37}c^{2}_{2}c_{14}+\\
+ \beta _{38}c_{6}c_{12}+ \beta _{39}c^{4}_{2}c_{10}+
\beta _{40}c^{4}_{2}c^{2}_{5}+ \beta _{41}c^{2}_{2}c^{2}_{4}c_{6}+
\beta _{42}c^{6}_{2}c_{6}+ \beta _{43}c^{3}_{6}+\\
+ \beta _{44}c^{2}_{4}c_{10}+ \beta _{45}c^{2}_{4}c^{2}_{5})+
\varphi _{5}(\beta _{46}c^{2}_{8}+ \beta _{47}c^{4}_{4}+
\beta _{48}c^{8}_{2}+ \beta _{49}c_{6}c_{10}+\\
+ \beta _{50}c^{2}_{2}c^{2}_{6}+ \beta _{51}c^{2}_{2}c_{12}+
\beta _{52}c^{2}_{5}c_{6}+ \beta _{53}c^{4}_{2}c^{2}_{4})+
\varphi _{6}(\beta _{54}c_{14}+ \beta _{55}c^{2}_{2}c_{10}+\\
+ \beta _{56}c^{2}_{2}c^{2}_{5}+ \beta _{57}c^{2}_{4}c_{6}+
\beta _{58}c^{4}_{2}c_{6})+ \varphi _{7}(\beta _{59}c_{12}+
\beta _{60}c^{2}_{6}+ \beta _{61}c^{2}_{2}c^{2}_{4}+\\
+ \beta _{62}c^{6}_{2})+ u_{5}(\beta _{63}c_{10}+
\beta _{64}c^{2}_{5}+ \beta _{65}c^{2}_{2}c_{6})+
\varphi _{9}(\beta _{66}c^{2}_{4}+ \beta _{67}c^{4}_{2})+\\
+ \varphi _{10}\beta _{68}c_{6}+ \varphi _{11}\beta _{69}c^{2}_{2}.
\end{multline*}
From the action of the operation $S_{11,11}$:
 $$S_{11,11}\varphi _{13}= 0= u_{3}(\beta _{1}+ 1)$$
we conclude:
\[
\beta _{1}= 1.
\]
From the action of the operation $S_{20}$, we have:
\[
S_{20}\varphi _{13}= \varphi _{3}= u_{1}S_{20}c_{25}+
u_{2}S_{20}c_{24}+ \varphi _{3}\beta _{20}+ u_{2}\beta _{69}c^{2}_{2},
\]
Hence, $\beta _{20}= 1$, $S_{20}c_{25}= 0$, $S_{20}c_{24}=
\beta _{69}c^{2}_{2}$.

Applying the operation $S_{10,10}$, we get:
\begin{multline*}
S_{10,10}\varphi _{13}= \varphi _{3}= u_{1}S_{10,10}c_{25}+
u_{2}S_{10,10}c_{24}+ (\beta _{21}+ \beta _{63})\varphi _{3}+\\
+ u_{2}c^{2}_{2}(\beta _{54}+ \beta _{55}+ \beta _{69}).
\end{multline*}
Hence,
\begin{equation}\label{eq:c2p9f1}
S_{10,10}c_{25}= 0, \ S_{10,10}c_{24}= c^{2}_{2}(\beta _{54}+
\beta _{55}+ \beta _{69}), \ \beta _{21}+ \beta _{63}= 1.
\end{equation}
If we apply the operation $S_{18}$, we get:
\[
S_{18}\varphi _{13}= u_{4}= u_{1}S_{18}c_{25}+ u_{2}(S_{18}c_{24}+
\beta _{68}c_{6})+ u_{3}c^{2}_{2}(\beta _{2}+ \beta _{69})+ u_{4}\beta _{35}.
\]
Hence,
\begin{equation}\label{eq:c2p9f2}
\beta _{35}= 1, \ S_{18}c_{25}= 0, \ S_{18}c_{24}= \beta _{68}c_{6},
\ \beta _{2}= \beta _{69}.
\end{equation}
From the action $S_{9,9}$, we conclude:
\begin{multline*}
S_{9,9}\varphi _{13}= 0= u_{1}S_{9,9}c_{25}+ u_{2}(S_{9,9}c_{24}+
\beta _{68}c_{6})+ u_{3}c^{2}_{2}(\beta _{2}+ \beta _{3})+\\
+ u_{4}(\beta _{36}+ 1),
\end{multline*}
hence,
\begin{equation}\label{eq:c2p9f3}
S_{9,9}c_{25}= 0, \ S_{9,9}c_{24}= \beta _{68}c_{6}, \ \beta _{36}= 1,
\ \beta _{2}= \beta _{3}.
\end{equation}
Using the operation $S_{6,6,6}$, we get the equality:
\begin{multline*}
S_{6,6,6}\varphi _{13}=u_{4}= u_{1}S_{6,6,6}c_{25}+
u_{2}(S_{6,6,6}c_{24}+ (1+ \beta _{38}+ \beta _{59}+\beta _{43})c_{6})+\\
+ u_{3}c^{2}_{2}(1+ \beta _{10}+ \beta _{13}+ \beta _{46}+
\beta _{50}+ \beta _{69})+ u_{4}(\beta _{43}+ \beta _{60}).
\end{multline*}
So,
\begin{multline}\label{eq:c2p9f4}
S_{6,6,6}c_{25}=0, \ S_{6,6,6}c_{24}= (1+ \beta _{38}+ \beta _{53}+
\beta _{43})c_{6},\\
\beta _{43}+\beta _{60}= 1, \ \beta _{10}+ \beta _{13}+ \beta _{46}+
\beta _{50}+ \beta _{69}=1.
\end{multline}
If we apply the operation $S_{16}$, then we get:
\begin{multline*}
S_{16}\varphi _{13}= \varphi _{5}= u_{1}S_{16}c_{25}+
u_{2}(S_{16}c_{24}+ \beta _{66}c^{2}_{4}+ \beta _{67}c^{4}_{2}))+
u_{2}c_{8}+ u_{4}c_{2}+\\
+ u_{3}\beta _{68}c_{6}+ \varphi _{3}c^{2}_{2}\beta _{69},
\end{multline*}
So, it will be: $S_{16}c_{25}= c_{9}$, $S_{16}c_{24}=
\beta _{66}c^{2}_{4}+\beta _{67}c^{4}_{2}$, $\beta _{68}=1$,
$\beta _{69}= 0$. From these relations and from (\ref{eq:c2p9f2}) and
(\ref{eq:c2p9f3}) we get that
$\beta _{2},\beta _{3}= 0$.

Let us act by the operation $S_{8,8}$, we get:
\begin{multline*}
S_{8,8}\varphi _{13}= \varphi _{5}= u_{1}S_{8,8}c_{25}+
u_{2}(S_{8,8}c_{24}+ (1+\beta _{66})c^{2}_{4}+
\beta _{67}c^{4}_{2})+u_{3}\beta _{13}c_{6}+\\
+ \varphi _{3}\beta _{22}c^{2}_{2}+ \varphi _{5}\beta _{46}.
\end{multline*}
Hence, $S_{8,8}c_{25}= 0$, $S_{8,8}c_{24}=
(1+\beta _{66})c^{2}_{4}+ \beta _{67}c^{4}_{2}$,
$\beta _{13}= 0$, $\beta _{22}= 0$, $\beta _{46}= 1.$

Let us apply the operation $S_{14}$, then it will be:
\begin{multline*}
S_{14}\varphi _{13}= \varphi _{6}= u_{1}S_{14}c_{25}+
u_{2}(S_{14}c_{24}+ \beta _{63}c_{10}+ \beta _{64}c^{2}_{5}+
\beta _{65}c^{2}_{2}c_{6})+\\
+ u_{3}(c^{2}_{4}(1+ \beta _{17}+ \beta _{66})+
c^{4}_{2}(\beta _{4}+ \beta _{67}))+ \varphi _{3}\beta _{30}c_{6}+\\
+ u_{4}c^{2}_{2}(\beta _{37}+ 1)+ \varphi _{6}\beta _{54}.
\end{multline*}
From these relations we obtain:
\begin{multline}\label{eq:c2p9f5}
S_{14}c_{25}= 0, \ S_{14}c_{24}= \beta _{63}c_{10}+
\beta _{64}c^{2}_{5}+ \beta _{65}c^{2}_{2}c_{6},\\
\beta _{30}=0, \ \beta _{37}=1, \ \beta _{54}=1, \ \beta _{17}+
\beta _{66} =1, \ \beta _{4}= \beta _{67}.
\end{multline}
Applying the operation $S_{12}$, we get the equality:
\begin{multline*}
S_{12}\varphi _{13}= \varphi _{7}= u_{1}S_{12}c_{25}+
u_{2}c_{4}c_{8}+ u_{3}c_{2}c_{8}+ u_{4}c_{2}c_{4}+
\varphi _{3}(c^{2}_{4}+ c^{4}_{2})+u_{4}c_{6}+\\
+ u_{2}(S_{12}c_{24}+\beta _{59}c_{12}+ \beta _{60}c^{2}_{6}+
\beta _{61}c^{2}_{2}c^{2}_{4}+ \beta _{62}c^{6}_{2})+
u_{3}(c_{10}(1+ \beta _{18}+ \beta _{63})+\\
+ c^{2}_{5}(\beta _{19}+ \beta _{64})+
c^{2}_{2}c_{6}(\beta _{5}+ \beta _{65}))+
\varphi _{3}(c^{4}_{2}(\beta _{27}+ \beta _{67})+
c^{2}_{4}(\beta _{32}+ \beta _{66}))+\\
+ u_{4}c_{6}\beta _{38}+ \varphi _{5}\beta _{51}c^{2}_{2}+
\varphi _{7}\beta _{59}.
\end{multline*}
This gives the result:
\begin{multline}\label{eq:c2p9f6}
S_{12}c_{25}= c_{13}, \ S_{12}c_{24}= (\beta _{59}+1)c_{12}+
\beta _{60}c^{2}_{6}+ \beta _{61}c^{2}_{2}c^{2}_{4}+
\beta _{62}c^{6}_{2},\\
\beta _{59}= 0, \ \beta _{51}= 0, \ \beta _{38}= 0, \ \beta _{19}+
\beta _{64}= 1, \ \beta _{5}+\beta _{65}=1, \
\beta _{18}=\beta _{63},  \\
\beta _{27}= \beta _{67}, \ \beta _{32}= \beta _{66}.
\end{multline}
Applying the operation $S_{6,6}$ we obtain:
\begin{multline*}
S_{6,6}\varphi _{13}= \varphi _{7}= u_{1}S_{6,6}c_{25}+
u_{2}c^{2}_{2}c_{8}+ u_{3}c^{2}_{2}c_{6}+ u_{4}c^{3}_{2}+
\varphi _{5}c^{2}_{2}(\beta _{50}+ \beta _{65})+\\
+ \varphi _{7}+ u_{2}(S_{6,6}c_{24}+
c^{2}_{2}c^{2}_{4}(1+\beta _{41}+ \beta _{61})+
c^{6}_{2}(\beta _{42}+ \beta _{62})+ c^{2}_{6}(\beta _{43}+\\
+ \beta _{60}))+ u_{3}(c_{10}(1+ \beta _{15}+\beta _{49})+
c^{2}_{5}(\beta _{16}+ \beta _{52})+
c^{2}_{2}c_{6}(\beta _{5}+\beta _{10}))+\\
+ \varphi _{3}(c^{2}_{4}(1+ \beta _{31}+ \beta _{57}+
\beta _{66})+ c^{4}_{2}(\beta _{26}+\beta _{58}+\beta _{67}))+\\
+ u_{4}c_{6}(1+\beta _{43})+ \varphi _{7}\beta _{60}.
\end{multline*}
We have:
\begin{multline}\label{eq:c2p9f7}
S_{6,6}c_{25}= c_{9}c^{2}_{2}, \ S_{6,6}c_{24}=
c^{2}_{2}c^{2}_{4}(1+\beta _{41}+ \beta _{61})+
c^{6}_{2}(\beta _{42}+ \beta _{62})+\\
+ c^{2}_{6}(\beta _{43}+ \beta _{60}), \ \beta _{60}= 0, \
\beta _{43}= 1, \  \beta _{15}+\beta _{49} =1, \ \beta _{16}=
\beta _{52},\\
\beta _{5}=\beta _{10}, \ \beta _{31}+ \beta _{57}+ \beta _{66}=1,
\ \beta _{26}+\beta _{58}+\beta _{67}= 0.
\end{multline}
Let us calculate the result of the action of the operation $S_{10}$
on $\varphi _{13}$, we have:
\begin{multline*}
S_{10}\varphi _{13}= u_{5}= u_{1}S_{10}c_{25}+
u_{2}c_{8}(c_{2}c_{4}+ c_{6})+ u_{4}c_{2}(c_{2}c_{4}+c_{6})+
\varphi _{3}c_{2}c_{8}+\\
+ \varphi _{5}c_{6}+ \varphi _{6}c^{2}_{2}+
u_{3}c^{2}_{6}\beta _{15}+ u_{2}(S_{10}c_{24}+
c_{14}+ c^{2}_{2}c_{10}\beta _{55}+ c^{2}_{2}c^{2}_{5}\beta _{56}+\\
+ c^{2}_{4}c_{6}\beta _{57}+ c^{4}_{2}c_{6}\beta _{58})+
u_{3}(c_{12}(1+ \beta _{18})+ c^{6}_{2}(\beta _{4}+ \beta _{6}+
\beta _{62})+\\
+ c^{2}_{2}c^{2}_{4}(\beta _{11}+ \beta _{17}+ \beta _{61}))+
\varphi _{3}(c_{10}(1+ \beta _{63})+ c^{2}_{5}(1+ \beta _{33}+
\beta _{64})+\\
+ c^{2}_{2}c_{6}(1+ \beta _{25}+ \beta _{65}))+
u_{4}(c^{4}_{2}(\beta _{39}+ \beta _{67})+
c^{2}_{4}(\beta _{44}+ \beta _{66}))+\\
+ \varphi _{5}c_{6}\beta _{49}+ \varphi _{6}c^{2}_{2}\beta _{55}+
u_{5}\beta _{63}.
\end{multline*}
It follows from these relations that
\begin{multline}\label{eq:c2p9f8}
S_{10}c_{24}= c_{14}+ c^{2}_{2}c_{10}+ c^{2}_{2}c^{2}_{5}\beta _{56}+
c^{2}_{4}c_{6}\beta _{57}+ c^{4}_{2}c_{6}\beta _{58},\\
S_{10}c_{25}= c^{2}_{2}c_{11}+ c_{2}c_{5}c_{8}+ c_{6}c_{9}, \
\beta _{4}+ \beta _{6}+ \beta _{62}= 1, \ \beta _{11}+ \beta _{17}+
\beta _{61}= 0,\\
\beta _{44}= \beta _{66}, \ \beta _{63}= 1, \  \beta _{15}= 1, \
\beta _{18}= 1, \ \beta _{55}= 0, \ \beta _{49}= 0, \
\beta _{39}= \beta _{67},\\
\beta _{33}+ \beta _{64}=1, \ \beta _{25}+ \beta _{65}=1.
\end{multline}
From these relations and from (\ref{eq:c2p9f1}) it follows that
$\beta _{21}= 0$.
Let us use now the operation $S_{5,5}$, we have:
\begin{multline*}
S_{5,5}\varphi _{13}= 0= u_{1}S_{5,5}c_{25}+ u_{2}c_{8}(c_{2}c_{4}+
c_{6})+ u_{4}c_{2}(c_{2}c_{4}+ c_{6})+ \varphi _{3}c_{2}c_{8}+ \\
+ \varphi _{5}c_{6}+ \varphi _{6}c^{2}_{2}+ u_{3}c^{2}_{6}+
u_{2}(S_{5,5}c_{24}+c_{14}+c^{2}_{2}c^{2}_{5}\beta _{56}+
c^{2}_{4}c_{6}\beta _{57}+\\
+ c^{4}_{2}c_{6}\beta _{58})+
u_{3}(c^{6}_{2}(\beta _{4}+ \beta _{6}+ \beta _{7})+
c^{2}_{2}c^{2}_{4}(\beta _{11}+ \beta _{12}+ \beta _{17})+\\
+ c^{2}_{6}\beta _{16}+c_{12}\beta _{19})+
\varphi _{3}(c^{2}_{5}(1+ \beta _{33}+
\beta _{64})+ c^{2}_{2}c_{6}(1+ \beta _{25}+ \beta _{28}+\\
+ \beta _{65}))+ u_{4}(c^{4}_{2}(\beta _{39}+ \beta _{40})+
c^{2}_{4}(\beta _{44}+ \beta _{45}))+ \varphi _{5}c_{6}\beta _{52}+\\
+\varphi _{6}c^{2}_{2}\beta _{56}+ u_{5}(1+ \beta _{64}).
\end{multline*}
This means that:
\begin{multline}\label{eq:c2p9f9}
S_{5,5}c_{24}= c_{14}+c^{2}_{2}c_{10}+ c^{2}_{4}c_{6}\beta _{57}+
c^{4}_{2}c_{6}\beta _{58}, \ S_{5,5}c_{25}= c_{2}c_{5}c_{8}+
c^{2}_{2}c_{11}+ c_{6}c_{9},\\
\beta _{4}+ \beta _{6}+ \beta _{7}= 1, \ \beta _{11}+ \beta _{12}+
\beta _{17}= 0, \ \beta _{25}+ \beta _{28}+ \beta _{65} =1, \
\beta _{16}= 0, \\
\beta _{19}= 0, \ \beta _{52}= 0, \ \beta _{64}= 1, \ \beta _{33}=0,
\ \beta _{56}=0, \ \beta _{39}= \beta _{40}, \ \beta _{44}= \beta _{45}.
\end{multline}
From these relations and from (\ref{eq:c2p9f8}) it follows that
$\beta _{28}= 0$.
Let us act on $\varphi _{13}$ by the operation $S_{8}$:
\begin{multline*}
S_{8}\varphi _{13}= \varphi _{9}= u_{1}S_{8}c_{25}+ u_{2}c_{16}+
u_{3}c_{14}+u_{4}(c_{10}+c^{2}_{5}+c^{2}_{2}c_{6}\beta _{65})+
\varphi _{6}c_{6}+\\
+u_{5}c_{2}+ c^{2}_{4}(u_{2}c_{8}+ u_{3}c_{6}\beta _{57}+u_{4}c_{2})+
u_{3}c^{4}_{2}c_{6}\beta _{58}+u_{2}(S_{8}c_{24}+
\beta _{47}c^{4}_{4}+ \beta _{48}c^{8}_{2}+\\
+ \beta _{50}c^{2}_{2}c^{2}_{6}+ \beta _{53}c^{4}_{2}c^{2}_{4})+
\varphi _{3}(c^{6}_{2}\beta _{62}+ c^{2}_{2}c^{2}_{4}\beta _{61})+
\varphi _{5}(c^{2}_{4}\beta _{66}+ c^{4}_{2}\beta _{67}).
\end{multline*}
This means that
\begin{multline*}
S_{8}c_{25}= c_{17}+ c_{9}c^{2}_{4}, \ S_{8}c_{24}=
\beta _{47}c^{4}_{4}+ \beta _{48}c^{8}_{2}+
\beta _{50}c^{2}_{2}c^{2}_{6}+\beta _{53}c^{4}_{2}c^{2}_{4}.\\
\beta _{65}= 1, \ \beta _{57}= 1, \ \beta _{58}= 0, \ \beta _{61}= 0,
\ \beta _{62}=0, \ \beta _{66}=0, \ \beta _{67}= 0,
\end{multline*}
From these relations and from (\ref{eq:c2p9f5}) it follows that
$\beta _{17} = 1$,
$\beta _{4}= 0$; from (\ref{eq:c2p9f6}) it follows that $\beta _{5}= 0$,
$\beta _{27}= 0$, $\beta _{32}= 0$; from (\ref{eq:c2p9f7}) it follows that
$\beta _{10}= 0$, $\beta_{31}= 0$, $\beta _{26}=0$; from (\ref{eq:c2p9f8})
it follows
that $\beta _{6}=1$, $\beta _{11}= 1$, $\beta _{25}= 0$,
$\beta _{39}= 0$,
$\beta _{44}= 0$; from (\ref{eq:c2p9f9}) it follows
$\beta _{7}= 0$, $\beta _{12}= 0$,
$\beta _{40}=0$, $\beta _{45}= 0$; from (\ref{eq:c2p9f4}) it follows that
$\beta _{50}= 0$.

Applying the operation $S_{6}$, we obtain the equality:
\begin{multline*}
S_{6}\varphi _{13}= \varphi _{10}= u_{1}S_{6}c_{25}+
u_{2}c_{8}(c_{10}+ c^{2}_{5}+ c^{2}_{2}c_{6}+ c_{4}c_{6}+
c_{2}c_{8})+ u_{4}c_{2}(c_{10}+ \\
+ c^{2}_{5}+ c^{2}_{2}c_{6}+ c_{4}c_{6}+ c_{2}c_{8})+
\varphi _{5}c_{2}c_{8}+ \varphi _{5}(c_{10}+ c^{2}_{5}+
c^{2}_{2}c_{6})+ \varphi _{7}c_{6}+\\
+ \varphi _{3}(c^{2}_{4}+ c^{4}_{2})c_{6}+ u_{4}c^{2}_{6}+
u_{2}(S_{6}c_{24}+ c_{18}+ c^{2}_{9}+c^{2}_{2}c_{14}+
\beta _{41}c^{2}_{2}c^{2}_{4}c_{6}+\\
+ \beta _{42}c^{6}_{2}c_{6}+ c^{3}_{6})+
u_{3}(c^{8}_{2}(1+ \beta _{9}+ \beta _{48})+
c^{4}_{2}c^{2}_{4}(\beta _{8}+\beta _{53})+\\
+c^{4}_{4}(1+\beta _{14}+ \beta _{47}))+
u_{4}(c^{2}_{2}c^{2}_{4}(1+ \beta _{41})+
c^{6}_{2}\beta _{42})+ \varphi _{10}.
\end{multline*}
From these relations we obtain:
\begin{multline}\label{eq:c2p9f10}
S_{6}c_{25}= c_{9}(c_{2}c_{8}+ c_{10}+ c^{2}_{5}+ c^{2}_{2}c_{6})+
c_{6}c_{13},\\
S_{6}c_{24}= c_{18}+ c^{2}_{9}+ c^{2}_{2}c_{14}+ c_{6}c_{12}+
c^{2}_{2}c^{2}_{4}c_{6},\\
\beta _{9}+ \beta _{48}= 1, \ \beta _{8}= \beta _{53}, \
\beta _{14}+ \beta _{47}= 1,\\
\beta _{41}= 1, \ \beta _{42}= 0.
\end{multline}
Developing the relation $S_{4}\varphi _{13}= \varphi _{11}$, we get after
cancellation:
\begin{multline*}
u_{1}S_{4}c_{25}+ u_{2}c_{8}(c_{12}+ c_{4}c_{8}+ c_{4}c^{4}_{2})+
(u_{1}c_{11}+ u_{3}c_{8}+u_{3}c^{4}_{2})c_{2}c_{8}+ u_{2}S_{4}c_{24}+\\
+\varphi _{3}c^{4}_{2}c^{2}_{4}+ u_{2}(\beta _{23}c^{2}_{2}c^{4}_{4}+
\beta _{24}c^{10}_{2}+ \beta _{29}c^{6}_{2}c^{2}_{4}+
\beta _{34}c^{4}_{5})+\\
+ \varphi _{3}(c^{2}_{8}+ \beta _{47}c^{4}_{4}+ \beta _{48}c^{8}_{2}+
\beta _{53}c^{4}_{2}c^{2}_{4})+ u_{4}c_{2}(c_{12}+ c_{4}c^{4}_{2})=
u_{1}c_{21}+\\
+ u_{3}((c_{10}+ c^{2}_{5}+ c^{2}_{2}c_{6})c^{4}_{2}+ c_{6}c_{12})+
u_{4}c^{4}_{2}c_{6}+ \varphi _{5}c_{12}+ \varphi _{7}c^{4}_{2}.
\end{multline*}
After regrouping we have:
\begin{multline*}
u_{1}(S_{4}c_{25}+ c_{21}+ c_{2}c_{8}c_{11})+ u_{2}c_{4}c^{2}_{8}+
u_{3}c^{2}_{8}c_{2}+ \varphi _{3}c^{2}_{8}+ u_{2}c_{8}c_{12}+
u_{3}c_{6}c_{12}+ u_{4}c_{2}c_{12}+\\
+\varphi _{5}c_{12}+ u_{2}c_{8}c^{4}_{2}+ u_{3}c_{2}c_{8}c^{4}_{2}+
u_{4}c_{2}c_{4}c^{4}_{2}+ u_{3}(c_{10}+c^{2}_{5}+
c^{2}_{2}c_{6})c^{4}_{2}+ \\
+\varphi _{3}(c^{2}_{4}+ c^{4}_{2}\beta _{48})c^{4}_{2}+
u_{4}c_{6}c^{4}_{2}+ \varphi _{7}c^{4}_{2}+ u_{2}(S_{4}c_{24}+
\beta _{23}c^{2}_{2}c^{4}_{4}+ \beta _{24}c^{10}_{2}+\\
+ \beta _{29}c^{6}_{2}c^{2}_{4}+ \beta _{34}c^{4}_{5})+
\varphi _{3}(\beta _{47}c^{4}_{4} + \beta _{53}c^{4}_{2}c^{2}_{4}) = 0.
\end{multline*}
Hence,
\begin{multline*}
S_{4}c_{24}= \beta _{23}c^{2}_{2}c^{4}_{4}+ \beta _{24}c^{10}_{2}+
\beta _{29}c^{6}_{2}c^{2}_{4}+ \beta _{34}c^{4}_{5}+ c^{4}_{2}c_{12},\\
S_{4}c_{25}= c_{21}+ c_{2}c_{8}c_{11}+c_{5}c^{2}_{8}+c_{9}c_{12}+
c_{11}c^{4}_{2},\\
\beta _{47}= 0, \ \beta _{53}= 0, \ \beta _{48}= 1.
\end{multline*}
From these relations and from (\ref{eq:c2p9f10}), it follows that
$\beta _{9}= 0$,
$\beta _{8}= 0$, $\beta _{14}= 1$.

Using the relation $S_{2}\varphi _{13}= \varphi _{12}$, we get after the
cancellation:
\begin{multline*}
u_{1}(S_{2}c_{25}+ c_{23}+ c_{2}c_{5}c_{8}(c^{2}_{4}+c^{4}_{2})+
c_{2}c_{5}c_{16}+c_{2}c_{8}c_{13}+c_{13}(c_{10}+c^{2}_{5}+
c^{2}_{2}c_{6})+ \\
+ c_{9}c_{14}+ c_{6}c_{17}+(c_{19}+c_{9}c_{10}+c_{11}(c^{2}_{4}+
c^{4}_{2}))c^{2}_{2})+ u_{2}(S_{2}c_{24}+c_{6}c^{2}_{8}+
c^{2}_{4}c_{14}+ \\
+ c^{2}_{11}+ c_{6}c^{4}_{4}+ c^{2}_{6}c_{10}+ c_{12}(c^{2}_{5}+
c^{2}_{2}c_{6})+ c^{2}_{2}c_{18})+u_{3}(c^{4}_{5}(1+\beta _{34})+
\beta _{24}c^{10}_{2}+\\
+ \beta _{23}c^{2}_{2}c^{4}_{4}+ \beta _{29}c^{6}_{2}c^{2}_{4})= 0.
\end{multline*}
Hence,
\begin{multline*}
S_{2}c_{24}= c_{6}c^{2}_{8}+ c^{2}_{11}+ c_{6}c^{4}_{4}+
c^{2}_{6}c_{10}+ c_{12}(c^{2}_{5}+ c^{2}_{2}c_{6})+ c^{2}_{2}c_{18}+
c^{2}_{4}c_{14},\\
S_{2}c_{25}= c_{23}+ c_{2}c_{5}c_{8}(c^{2}_{4}+c^{4}_{2})+
c_{2}c_{5}c_{16}+c_{2}c_{8}c_{13}+ c_{13}(c_{10}+ c^{2}_{5}+
c^{2}_{2}c_{6})+\\
+ c_{9}c_{14}+ c_{6}c_{17}+ (c_{19}+ c_{9}c_{10}+ c_{11}(c^{2}_{4}+
c^{4}_{2}))c^{2}_{2},\\
\beta _{34}= 1, \ \beta _{23}= 0, \ \beta _{24}= 0, \ \beta _{29}= 0.
\end{multline*}
The final form of the projection of $\Phi _{13}$ is the following:
\begin{multline*}
\varphi _{13}= u_{1}c_{25}+ u_{2}c_{8}c_{16}+ u_{4}c_{2}c_{16}+
u_{5}c_{2}c_{8}+ u_{2}c_{24}+ u_{3}(c_{22}+c^{2}_{11}+c^{6}_{2}c_{10}+\\
+ c^{2}_{2}c^{2}_{4}c_{10}+ c_{6}c^{4}_{4}+ c^{2}_{6}c_{10}+
c^{2}_{4}c_{14}+ c_{10}c_{12})+ \varphi _{3}(c_{20}+ c^{4}_{5})+
u_{4}(c_{18}+\\
+ c^{2}_{9}+ c^{2}_{2}c_{14}+ c^{2}_{2}c^{2}_{4}c_{6}+ c^{3}_{6})+
\varphi _{5}(c^{2}_{8}+ c^{8}_{2})+
\varphi _{6}(c_{14}+ c^{2}_{4}c_{6})+ u_{5}(c_{10}+\\
+ c^{2}_{5}+ c^{2}_{2}c_{6})+ \varphi _{10}c_{6}.
\end{multline*}
Using this decomposition let us calculate the result of action of the
operations $S_{\omega }$ on the elements $c_{25}$ and $c_{24}$.

10. The projection of the element $\Phi _{14}$ has the following form:
\begin{multline*}
\varphi _{14}= u_{1}c_{27}+ u_{3}c_{8}c_{16}+ u_{4}c_{4}c_{16}+
u_{5}c_{4}c_{8}+ u_{2}c_{26}+ u_{3}(c_{24}+ \beta _{1}c^{2}_{12}+ \\
+ \beta _{2}c^{2}_{5}c_{14}+ \beta _{3}c_{10}c_{14}+
\beta _{4}c^{6}_{4}+ \beta _{5}c^{2}_{4}c^{2}_{8}+
\beta _{6}c^{2}_{4}c^{2}_{5}c_{6}+ \beta _{7}c^{2}_{4}c_{6}c_{10}+
\beta _{8}c^{4}_{6}+\\
+ \beta _{9}c^{2}_{6}c_{12}+ \beta _{10}c_{6}c^{2}_{9}+
\beta _{11}c_{6}c_{18}+ \beta _{12}c^{2}_{2}c_{20}+
\beta _{13}c^{2}_{2}c^{2}_{10}+\beta _{14}c^{2}_{2}c^{2}_{5}c_{10}+\\
+ \beta _{15}c^{2}_{2}c^{4}_{5}+ \beta _{16}c^{4}_{2}c^{2}_{8}+
\beta _{17}c^{4}_{2}c^{4}_{4}+ \beta _{18}c^{6}_{2}c^{2}_{6}+
\beta _{19}c^{4}_{2}c_{6}c_{10} + \beta _{20}c^{6}_{2}c_{12}+\\
+ \beta _{21}c^{8}_{2}c^{2}_{4}+ \beta _{22}c^{4}_{2}c^{2}_{5}c_{6}+
\beta _{23}c^{2}_{2}c_{6}c_{14}+
\beta _{24}c^{2}_{2}c^{2}_{4}c^{2}_{6} +
\beta _{25}c^{2}_{2}c^{2}_{4}c_{12}+\\
+ \beta _{26}c^{12}_{2})+
\varphi _{3}(\beta _{27}c_{22}+ \beta _{28}c^{2}_{11}+
\beta _{29}c^{2}_{5}c_{12}+ \beta _{30}c_{10}c_{12}+
\beta _{31}c^{2}_{4}c_{14}+\\
+ \beta _{32}c^{2}_{5}c^{2}_{6}+ \beta _{33}c^{2}_{6}c_{10}+
\beta _{34}c_{6}c^{4}_{4}+ \beta _{35}c_{6}c^{2}_{8}+
\beta _{36}c^{2}_{2}c_{18}+ \beta _{37}c^{2}_{2}c^{2}_{9}+\\
+ \beta _{38}c^{4}_{2}c_{14}+ \beta _{39}c^{2}_{2}c_{6}c_{12}+
\beta _{40}c^{6}_{2}c_{10}+ \beta _{41}c^{6}_{2}c^{2}_{5}+
\beta _{42}c^{4}_{2}c^{2}_{4}c_{6}+\beta _{43}c^{8}_{2}c_{6}+\\
+ \beta _{44}c^{2}_{2}c^{3}_{6}+ \beta _{45}c^{2}_{2}c^{2}_{4}c_{10}+
\beta _{46}c^{2}_{2}c^{2}_{4}c^{2}_{5})+u_{4}(\beta _{47}c_{20}+
\beta _{48}c^{2}_{10}+\beta _{49}c^{2}_{5}c_{10}+ \\
+ \beta _{50}c^{4}_{5}+ \beta _{51}c^{2}_{4}c_{12}+
\beta _{52}c^{2}_{4}c^{2}_{6}+ \beta _{53}c_{6}c_{14}+
\beta _{54}c^{2}_{2}c^{2}_{8}+\beta _{55}c^{2}_{2}c^{4}_{4}+
\beta _{56}c^{10}_{2}+\\
+ \beta _{57}c^{2}_{2}c_{6}c_{10}+ \beta _{58}c^{4}_{2}c^{2}_{6}+
\beta _{59}c^{4}_{2}c_{12}+\beta _{60}c^{2}_{2}c^{2}_{5}c_{6}+
\beta _{61}c^{6}_{2}c^{2}_{4})+\varphi _{5}(\beta _{62}c_{18}+ \\
+ \beta _{63}c^{2}_{9}+ \beta _{64}c^{2}_{2}c_{14}+
\beta _{65}c_{6}c_{12}+ \beta _{66}c^{4}_{2}c_{10}+
\beta _{67}c^{4}_{2}c^{2}_{5} + \beta _{68}c^{2}_{2}c^{2}_{4}c_{6}+\\
+ \beta _{69}c^{6}_{2}c_{6}+ \beta _{70}c^{3}_{6}+
\beta _{71}c^{2}_{4}c_{10}+ \beta _{72}c^{2}_{4}c^{2}_{5})+
\varphi _{6}(\beta _{73}c^{2}_{8}+ \beta _{74}c^{4}_{4}+
\beta _{75}c^{8}_{2}+\\
+ \beta _{76}c_{6}c_{10}+ \beta _{77}c^{2}_{2}c^{2}_{6}+
\beta _{78}c^{2}_{2}c_{12}+ \beta _{79}c^{2}_{5}c_{6}+
\beta _{80}c^{4}_{2}c^{2}_{4})+ \varphi _{7}(\beta _{81}c_{14}+\\
+ \beta _{82}c^{2}_{2}c_{10}+ \beta _{83}c^{2}_{2}c^{2}_{5}+
\beta _{84}c^{2}_{4}c_{6}+ \beta _{85}c^{4}_{2}c_{6})+
u_{5}(\beta _{86}c_{12}+ \beta _{87}c^{2}_{6}+\\
+ \beta _{88}c^{2}_{2}c^{2}_{4}+ \beta _{89}c^{6}_{2})+
\varphi _{9}(\beta _{90}c_{10}+ \beta _{91}c^{2}_{5}+
\beta _{92}c^{2}_{2}c_{6})+ \varphi _{10}(\beta _{93}c^{2}_{4}+\\
+ \beta _{94}c^{4}_{2})+ \varphi _{11}c_{6}\beta _{95}+
\varphi _{12}c^{2}_{2}\beta _{96}.
\end{multline*}
Using the operation $S_{12,12}$, we get:
\[
S_{12,12}\varphi _{14}= 0= u_{3}(1+ \beta _{1}+ \beta _{86}),
\]
so,
\begin{equation}\label{eq:c2p10f1}
\beta _{1}+ \beta _{86}=1.
\end{equation}
If we use the operation $S_{22}$, we get:
\[
S_{22}\varphi _{14}= \varphi _{3}= u_{1}S_{22}c_{27}+
u_{2}(S_{22}c_{26}+ \beta _{96}c^{2}_{2})+ \varphi _{3}\beta _{27}.
\]
i.e.
\[
\beta _{27}= 1, \ S_{22}c_{27}= 0, \ S_{22}c_{26}= \beta _{96}c^{2}_{2}.
\]
Using the operation $S_{11,11}$, we get the equality:
\[
S_{11,11}\varphi _{14}= \varphi _{3}= u_{1}S_{11,11}c_{27}+
u_{2}(S_{11,11}c_{26}+ \beta _{96}c^{2}_{2})+
\varphi _{3}(1+\beta _{28}),
\]
or,
\[
\beta _{28}= 0, \ S_{11,11}c_{27}= 0, \
S_{11,11}c_{26}= \beta _{96}c^{2}_{2}.
\]
Application of the operation $S_{20}$, $S_{20}\varphi _{14}= u_{4}$,
gives us:
\[
u_{4}= u_{1}S_{20}c_{27}+ u_{2}(S_{20}c_{26}+ \beta _{95}c_{6})+
u_{3}c^{2}_{2}(\beta _{12}+ \beta _{96})+ u_{4}\beta _{47},
\]
or
\begin{equation}\label{eq:c2p10f2}
S_{20}c_{27}= 0,\  S_{20}c_{26}= \beta _{95}c_{6}, \ \beta _{47}= 1,
\ \beta _{12}= \beta _{96}.
\end{equation}
From the condition $S_{10,10}\varphi _{14}= 0,$ it follows that
\begin{multline*}
0 = u_{1}S_{10,10}c_{27}+ u_{2}(S_{10,10}c_{26}+
c_{6}(\beta _{76}+ \beta _{95}))+\\
+u_{3}c^{2}_{2}(\beta _{82} + \beta _{13}+ \beta _{3}+ \beta _{96}+ 1)+
u_{4}(\beta _{48}+ \beta _{90}).
\end{multline*}
Hence,
\begin{multline}\label{eq:c2p10f3}
S_{10,10}c_{27}= 0, \ S_{10,10}c_{26}= c_{6}(\beta _{76}+ \beta _{95}),\\
\beta _{48}=\beta _{90}, \ \beta _{82}+ \beta _{13}+ \beta _{3}+
\beta _{96}= 1.
\end{multline}
Applying the operation $S_{6,6,6,6}$ we get:
\[
S_{6,6,6,6}\varphi _{14}= u_{3}= u_{1}S_{6,6,6,6}c_{27}+
u_{2}S_{6,6,6,6}c_{26}+ u_{3}(\beta _{8}+ \beta _{70}+\beta _{95})
\]
So,
\begin{equation}\label{eq:c2p10f4}
S_{6,6,6,6}c_{27}= 0, \ S_{6,6,6,6}c_{26}= 0, \ \beta _{8}+ \beta _{70}+
\beta _{95}= 1.
\end{equation}
From the condition $S_{4,4,4,4,4,4}\varphi _{14}= 0,$ we get the equality:
\begin{equation}\label{eq:c2p10f5}
\beta _{1}+ \beta _{4}+ \beta _{5}+ \beta _{51}+ \beta _{93}= 0.
\end{equation}
Let us act by the operation $S_{5,5,5,5}$ on the element $\varphi _{14}$.
Then we have:
\begin{multline*}
S_{5,5,5,5}\varphi _{14}= 0= u_{1}S_{5,5,5,5}c_{27}+
u_{2}(S_{5,5,5,5}c_{26}+ c_{6}(\beta _{76}+ \beta _{79}))+\\
+ u_{3}c^{2}_{2}(1+ \beta _{13}+ \beta _{14}+ \beta _{15}+
\beta _{2}+ \beta _{3}+ \beta _{96})+
u_{4}(\beta _{48}+ \beta _{49}+ \beta _{50}).
\end{multline*}
This gives
\begin{multline}\label{eq:c2p10f6}
S_{5,5,5,5}c_{27}= 0, \ S_{5,5,5,5}c_{26}= c_{6}(\beta _{76}+
\beta _{79}), \ \beta _{48}+ \beta _{49}= \beta _{50},\\
\beta _{13}+ \beta _{14}+ \beta _{15}+ \beta _{2}+ \beta _{3}+
\beta _{96} = 1.
\end{multline}
If we use the operation $S_{18}$we get:
\begin{multline*}
S_{18}\varphi _{14}= \varphi _{5}= u_{1}S_{18}c_{27}+
u_{2}(S_{18}c_{26}+ \beta _{93}c^{2}_{4}+ \beta _{94}c^{4}_{2})+ \\
+u_{3}c_{6}(1+ \beta _{11}+\beta _{95})+
\varphi _{3}c^{2}_{2}(\beta _{36}+ \beta _{96})+ \varphi _{5}\beta _{62}.
\end{multline*}
Hence
\begin{multline}\label{eq:c2p10f7}
S_{18}c_{27}= 0, \ S_{18}c_{26}= \beta _{93}c^{2}_{4}+
\beta _{94}c^{4}_{2},\\
\beta _{62} = 1, \ \beta _{36}= \beta _{96}, \ \beta _{11}+ \beta _{95}=1.
\end{multline}
Let us apply now the operation $S_{9,9}$:
\begin{multline*}
S_{9,9}\varphi _{14}= \varphi _{5}= u_{1}S_{9,9}c_{27}+
u_{2}(S_{9,9}c_{26}+ \beta _{93}c^{2}_{4}+ \beta _{94}c^{4}_{2})+\\
+u_{3}c_{6}(1+\beta _{11}+ \beta _{10})+
\varphi _{3}c^{2}_{2}(\beta _{36}+ \beta _{96}+ \beta _{37})+
\varphi _{5}(\beta _{63}+ 1).
\end{multline*}
This gives
\begin{multline}\label{eq:c2p10f8}
S_{9,9}c_{27}= 0, \ S_{9,9}c_{26}=
\beta _{93}c^{2}_{4}+ \beta _{94}c^{4}_{2}, \\
\beta _{63}=0, \ \beta _{11}+\beta _{10} = 1,\
\beta _{36}+ \beta _{96}= \beta _{37}.
\end{multline}
It follows from (\ref{eq:c2p10f7}) and (\ref{eq:c2p10f8}) that
$\beta _{37}= 0$.

Let us apply the operation $S_{16}$, this gives:
\begin{multline*}
S_{16}\varphi _{14}= \varphi _{6}= u_{1}S_{16}c_{27}+
u_{2}S_{16}c_{26}+ u_{3}c_{8}+ u_{4}c_{4}+ u_{3}c^{4}_{2}\beta _{94}+\\
+u_{2}(\beta _{90}c_{10}+ \beta _{91}c^{2}_{5}+
\beta _{92}c^{2}_{2}c_{6})+ u_{3}c^{2}_{4}\beta _{93}+
\varphi _{3}c_{6}\beta _{95}+ u_{4}c^{2}_{2}\beta _{96}.
\end{multline*}
This means that
\begin{multline*}
S_{16}c_{27}=c_{11}, \ S_{16}c_{26}= (\beta _{90}+ 1)c_{10}+
\beta _{91}c^{2}_{5}+ \beta _{92}c^{2}_{2}c_{6},\\
\beta _{94}= 1, \ \beta _{93}= 0, \ \beta _{95}= 0, \ \beta _{96}= 0.
\end{multline*}
From these relations and from (\ref{eq:c2p10f7}) it follows that
$\beta _{36}= 0$,
$\beta _{11}= 1$; from (\ref{eq:c2p10f8}) it follows that
$\beta _{10}= 0$; from
(\ref{eq:c2p10f2}) it follows that $\beta _{12}= 0$.

Let us use the operation $S_{8,8}$, this gives:
\begin{multline*}
S_{8,8}\varphi _{14}= 0= u_{1}S_{8,8}c_{27}+ u_{2}(S_{8,8}c_{26}+
\beta _{90}c_{10}+ \beta _{91}c^{2}_{5}+ \beta _{92}c^{2}_{2}c_{6})+\\
+ u_{3}(\beta _{16}c^{4}_{2}+ \beta _{5}c^{2}_{4})+
\varphi _{6}\beta _{73}+ \varphi _{3}c_{6}\beta _{35}+
u_{4}c^{2}_{2}\beta _{54}.
\end{multline*}
Hence
\begin{multline*}
S_{8,8}c_{27}= 0, \ S_{8,8}c_{26}= \beta _{90}c_{10}+
\beta _{91}c^{2}_{5}+ \beta _{92}c^{2}_{2}c_{6}, \\
\beta _{16}= 0,\ \beta _{5}=0,\ \beta _{73}= 0, \ \beta _{35}= 0,
\ \beta _{54}= 0.
\end{multline*}
Using the operation $S_{14}$ we obtain:
\begin{multline*}
S_{14}\varphi _{14}= \varphi _{7}= u_{1}S_{14}c_{27}+
u_{2}S_{14}c_{26}+ u_{2}c_{4}c_{8}+ u_{3}c_{2}c_{8}+
u_{4}c_{2}c_{4}+\\
+u_{4}c_{6}+ u_{3}(c_{10}+ c^{2}_{5}+ c^{2}_{2}c_{6})+
\varphi _{3}(c^{2}_{4}+ c^{4}_{2})+ u_{2}(\beta _{86}c_{12}+
\beta _{87}c^{2}_{6}+\\
+\beta _{88}c^{2}_{2}c^{2}_{4}+ \beta _{89}c^{6}_{2})+
u_{3}(c_{10}(\beta _{3}+ \beta _{90})+ c^{2}_{5}(\beta _{2}+
\beta _{91})+ \\
+c^{2}_{2}c_{6}(\beta _{23}+ \beta _{92}))+
\varphi _{3}(\beta _{31}c^{2}_{4}+ \beta _{38}c^{4}_{2})+
u_{4}c_{6}\beta _{53}+ \varphi _{5}c^{2}_{2}\beta _{64}+
\varphi _{7}\beta _{81}.
\end{multline*}
It follows from the relation
\begin{multline}\label{eq:c2p10f9}
S_{14}c_{27}= c_{13}, \ S_{14}c_{26}= (\beta _{86}+ 1)c_{12}+
\beta _{87}c^{2}_{6}+ \beta _{88}c^{2}_{2}c^{2}_{4}+
\beta _{89}c^{6}_{2};\\
\beta _{3}= \beta _{90}, \ \beta _{2}= \beta _{91}, \
\beta _{23}= \beta _{92},\ \beta _{31}, \ \beta _{38}, \ \beta _{53},
\ \beta _{64}, \ \beta _{81}= 0;$$
\end{multline}
Let us apply the operation $S_{12}$, we have:
\begin{multline*}
S_{12}\varphi _{14}= u_{5}= u_{1}S_{12}c_{27}+ u_{2}(S_{12}c_{26}+
\beta _{82}c^{2}_{2}c_{10}+ \beta _{83}c^{2}_{2}c^{2}_{5}+
\beta _{84}c^{2}_{4}c_{6}+\\
+ \beta _{85}c^{4}_{2}c_{6})+
u_{3}(c_{12}(1+ \beta _{86})+ c^{2}_{6}(\beta _{9}+ \beta _{87}) +
c^{2}_{2}c^{2}_{4}(\beta _{25}+ \beta _{88})+\\
+ c^{6}_{2}(\beta _{20}+ \beta _{89}))+
\varphi _{3}(c_{10}(1+ \beta _{30}+ \beta _{90})+
c^{2}_{5}(\beta _{29}+ \beta _{91})+ (\beta _{79}+\\
+ \beta _{92})c^{2}_{2}c_{6})+ u_{4}(c^{2}_{4}\beta _{51}+
c^{4}_{2}\beta _{59})+ \varphi _{5}c_{6}\beta _{65}+
\varphi _{6}c^{2}_{2}\beta _{78}+ u_{5}\beta _{86}.
\end{multline*}
This gives:
\begin{multline}\label{eq:c2p10f10}
S_{12}c_{27}= 0, \ S_{12}c_{26}=\beta _{82}c^{2}_{2}c_{10}+
\beta _{83}c^{2}_{2}c^{2}_{5}+ \beta _{84}c^{2}_{4}c_{6}+
\beta _{85}c^{4}_{2}c_{6};\\
\beta _{9}= \beta _{87}, \ \beta _{25}=\beta _{88}, \
\beta _{20}= \beta _{89}, \ \beta _{29}=\beta _{91}, \ \beta _{86}= 1,\\
\beta _{78}, \ \beta _{65}, \ \beta _{51}, \ \beta _{59}= 0; \
\beta _{79}= \beta _{92}, \ \beta _{30}+ \beta _{90}= 1.
\end{multline}
From these relations and from (\ref{eq:c2p10f1}) it follows
that $\beta _{1}= 0$,
from (\ref{eq:c2p10f5}) it follows that $\beta _{4}= 0$.

Calculating the action of the operation $S_{6,6}$ on $\varphi_{14}$ we
obtain:
\begin{multline*}
S_{6,6}\varphi _{14}= 0= u_{1}S_{6,6}c_{27}+
u_{2}c_{4}(c_{10}+ c^{2}_{5}+c^{2}_{2}c_{6})+
u_{3}c_{2}(c_{10}+c^{2}_{5}+c^{2}_{2}c_{6})+\\
+ \varphi _{3}(c_{10}+ c^{2}_{5}+ c^{2}_{2}c_{6}) +
u_{2}c_{2}c_{4}c_{8} + u_{3}c_{2}c_{4}c_{6}+
u_{4}c^{2}_{2}c_{4}+ \varphi _{5}c_{2}c_{4}+\\
+ u_{2}(S_{6,6}c_{26}+ c_{14}+
c^{4}_{2}c_{6}(1+ \beta _{85})+ c^{2}_{4}c_{6}\beta _{84}+
c^{2}_{2}c_{10}(\beta _{57}+ \beta _{82})+\\
+ c^{2}_{2}c^{2}_{5}(\beta _{60}+ \beta _{83}))+
u_{3}(c^{2}_{6}\beta _{70}+ c_{12}\beta _{9}+
c^{6}_{2}(\beta _{18}+ \beta _{69})+c^{2}_{2}c^{2}_{4}(\beta _{23}+\\
+ \beta _{24}+ \beta _{68}))+ \varphi _{3}(c_{10}(\beta _{33}+
\beta _{76} + \beta _{90}) + c^{2}_{5}(\beta _{32} + \beta _{79}+
\beta _{91})+\\
+ c^{2}_{2}c_{6}(\beta _{39}+ \beta _{44}+\beta _{92}))+
u_{4}(c^{2}_{4}(\beta _{52}+\beta _{84})+
c^{4}_{2}(\beta _{58}+\beta _{85}))+\\
+ u_{5}\beta _{87}+ \varphi _{5}c_{6}\beta _{70}+
\varphi _{6}c^{2}_{2}(\beta _{92}+ \beta _{77}).
\end{multline*}
Hence,
\begin{multline}\label{eq:c2p10f11}
S_{6,6}c_{26}= c_{14}+ c^{4}_{2}c_{6}(1+ \beta _{85})+
c^{2}_{4}c_{6}\beta _{84}+ c^{2}_{2}c_{10}(\beta _{57}+
\beta _{82})+ c^{2}_{2}c^{2}_{5}(\beta _{60}+ \beta _{83});\\
S_{6,6}c_{27}= c_{5}(c_{10}+c^{2}_{5}+ c^{2}_{2}c_{6})+
c_{2}c_{4}c_{8},\beta _{18}= \beta _{69}, \ \beta _{23}+
\beta _{24}= \beta _{68},\\
\beta _{33}+ \beta _{76}= \beta _{90}, \ \beta _{32}+ \beta _{79}=
\beta _{91}, \ \beta _{39}+ \beta _{44}= \beta _{92},\\
\beta _{52}= \beta _{84}, \ \beta _{58}= \beta _{85}, \
\beta _{92}= \beta _{77}, \ \beta _{70}, \ \beta _{9}, \
\beta _{87}= 0.
\end{multline}
Applying the operation $S_{10}$ we obtain:
\begin{multline*}
S_{10}\varphi _{14}= \varphi _{9}= u_{1}S_{10}c_{27}+
u_{2}c_{6}c_{10}\beta _{76} + u_{3}c_{2}c_{4}c_{8} +
u_{3}c_{6}c_{8} + u_{4}c_{2}c^{2}_{4}+ \\
+ u_{4}c_{4}c_{6}+ \varphi _{3}c_{4}c_{8}+ u_{3}c_{6}c^{2}_{4}+
u_{3}c_{6}c^{4}_{2}+ \varphi _{5}c^{2}_{4}\beta _{71}+
\varphi _{6}c_{6}\beta _{76}+\\
+u_{2}(S_{10}c_{26}+\beta _{74}c^{4}_{4}+ \beta _{75}c^{8}_{2}+
\beta _{77}c^{2}_{2}c^{2}_{6}+ \beta _{79}c^{2}_{5}c_{6} +
\beta _{80}c^{4}_{2}c^{2}_{4})+\\
+ u_{3}(c_{14}(1+ \beta _{3})+ c^{2}_{2}c_{10}(1+ \beta _{3}+
\beta _{82}) + c^{2}_{2}c^{2}_{5}(\beta _{2} + \beta _{14} +
\beta _{83}) + c^{2}_{4}c_{6}(\beta _{7}+\\
+ \beta _{84})+ c^{4}_{2}c_{6}(\beta _{19}+ \beta _{23}+
\beta _{85}))+\varphi _{3}(c_{12}\beta _{30}+c^{2}_{6}\beta _{33}+
c^{6}_{2}(\beta _{40}+\beta _{89})+\\
+c^{2}_{2}c^{2}_{4}(\beta _{45}+ \beta _{88}))+
u_{4}(c_{10}(1+ \beta _{90})+ c^{2}_{5}(1+ \beta _{49}+ \beta _{91})+\\
+ c^{2}_{2}c_{6}(\beta _{57}+ \beta _{92}+1))+
\varphi _{7}c^{2}_{2}\beta _{82}+ \varphi _{9}\beta _{90}+
\varphi _{5}c^{4}_{2}\beta _{66}.
\end{multline*}
This means that:
\begin{multline}\label{eq:c2p10f12}
S_{10}c_{26}= \beta _{74}c^{4}_{4}+ \beta _{75}c^{8}_{2}+
\beta _{77}c^{2}_{2}c^{2}_{6}+ \beta _{79}c^{2}_{5}c_{6} +
\beta _{80}c^{4}_{2}c^{2}_{4},\\
S_{10}c_{27}=c_{4}c_{5}c_{8}+c_{9}c^{2}_{4}+ c_{11}c_{6}; \
\beta _{2}+ \beta _{14}= \beta _{83}, \ \beta _{7}= \beta _{84}, \
\beta _{19}+ \beta _{23}= \beta _{85},\\
\beta _{3}, \ \beta _{71}, \ \beta _{76}, \ \beta _{90}= 1, \
\beta _{30}, \
\beta _{33}, \ \beta _{82}, \ \beta _{66}=0, \ \beta _{40}= \beta _{89},
\ \beta _{45}= \beta _{88},\\
\beta _{49}+ \beta _{91}= 1, \ \beta _{57}+ \beta _{92}= 1.
\end{multline}
It follows from (\ref{eq:c2p10f12}) and (\ref{eq:c2p10f3})
that $\beta _{48}= 1,\beta _{13}= 0$,
from (\ref{eq:c2p10f4}) it follows that $\beta _{8}= 1.$

From the equality $S_{8}\varphi _{14}= \varphi _{10}$, we obtain after
cancellations:
\begin{multline*}
u_{1}c_{19}+ u_{3}(c^{4}_{4}+ c^{4}_{2}c^{2}_{4}) +
\varphi _{6}c^{2}_{4} = u_{1}S_{8}c_{27} + u_{3}c^{2}_{4}c_{8} +
u_{4}c^{3}_{4}+ u_{2}c^{2}_{4}c_{10}+\\
+ u_{2}(S_{8}c_{26}+ c^{4}_{2}c^{2}_{5}\beta _{67}+
\beta _{68}c^{2}_{2}c^{2}_{4}c_{6}+ \beta _{69}c^{6}_{2}c_{6}+
\beta _{72}c^{2}_{4}c^{2}_{5})+\\
+u_{3}(c^{8}_{2}(1+\beta _{75})+ c^{4}_{4}\beta _{74}+
c^{2}_{2}c^{2}_{6}\beta _{77}+c^{2}_{5}c_{6}\beta _{79}+
c^{4}_{2}c^{2}_{4}\beta _{80})+
\varphi _{3}(c^{2}_{2}c^{2}_{5}\beta _{83}+\\
+c^{2}_{4}c_{6}\beta _{84}+c^{4}_{2}c_{6}\beta _{85})+
u_{4}(c^{2}_{2}c^{2}_{4}\beta _{88}+ c^{6}_{2}\beta _{89})+
\varphi _{5}(c^{2}_{5}\beta _{91}+ \beta _{92}c^{2}_{2}c_{6}).
\end{multline*}
Hence,
\begin{multline*}
S_{8}c_{26}= c^{4}_{2}c^{2}_{5}\beta _{67}+
\beta _{68}c^{2}_{2}c^{2}_{4}c_{6}+ \beta _{69}c^{6}_{2}c_{6}+
\beta _{72}c^{2}_{4}c^{2}_{5},\\
\beta _{77}, \ \beta _{79}, \ \beta _{80}, \ \beta _{83}, \
\beta _{84}, \
\beta _{85}, \ \beta _{88}, \ \beta _{89}, \ \beta _{91}, \
\beta _{92}= 0,\\
\beta _{74}, \ \beta _{75}=1, \ S_{8}c_{27}= c_{19}+ c_{11}c^{2}_{4}.
\end{multline*}
From these relations and from (\ref{eq:c2p10f12}) we get $\beta _{57},
\beta _{49}= 1$, $\beta _{45},\beta _{40},\beta _{7}=0$; from
(\ref{eq:c2p10f11}) it
follows that $\beta _{58},\beta _{52},\beta _{32}= 0$; from
(\ref{eq:c2p10f10}) it
follows that $\beta _{29},\beta _{20},\beta _{25}= 0$; from
(\ref{eq:c2p10f9}) it
follows that $\beta _{2}, \beta _{23}= 0$; from (\ref{eq:c2p10f12})
it follows that
$\beta _{14},\beta _{19}= 0$; from (\ref{eq:c2p10f6}) it follows that
$\beta _{15}= 0$; from (\ref{eq:c2p10f6}) it follows also that
$\beta _{50}=0.$

Using the equality $S_{6}\varphi _{14}= \varphi _{11}$, we obtain after
the cancellation:
\begin{multline*}
u_{1}c_{21}+ u_{3}c^{4}_{2}(c_{10}+ c^{2}_{5}+ c^{2}_{2}c_{6})+
\varphi _{6}(c_{10}+ c^{2}_{5}+ c^{2}_{2}c_{6})=
u_{1}S_{6}c_{27}+u_{3}c_{2}c^{2}_{8}+\\
+ u_{3}c_{8}(c_{10}+ c^{2}_{5}+ c^{2}_{2}c_{6})+
u_{4}c_{4}(c_{10}+ c^{2}_{5}+ c^{2}_{2}c_{6})+
u_{2}c_{10}(c_{10}+c^{2}_{5}+c^{2}_{2}c_{6})+\\
+ u_{3}c_{4}c_{6}c_{8}+ u_{4}c_{2}c_{4}c_{8}+
\varphi _{5}c_{4}c_{8}+\varphi _{3}c^{2}_{8}+
u_{2}(S_{6}c_{26}+c^{4}_{2}c_{12}+c^{4}_{2}c_{12}+\\
+c^{10}_{2}\beta _{56}+ c^{2}_{2}c^{4}_{4}\beta _{55}+
c^{2}_{2}c^{2}_{5}c_{6}\beta _{60}+ \beta _{61}c^{6}_{2}c^{2}_{4})+
u_{3}(c^{2}_{4}c^{2}_{5}(\beta _{6}+\beta _{72})+\\
+c^{4}_{2}c^{2}_{5}(\beta _{22}+\beta _{67})+
c^{2}_{2}c^{2}_{4}c_{6}\beta _{68} + c^{6}_{2}c_{6}\beta _{69}) +
\varphi _{3}(c^{4}_{2}c^{2}_{4}\beta _{42} + c^{8}_{2}\beta _{43} +\\
+ c^{4}_{4}\beta _{34}+ c^{2}_{2}c^{2}_{6}\beta _{44})+
u_{4}c^{2}_{2}c^{2}_{5}\beta _{60}+
\varphi _{5}(c^{2}_{2}c^{2}_{4}\beta _{68}+ c^{6}_{2}\beta _{69}).
\end{multline*}
Hence,
\begin{multline}\label{eq:c2p10f13}
S_{6}c_{26}= c^{2}_{2}c^{2}_{4}\beta _{55}+ c^{10}_{2}\beta _{56}+
c^{6}_{2}c^{2}_{4}\beta _{61}; \ \beta _{6}= \beta _{72}, \ \beta _{22}=
\beta _{67}.\\
S_{6}c_{27}=c_{21}+c_{5}c^{2}_{8}+ c_{11}(c_{10}+ c^{2}_{5}+
c^{2}_{2}c_{6})+ c_{9}c_{4}c_{8}, \\
\beta _{69}, \ \beta _{68}, \ \beta _{60}, \ \beta _{43}, \ \beta _{42},
\ \beta _{44}, \ \beta _{34}= 0.
\end{multline}
From these relations and from (\ref{eq:c2p10f11}) it follows that
$\beta _{18},\beta _{24},\beta _{39}= 0$.

From the equality $S_{4}\varphi _{14}= \varphi _{12}$ we get after the
cancellation:
\begin{multline*}
u_{1}S_{4}c_{27}+ u_{3}c_{8}(c_{12}+ c_{2}c_{10}+ c_{4}c_{8}+
c_{4}c^{4}_{2})+ u_{4}c_{4}(c_{12}+ c_{2}c_{10}+ c_{4}c_{8})+ \\
+ \varphi _{6}c_{4}c_{8}+ u_{3}c^{4}_{2}c_{12}+
\varphi _{3}c^{4}_{2}c_{10}+ u_{3}c_{10}(c_{10}+
c^{2}_{5}+c^{2}_{2}c_{6})+u_{4}c_{6}c_{10}+\\
+\varphi _{6}c_{12}+ \varphi _{7}c_{10}+ \varphi _{3}c^{2}_{4}c_{10}+
u_{2}(S_{4}c_{26}+ \beta _{41}c^{6}_{2}c^{2}_{5} +
\beta _{46}c^{2}_{2}c^{2}_{4}c^{2}_{5}) +\\
+ u_{3}(\beta _{55}c^{2}_{2}c^{2}_{4}+ \beta _{56}c^{10}_{2}+
\beta _{61}c^{6}_{2}c^{2}_{4})+
\varphi _{3}(\beta _{67}c^{4}_{2}c^{2}_{5}+
\beta _{72}c^{2}_{4}c^{2}_{5})= u_{1}c_{23}.
\end{multline*}
This means that
\begin{multline*}
S_{4}c_{27}= c_{23}+ c_{11}c_{12}+ c_{13}c_{10}+ c_{4}c_{8}c_{11},\\
S_{4}c_{26}= \beta _{41}c^{6}_{2}c^{2}_{5} +
\beta _{46}c^{2}_{2}c^{2}_{4}c^{2}_{5}; \ \beta _{67}, \
\beta _{72}= 0, \ \beta _{55}, \ \beta _{56}, \ \beta _{61}= 0.
\end{multline*}
From these relations and from (\ref{eq:c2p10f13}) it follows that
$\beta _{6},\beta _{22}= 0.$

To determine the coefficients $\beta _{41}$ and $\beta _{46}$ let us
apply the operation $S_{5,5}$.
\begin{multline*}
S_{5,5}\varphi _{14}= \varphi _{9}= u_{1}S_{5,5}c_{27}+
u_{3}(c_{2}c_{4}c_{8}+ c_{6}c_{8}+ c^{2}_{4}c_{6} +
c^{4}_{2}c_{6})+ \varphi _{3}c_{4}c_{8}+\\
+ u_{2}c_{6}c_{10}+ u_{4}(c_{2}c^{2}_{4}+ c_{4}c_{6})+
\varphi _{5}c^{2}_{4}+ \varphi _{6}c_{6}+ \varphi _{9}+\\
+ u_{2}(S_{5,5}c_{26}+ c^{8}_{2}+c^{4}_{4})+
\varphi _{3}(\beta _{41}c^{6}_{2}+ \beta _{46}c^{2}_{2}c^{2}_{4}).
\end{multline*}
Hence,
\[
\beta _{41}, \ \beta _{46}= 0; \ S_{5,5}c_{27}= c_{4}c_{5}c_{8}+
c_{9}c^{2}_{4}+ c_{11}c_{6}, \ S_{5,5}c_{26}= c^{8}_{2}+ c^{4}_{4}.
\]
To determine the coefficient $\beta _{21}$ we apply the operation
$S_{4,4}$:
\begin{multline*}
S_{4,4}\varphi _{14}= 0= u_{1}S_{4,4}c_{27}+
u_{2}(c^{2}_{4}+ c^{4}_{2})c_{10}+ u_{3}(c^{2}_{4}c_{8}+ c^{2}_{8} +
c^{8}_{2} +c_{6}c_{10}+\\
+ c^{4}_{2}c^{2}_{4})+ u_{4}(c^{3}_{4}+ c_{2}c_{10}+ c_{4}c_{8}+
c_{4}c^{4}_{2})+\varphi _{6}(c_{8}+c^{4}_{2}+c^{2}_{4})+
\varphi _{5}c_{10}+\\
+ u_{3}c^{8}_{2}\beta _{21}+ u_{2}(S_{4,4}c_{26}+ c_{18}).
\end{multline*}
It follows from these relations that:
\[
\beta _{21}= 0; \ S_{4,4}c_{27}= c_{11}(c_{8}+ c^{2}_{4}+ c^{4}_{2})+
c_{9}c_{10}, \ S_{4,4}c_{26}= c_{18}.
\]
To determine $\beta _{17}$ let us use the operation $S_{4,4,4,4}$:
\[
S_{4,4,4,4}\varphi _{14}= \varphi _{6}= u_{1}S_{4,4,4,4}c_{27}+
u_{2}S_{4,4,4,4}c_{26}+ u_{3}c^{4}_{2}(1+\beta _{17})+\varphi _{6}.
\]
So,
\[
\beta _{17}= 1; \ S_{4,4,4,4}c_{27}= S_{4,4,4,4}c_{26}= 0.
\]
To determine the coefficient $\beta_{26}$ we apply the operation
$S_{2,2,2,2,2,2,2,2,2,2,2,2}$:
\[
S_{2,2,2,2,2,2,2,2,2,2,2,2}\varphi _{14}= 0= u_{3}(1+ \beta _{26}),
\]
hence, $\beta _{26}= 1$.
The final form of the projection of the element $\Phi _{14}$ is the
following:
\begin{multline*}
\varphi _{14}= u_{1}c_{27}+ u_{3}c_{8}c_{16}+ u_{4}c_{4}c_{16}+
u_{5}c_{4}c_{8}+ u_{2}c_{26}+ u_{3}(c_{24}+ c_{10}c_{14}+\\
+ c^{4}_{6}+ c_{6}c_{18}+ c^{4}_{2}c^{4}_{4}+ c^{12}_{2})+
\varphi _{3}c_{22}+ u_{4}(c_{20}+c^{2}_{10}+c^{2}_{5}c_{10}+
c^{2}_{2}c_{6}c_{10})+\\
+ \varphi _{5}(c_{18}+ c^{2}_{4}c_{10})+ \varphi _{6}(c^{4}_{4}+
c^{8}_{2}+ c_{6}c_{10})+ u_{5}c_{12}+ \varphi _{9}c_{10}+
\varphi _{10}c^{4}_{2}.
\end{multline*}
Using this decomposition we calculate the action of the operation
$S_{\omega }$ on the elements $c_{27}$, and $c_{26}$.

\section{Matrix Massey products in MASS}

If elements $\xi $, $\eta $, $\zeta $ belong to $E^{0,1,t}_{2}$,
for $t< 106$ then the following triple Massey products are defined
and almost all of them contain zero: $<\xi , h_{0},\eta >$,
$<\xi , h_{0},\zeta >$, $<\zeta , h_{0},\eta >$. Hence for all such
triples the matrix Massey products of the following two types are
defined: $<\xi , h_{0},\eta ,h_{0}>$,
$< \xi , h_{0}, ( \eta ,\zeta ) ,\begin{pmatrix} \zeta \\
\eta \end{pmatrix}>$.
Really, let $c_{\zeta,\eta }$ be the element belonging to $E^{0,0,t}_{1}$
which is defined uniquely up to cycles of the differential $d_{1}$,
by the property: $d_{1}(c_{\zeta,\eta }) \in  <\zeta, h_{0},\eta >$.
Let also $h_{\zeta }$ denotes the element belonging to $E^{1,0,t}_{1}$,
and having the property $d_{1}(h_{\zeta }) = h_{0}\zeta $. Then we have:
\[
<\xi, h_{0}, \eta,h_{0} > = \left( \begin{array}{lcccr}
0 & \xi & h_{\xi } & c_{\xi ,\eta } & *  \\
0 & 0 & h_{0} & h_{\eta } & 0 \\
0 & 0 & 0 & \eta & h_{\eta } \\
0 & 0 & 0 & 0 & h_{0} \\
\end{array} \right)
= h_{0}c_{\xi ,\eta } + h_{\xi }h_{\eta },
\]
\begin{multline*}
< \xi , h_{0}, ( \eta ,\zeta ) ,\begin{pmatrix} \zeta \\ \eta
\end{pmatrix}> =
\left( \begin{array}{lcccc}
0 & \xi & h_{\xi } & ( c_{\xi ,\eta },c_{\xi,\zeta } ) & *  \\
0 & 0 & h_{0} & ( h_{\eta },h_{\zeta } ) & c_{\eta,\zeta } \\
0 & 0 & 0 & ( \eta,\zeta ) & 0 \\
0 & 0 & 0 & 0 & \begin{pmatrix} \zeta \\ \eta \end{pmatrix}  \\
\end{array} \right) = \\
= \xi c_{\eta ,\zeta } + \zeta c_{\xi ,\eta } +\eta c_{\xi ,\zeta }.
\end{multline*}
Let us denote the first product by ${\cal A}_{\xi ,\eta }$, and the
second by ${\cal F}_{\xi ,\zeta ,\eta }$. There is the equality:
$\varphi _{3} = {\cal F}_{u_{1},u_{2},u_{3}}$, and in the notations
of the work [6] we have: $c_{1} = {\cal A} _{u_{1},u_{1}},
a_{6} = {\cal A} _{u_{1},\varphi _{3}}$ .

As a canonical representative for ${\cal F}_{u_{1},u_{i},u_{j}}$ we choose
the element $\tilde{\varphi }_{i,j}$. As a canonical representative for
$c_{u_{1},u_{j}}$ we take $c_{1,j}$. If $\xi \in  E^{0,1,t}_{2}$ has in
$E^{0,1,t}_{1}$ the following decomposition:
$\xi =\sum^{}_{i} u_{i}\tilde{c}_{i}$, where
$\tilde{c}_{i}\in  E^{0,0,t}_{1}$, then as $h_{\xi }$ we take the
following element: $h_{\xi }= \sum^{}_{i} h_{i}\tilde{c}_{i}$.
As a canonical representative for ${\cal A}_{\xi ,\xi }$ we take
$h^{2}_{\xi }$. Under these conditions the elements
${\cal F}_{\xi ,\zeta ,\eta }$ and ${\cal A}_{\xi ,\zeta }$
are defined uniquely for $t<108.$
For the simplicity let us introduce the following new notations:
\[
\begin{array}{lcl}
 old  &  notation  &  new  \\
 {\cal F}_{u_{1},u_{i},u_{j}} &  & \tilde{\varphi }_{i,j} \\
 {\cal F}_{u_{k},u_{i},u_{j}} &  & \omega _{k,i,j} \\
 {\cal F}_{u_{1},u_{i},\omega _{i,j,k}} &  & \psi_{\hat{i},j,k} \\
 {\cal F}_{u_{1},u_{i,j},\omega _{i,j,k}} &  & \psi_{\hat{i},\hat{j},k}
\end{array}
\]

\medskip
\section{The cell $E^{0,1,t}_{2}$ for $t < 108$ }

Generators are given in the Table~3. The following brief notations are
used:

\medskip
\begin{tabular}{|l|c|c|c|c|c|c|c|} \hline
complete & $\omega_{2,3,4}$ & $\omega_{2,3,5}$ & $\omega_{2,4,5}$ &
$\omega_{3,4,5}$ & $\psi_{\hat{2},3,4}$ & $\psi_{2,\hat{3},4}$ &
$\psi_{\hat{2},\hat{3},4}$ \\ \hline
brief & $\omega_{1}$ & $\omega_{2}$ & $\omega_{3}$ & $\omega_{4}$ &
$\psi_{1}$ & $\psi_{2}$ & $\psi_{3}$  \\ \hline
\end{tabular}

\medskip
\begin{tabular}{|l|c|c|c|c|c|c|} \hline
complete & $\psi_{2,3,\hat{4}}$ & $\psi_{\hat{2},3,\hat{4}}$ &
$\psi_{\hat{2},3,5}$ & $\psi_{2,\hat{3},\hat{4}}$ &
$\psi_{2,\hat{3},5}$ & $\psi_{\hat{2},\hat{3},5}$\\ \hline
brief & $\psi_{4}$ & $\psi_{5}$ & $\psi_{6}$ & $\psi_{7}$ &
$\psi_{8}$ & $\psi_{9}$ \\ \hline
\end{tabular}

\medskip

\begin{lemma} Let all the elements given below are taken from the cell
$E^{0,1,t}_{2}$ for $t < 108$. Let also the following conditions hold:
$i \neq j \neq k \neq i; i,j,k \in \{ 2,3,4,5 \}$. Let
$\xi \neq \zeta \neq \eta \neq \theta \neq \xi \neq \eta , \theta \neq
\zeta $. Then the following conditions hold:
\begin{align*}
1) & \  u_{i}\tilde{\varphi }_{j,k} + u_{j}\tilde{\varphi }_{i,k} +
u_{k}\tilde{\varphi }_{i,j} = u_{1}\omega _{i,j,k}.\\
2) & \ u_{i}\tilde{\varphi }_{i,j,k} +
\tilde{\varphi }_{i,j}\tilde{\varphi }_{j,k} =
u_{1}\psi_{ \hat{i},j,k} + u_{j}u_{k}c^{2}_{1,i} .\\
3) & \ \tilde{\varphi }_{i,j}\tilde{\varphi }_{i,j,k} =
u_{1}\psi_{ \hat{i},\hat{j},k} +
u_{i}\tilde{\varphi }_{i,k}c^{2}_{1,j} +
u_{j}\tilde{\varphi }_{j,k}c^{2}_{1,i} .\\
4) & \ u_{i}\psi _{i,\hat{j},k} + u_{j}\psi_{ \hat{i},j,k} =
\tilde{\varphi }_{i,j}\omega _{i,j,k}.\\
5) & \ \tilde{\varphi }^{2}_{i,j} = u^{2}_{1}c^{2}_{i,j} +
u_{i}c^{2}_{1,j} + u_{j}c^{2}_{1,i} .\\
6) & \ u_{i}\psi_{ \hat{i},\hat{j},k} +
\tilde{\varphi }_{i,j}\psi_{ \hat{i},j,k} =
u_{1}\tilde{\varphi }_{i,k}c^{2}_{i,j} +
u_{j}\omega _{i,j,k}c^{2}_{1,i}.\\
7) & \ u_{i}\psi _{i,\hat{j},\hat{k}} +
\tilde{\varphi }_{i,j}\psi _{i,j,\hat{k}} +
\tilde{\varphi }_{i,k}\psi _{i, \hat{j},k} =
\tilde{\varphi }_{i,j,k}\omega _{i,j,k}.\\
8) & \ \omega ^{2}_{i,j,k} = u^{2}_{i}c^{2}_{j,k} +
u^{2}_{j}c^{2}_{i,k} + u^{2}_{k}c^{2}_{i,j}.\\
9) & \ \tilde{\varphi }^{2}_{i,j,k}= u^{2}_{1}c^{2}_{i,j,k} +
u^{2}_{i}c^{2}_{1,j}c^{2}_{1,k} + u^{2}_{j}c^{2}_{1,i}c^{2}_{1,k} +
u^{2}_{k}c^{2}_{1,j}c^{2}_{1,i}.\\
10)& \ \xi {\cal F}_{\zeta ,\eta ,\theta } +
\zeta {\cal F}_{\xi ,\eta ,\theta } +
\eta {\cal F}_{\xi ,\zeta ,\theta } +
\theta {\cal F}_{\xi ,\zeta ,\eta } = 0 .\\
\end{align*}
\end{lemma}
\begin{proof} It consists in direct check writing the decompositions
of participating elements. \end{proof}

This gives all the relations in $E^{0,1,t}_{2}$, $t < 108$.

\medskip
\section{The cell $E^{ 2,0,t }_{ 2 }$, for $ t < 108$}
\medskip

\begin{lemma} i) If $\xi ,\zeta ,\eta \in E^{0,1,t}_{2}$, and
$\theta \in  E^{*,*,*}_{2}$ is an arbitrary element, and expression
${\cal A}_{\xi ,\zeta }$ is defined, then the expressions
${\cal A}_{\theta \xi ,\zeta }$ and ${\cal A}_{\xi ,\theta \zeta }$
are also defined and we have the equality
$\theta {\cal A}_{\xi ,\zeta } = {\cal A}_{\theta \xi ,\zeta } =
{\cal A}_{\xi ,\theta \zeta }$.
ii) If ${\cal A}_{\xi,\zeta }$ and ${\cal A}_{\xi ,\eta }$ are defined,
then the relation ${\cal A}_{\xi,\zeta + \eta }$ is also defined, and
the following equality holds:
${\cal A}_{\xi,\zeta } + {\cal A}_{\xi,\eta } =
{\cal A}_{\xi,\zeta + \eta }$.
\end{lemma}
\begin{proof} i) If ${\cal A}_{\xi ,\zeta }$ is defined by the expression:
\[
<\xi, h_{0}, \eta,h_{0} > = \left( \begin{array}{lcccr}
0 & \xi & h_{\xi } & c_{\xi ,\eta } & *  \\
0 & 0 & h_{0} & h_{\eta } & 0 \\
0 & 0 & 0 & \eta & h_{\eta } \\
0 & 0 & 0 & 0 & h_{0} \\
\end{array} \right)
= h_{0}c_{\xi ,\eta } + h_{\xi }h_{\eta },
\]
then ${\cal A}_{\theta \xi ,\zeta }$ can be given by the formula:
\[
<\theta \xi, h_{0}, \eta,h_{0} > = \left( \begin{array}{lcccr}
0 & \theta \xi & \theta h_{\xi } & \theta c_{\xi ,\eta } & *  \\
0 & 0 & h_{0} & h_{\eta } & 0 \\
0 & 0 & 0 & \eta & h_{\eta } \\
0 & 0 & 0 & 0 & h_{0} \\
\end{array} \right)
= \theta (h_{0}c_{\xi ,\eta } + h_{\xi }h_{\eta }),
\]
ii) is proved analogously. \end{proof}

\begin{lemma} Let $\sum^{}_{i} \xi _{i}\zeta _{i} = 0$ be one of the
relations of Lemma~1, where the elements
$\xi _{i},\zeta _{i} \in  E^{0,1,t}_{2}$ are such that the sum of their
$t$-gradings is less than 108, and the elements
${\cal A}_{\xi _{i},\zeta _{i}}$ are defined for all i. Then
$\sum^{}_{i} {\cal A}_{\xi _{i},\zeta _{i}} = 0$.
\end{lemma}
\begin{proof} Let us consider for example the first of relations:
\[
u_{i}\tilde{\varphi }_{j,k} + u_{j}\tilde{\varphi }_{i,k} +
u_{k}\tilde{\varphi }_{i,j} = u_{1}\omega _{i,j,k}.
\]
Because of the equality
$\tilde{\varphi }_{j,k}= u_{1}c_{j,k}+ u_{k}c_{1,j}+ u_{j}c_{1,k}$,
we have $h\tilde{\varphi }_{j,k}= h_{1}c_{j,k}+ h_{k}c_{1,j}+h_{j}c_{1,k}$
and
$c_{u_{i},}\tilde{\varphi }_{j,k}= c_{i,j,k}+ c_{1,i}c_{j,k}$.
Because of equality
$\omega _{i,j,k} = u_{i}c_{j,k} + u_{j}c_{i,k}+ u_{k}c_{i,j}$,
we have $h_{\omega _{i,j,k}}= h_{i}c_{j,k}+ h_{j}c_{i,k}+ h_{k}c_{i,j}$,
and
$c_{u_{1},\omega _{i,j,k}}= c_{i,j,k}+ c_{1,i}c_{j,k}+
c_{1,j}c_{i,k}+ c_{1,k}c_{i,j}$. Hence,
\begin{align*}
{\cal A}_{u_{i},}\tilde{\varphi }_{j,k} =
& h_{0}(c_{i,j,k} + c_{1,i}c_{j,k}) + h_{i}(h_{1}c_{j,k}+
h_{k}c_{1,j}+ h_{j}c_{1,k}), \\
{\cal A}_{u_{j},}\tilde{\varphi }_{i,k} = &
h_{0}(c_{i,j,k} + c_{1,j}c_{i,k}) + h_{j}(h_{1}c_{i,k}+
h_{k}c_{1,i}+ h_{i}c_{1,k}), \\
{\cal A}_{u_{k},}\tilde{\varphi }_{i,j} = & h_{0}(c_{i,j,k} +
c_{1,k}c_{i,j}) + h_{k}(h_{1}c_{i,j}+ h_{i}c_{1,j}+ h_{j}c_{1,i}), \\
{\cal A}_{u_{1},\omega _{i,j,k}}= & h_{0}(c_{i,j,k}+
c_{1,i}c_{j,k}+ c_{1,j}c_{i,k}+ c_{1,k}c_{i,j}) + \\
+ & h_{1}(h_{i}c_{j,k}+ h_{j}c_{i,k}+ h_{k}c_{i,j}).
\end{align*}
Adding these equalities we get the assertion of Lemma. The other
equalities are proved analogously.
\end{proof}

Now we can describe the generators of the cell
$E^{2,0,t}_{2} , t < 108.$
As it will be shown in Subsection~2.4, there are the following
equalities for the Massey products
$<\tilde{\varphi }_{7},h_{0},\omega _{1}>=
<\psi _{1},h_{0},\tilde{\varphi }_{6}> =
<\psi _{2},h_{0},\tilde{\varphi }_{5}>=$
$<\psi _{3},h_{0},u_{4}>=$ $=<\psi _{4},h_{0},\tilde{\varphi }_{3}>=
<\psi _{5},h_{0},u_{3}>= <\psi _{7},h_{0},u_{2}>,$
and they are not equal to zero in $E^{*,*,*}_{2}$. Equalities are
understood in $E^{*,*,*}_{2}$, indeterminacy of each product is equal
to zero in $E^{*,*,*}_{2}$. Let us call the pairs
$(\tilde{\varphi }_{7},\omega _{1})$,
$(\tilde{\varphi }_{6},\psi _{1})$,
$(\tilde{\varphi }_{5},\psi _{2})$, $(u_{4},\psi _{3})$,
$(\tilde{\varphi }_{3},\psi _{4})$, $(u_{3},\psi _{5})$,
$(u_{2},\psi_{7})$
{\it forbidden}. For each forbidden pair $\xi,\zeta $
the element ${\cal A}_{\xi,\zeta }$ is not defined.

Let us take for each non-forbidden pair
$\xi ,\zeta  \in  E^{0,1,*}_{2}$,
such that the sum of $t$-gradings of $\xi$ and $\zeta $ is less
than 108, the element ${\cal A}_{\xi,\zeta }$.
Let us delete from this set the elements ${\cal A}_{\xi,\zeta }$
which according to Lemma~2.3.2 can be expressed by the others. Let
us add one element which we shall denote by
${\cal A}_{<u_{2},\psi _{7}>+<\tilde{\varphi }_{7},\omega_{1}>}$,
and having the following decomposition:
\begin{align*}
{\cal A}_{<u_{2},\psi _{7}>+<\tilde{\varphi }_{7},\omega_{1}>} =
& h_{0}(c_{5}c_{8}c_{13}+ c_{4}c_{9}c_{13})+ h_{1}h_{3}c_{9}c_{13}+
h_{1}h_{4}c_{5}c_{13}+\\
+ & h_{2}h_{3}c_{8}(c_{13}+ c_{8}c_{5}+ c_{4}c_{9})+
h^{2}_{3}c_{2}c_{8}c_{9}+ h^{2}_{4}c_{2}c_{4}c_{5}+\\
+ & h_{2}h_{4}c_{4}(c_{13}+ c_{8}c_{5}+ c_{4}c_{9})+
h_{3}h_{4}c_{2}(c_{8}c_{5}+ c_{4}c_{9}).
\end{align*}
This gives the final system of generators of the cell $E^{2,0,t}_{2}$,
for $t < 108$. Let us describe relations.

\begin{lemma} There are the following relations:
\begin{align*}
1) & \ \xi {\cal A}_{\xi,\eta } = \eta {\cal A}_{\xi,\xi };\\
2) & \ \xi {\cal A}_{\zeta,\eta } = \zeta {\cal A}_{\xi,\eta } =
\eta {\cal A}_{\xi,\zeta };\\
3) & \ {\cal A}^{2}_{\xi,\eta } = h_{0}c^{2}_{\xi ,\eta } +
h^{2}_{\xi }h^{2}_{\eta };\\
4) & \ {\cal A}_{\xi ,\eta }{\cal A}_{\xi ,\zeta } =
h^{2}_{\xi }{\cal A}_{\eta ,\zeta } +
h_{0}{\cal A}_{\xi ,{\cal F}_{\xi,\eta,\zeta }},\xi ,\eta ,\zeta
\in  E^{0,1,t}_{2};
\end{align*}
under the condition that all the expressions appearing in the
formulas above are defined.
\end{lemma}
\begin{proof} 1) We have the formula for the first differential:
$d_{1}(h_{\xi }c_{\xi ,\eta }) = \xi {\cal A}_{\xi,\eta }+
\eta {\cal A}_{\xi,\xi }$.

2) Also we have:
$ d_{1}(h_{\xi }c_{\zeta ,\eta } + h_{\zeta }c_{\xi ,\eta }) =
\xi {\cal A}_{\zeta,\eta }+ \zeta {\cal A}_{\xi,\eta }$.

Relations 3) and 4) can be checked directly.
\end{proof}

\medskip
\section{The cell $E^{ 1,1,t }_{ 2 }$ for $ t < 108$ }
\medskip

\begin{lemma} i) The element
$\kappa \in <\tilde{\varphi }_{7},h_{0},\omega _{1}>$ is defined and not
equal to zero in $E^{1,1,t}_{2}$;

ii) There are the equalities:
$$<\tilde{\varphi }_{7},h_{0},\omega _{1}> =
<\tilde{\varphi }_{6},h_{0},\psi _{1}> =
<\tilde{\varphi }_{5},h_{0},\psi _{2}> =
 <u_{4},h_{0},\psi _{3}> = $$
$$ = <\tilde{\varphi }_{3},h_{0},\psi _{4}>=
<u_{3},h_{0},\psi _{5}>= <u_{2},h_{0},\psi _{7}>, $$
which are understood without an ambiguity because the indeterminacy
in the given dimension is equal to zero.

iii) There is the equality $h_{0}\kappa = 0$ in $E^{*,*,*}_{2}$.
\end{lemma}
\begin{proof} i) The element $\kappa$ has the following decomposition:
\begin{align*}
\kappa  =& (\alpha _{1,2}c_{11} + \alpha _{1,3}c_{9} +
\alpha _{1,4}c_{5})c_{13} + (\alpha _{2,3}c_{2}c_{8} +
\alpha _{2,4}c_{2}c_{4})c_{11}+\\
+& (\alpha _{2,3}c_{4}c_{8} + \alpha _{3,4}c_{2}c_{4})c_{9} +
(\alpha _{2,4}c_{4}c_{8} + \alpha _{3,4}c_{2}c_{8})c_{5}.
\end{align*}
Here $\alpha _{i,j} = u_{i}h_{j} + u_{j}h_{i}$. Because of
\[
d_{1}[(c_{2}c_{11} + c_{4}c_{9} + c_{5}c_{8})c_{13}] = \kappa  + \alpha _{3,4}c_{11}c^{2}_{2} + \alpha _{2,4}c_{9}c^{2}_{4} + \alpha _{2,3}c_{5}c^{2}_{8},
\]
and
\[
d_{1}(c_{11}) = \alpha _{3,4}, d_{1}(c_{9}) = \alpha _{2,4}, d_{1}(c_{5})= \alpha _{2,3},
\]
the fact is proved.

ii)
\begin{align*}
d_{1}((c_{5}c_{8}+ c_{4}c_{9})c_{13}) = &\kappa  +
<u_{2},h_{0},\psi _{7}>, \\
d_{1}((c_{5}c_{8}+ c_{2}c_{11})c_{13}) = &\kappa  +
<u_{3},h_{0},\psi _{5}>, \\
d_{1}((c_{4}c_{9}+ c_{2}c_{11})c_{13}) = &\kappa  +
<u_{4},h_{0},\psi _{3}>, \\
<\tilde{\varphi }_{6},h_{0},\psi _{1}>= &<u_{2},h_{0},\psi _{7}>, \\
<\tilde{\varphi }_{5},h_{0},\psi _{2}> = &<u_{3},h_{0},\psi _{5}>, \\
<u_{4},h_{0},\psi _{3}>= &<\tilde{\varphi }_{3},h_{0},\psi _{4}>,
\end{align*}
all the equalities are fulfilled strictly for the decompositions
of the given elements in $E^{*,*,*}_{1}$).

iii)  It follows from the formula:
\begin{align*}
d_{1}(h_{1}h_{2}c_{11}c_{13} + &h_{2}h_{3}c_{2}c_{8}c_{11} +
h_{2}h_{4}c_{2}c_{4}c_{11} + h_{1}h_{3}c_{9}c_{13} +\\
+ h_{2}h_{3}c_{4}c_{8}c_{9} + &h_{3}h_{4}c_{2}c_{4}c_{9} +
h_{1}h_{4}c_{5}c_{13} + h_{2}h_{4}c_{4}c_{5}c_{8} + \\
+ h_{3}h_{4}c_{2}c_{5}c_{8} + &h^{2}_{2}c_{4}c_{8}c_{11} +
h^{2}_{3}c_{2}c_{8}c_{9} + h^{2}_{4}c_{2}c_{4}c_{5}) = h_{0}\kappa .
\end{align*}
\end{proof}

\begin{remarka} Described elements give a complete set of generators
of $E^{*,*,t}_{2}$ for $t < 108.$ \end{remarka}

\medskip
\section{On the action of differentials $ d_{r}$ $ ( r > 1)$}
\medskip

Because of multiplicative properties it is sufficient to describe an
action of $d_{r}$ on $E^{2,0,t}_{r}$ and $E^{0,1,t}_{r}$. Remind that
the cell $E^{0,0,t}_{2}$ consists of cycles of all the differentials
$d_{r}$. If $t < 104$ in all the cells $E^{q,1,t}_{2}$ nonzero elements
have $t$-grading equal to $4m+2$, and in the cells $E^{q,2,t}_{2}$
nonzero elements have $t$-grading equal to $4m$. First "irregular"
elements are  $\kappa $, $u_{1}\kappa h^{2}_{1}$ and
$(h^{2}_{1})^{2}\kappa$; $\deg (u_{1}\kappa h^{2}_{1}) = (3,2,110)$,
$\deg [(h^{2}_{1})^{2}\kappa ] = (2,5,112)$. Hence first elements
of the cells $E^{2,0,t}_{r}$
and $E^{0,1,t}_{r}$ for which a differential $d_{r}$ can be
non-equal to zero for $r > 1$ must have  $t$-grading not less than
110.

\begin{teore} There is an isomorphism of the terms $E^{*,*,t}_{2}$
and $E^{*,*,t}_{\infty }$ of MASS up to dimension 108:
\end{teore}
\begin{proof} We need to prove only the fact that the element
$\kappa $ is an infinite cycle. In the chapter~3 it will be shown
that there exists an element $\Omega _{1} \in  MSp_{49}$,
having the order 2 and whose projection to the term  $E_{2}$
of the Adams-Novikov spectral sequence is an element
(having the same notation) $\Omega _1 \in E^{1,50}_{2}$.
This last element has the order 2 and projects into the element
$\omega _{1} \in  E^{0,1,50}_{\infty }$ in MASS. Hence in the term
$E_{2}$ of the Adams-Novikov spectral sequence the following
Massey product $<\Omega _{1},2,\Phi _{7}>$ is defined and is
associated to the element $\kappa  \in  <\omega _{1},h_{0},\phi _{7}>$
defined in MASS. All the elements in $E^{0,0,*}_{2}$ are
infinite cycles. Hence for the last product all the conditions of
Theorem~3 of the work~[11] about the convergence of Massey
products in spectral sequences are fulfilled. Hence the element
$\kappa $ is an infinite cycle.
\end{proof}

\medskip                          
\medskip
\chapter[The Adams-Novikov spectral sequence]{The Adams-Novikov
spectral sequence for $t - s< 56$}
\section{Algebraic structure}

In the work \cite{V2} of the second author the ring $MSp_{*}$ was
calculated
up to dimensions $* < 32$. We continue these calculations and
compute this ring in dimensions  $* < 56$ (up to some integer relations).
On the base
of these calculations we construct the element
$\Omega _{49} \in MSp_{49}$ of the order 2 and with the following
properties: $\theta ^{2}_{1}\Omega _{1}$ = 0 and
$0 \in <\theta _{1},2,\Omega _{1}>$ in $MSp_{*}$. Let us determine
the term $E_{2}$ of the Adams-Novikov spectral sequence, which is
associated to the term $E_{\infty }$ of the MASS  which is calculated
up to $t < 106$ in the Chapter~2. We also calculate the action of the
differentials in the Adams-Novikov spectral sequence for
$32 < t-s < 52.$ These results are given in Tables~6 and 7.

1. We follow the work \cite{V2} and denote by $\pi ^{1}_{i}(x)$ the
projection
of the element $x \in MSp_{*}$ (if $x$ is in the i-th module of
filtration) into the term $E_{2}$ of the Adams-Novikov spectral sequence.
Let us introduce the following notations for the projections of the elements
$\theta _{1} \in MSp_{1}$, $\Phi _{i} \in MSp_{8i-3}$:
\[
U_{1} = \pi ^{1}_{1}(\theta _{1}) \in  E^{1,2}_{2}, \ U_{i+2} =
\pi ^{1}_{1}(\Phi _{2}i) \in  E^{1,2(2^{i}-1)}_{2},
\]
\[
\Phi _{j} = \pi ^{1}_{1}(\Phi _{j}) \in  E^{1,8j-2}_{2}, \
\mbox{if } \ j \neq  2^{i}.
\]
We have: $2U_{i} = 0, \ 2\Phi _{j} = 0$ and $U_{i},\ \Phi _{j}$ are the
infinite cycles of the Adams-Novikov spectral sequence.

2. We follow the work \cite{V2} and denote by $\pi ^{2}_{i}(x)$ the
projection
of the element $x \in  E_{2} \cong {\rm Ext}_{A}(BP^{*}(MSp),BP^{*})$,
(here we put $A = A^{BP}$) into the i-th line of the term
$E_{\infty }$ of MASS if $x$ is in the i-th module of filtration
corresponding to MASS. For the shortening let us introduce new
notations of some generators of $E^{2,0,t}_{\infty }$, $t < 56$, results
are given in Table~5. Let us consider the elements
$z_{i},y_{j}$  $(1 \leq{i,j} \leq 8)$, having the following
projections into the term $E_{\infty }$ of MASS (notations for the
elements of MASS are taken from Table~5):
\begin{align*}
&\pi ^{2}_{2}(z_{1}) = a_{1}, \ \pi ^{2}_{2}(z_{2}) =
a_{2}, \ \pi ^{2}_{2}(z_{3}) = a_{3}, \ \pi ^{2}_{2}(z_{4}) = a_{4},\\
&\pi ^{2}_{0}(y_{4}) = e_{4}, \ \pi ^{2}_{2}(z_{5}) =
a_{5}, \ \pi ^{2}_{2}(z_{6}) = a_{6}, \ \pi ^{2}_{0}(y_{6}) = c_{6},\\
&\pi ^{2}_{2}(z_{7}) = a_{7}, \ \pi ^{2}_{2}(y_{7}) =
b_{7}, \ \pi ^{2}_{2}(z_{8}) = a_{8}, \ \pi ^{2}_{0}(y_{8}) = e_{8}.
\end{align*}
These notations differ from notations of the work \cite{V2} by
the transposition
of the elements $y_{6}$ and $z_{6}$. The action of the differentials
on these elements are described in the work \cite{V2}.
Also in the work \cite{V2} the following elements are introduced:
\[
\tau_{1} = U_{1}y_{4} + U_{2}z_{3} \in  E^{1,18}_{2}, \ \tau_{2} =
U_{1}y_{6} + U_{2}z_{5} \in  E^{1,26}_{2},
\]
\[
\tau_{3} = U_{2}y_{6} + U_{3}y_{4} \in  E^{1,30}_{2}.
\]
These elements are the infinite cycles in the Adams-Novikov spectral
sequence and define the elements $\tau_{1} \in MSp_{17}$,
$\tau_{2} \in  MSp_{25}$, $\tau_{3} \in  MSp_{29}$. Let us denote by
$F^{i}E^{*,*}_{2}$ the set of elements from $E^{*,*}_{2} \cong
{\rm Ext}^{*,*}_{A}(BP^{*}(MSp),BP^{*})$ having the filtration
corresponding MASS not less than $i$.

3. It follows from the relation $h^{2}_{0}e_{8} = a_{1}a_{7} + a^{2}_{4}$,
which is fulfilled in the term $E_{\infty }$ of MASS that in the term
$E_{2}$ of the Adams-Novikov spectral sequence we have the relation:
\begin{align*}
4y_{8} &= z_{1}z_{7}+ z^{2}_{4} + \beta _{1}4z_{8}+
\beta _{2}(z^{2}_{1}z^{3}_{2} + z^{4}_{1}y_{4}) +
\beta _{3}(z_{1}z_{3}z_{4} + z^{2}_{1}y_{6}) +\\
&+ \beta _{4}2z_{1}y_{7} + \beta _{5}2z_{2}z_{6} +
\beta _{6}2z_{1}z_{7} + \beta _{7}2z_{3}z_{5} +
\beta _{8}z^{3}_{1}z_{5} + \beta _{9}2z^{2}_{4}+\\
&+ \beta _{10}z^{2}_{1}z_{2}z_{4}+ \beta _{11}z^{2}_{1}z^{2}_{3}+
\beta _{12}z_{1}z^{2}_{2}z_{3}+ \beta _{13}z^{5}_{1}z_{3}+
\beta _{14}z^{4}_{2}+ \beta _{15}z^{4}_{1}z^{2}_{2}+\\
&+ \beta _{16}z^{8}_{1} + \beta _{17}2z^{2}_{1}z_{2}y_{4} +
\beta _{18}2z^{4}_{1}z_{4} + \beta _{19}2z^{3}_{1}z_{2}z_{3} +
\beta _{20}2z^{6}_{1}z_{2} +\\
&+ \beta _{21}2z^{2}_{1}z_{6} + \beta _{22}2z_{1}z_{3}y_{4}+
\beta _{23}2z_{1}z_{2}z_{5}+ \beta _{24}2z^{2}_{2}z_{4}+
\beta _{25}2z_{2}z^{2}_{3} +\\
&+ \beta _{26}2z^{2}_{2}y_{4} + \beta _{27}4y_{4}z_{4} +
\beta _{28}8y^{2}_{4} + \beta _{29}4z_{2}y_{6}.
\end{align*}
Let us choose the element $y_{8}$ in such a way that
$\beta _{28} = 0$ (changing $y_{8}$ for
$y^{*}_{8} = y_{8} + 2\beta _{28}y^{2}_{4}$).
Let us multiply by $U_{1}$ both parts of studying equality, then we have:
\begin{multline*}
\beta _{2}z^{4}_{1}\tau_{1} + \beta _{3}z^{2}_{1}\tau_{2} +
(\beta _{8}+ \beta _{10})z^{3}_{1}z_{5}U_{1} + \\
+ (\beta _{11}+ \beta _{12}+ \beta _{14})z^{4}_{2}U_{1} +
(\beta _{13}+ \beta _{15})z^{5}_{1}z_{3}U_{1} +
\beta _{16}z^{8}_{1}U_{1} = 0.
\end{multline*}
we obtain that $\beta _{2}$, $\beta _{3}$, $\beta _{16} \equiv 0 \mod{2}$;
$\beta _{8}\equiv  \beta _{10}$, $\beta _{13}\equiv  \beta _{15} \mod{2}$;
$\beta _{11}+ \beta _{12}+ \beta _{14} \equiv  0 \mod{2}$.
Not loosing the generality in the last relation we can have
$\beta _{11} \equiv  \beta _{12} \mod{2}$, then
$\beta _{14} \equiv 0 \mod{2}$.
We can choose the element $z_{7}$ in such a way that the action of
the differential $d_{3}(z_{7}) = U_{1}U^{2}_{3}$ and the relation
$U_{1}z_{7} = U_{3}z_{4}$ are conserved and our equation takes the
form:
\begin{align*}
4y_{8} =& z_{1}z_{7} + z^{2}_{4} +
\beta _{1}4z_{8}\hbox{ + }\beta _{5}2z_{2}z_{6} +
\beta _{7}2z_{3}z_{5} + \beta _{8}(z^{3}_{1}z_{5} +
z^{2}_{1}z_{2}z_{4}) +\\
+& \beta _{9}2z^{2}_{4} + \beta _{11}(z^{2}_{1}z^{2}_{3} +
z_{1}z^{2}_{2}z_{3}) + \beta _{13}(z^{5}_{1}z_{3}+
z^{4}_{1}z^{2}_{2}) + \beta _{14}2z^{4}_{2} + \\
+& \beta _{24}2z^{2}_{2}z_{4}+ \beta _{25}2z_{2}z^{2}_{3} +
\beta _{26}2z^{2}_{2}y_{4} + \beta _{27}4y_{4}z_{4} +
\beta _{29}4z_{2}y_{6}.
\end{align*}
Let us apply the operation $S_{4,4}$ to the both parts of the last
equality
$$S_{4,4}y_{8} \equiv ((S_{2}z_{4})/{2})^{2} (1 + 2\beta _{9}) \mod{4}.$$
We can suppose (changing $y_{8}$ if necessary), that
$S_{4,4}y_{8} \equiv  ((S_{2}z_{4})/{2})^{2} \mod{4}$
Hence $\beta _{9} \equiv 0 \mod{2}$. Then by the choice of $y_{8}$
we can achieve that our relation take the form:
\begin{align*}
4y_{8} =& z_{1}z_{7} + z^{2}_{4} + \beta _{1}4z_{8} +
\beta _{5}2z_{2}z_{6} + \beta _{7}2z_{3}z_{5} +
\beta _{8}(z^{3}_{1}z_{5} + z^{2}_{1}z_{2}z_{4}) +\\
+& \beta _{11}(z^{2}_{1}z^{2}_{3} + z_{1}z^{2}_{2}z_{3}) +
\beta _{13}(z^{5}_{1}z_{3} + z^{4}_{1}z^{2}_{2}) +
\beta _{14}2z^{4}_{2} +\\
+& \beta _{24}2z^{2}_{2}z_{4} + \beta _{25}2z_{2}z^{2}_{3} +
\beta _{26}2z^{2}_{2}y_{4}+ \beta _{27}4y_{4}z_{4} +
\beta _{29}4z_{2}y_{6}.
\end{align*}
Because of the fact  $S_{7}(U_{2}U^{2}_{3}) = 0$, it follows that
$S_{7}y_{8} \equiv  0 \mod{(2z_{1})}$, and because of the fact
$S_{7}U_{4} = U_{1}$, it follows that
$S_{7}z_{8} \equiv  z_{1} \mod{(2z_{1})}$. Applying $S_{7}$ to our
equality we obtain $4z_{1} \equiv 4\beta _{1}z_{1} \mod{(8z_{1})}$,
hence $\beta _{1} \equiv 1 \mod{2}$, and so one can choose the
element $y_{8}$ in such a way that our equality takes the form:
\begin{align*}
4y_{8} + 4z_{8} =& z_{1}z_{7}+ z^{2}_{4}+ \beta _{5}2z_{2}z_{6}+
\beta _{7}2z_{3}z_{5}+ \beta _{8}(z^{3}_{1}z_{5} +
z^{2}_{1}z_{2}z_{4})+\\
+& \beta _{11}(z^{2}_{1}z^{2}_{3} + z_{1}z^{2}_{2}z_{3}) +
\beta _{13}(z^{5}_{1}z_{3} + z^{4}_{1}z^{2}_{2}) +
\beta _{14}2z^{4}_{2} +\\
+& \beta _{24}2z^{2}_{2}z_{4}+ \beta _{25}2z_{2}z^{2}_{3}+
\beta _{26}2z^{2}_{2}y_{4}+ \beta _{27}4y_{4}z_{4}+
\beta _{29}4z_{2}y_{6}.
\end{align*}
From the relations $2z_{6} = z_{1}z_{5} + z_{2}z_{4}$ and
$4z_{4} + 4y_{4} = z_{1}z_{3} + z^{2}_{2}$ it follows that
$\beta _{8}(z^{3}_{1}z_{5} + z^{2}_{1}z_{2}z_{4}) =
\beta _{8}2z^{2}_{1}z_{6}$, $\beta _{11}(z^{2}_{1}z^{2}_{3} +
z_{1}z^{2}_{2}z_{3}) =
\beta _{11}z_{1}z_{3}4(z_{4} + y_{4})$,
$\beta _{13}(z^{5}_{1}z_{3} + z^{4}_{1}z^{2}_{2}) =
\beta _{13}z^{4}_{1}4(z_{4} + y_{4})$.
Hence one can choose $z_{7}$ (not changing the relation
$U_{1}z_{7} = U_{3}z_{4}$) in such a way that our expression takes
the form:
\begin{align*}
4y_{8} + 4z_{8} =& z_{1}z_{7}+ z^{2}_{4}+ \beta _{5}2z_{2}z_{6}+
\beta _{7}2z_{3}z_{5} + \beta _{14}2z^{4}_{2}+
\beta _{24}2z^{2}_{2}z_{4}+\\
+& \beta _{25}2z_{2}z^{2}_{3}+ \beta _{26}2z^{2}_{2}y_{4}+
\beta _{27}4y_{4}z_{4}+ \beta _{29}4z_{2}y_{6}.
\end{align*}
From the relation $4z_{4} + 4y_{4} = z_{1}z_{3} + z^{2}_{2}$ it follows
that
$2\beta _{24}z^{2}_{2}z_{4} = 2\beta _{24}z_{4}(4z_{4} + 4y_{4} -
z_{1}z_{3})$,
$\beta _{26}2z^{2}_{2}y_{4} = 2\beta _{26}y_{4}(4z_{4} + 4z_{4} -
z_{1}z_{3})$,
$\beta _{14}2z^{4}_{2} = \beta _{14}2z^{2}_{2}(4z_{4} + 4y_{4} -
z_{1}z_{3})$. Hence we can choose $y_{8}$ and $z_{8}$ that our
equality takes the form:
\begin{align*}
4y_{8} + 4z_{8} =& z_{1}z_{7}+ z^{2}_{4} + \beta _{5}2z_{2}z_{6} +
\beta _{7}2z_{3}z_{5} + \beta _{25}2z_{2}z^{2}_{3} + \\
+& \beta _{27}4y_{4}z_{4}+ \beta _{29}4z_{2}y_{6}.
\end{align*}
From the conditions $S_{6}c_{6} = 1$ and $S_{6}c_{8} = c_{2}$ it follows
that $S_{6}y_{6} \equiv 1 \mod{2}$;
$S_{6}z_{8} \equiv z_{2} \mod{(2z_{2},z^{2}_{1})}$;
$S_{6}y_{8} \equiv  0 \mod{(2z_{2},z^{2}_{1})}$;
$S_{6}z_{7} \equiv  z_{1} \mod{(2z_{1})}$.
It follows from these facts that
$4z_{2} \equiv  \beta _{29}4z_{4} \mod{(8z_{4})}$, hence
$\beta _{29} \equiv 1 \mod{2}$,
and one can choose $y_{8}$ that our equality takes the form:
\[
4y_{8} + 4z_{8} + 4z_{2}y_{6} = z_{1}z_{7}+ z^{2}_{4} +
\beta _{7}2z_{3}z_{5} + \beta _{25}2z_{2}z^{2}_{3} +
\beta _{27}4y_{4}z_{4}+ \beta _{5}2z_{2}z_{6}.
\]
From the relations $S_{4}c_{4} = 1$, $S_{4}c_{6} = 0$,
$S_{4}c_{8} = c_{4}$
we obtain
\begin{align*}
S_{4}z_{6} \equiv z_{2} \mod{(2z_{2},z^{2}_{1})}; \
S_{4}z_{4} \equiv  0 \mod{2}; \
S_{4}z_{7} \equiv  0 \mod{(2z_{3},z_{1}z_{2},2z^{3}_{1})};\\
S_{4}y_{6} \equiv  0 \mod{(2z_{2},z^{2}_{1})}; \
S_{4}y_{8} \equiv  0
\mod{(2z_{4},2y_{4},2z^{2}_{1}z_{2},z^{2}_{2},z_{1}z_{3},z^{4}_{1})};\\
S_{4}z_{8} \equiv  z_{4}
\mod{(2z_{4},2y_{4},2z^{2}_{1}z_{2},z^{2}_{2},z_{1}z_{3},z^{4}_{1})}.
\end{align*}
If we apply the operation $S_{4}$ to our equation then we get the
relation: $ 2\beta _{5}z^{2}_{2} \equiv 0 \mod{(2z^{2}_{2})}$,
hence it follows that $\beta _{5} \equiv 0 \mod{2}$, hence one can choose
$y_{8}$ in such a way that the relation holds:
\[
4y_{8} + 4z_{8} + 4z_{2}y_{6} = z_{1}z_{7}+ z^{2}_{4} +
\beta _{7}2z_{3}z_{5} + \beta _{25}2z_{2}z^{2}_{3} +
\beta _{27}4y_{4}z_{4}.
\]
From the following relations in MASS: $S_{2,2}c_{8} = 0$,
$S_{2}c_{6} = c^{2}_{2}$, $S_{2}c_{4} = c_{2}$,
$S_{2,2}c_{4} = 0$, and the condition $S_{2,2}U_{4} = 0$ we get the
relations:
\begin{align*}
S_{2,2}y_{8} &\equiv  y_{4}
\mod{(2y_{4},2z_{4},z_{1}z_{3},z^{2}_{2},2z^{2}_{1}z_{2},z^{4}_{1})};
\ S_{2,2}y_{6} \equiv  0 \mod{(2z_{2},z^{2}_{1})};\\
S_{2,2}z_{7} &\equiv  z_{3} \mod{(2z_{3},z_{1}z_{2},z^{3}_{1})};
\ S_{2}z_{5} \equiv  z_{3} \mod{(2z_{3},z_{1}z_{2},z^{3}_{1})}; \\
S_{2,2}z_{8} &\equiv  0
\mod{(2y_{4},2z_{4},z_{1}z_{3},z^{2}_{2},2z^{2}_{1}z_{2},z^{4}_{1})};
\ S_{2,2}z_{5} \equiv  0 \mod{(2z_{1})};\\
S_{2}y_{6} &\equiv  y_{4}
\mod{(2y_{4},2z_{4},z_{1}z_{3},z^{2}_{2},2z^{2}_{1}z_{2},z^{4}_{1})};
\ S_{2}z_{4} \equiv  z_{2} \mod{(2z_{2},z^{2}_{1})}; \\
S_{2,2}z_{4} &\equiv  0 \mod 4; \  S_{2}z_{3} \equiv  0 \mod{(2z_{1})}.
\end{align*}
Hence if we apply $S_{2,2}$ to our equation we get a relation:
\[
4\beta _{27}z_{4} + 4y_{4} \equiv  z_{1}z_{3} + z^{2}_{2}
\mod{(8y_{4},8z_{4},2z_{1}z_{3},2z^{2}_{2})}.
\]
Hence, $\beta _{27} \equiv 1 \mod{2}$, and we can choose $y_{8}$,
in such a way that our equation takes the form:
\[
4y_{8} + 4z_{8} +\ 4z_{2}y_{6} + 4y_{4}z_{4} = z_{1}z_{7} +
 z^{2}_{4} + \beta _{7}2z_{3}z_{5} + \beta _{25}2z_{2}z^{2}_{3}.
\]
From the conditions $S_{3}(U_{1}U_{2}U_{3}) = U^{2}_{1}U_{2}$ and
$S_{3}(U_{1}U^{2}_{3}) = 0$ it follows that
\begin{align*}
&S_{3}z_{5} \equiv  z_{2} \mod{(2z_{2},z^{2}_{1})},
S_{3}z_{4} \equiv  z_{1} \mod{(2z_{1})}, \\
&S_{3}z_{7} \equiv  0
\mod{(2z_{4},2y_{4},z^{2}_{2},z_{1}z_{3},2z^{2}_{1}z_{2},z^{4}_{1})},
\end{align*}
so the application of the operation $S_{3}$ to our equation gives
the relation:
$$2\beta _{7}z_{2}z_{3}  \equiv  0\mod{(4z_{2}z_{3})}.$$
Hence $\beta _{7} \equiv  0 \mod{2}$. Let us choose $y_{8}$ in such a
way that our equation takes the form:
\[
4y_{8} + 4z_{8} + 4z_{2}y_{6} + 4y_{4}z_{4} = z_{1}z_{7}+ z^{2}_{4} +
\beta _{25}2z_{2}z^{2}_{3}.
\]
Multiplying this equation by $z_{1}$, we obtain:
\[
(4y_{8} + 4z_{8} + 4z_{2}y_{6} + 4y_{4}z_{4})z_{1} = z^{2}_{1}z_{7}+
z^{2}_{4}z_{1} + \beta _{25}2z_{2}z^{2}_{3}z_{1}.
\]
From the relation~4 of the work \cite{V2} it follows that
$z_{1}z_{3}= 4y_{4}+ 4z_{4} - z^{2}_{2}$,
hence our equation can be rewritten in the form:
\[
4(y_{8} + z_{8}+ z_{2}y_{6} + y_{4}z_{4})z_{1}= z^{1}_{1}z_{7}+
z^{2}_{4}z_{1} + \beta _{25}2z_{2}z_{3}(4z_{4}+ 4y_{4}- z^{2}_{2}).
\]
From the relations $S_{1,1}U^{2}_{3} = 0; U_{1}z_{7} = U_{3}z_{4};
U_{1}z_{4} = U_{3}z_{1}$ it follows that
\begin{align*}
&S_{1,1}z_{4} \equiv  z_{2}\mod{(2z_{2},z^{2}_{1})};
S_{1}z_{4} \equiv 0\mod{(2z_{3},z_{1}z_{2},2z^{3}_{1})}; \\
&S_{1,1}z_{7} \equiv 0
\mod{(2z_{5},2z_{1}z_{4},2z_{2}z_{3},2z^{2}_{1}z_{3},
2z_{1}z^{2}_{2},2z^{3}_{1}z_{2},2z^{5}_{1})};\\
&S_{1,1}z_{3} \equiv  z_{1}\mod{(2z_{1})}.
\end{align*}
Applying to the studying equality the operation $S_{1,1}$, we have
$0 \equiv \beta _{25}2z_{1}z^{3}_{2}\mod{(4z_{1}z^{3}_{2})}$, or
$0 \equiv  \beta _{25} \mod{2}$. Choosing $y_{6}$ transform our
equality to the form:
\[
4(y_{8} + z_{8}+ z_{2}y_{6} + y_{4}z_{4}) = z_{1}z_{7}+ z^{2}_{4}.
\]

4. Consider the expression $U_{1}z_{8}$. Because of the equality
$d_{3}(U_{1}z_{8}) = U^{3}_{1}U_{4}$, and the identity
$u_{1}a_{8} = u_{4}a_{1}$, fulfilled in the term $E_{\infty }$ of MASS
we have the following equality in the term $E_{2}$ of the Adams-Novikov
spectral sequence:
\begin{align*}
U_{1}z_{8} &= U_{4}z_{1} + \beta _{1}z^{4}_{1}\tau_{1} +
\beta _{2}z^{2}_{2}\tau_{1} + \beta _{3}z^{2}_{1}\tau_{2} +
\beta _{4}z_{1}y_{7}U_{1} + \beta _{5}z_{1}z_{7}U_{1} +\\
&+ \beta _{6}z_{3}z_{5}U_{1} + \beta _{7}z^{3}_{1}z_{5}U_{1} +
\beta _{8}z^{4}_{2}U_{1} + \beta _{9}z^{8}_{1}U_{1} +
\beta _{10}z^{5}_{1}z_{3}U_{1}.
\end{align*}
Changing $z_{8}$ and $z_{7}$ in such a way that the equality of the item~3
and the relation $U_{1}z_{7} = U_{3}z_{4}$ are conserved, it is possible
to transform our equality to the form:
\[
U_{1}z_{8} = U_{4}z_{1} + \beta _{6}z_{3}z_{5}U_{1} +
\beta _{8}z^{4}_{2}U_{1}.
\]
Changing $z_{8}$ and $z_{4}$ in such a way that the equality of the
item~3 and the equalities 3,4,5,6,9 and 11 of the work [3] are
conserved, it is possible to transform our equality to the form:
\[
U_{1}z_{8} = U_{4}z_{1} + \beta _{6}z_{3}z_{5}U_{1}.
\]
Changing $z_{8}$ simultaneously with the changing of $y_{8}$ in order to
conserve the equality of the item~3, it is possible to transform our
equality to the form:
\[
U_{1}z_{8} = U_{4}z_{1}.
\]

5. Let us consider the expression $U_{2}y_{7}$. Because of the
equality $d_{3}(U_{2}y_{7}) = U_{1}U^{2}_{2}\Phi _{3}$ and the identity
$u_{2}b_{7} = \varphi _{3}a_{3}$, fulfilled for the term
$E_{\infty }$ of the MASS the following equation is valid in the term
$E_{2}$ of the Adams-Novikov spectral sequence:
\begin{align*}
U_{2}y_{7} &= \Phi _{3}z_{3} + \beta _{1}z^{4}_{1}\tau_{1} +
\beta _{2}z^{2}_{2}\tau_{1} + \beta _{3}z^{2}_{1}\tau_{2} +
\beta _{4}z_{1}y_{7}U_{1} + \beta _{5}z_{1}z_{7}U_{1} +\\
&+ \beta _{6}z_{3}z_{5}U_{1} + \beta _{7}z^{3}_{1}z_{5}U_{1} +
\beta _{8}z^{4}_{2}U_{1} + \beta _{9}z^{8}_{1}U_{1} +
\beta _{10}z^{5}_{1}z_{3}U_{1}.
\end{align*}
Let us multiply both parts of this equality by $U_{1}$ and use the
relations
$U_{1}z_{3} = U_{2}z_{2}$ and $U_{1}y_{7} = \Phi _{3}z_{2}$
from [3]. We get the relations $\beta _{i} \equiv  0 \mod{2}$ for
each $i$. Hence we have:
\[
U_{2}y_{7} = \Phi _{3}z_{3}.
\]

6. Let us consider the expression $U_{2}z_{7}$. Because of the
equality $d_{3}(U_{2}z_{7}) = U_{1}U_{2}U^{2}_{3}$, and the identity
$u_{2}a_{7} = u_{3}a_{5}$, fulfilled for the term $E_{\infty }$ of the
MASS the following equation is valid in the term
$E_{2}$ of the Adams-Novikov spectral sequence:
\begin{align*}
U_{2}z_{7} &= U_{3}z_{5} + \beta _{1}z^{4}_{1}\tau_{1} +
\beta _{2}z^{2}_{2}\tau_{1} + \beta _{3}z^{2}_{1}\tau_{2} +
\beta _{4}z_{1}y_{7}U_{1} + \beta _{5}z_{1}z_{7}U_{1} +\\
&+ \beta _{6}z_{3}z_{5}U_{1} + \beta _{7}z^{3}_{1}z_{5}U_{1} +
\beta _{8}z^{4}_{2}U_{1} + \beta _{9}z^{8}_{1}U_{1} +
\beta _{10}z^{5}_{1}z_{3}U_{1}.
\end{align*}
Let us multiply both parts of this equality by $U_{1}$ and use the
relations $U_{1}z_{7} = U_{3}z_{4}$ and $U_{1}z_{5} = U_{2}z_{4}$ from
[3]. We obtain the relations $\beta _{i} \equiv  0 \mod{2}$ for each
$i$. Hence we have:
\[
U_{2}z_{7} = U_{3}z_{5}.
\]

7. An element $z_{9} \in  E^{0,36}_{2}$ we choose in such a way
that $\pi ^{2}_{2}(z_{9}) = a_{9}$. From the action of the
Landweber-Novikov operation $S_{1}$: $S_{1}a_{9} = a_{8}$ it follows
that
$$S_{1}(z_{9} + \text{any element in this dimension not equal to} \
z_{9}) = z_{8} +\text{decomposables}.$$
Hence, the following conditions must be fulfilled:
$$d_{3}(z_{9}) \neq  0,$$
$$d_{3}(z_{9} + \text{any element in this dimension not equal to} \
z_{9}) \neq  0.$$
So,
\begin{align*}
d_{3}(z_{9}) &= \beta _{1}U_{1}U_{2}U_{4} +
\beta _{2}U_{1}U_{3}\Phi _{3} + \beta _{3}U^{2}_{2}(U_{1}y_{6} +
U_{3}z_{3}) + \\
&+ \beta _{4}U^{2}_{1}(U_{1}y_{8} + U_{2}z_{7}) +
\beta _{5}U_{2}U_{3}(U_{1}y_{4} + U_{2}z_{3}) +
\beta _{6}U^{3}_{1}y^{2}_{4}.
\end{align*}
Applying the operation $S_{4,4}$ to this equality and because of the
relations
$S_{4,4}z_{9} \equiv 0 \mod{2}$, $S_{4,4}y_{8} \equiv 1 \mod{2}$
we obtain $\beta _{4}U^{3}_{1} = 0$, or $\beta _{4} \equiv 0\mod{2}$.
Let us use the operation $S_{7}$, then we have
$S_{7}z_{9} \equiv  z_{2} \mod{(2z_{2},z^{2}_{1})}$,
$S_{7}U_{4} = U_{1}$. So
$d_{3}(z_{2}) = U^{2}_{1}U_{2} = \beta _{1}U^{2}_{1}U_{2}$, i.e.
$\beta _{1} \equiv  1 \mod{2}$.
Let us use the operation $S_{2,2,2,2}$; it follows from the relations
$S_{2,2,2,2}z_{9} \equiv 0 \mod{(2z_{1})}$;
$S_{2,2}y_{4} \equiv 1 \mod{2}$, that $\beta _{6}U^{3}_{1} = 0$, or
$\beta _{6} \equiv  0 \mod{2}$. Let us use the operation $S_{6}$, we
have $S_{6}z_{9} \equiv  z_{3} \mod{(2z_{3},z_{1}z_{2},2z^{3}_{1})}$.
Hence, $d_{3}(z_{3}) = U_{1}U^{2}_{2} = (\beta _{3} + 1)U_{1}U^{2}_{2}$,
so $\beta _{3} \equiv  0 \mod{2}$. Applying the operation $S_{5}$ and
having in mind
$S_{5}z_{9} \equiv 0
\mod{(z^{2}_{2},2z_{4},2y_{4},z_{1}z_{3},2z^{2}_{1}z_{2},z^{4}_{1})}$,
and also $S_{5}\Phi _{3} = U_{1}$, we obtain
$\beta _{2}U^{2}_{1}U_{3} = 0$, or $\beta _{2} \equiv  0 \mod{2}$.
From the conditions $S_{2,2}U_{4} = 0$, $S_{2,2}\Phi _{3} = U_{2}$,
$S_{2,2}z_{9} \equiv  0
\mod{(2z_{5},z_{1}z_{4},z_{2}z_{3},2z^{2}_{1}z_{3},z^{3}_{1}z_{2},
2z^{5}_{1})}$ and $S_{2,2}y_{4} \equiv  1 \mod{2}$, it follows the
relation
$\beta _{5}U_{1}U_{2}U_{3} = 0$, from which it follows that
$\beta _{5} \equiv  0 \mod{2}$. So
\[
d_{3}(z_{9}) = U_{1}U_{2}U_{4}.
\]

8. An element $y_{9} \in  E^{0,36}_{2}$ we choose so that
$\pi ^{2}_{2}(y_{9}) = b_{9}$. From the action of the Landweber-Novikov
operation $S_{2}: S_{2}a_{9} = a_{7}+b_{7}$, it follows that
$$S_{2}(y_{9} + \text{any element in this dimension not equal to }
y_{9})=$$
$$ z_{7} + y_{7} + \text{decomposables}.$$
Hence the following conditions must be fulfilled:
$$d_{3}(y_{9}) \neq  0,$$
$$d_{3}(y_{9} + \text{any element in this dimension not equal to }
y_{9}) \neq  0.$$
So,
\begin{align*}
d_{3}(y_{9}) &= U_{1}U_{3}\Phi _{3} + \beta _{1}U_{1}U_{2}U_{4} +
\beta _{2}U^{2}_{2}(U_{1}y_{6} + U_{3}z_{3}) +\\
&+ \beta _{3}U^{2}_{1}(U_{1}y_{8} + U_{2}z_{7}) +
\beta _{4}U_{2}U_{3}(U_{1}y_{4} + U_{2}z_{3}) +
\beta _{5}U^{3}_{1}y^{2}_{4}.
\end{align*}
Applying the operation $S_{4,4}$ we obtain $\beta _{3}U^{3}_{1} = 0$,
so $\beta _{3} \equiv 0 \mod{2}$. Applying the operation $S_{2,2,2,2}$
we obtain the relation $\beta _{5}U^{3}_{1} = 0$, that means that
$\beta _{5} \equiv 0 \mod{2}$. Acting by the operation $S_{6}$ we
arrive to the expression $0 = \beta _{2}U_{1}U^{2}_{2}$, i.e.
$\beta _{2} \equiv 0 \mod{2}$. From the condition $S_{2,2}b_{9} = 0$,
fulfilled in MASS the equality follows
$S_{2,2}y_{9} \equiv 0 \mod{(2z_{5},z_{1}z_{4},z_{2}z_{3},2z^{5}_{1})}$.
Hence applying the operation $S_{2,2}$ to the equality that we are
studying we obtain $\beta_{4}U_{1}U_{2}U_{3} = 0$, hence,
$\beta _{4} \equiv 0 \mod{2}$. Because of the equality $S_{7}b_{9} = 0$,
it follows that $S_{7}y_{9} \equiv  0 \mod{(2z_{2},z^{2}_{1})}$, this
gives $\beta _{1}U^{2}_{1}U_{2} = 0$, or $\beta _{1} \equiv 0 \mod{2}$.
We get the following final form of the action of the differential
$d_{3}$ on $y_{9}$:
\[
d_{3}(y_{9}) = U_{1}U_{3}\Phi _{3}.
\]

9. From the condition $h_{0}b_{9} = a_{2}a_{7} + a_{4}a_{5}$,
fulfilled in the term $E_{\infty }$ of the MASS it follows that one
can choose an element $y_{9}$ in such a way that the following
equality is fulfilled:
\begin{align*}
2y_{9} &= z_{2}z_{7} + z_{4}z_{5} + \beta _{1}2z_{4}z_{5} +
\beta _{2}2z^{3}_{3} + \beta _{3}2z_{3}(z_{2}y_{4} + z^{2}_{3}) +
\beta _{4}2z^{5}_{1}y_{4} +&\\
&+ \beta _{5}2z^{2}_{1}z_{3}y_{4} + \beta _{6}2z_{1}z^{2}_{2}y_{4} +
\beta _{7}2z^{2}_{1}y_{7} + \beta _{8}z^{3}_{1}(z_{2}y_{4} +
z^{2}_{3}) +&\\
&+ \beta _{9}2z_{1}z_{2}z_{6}+ \beta _{10}2z^{2}_{1}z_{7} +
\beta _{11}2z_{1}z^{2}_{4} + \beta _{12}2z_{1}z_{3}z_{5} +&\\
&+ \beta _{13}2z^{2}_{2}z_{5} + \beta _{14}2z^{3}_{1}z_{2}z_{4} +
\beta _{15}2z^{4}_{1}z_{5} + \beta _{16}2z_{1}z^{4}_{2} +&\\
&+ \beta _{17}2z^{9}_{1} + \beta _{18}2z^{3}_{1}y_{6} +
\beta _{19}2z^{6}_{1}z_{3} + \beta _{20}2z^{5}_{1}z^{2}_{2} +& \\
&+ \beta _{21}2z^{2}_{1}z^{2}_{2}z_{3} + \beta _{22}2z_{2}z_{3}z_{4} +
\beta _{23}z_{2}(z_{1}z_{6} + z_{2}z_{5}) +& \\
&+ \beta _{24}z_{4}(z_{1}y_{4}+ z_{2}z_{3})+
\beta _{25}z^{3}_{2}z_{3}+ \beta _{26}z^{2}_{1}z_{2}z_{5}+
\beta _{27}z^{2}_{1}z_{3}z_{4} +&\\
&+ \beta _{28}z^{3}_{1}z_{6} + \beta _{29}z_{1}z_{2}z^{2}_{3} +
\beta _{30}z_{1}z^{2}_{2}z_{4} + \beta _{31}z^{5}_{1}z_{4} +&\\
&+ \beta _{32}z^{3}_{1}z^{3}_{2} + \beta _{33}z^{4}_{1}z_{2}z_{3} +
\beta _{34}z^{7}_{1}z_{2} + \beta _{35}2z^{3}_{1}z^{2}_{3}.
\end{align*}
Let us multiply both parts of the equality by $U_{1}$. Then we have the
following identity:
\begin{align*}
0 &= \beta _{8}z^{3}_{1}z_{2}\tau_{1} +
\beta _{18}z^{3}_{1}\tau_{2} + \beta _{23}z_{1}z_{2}\tau_{2} +
\beta _{24}z_{1}z_{4}\tau_{1} + \beta _{28}U_{1}z^{3}_{1}z_{6} + \\
&+ (\beta _{25}+ \beta _{29})U_{1}z^{3}_{2}z_{3} + (\beta _{26}+
\beta _{27} + \beta _{30})U_{1}z^{2}_{1}z_{2}z_{5}+
\beta _{31}U_{1}z^{5}_{1}z_{4} + \\
&+ (\beta _{32} + \beta _{33})U_{1}z^{3}_{1}z^{3}_{2} +
\beta _{34}U_{1}z^{7}_{1}z_{2}.
\end{align*}
It follows from this identity that $\beta _{8}$, $\beta _{18}$,
$\beta _{23}$, $\beta _{24}$, $\beta _{28}$, $\beta _{31}$,
$\beta _{34} \equiv 0 \mod{2}$; $\beta _{25} \equiv \beta _{29} \mod{2}$;
$\beta _{32} \equiv  \beta _{33} \mod{2}$;
$\beta _{26} + \beta _{30} + \beta _{27} \equiv 0 \mod{2}$.
Without loss of generality we can consider that
$\beta _{30} \equiv 0 \mod{2}$, hence
$\beta _{26} \equiv \beta _{27} \mod{2}$,
and we can choose $y_{9}$ in such a way that the considering equality
takes the form:
\begin{align*}
2y_{9} &= z_{2}z_{7} + z_{4}z_{5} + \beta _{1}2z_{4}z_{5} +
\beta _{2}2z^{3}_{3} + \beta _{4}2z^{5}_{1}y_{4} +\\
&+ \beta _{5}2z^{2}_{1}z_{3}y_{4} +
\beta _{6}2z_{1}z^{2}_{2}y_{4} + \beta _{7}2z^{2}_{1}y_{7} +
\beta _{9}2z_{1}z_{2}z_{6} +\\
&+ \beta _{10}2z^{2}_{1}z_{7} + \beta _{11}2z_{1}z^{2}_{4} +
\beta _{12}2z_{1}z_{3}z_{5} + \beta _{13}2z^{2}_{2}z_{5} + \\
&+ \beta _{14}2z^{3}_{1}z_{2}z_{4} + \beta _{15}2z^{4}_{1}z_{5} +
\beta _{16}2z_{1}z^{4}_{2} + \beta _{17}2z^{9}_{1} + \\
&+ \beta _{18}2z^{3}_{1}y_{6} + \beta _{19}2z^{6}_{1}z_{3} +
\beta _{20}2z^{5}_{1}z^{2}_{2} +
\beta _{21}2z^{2}_{1}z^{2}_{2}z_{3} +\\
&+ \beta _{22}2z_{2}z_{3}z_{4} + \beta _{25}(z^{3}_{2}z_{3} +
z_{1}z_{2}z^{2}_{3}) + \beta _{35}2z^{3}_{1}z^{2}_{3} +\\
&+ \beta _{26}(z^{2}_{1}z_{2}z_{5} + z^{2}_{1}z_{3}z_{4}) +
\beta _{32}(z^{3}_{1}z^{3}_{2} + z^{4}_{1}z_{2}z_{3}).
\end{align*}
The following relations arrive from Table~2 of the work \cite{V2}:
\begin{align*}
\beta _{25}(z^{3}_{2}z_{3} + z_{1}z_{2}z^{2}_{3}) &=
\beta _{25}z_{2}z_{3}(4z_{4} + 4y_{4}),\\
\beta _{26}(z^{2}_{1}z_{2}z_{5} + z^{2}_{1}z_{3}z_{4}) &=
\beta _{26}2z_{1}z_{2}z_{6},\\
\beta _{32}(z^{3}_{1}z^{3}_{2} + z^{4}_{1}z_{2}z_{3}) &=
\beta _{32}z^{3}_{1}z_{3}(4z_{4} + 4y_{4}),\\
\beta _{5}2z^{2}_{1}z_{3}y_{4} + \beta _{6}2z_{1}z^{2}_{2}y_{4} &=
\beta _{6}2z_{1}y_{4}(4z_{4} + 4y_{4}) + (\beta _{5} -
\beta _{6})2z^{2}_{1}z_{3}y_{4},\\
\beta _{10}2z^{2}_{1}z_{7} + \beta _{11}2z_{1}z^{2}_{4} &=
\beta _{10}2z_{1}(4z_{8} + 4y_{8} + 4z_{2}y_{6} + 4z_{4}y_{4}) +\\
&+ (\beta _{10} - \beta _{11})2z_{1}z^{2}_{4},\\
\beta _{12}2z_{1}z_{3}z_{5} + \beta _{22}2z_{2}z_{3}z_{4} &=
2z_{3}y_{6}\beta _{12} + (\beta _{22} -
\beta _{12})2z_{2}z_{3}z_{4},\\
\beta _{14}2z^{3}_{1}z_{2}z_{4} + \beta _{15}2z^{4}_{1}z_{5} &=
\beta _{14}2z^{3}_{1}2y_{6} + (\beta _{15} -
\beta _{14})2z^{4}_{1}z_{5},\\
\beta _{16}2z_{1}z^{4}_{2} + \beta _{21}2z^{2}_{1}z^{2}_{2}z_{3} +
\beta _{35}2z^{3}_{1}z^{2}_{3} &=
\beta _{16}2z_{1}z^{2}_{2}(4y_{4} + 4z_{4})+\\
+ (\beta _{21} - \beta _{16})2z^{2}_{1}z_{3}(4y_{4} &+ 4z_{4}) +
(\beta _{35} - \beta _{21} + \beta _{16})2z^{3}_{1}z^{2}_{3};\\
\beta _{19}2z^{6}_{1}z_{3}+ \beta _{20}2z^{5}_{1}z^{2}_{2} &=
\beta _{19}2z^{5}_{1}(4y_{4}+ 4z_{4}) +
(\beta _{20}- \beta _{19})2z^{5}_{1}z^{2}_{2}.
\end{align*}
Hence one can choose the element $y_{9}$ in order to get the relation
(changing the notations of the elements $\beta _{i}$ for convenience):
\begin{align*}
2y_{9} &= z_{2}z_{7} + z_{4}z_{5} + \beta _{1}2z_{4}z_{5} +
\beta _{2}2z^{3}_{3} + \beta _{4}2z^{5}_{1}y_{4} +\\
&+ \beta _{5}2z^{2}_{1}z_{3}y_{4} + \beta _{7}2z^{2}_{1}y_{7} +
\beta _{9}2z_{1}z_{2}z_{6} + \beta _{11}2z_{1}z^{2}_{4} + \\
&+ \beta _{13}2z^{2}_{2}z_{5} + \beta _{15}2z^{4}_{1}z_{5} +
\beta _{17}2z^{9}_{1} + \beta _{18}2z^{3}_{1}y_{6} + \\
&+ \beta _{19}2z^{6}_{1}z_{3} + \beta _{22}2z_{2}z_{3}z_{4} +
\beta _{35}2z^{3}_{1}z^{2}_{3}.
\end{align*}
Using the relation $2y_{7} = z_{2}z_{5} + 3z_{3}z_{4}$ we can choose
the element $y_{9}$ so that $\beta _{22}$ = 0. Let us apply the
operation $S_{2,2}$ to our equality.
From the conditions
\begin{align*}
S_{2,2}y_{9} \equiv  0 \mod{(2z^{5}_{1},2z^{2}_{1}z_{3})},
S_{2}z_{5} \equiv  0 \mod{(2z^{3}_{1})},\\
S_{2}z_{4} \equiv  0 \mod{(2z^{2}_{1})}, S_{2}z_{7} \equiv  0
\mod{(2z^{5}_{1})},\\
S_{2}z_{2} \equiv  0 \mod{2},
\end{align*}
we obtain that $\beta _{4}2z^{5}_{1} \equiv  0 \mod{(4z^{5}_{1})}$,
or $\beta _{4} \equiv  0 \mod{2}$. By the choice of $y_{9}$ we
transform our equality to the form:
\begin{align*}
2y_{9} &= z_{2}z_{7} + z_{4}z_{5} + \beta _{1}2z_{4}z_{5} +
\beta _{2}2z^{3}_{3} + \beta _{5}2z^{2}_{1}z_{3}y_{4} +\\
&+ \beta _{7}2z^{2}_{1}y_{7} + \beta _{9}2z_{1}z_{2}z_{6} +
\beta _{11}2z_{1}z^{2}_{4} + \beta _{13}2z^{2}_{2}z_{5} + \\
&+ \beta _{15}2z^{4}_{1}z_{5} + \beta _{17}2z^{9}_{1} +
\beta _{18}2z^{3}_{1}y_{6} + \beta _{19}2z^{6}_{1}z_{3} +\\
&+ \beta _{35}2z^{3}_{1}z^{2}_{3}.
\end{align*}
From the condition $S_{6}y_{9} \equiv  0 \mod{(2z^{3}_{1})}$ we obtain
that $\beta _{18}2z^{3}_{1} \equiv  0 \mod{(4z^{3}_{1})}$, hence,
$\beta _{18} \equiv  0 \mod{2}$, and we can choose the element
$y_{9}$, in such a way that the corresponding summand will be equal
to zero. So, we get:
\begin{align*}
2y_{9} &= z_{2}z_{7} + z_{4}z_{5} + \beta _{1}2z_{4}z_{5} +
\beta _{2}2z^{3}_{3} + \beta _{5}2z^{2}_{1}z_{3}y_{4} + \\
&+ \beta _{7}2z^{2}_{1}y_{7} + \beta _{9}2z_{1}z_{2}z_{6} +
\beta _{11}2z_{1}z^{2}_{4} + \beta _{13}2z^{2}_{2}z_{5} + \\
&+ \beta _{15}2z^{4}_{1}z_{5} + \beta _{17}2z^{9}_{1} +
\beta _{19}2z^{6}_{1}z_{3} + \beta _{35}2z^{3}_{1}z^{2}_{3}.
\end{align*}
Multiplying our expression by $z_{1}$, and changing (according to the
formula in the item~3) the expression $z_{1}z_{2}z_{7}$ by
\[
z_{1}z_{2}z_{7} = z_{2}(4y_{8} + 4z_{8} + 4z_{2}y_{6} +
4z_{4}y_{4}) - z_{2}z^{2}_{4},
\]
and changing (according to the relation $2y_{6} = z_{1}z_{5} +
z_{2}z_{4}$ from the work [6]) the expression
$z_{1}z_{4}z_{5}$ by:
\[
z_{1}z_{4}z_{5} = z_{4}2z_{6} - z_{2}z^{2}_{4}.
\]
Now let us divide all coefficients by 2, we obtain:
\begin{align*}
z_{1}y_{9} &+ z_{4}y_{6} + 2(y_{8} + z_{8} + z_{2}y_{6} +
z_{4}y_{4})z_{2} = z_{2}z^{2}_{4} + \beta _{1}z_{1}z_{4}z_{5} +\\
&+  \beta _{2}z_{1}z^{3}_{3} + \beta _{5}z^{3}_{1}z_{3}y_{4} +
\beta _{7}z^{3}_{1}y_{7} + \beta _{9}z^{2}_{1}z_{2}z_{6} +
\beta _{11}z^{2}_{1}z^{2}_{4} +\\
&+  \beta _{13}z_{1}z^{2}_{2}z_{5} + \beta _{15}z^{5}_{1}z_{5} +
\beta _{17}z^{10}_{1} + \beta _{19}z^{7}_{1}z_{3} +
\beta _{35}z^{4}_{1}z^{2}_{3}.
\end{align*}
The rest of coefficients we determine later.

10. Let us consider the expression $U_{1}y_{9}$. From the action
of the differential $d_{3}$:
$d_{3}(U_{1}y_{9}) = U^{2}_{1}U_{3}\Phi _{3}$ and the relation
in the MASS: $u_{1}b_{9} = \varphi _{3}a_{4}$ it follows that
one can choose the element $y_{9}$ (not changing the relation
of the previous item) in such a way that the equality is
fulfilled:
\begin{align*}
U_{1}y_{9} &= \Phi _{3}z_{4} + U_{1}[\gamma _{1}z_{3}(z^{2}_{3} +
z_{2}y_{4}) + \gamma _{2}z_{2}(z_{3}z_{4} + z_{1}y_{6}) + \\
&+ \gamma _{3}z_{4}(z_{1}y_{4} + z_{2}y_{3}) +
\gamma _{4}z^{2}_{1}z_{2}(z_{1}y_{4} + z_{2}z_{3}) +
\gamma _{5}z_{2}y_{7} +  \\
&+ \gamma _{6}z_{2}z_{7} + \gamma _{7}z^{3}_{2}z_{3} +
\gamma _{8}z_{1}z_{8} + \gamma _{9}z^{2}_{1}z_{2}z_{5} +
\gamma _{10}z^{3}_{1}y_{6} +\\
&+ \gamma _{11}z^{5}_{1}z_{4} + \gamma _{12}z^{4}_{1}z^{3}_{2} +
\gamma _{13}z^{8}_{1}z_{2}] + \gamma _{14}U_{2}z_{3}z_{5}.
\end{align*}
Let us multiply this equality by $z_{1}$, and the last equality
of the item~9 we multiply by $U_{1}$ and add them up. Then we get
an equality:
\begin{align*}
(\beta _{1} + \gamma _{6} + 1)U_{1}z_{2}z^{2}_{4} +
\beta _{2}U_{1}z_{1}z^{3}_{3} + \beta _{5}U_{1}z^{3}_{1}z_{3}y_{4} +
(\beta _{7} + \\
+ \beta _{9})U_{1}z^{3}_{1}y_{7} +
\beta _{11}U_{1}z^{2}_{1}z^{2}_{4} +
\beta _{13}U_{1}z_{1}z^{2}_{2}z_{5} +
\beta _{15}U_{1}z^{5}_{1}z_{5} +\\
+ \beta _{17}U_{1}z^{10}_{1} + \beta _{19}U_{1}z^{7}_{1}z_{3} +
\beta _{35}U_{1}z^{4}_{1}z^{2}_{3} +
\gamma _{1}U_{1}z_{1}z_{3}(z^{2}_{3} +\\
+ z_{2}y_{4}) + \gamma _{2}U_{1}z_{1}z_{2}(z_{3}z_{4} + z_{1}y_{6}) +
\gamma _{3}U_{1}z_{1}z_{4}(z_{1}y_{4} +\\
+ z_{2}y_{3}) +
\gamma _{4}U_{1}z^{3}_{1}z_{2}(z_{1}y_{4} + z_{2}z_{3}) +
\gamma _{5}U_{1}z_{1}z_{2}y_{7} +\\
+ \gamma _{7}U_{1}z_{1}z^{3}_{2}z_{3} +
\gamma _{8}U_{1}z^{2}_{1}z_{8} +
\gamma _{9}U_{1}z^{3}_{1}z_{2}z_{5} +
\gamma _{10}U_{1}z^{4}_{1}y_{6} +\\
+ \gamma _{11}U_{1}z^{6}_{1}z_{4} +
\gamma _{12}U_{1}z^{5}_{1}z^{3}_{2} +
\gamma _{13}U_{1}z^{8}_{1}z_{2} +
\gamma _{14}U_{1}z_{2}z_{3}z_{5} = 0.
\end{align*}
from this we obtain that $\beta _{2}$, $\beta _{5}$,
$\beta _{11}$, $\beta _{13}$, $\beta _{15}$, $\beta _{17}$,
$\beta _{19}$, $\beta _{35}$, $\gamma _{1}$, $\gamma _{2}$,
$\gamma _{3}$, $\gamma _{4}$,
$\gamma _{5}$, $\gamma _{7}$, $\gamma _{8}$, $\gamma _{9}$,
$\gamma _{10}$, $\gamma _{11}$, $\gamma _{12}$, $\gamma _{13}$,
$\gamma _{14} \equiv 0 \mod{2}$,
$\beta _{1} + \gamma _{1} + 1 \equiv  0 \mod{2}$,
$\beta _{7} \equiv  \beta _{9} \mod{2}$.
Hence we can change the choice of $y_{9}$, in order to obtain the
relations:
\begin{align*}
2y_{9} &= z_{2}z_{7} + z_{4}z_{5} + 2(\beta  + 1)z_{4}z_{5} +
\gamma (z^{2}_{1}y_{7} + z_{1}z_{2}z_{6}),\\
U_{1}y_{9} &= \Phi _{3}z_{4} + \beta U_{1}z_{2}z_{7}.
\end{align*}
We can choose $y_{9}$ so that $\beta = 0$, then we obtain:
\[
U_{1}y_{9} = \Phi _{3}z_{4}, 2y_{9} = z_{2}z_{7} + 3z_{4}z_{5} +
\gamma (z^{2}_{1}y_{7} + z_{1}z_{2}z_{6}).
\]
Applying the operation $S_{4}$ to this equality and having in mind
the equality $S_{4}y_{7} \equiv  z_{3} \mod{(2z_{3})}$, we obtain
$\gamma z^{2}_{1}z_{3} \equiv  0
\mod{(2z^{2}_{1}z_{3},2z^{3}_{1}z_{2},2z^{5}_{1})}$,
hence it follows that $\gamma  \equiv 0 \mod{2}$, so one can choose
the element $y_{9}$ to obtain:
\[
2y_{9} = z_{2}z_{7} + 3z_{4}z_{5}.
\]

11. Let us consider the expression $U_{3}y_{6}$. From the action of
the differential $d_{3}$: $d_{3}(U_{3}y_{6}) = U^{2}_{1}U_{3}\Phi _{3}$
and the relation in the MASS: $u_{3}b_{3} = \varphi _{3}a_{4}$, we
obtain the following equality:
\begin{align*}
U_{3}y_{6} &= \Phi _{3}z_{4} + U_{1}[\beta _{1}z_{3}(z^{2}_{3} +
z_{2}y_{4}) + \beta _{2}z_{2}(z_{1}y_{6} + z_{3}z_{4}) + \\
&+ \beta _{3}z_{4}(z_{1}y_{4} + z_{2}z_{3}) +
\beta _{4}z^{2}_{1}z_{2}(z_{1}y_{4} + z_{2}z_{3}) +
\beta _{5}z_{2}y_{7} + \\
&+ \beta _{6}z_{2}z_{7} + \beta _{7}z^{3}_{2}z_{3} +
\beta _{8}z_{1}z_{8} + \beta _{9}z^{2}_{1}z_{2}z_{5} +
\beta _{10}z^{3}_{1}z_{6} +\\
&+ \beta _{11}z^{5}_{1}z_{4} + \beta _{12}z^{3}_{1}z^{3}_{2} +
\beta _{13}z^{7}_{1}z_{2}] + U_{2}\beta _{14}z_{3}z_{5}.
\end{align*}
Let us multiply this equation by $U_{1}$ and use the relations:
\[
U_{1}y_{6} = \Phi _{3}z_{1}\hbox{,  }U_{3}z_{1} = U_{1}z_{3},
\]
we obtain $\beta _{i} \equiv 0 \mod{2}$. Hence we obtain:
\[
U_{3}y_{6} = \Phi _{3}z_{4}.
\]

12. Let us consider the expression $U_{1}z_{9}$. From the condition
$d_{3}(U_{1}z_{9}) = U^{2}_{1}U_{2}U_{4}$ and the relation
$u_{1}a_{9} = u_{4}a_{2}$ fulfilled in the MASS it follows that
one can choose the element $z_{9}$, in such a way that the equality
is fulfilled:
\[
U_{1}z_{9} = U_{4}z_{2} + \beta U_{2}z_{2}z_{5}.
\]
Let us apply the operation $S_{2}$ to this equality. From the relations
\[
S_{2}z_{9} \equiv y_{7} \mod{(2y_{7},2z_{7},\ldots )},
S_{2}U_{4} = \Phi _{3}, S_{2}z_{5} \equiv  z_{3}
\mod{(2z_{3},z_{1}z_{2},2z^{3}_{1})}
\]
we obtain that $\beta U_{2}z^{2}_{3} \equiv 0 \mod{2}$, hence,
$\beta \equiv 0 \mod{2}$. So our equality takes the form:
\[
U_{1}z_{9} = U_{4}z_{2}.
\]

13. Let us consider the expression $U_{2}z_{8}$. Because of relation
$d_{3}(U_{2}z_{8}) = U_{1}U_{2}U_{4}$, of the term $E_{3}$ of the
Adams-Novikov spectral sequence and the relation
$u_{2}a_{8} = u_{4}a_{2}$, fulfilled in the MASS we obtain the
following:
\begin{align*}
U_{2}z_{8} &= U_{4}z_{2} +
\beta _{1}U_{1}[\beta _{1}z_{3}(z^{2}_{3} + z_{2}y_{4}) +
\beta _{2}z_{2}(z_{1}y_{6} + z_{3}z_{4}) + \\
&+ \beta _{3}z_{4}(z_{1}y_{4} + z_{2}z_{3}) +
\beta _{4}z^{2}_{1}z_{2}(z_{1}y_{4} + z_{2}z_{3}) +
\beta _{5}z_{2}y_{7} + \\
&+ \beta _{6}z_{2}z_{7} + \beta _{7}z^{3}_{2}z_{3} +
\beta _{8}z_{1}z_{8} + \beta _{9}z^{2}_{1}z_{2}z_{5} +
\beta _{10}z^{3}_{1}z_{6} + \\
&+ \beta _{11}z^{5}_{1}z_{4} + \beta _{12}z^{3}_{1}z^{3}_{2} +
\beta _{13}z^{7}_{1}z_{2}] + U_{2}\beta _{14}z_{3}z_{5}.
\end{align*}
If we multiply this equality by $U_{1}$, and then use the relation:
$U_{1}z_{8} = U_{4}z_{1}$ and $U_{2}z_{1} = U_{1}z_{2}$, then we get
the equality:
\begin{align*}
U^{2}_{1}[\beta _{1}z_{3}(z^{2}_{3} + z_{2}y_{4}) +
\beta _{2}z_{2}(z_{1}y_{6} + z_{3}z_{4}) +
\beta _{3}z_{4}(z_{1}y_{4} + z_{2}z_{3}) +\\
+ \beta _{4}z^{2}_{1}z_{2}(z_{1}y_{4} + z_{2}z_{3}) +
\beta _{5}z_{2}y_{7} + \beta _{6}z_{2}z_{7} +
\beta _{7}z^{3}_{2}z_{3} + \beta _{8}z_{1}z_{8} +\\
+ \beta _{9}z^{2}_{1}z_{2}z_{5} + \beta _{10}z^{3}_{1}z_{6} +
\beta _{11}z^{5}_{1}z_{4} + \beta _{12}z^{3}_{1}z^{3}_{2} +
\beta _{13}z^{7}_{1}z_{2}] + \\
+ U_{1}U_{2}\beta _{14}z_{3}z_{5} = 0.
\end{align*}
From the last equality we conclude that all the elements $\beta _{i}$
are equal to 0 $\mod 2$ and so the initial equality can be transformed
to the form:
\[
U_{2}z_{8} = U_{4}z_{2}.
\]

14. Let us choose an element $z_{10} \in  E^{0,40}_{2}$ so that
the following equality is fulfilled:
$\pi ^{2}_{2}(z_{10}) = b_{10} + a_{4}c_{6}$. From the action of
the Landweber-Novikov operation $S_{2}$ in the MASS:
$S_{2}b_{10} = b_{8}$, we conclude that in the Adams-Novikov
spectral sequence the following relation need to be fulfilled:
$$S_{2}(z_{10} + \text{any element from \ $F^{2}E^{0,40}_{2}$, not
equal to } z_{10}) = z_{8} + \text{decomposables}.$$
Hence
$$d_{3}(z_{10}) \neq  0,$$
$$d_{3}(z_{10} + \text{decomposable}) \neq 0. $$
Hence we choose the element $z_{10}$ so that the following equality
will be fulfilled:
\begin{align*}
d_{3}(z_{10}) &= U^{2}_{1}\Phi _{5} + \beta _{1}U^{3}_{3} +
\beta _{2}U_{2}U_{3}\Phi _{3} + \beta _{3}U^{2}_{2}U_{4} +\\
&+ \beta _{4}U_{1}U_{2}(U_{1}y_{8} + U_{3}z_{7}) +
\beta _{5}U^{2}_{1}U_{2}y^{2}_{4} +\\
&+ \beta _{6}U^{2}_{2}(U_{2}y_{6} + U_{3}y_{4}) +
\beta _{7}U_{1}\Phi _{3}(U_{1}y_{4} + U_{2}z_{3})+\\
&+ \beta _{8}U_{1}U_{3}(U_{1}y_{6} + U_{3}z_{3}).
\end{align*}
Let us apply the operation $S_{4,4}$ to this equality. From the
conditions: $S_{4,4}\Phi _{5} = \Phi _{1}$,
$S_{4,4}z_{10} \equiv  z_{2} \mod{(2z_{2},z^{2}_{1})}$,
$S_{4,4}y_{8} \equiv  1 \mod{2}$,
we get the relation $\beta _{4} \equiv 0 \mod{2}$. From the relations
fulfilled for the operation $S_{2,2,2,2}$:
$S_{2,2,2,2}\Phi _{5} = \Phi _{1}$, $S_{2,2,2,2}z_{10} \equiv  z_{2} \mod{(2z_{2},z^{2}_{1})}$,
and the operation $S_{2,2}$:
$S_{2,2}y_{4} \equiv 1 \mod{2}$, $S_{2,2}\Phi _{3} = \Phi _{1}$,
we obtain the equality
$$d_{3}(z_{2}) =
(1 + \beta _{5} + \beta _{7} + \beta _{8})U^{2}_{1}U_{2},$$
from which it follows that
$\beta _{5} + \beta _{7} + \beta _{8} \equiv 0 \mod{2}$.
Using the operation $S_{7}$, we arrive to the relation
$\beta _{3}U_{1}U^{2}_{2}$, from which we conclude, that
$\beta _{3} \equiv 0 \mod{2}$. Acting by the operation $S_{6}$
on our equality and having in mind the relation:
$S_{6}z_{10} \equiv z_{4}
\mod{(2z_{4}, 2y_{4}, z_{1}z_{3}, 2z^{2}_{1}z_{2}, z^{2}_{2}, z^{4}_{1})}$,
we get the equality
$d_{3}(z_{4}) = (1 + \beta _{6})U^{2}_{1}U_{3} + \beta _{8}U^{3}_{2}$.
Hence, $\beta _{6}$,$\beta _{8} \equiv 0 \mod{2}$. Using the conditions
\begin{align*}
S_{2,2}z_{10} &\equiv  z_{6}
\mod{(2z_{6}, 2y_{6}, z_{1}z_{5}, z_{2}z_{4}, z^{2}_{3}, 2z^{2}_{1}z_{4},
z^{3}_{1}z_{3}, 2z^{4}_{1}z_{2}, z^{2}_{1}z^{2}_{2}, z^{6}_{1})}, \\
S_{2,2}\Phi _{5} &= \Phi _{3},
\end{align*}
and applying the operation $S_{2,2}$ to the considering equality, we
obtain:
\[
d_{3}(y_{6}) = \beta _{7}U_{1}U_{2}(U_{1}y_{4} + U_{2}z_{3}) +
\beta _{1}U^{2}_{2}U_{3} + (1 + \beta _{7})U^{2}_{1}\Phi _{3}.
\]
This gives the relations: $\beta _{1},\beta _{7} \equiv  0 \mod{2}$.
From the relation $\beta _{5} + \beta _{7} + \beta _{8} \equiv 0 \mod{2}$
we conclude $\beta _{5} \equiv 0 \mod{2}$. Applying the operation
$S_{2}$ to our equality. From the equality
$S_{2}(U_{2}U_{3}\Phi _{3}) = U^{2}_{2}\Phi _{3} + U_{2}U^{2}_{3}$
it follows that $\beta _{2} \equiv  0 \mod{2}$.
So, finally we get:
\[
d_{3}(z_{10}) = U^{2}_{1}\Phi _{5}.
\]

15. Choose the element $y_{10} \in  E^{0,40}_{2}$ so that
$\pi ^{2}_{0}(y_{10}) = c_{10}$. From the condition:
$$S_{2}(y_{10} + \text{any element in this dimension not equal to} \
y_{10}) \equiv $$
$$y_{8} \mod (\text{elements from } \ F^{2}E^{0,32}_{2})$$
the relations follow: $$d_{3}(y_{10}) \neq 0,$$
$$d_{3}(y_{10} + \text{any element in this dimension not equal to} \
y_{10}) \neq 0.$$
Hence we can choose the element $y_{10}$ so that the equality will be
fulfilled:
\[
d_{3}(y_{10}) = \beta _{1}U^{3}_{3} + \beta _{2}U_{2}U_{3}\Phi _{4} +
\beta _{3}U^{2}_{2}U_{4} + \beta _{4}U^{2}_{2}(U_{2}y_{6} + U_{3}y_{4}).
\]
Apply the operation $S_{6}$ to this equality and notice that
$S_{6}y_{10} \equiv 0 \mod{(2y_{4})}$, so we get the relation
$(\beta _{3} + \beta _{4})U^{3}_{2} = 0$, hence
$\beta _{3} \equiv  \beta _{4} \mod{2}$. Because of the fact
$S_{2,2}y_{10} \equiv  y_{6} \mod (\text{elements from} \
F^{2}E^{0,24}_{2})$,
after the application of the operation $S_{2,2}$ to our equality we
obtain
$d_{3}(y_{6}) = (\beta _{1} + \beta _{4})U^{2}_{2}U_{3}$.
This gives the relation $\beta _{1}$ + $\beta _{4} \equiv 1 \mod{2}$.
Using the operation $S_{2}$ we arrive to the equality
\[
d_{3}(y_{8}) = (\beta _{1} + \beta _{2})U_{2}U^{2}_{3} +
(\beta _{2} + \beta _{3})U^{2}_{2}\Phi _{3}.
\]
This gives us the relations: $\beta _{2} \equiv  \beta _{3} \mod{2}$
$\beta _{1} + \beta _{2} \equiv 1 \mod{2}$.
Finally if we use the operation $S_{1,1}$, and have in mind the fact that
$S_{1,1}y_{10} \equiv  y_{8} + y^{2}_{4}
\mod(\text{elements from} \ F^{2}E^{0,32}_{2})$, the we get the following
expression:
$d_{3}(y_{8}) = \beta _{1}U_{2}U^{2}_{3}$ + $(\beta _{2}$ + $\beta _{3})U^{2}_{2}\Phi _{3}$ + $\beta _{4}U^{2}_{1}(U_{2}y_{6} + U_{3}y_{4})$.
Hence
$\beta _{1} \equiv 1 \mod{2}$, $\beta _{2},\beta _{3},\beta _{4} \equiv 0 \mod{2}$.
Finally we have the formula for the action of $d_{3}$:
\[
d_{3}(y_{10}) = U^{3}_{3}.
\]

16. Choose the element $y^{*}_{10} \in  E^{0,40}_{2}$ so that
$\pi ^{2}_{0}(y^{*}_{10}) = c^{2}_{5} + c^{2}_{2}c_{6}$.
From the condition:
$$S_{3,3}(y^{*}_{10} + \text{any element from this dimension not equal to}
\ y^{*}_{10}) \equiv $$
$$ y_{4} \mod(\text{elements form} \ F^{2}E^{0,16}_{2})$$
the relations follow:
$$d_{3}(y_{10}) \neq  0,$$
$$d_{3}(y_{10} + \text{any element from his dimension not equal to}
 \ y^{*}_{10}) \neq  0.$$
Hence it is possible to choose the element $y^{*}_{10}$ so that the
following equality is fulfilled:
\[
d_{3}(y^{*}_{10}) = \beta _{1}U^{3}_{3} +
\beta _{2}U_{2}U_{3}\Phi _{4} + \beta _{3}U^{2}_{2}U_{4} +
\beta _{4}U^{2}_{2}(U_{2}y_{6} + U_{3}y_{4}).
\]
Applying the operation $S_{3,3}$ to this equality and noticing
that $S_{3,3}y_{10} \equiv  y_{4} \mod{(2y_{4})}$,
$S_{3,3}y_{6} \equiv 0 \mod{2}$, $S_{3,3}U_{4} = U_{2}$,
we obtain the relation $d_{3}(y_{4}) = \beta _{3}U^{3}_{2}$,
hence, $\beta _{3} \equiv  1 \mod{2}$.
If we apply the operation $S_{6}$ to our equality and notice that
$S_{6}y^{*}_{10} \equiv  y_{4} \mod{(2y_{4})}$,
we get the relation $d_{3}(y_{4})$ = $(1 + \beta _{4})U^{3}_{2}$,
hence
$\beta _{4} \equiv 0 \mod{2}$.
Using the operation $S_{2}$, and having in mind the equality
$$S_{2}y^{*}_{10} \equiv  y^{2}_{4} \mod(\text{elements from} \
F^{2}E^{0,32}_{2} \text{and decomposables}),$$
we come to the expression
$$0 = (\beta _{1} + \beta _{2})U_{2}U^{2}_{3} +
(\beta _{2} + 1)U^{2}_{2}\Phi _{3},$$
hence,
$\beta _{1} \equiv 1 \mod{2}$, $\beta _{2} \equiv 1 \mod{2}$.
Finally our equality has the form:
\[
d_{3}(y^{*}_{10}) = U^{2}_{2}U_{4} + U_{2}U_{3}\Phi _{3} + U^{3}_{3}.
\]

17. The is an equality: $h^{2}_{0}e_{10} = a_{3}a_{7} + a^{2}_{5}$
in the MASS. Hence, one can choose the element $y^{*}_{10}$ so that
the following equality will be fulfilled:
\[
4(y^{*}_{10} + y_{4}y_{6}) = (1 + 2\beta _{1})z_{3}z_{7} +
(1 + 2\beta _{2})z^{2}_{5} +
+ X^{1}_{40}(\widehat{y^{*}_{10}}, \widehat{y_{4}y_{6}},
\widehat{z_{3}z_{7}}, \widehat{{z}^{2}_{5}} ).
\]
where $X^{1}_{40} \in  F^{6}(E^{0,40}_{2})$, $d_{3}(X^{1}_{40}) = 0,$
and the sign ($\widehat{\quad}$) over an element means that the given
element does not appear in the expression for $X^{1}_{40}$.
Let us apply the operation $S_{5,5}$ to our equality, we get
the relation
$S_{5,5}y^{*}_{10} \equiv  ((S_{5}z_{5})/2)^{2}(1 + 2\beta _{2})
\mod{4}$;
we may suppose that
$S_{5,5}y^{*}_{10} \equiv ((S_{5}z_{5})/2)^{2} \mod{4}$.
Hence, $\beta _{2} \equiv 0 \mod{2}$,
so, by the choice of $y_{10}$ we can transform our equality
to the form
\[
4(y^{*}_{10}+ y_{4}y_{6}) = (1 + 2\beta _{1})z_{3}z_{7} + z^{2}_{5} +
X^{1}_{40}(\widehat{y^{*}_{10}}, \widehat{y_{4}y_{6}},
\widehat{z_{3}z_{7}}, \widehat{{z}^{2}_{5}} ).
\]
Let us apply the operation $S_{7}$ to this expression, we get the
relation:
$S_{7}y^{*}_{10} \equiv ((S_{7}z_{7})/4)(1 + 2\beta _{1})z_{3}
\mod{(8z_{3})}$. We may suppose that
$S_{7}y^{*}_{10} \equiv ((S_{7}z_{7})/4)z_{3} \mod{(8z_{3})}$.
Hence, $\beta _{1} \equiv 0 \mod{2}$ and by the choice of
$y^{*}_{10}$ it is possible to transform our expression to the form:
\[
4(y^{*}_{10}+ y_{4}y_{6}) = z_{3}z_{7} + z^{2}_{5} +
X^{1}_{40}(\widehat{y^{*}_{10}}, \widehat{y_{4}y_{6}},
\widehat{z_{3}z_{7}}, \widehat{{z}^{2}_{5}} ).
\]

18. There is the equality: $h_{0}b_{10} = a_{2}a_{8} + a_{1}a_{9}$
in the MASS. Hence it is possible to choose the element $z_{10}$ so
that the following equality is fulfilled:
\[
2(z_{10} + z_{4}y_{6}) = z_{2}z_{8} + (1 + 2\beta )z_{1}z_{9} +
X^{2}_{40}(\widehat{z_{10}},\widehat{z_{4}y_{6}},
\widehat{z_{2}z_{8}},\widehat{z_{1}z_{9}}).
\]
In this expression we have: $X^{2}_{40} \in  F^{6}(E^{0,40}_{2})$,
$d_{3}(X^{2}_{40}) = 0$. Applying the operation $S_{9}$ to this expression
we get the following relation:
$S_{9}z_{10} \equiv$ $ ((S_{9}z_{9})/2)(1 + 2\beta )z_{1}$
$\mod{(4z_{1})}$.
We may suppose that $S_{9}z_{10} \equiv$
$ ((S_{9}z_{9})/2)3z_{1}$ $\mod{(4z_{1})}$.
Hence it is possible to choose $z_{10}$ so that the following equality
is fulfilled:
\[
2(z_{10} + z_{4}y_{6}) = z_{2}z_{8} + 3z_{1}z_{9} +
X^{2}_{40}(\widehat{z_{10}},\widehat{z_{4}y_{6}},
\widehat{z_{2}z_{8}},\widehat{z_{1}z_{9}}).
\]
Applying the operation $S_{7}$ to this relation we obtain
$$S_{7}z_{10} \equiv 0 \mod{(2z_{1}z_{2},2z^{3}_{1},4z_{3})}.$$

19. From the relation $u_{1}(b_{10} + a_{4}c_{6}) = a_{1}\varphi _{6}$
in the MASS and the action of the differential:
$d_{3}(U_{1}z_{10}) = U^{3}_{1}\Phi _{5}$
it follows that we may choose the element $z_{10}$ so that
the equality is fulfilled:
\[
U_{1}z_{10} = \Phi _{5}z_{1} + \beta U_{1}z_{2}z_{8}.
\]
Applying the operation $S_{7}$ to this equality we obtain the
following relation:
$\beta U_{1}z_{1}z_{2} \equiv 0 \mod{(2z_{1}z_{2})}$,
from where it follows that one may choose the element
$z_{10}$ so that the equality is fulfilled:
\[
U_{1}z_{10} = \Phi _{5}z_{1} .
\]

20. Let us consider the expression: $U_{2}z_{9}$.
From the condition: $d_{3}(U_{2}z_{9}) = U^{2}_{1}U_{2}U_{4}$
and $u_{2}a_{9} = u_{4}a_{3}$ in the MASS it follows that one can
choose the element $z_{9}$ (not violating the previous relations)
that the following relation is fulfilled:
$U_{2}z_{9} = U_{4}z_{2} + Y[\beta _{i}]$ (*), where
$Y[\beta _{i}]$ has the following form:
\begin{align*}
Y[\beta _{i}] &= [\beta _{1}z_{3}(z_{3}y_{4} + z_{3}z_{4}) +
\beta _{2}z_{1}(z_{1}y_{8} + z_{2}z_{7}) +
\beta _{3}z_{1}z^{2}_{2}(z_{1}y_{4} + z_{2}z_{3}) +\\
&+ \beta _{4}z^{3}_{1}(z_{1}y_{6} + z_{2}z_{5}) +
\beta _{5}z^{4}_{1}(z^{2}_{1}y_{4} + z^{3}_{2}) +
\beta _{6}z_{3}(z_{1}y_{6} + z_{2}z_{5}) +\\
&+ \beta _{7}z_{5}(z_{1}y_{4} + z_{2}z_{3}) +
\beta _{8}z_{1}z_{9} + \beta _{9}z_{3}z_{7} +
\beta _{10}z^{3}_{1}z_{7} + \beta _{11}z^{3}_{2}z_{4} +\\
&+ \beta _{12}z^{5}_{1}z_{5} + \beta _{13}z^{2}_{2}z^{2}_{3} +
\beta _{14}z^{2}_{1}z^{4}_{2} + \beta _{15}z^{6}_{1}z^{2}_{2} +
\beta _{16}z^{10}_{1} + \beta _{17}z_{4}y_{6} +\\
&+ \beta _{18}z_{3}y_{7} + \beta _{19}z^{3}_{1}y_{7} +
\beta _{20}z^{2}_{1}y^{2}_{4}]U_{1}+
\beta _{21}U_{2}z_{3}(z_{2}y_{4}+ z^{2}_{3})
\end{align*}
If we multiply now the expression (*) by $U_{1}$ and use
the relations
$U_{1}U_{2}z_{9} = U_{2}U_{4}z_{1} = U_{1}U_{4}z_{2}$,
then we get $U_{1}Y[\beta _{i}] = 0$
Hence, it follows for all $\beta _{i} \equiv 0 \mod{2}$, and so:
\[
U_{2}z_{9} = U_{4}z_{2}.
\]

21. From the action of the differential $d_{3}$ and the relation
$u_{2}a_{9} = u_{3}a_{7} = \varphi _{3}a_{5}$ in the MASS
it follows that one can choose the element $y_{9}$
(not violating the previous relations) that in the Adams-Novikov
spectral sequence the following relations will be fulfilled:
\[
U_{2}y_{9} = U_{3}y_{7} + Y[\beta _{i}], \quad U_{2}y_{9} =
\Phi _{3}z_{5} + Y[\beta _{i}] .
\]
Here $Y[\beta _{i}]$ denotes the expression from the item 20.
Multiplying both relations (term by term) by  $U_{1}$,
we obtain that in both cases all coefficients
$\beta _{i} \equiv 0 \mod{2}$.
So, we have:
\[
U_{2}y_{9} = U_{3}y_{7},\quad U_{2}y_{9} = \Phi _{3}z_{5} .
\]

22. There is the relation in the MASS:
$\varphi ^{2}_{3} = u^{2}_{1}c^{2}_{5} + u^{2}_{2}c^{4}_{4} +
u^{2}_{3}c^{2}_{2}$.
Hence it is possible to choose the element $y^{*}_{10}$
(not violating his properties) so that the following equality
is fulfilled in the Adams-Novikov spectral sequence:
\begin{multline*}
\Phi ^{2}_{3} = U_{1}[U_{1}y^{*}_{10} + U_{2}(y_{9} + z_{9}) +
U_{3}z_{7}] + U^{2}_{2}y_{8} + U^{2}_{3}y_{4} +\\
+ U_{1}[U_{1}y_{4}y_{6} + U_{2}(y_{4}z_{5} + y_{6}z_{3})] +
\beta U^{2}_{2}z_{3}z_{5}.
\end{multline*}
Let us apply the operation $S_{2}$ to our equality. From the relation
$U_{1}z_{8} = U_{4}z_{1}$ it follows that
$S_{2}z_{8} \equiv 0 \mod{(2z^{2}_{3})}$,
and because of the relation from the item~3 it follows that
$S_{2}y_{8} \equiv 0 \mod{(2z^{2}_{3})}$.
So, we obtain that
$\beta U^{2}_{2}z^{2}_{3} \equiv 0 \mod{(2U^{2}_{2}z^{2}_{3})}$,
Hence, $\beta  \equiv 0 \mod{2}$. Finally:
\begin{multline*}
\Phi ^{2}_{3} = U_{1}[U_{1}y^{*}_{10} + U_{2}(y_{9} + z_{9}) +
U_{3}z_{7}] + U^{2}_{2}y_{8} + U^{2}_{3}y_{4} +\\
+ U_{1}[U_{1}y_{4}y_{6} + U_{2}(y_{4}z_{5} + y_{6}z_{3})].
\end{multline*}

23. An element $z_{11}\in  E^{0,44}_{2}$ we choose so that
$\pi ^{2}_{2}(z_{11}) = a_{11}$.
Because of the relation $S_{3}a_{11} = a_{8}$ we obtain that
$$S_{3}(z_{11} + \text{any element from this dimension not equal to}
\ z_{11}) = $$
$$z_{8} + \text{decomposables}.$$
Hence,
$$d_{3}(z_{11}) \neq 0,$$
$$d_{3}(z_{11} + \text{any element from this dimension not equal to }
z_{11}) \neq 0.$$
Hence it is possible to choose the element $z_{11}$ so that for
the action of the differential $d_{3}$ the following equality
will be fulfilled:
\begin{align*}
d_{3}(z_{11}) &= U_{1}U_{3}U_{4} + \beta _{1}U_{1}\Phi ^{2}_{3} +
\beta _{2}U_{1}U_{2}\Phi _{5} + \beta _{3}U_{2}\Phi _{3}(U_{1}y_{4} +
U_{2}z_{3}) +\\
&+ \beta _{4}U^{2}_{2}(U_{1}y_{8}+ U_{2}z_{7}) +
\beta _{5}U_{2}U_{3}(U_{2}y_{6}+ U_{3}y_{4}) +
\beta _{6}U^{2}_{3}(U_{1}y_{4}+ U_{2}z_{3})+ \\
&+ \beta _{7}U_{1}U^{2}_{2}y^{2}_{4}.
\end{align*}
Let us apply the operation $S_{9}$ to this expression. From the
relation $S_{9}z_{11} \equiv 0 \mod{(2z_{2})}$, we obtain
the equality $\beta _{2}U^{2}_{1}U_{2} = 0$. Hence,
$\beta _{2} \equiv 0 \mod{2}$.
Acting on the equality under consideration by the operation
$S_{5,5}$ and using the relation
$S_{5,5}z_{11} \equiv 0 \mod{(2z_{1})}$,
we obtain $\beta _{1} \equiv 0 \mod{2}$.
Let us use now the operation $S_{4,4}$. Because of the equality
$S_{4,4}z_{11} \equiv 0 \mod{(2z_{3})}$,
it follows the relation:
$\beta _{4} \equiv  0 \mod{2}$.
Analogous use of the operation $S_{2,2,2,2}$ together with the
relations $S_{2,2,2,2}z_{11} \equiv 0 \mod{(2z_{3})}$,
$S_{2,2,2}y_{6} \equiv 1 \mod{2}$,
$S_{2,2}\Phi _{3} = U_{2}$,
gives us the condition
($\beta _{3}$ + $\beta _{5}$ + $\beta _{6}$ +
$\beta _{7}$)$U_{1}U^{2}_{2} = 0$.
Hence,
$\beta _{3}$ + $\beta _{5}$ + $\beta _{6}$ + $\beta _{7}
\equiv 0 \mod{2}$.
Using the operation $S_{6}$ we get the following:
$S_{6}z_{11} \equiv  z_{5} \mod{(2z_{5})}$, $S_{6}y_{6} \equiv  1 \mod{2}$,
and so,
$d_{3}(z_{5}) = U_{1}U_{2}U_{3} = U_{1}U_{2}U_{3} +
\beta _{5}U_{1}U_{2}U_{3}$.
Hence, $\beta _{5} \equiv 0 \mod{2}$.
From the following property of the operation $S_{5}$:
$S_{5}z_{11} \equiv  0 \mod{(2y_{6},2z_{6})}$,
it follows that $\beta _{3} \equiv 0 \mod{2}$.
Finally using the operation $S_{2,2}$ we have:
$S_{2,2}z_{11} \equiv  y_{7} \mod{(2y_{7},2z_{7})}$,
that gives: $$d_{3}(y_{7}) = U_{1}U_{2}\Phi _{3}=
U_{1}U_{2}\Phi _{3} + \beta _{6}U_{1}U^{2}_{3} +
\beta _{6}U^{2}_{2}(U_{1}y_{4} + U_{2}z_{3}).$$
Hence, $\beta _{6} \equiv 0 \mod{2}$.
From the previous relations we obtain now that
$\beta _{7} \equiv 0 \mod{2}$. Summarizing all relations
we get finally:
\[
d_{3}(z_{11}) = U_{1}U_{3}U_{4}.
\]

24. An element $y_{11}\in  E^{0,44}_{2}$ we choose so that
$\pi ^{2}_{2}(y_{11}) = b_{11} + a_{5}c_{6}$.
Because of the relation $S_{9}b_{11} = a_{2}$, we obtain that
$$S_{9}(y_{11} + \text{any element of this dimension not equal to}
\ y_{11}) \equiv  z_{2} \mod{(2z_{2},z^{2}_{1})}.$$
Hence,
$$d_{3}(y_{11}) \neq 0,$$
$$d_{3}(y_{11} + \text{any element of this dimension not equal to }
y_{11}) \neq 0.$$
Hence it is possible to choose the element $y_{11}$, so that for the
action of the differential $d_{3}$ the following equality is fulfilled:
\begin{align*}
d_{3}(y_{11}) &= U_{1}U_{2}\Phi _{5} + \beta _{1}U_{1}U_{3}U_{4} +
\beta _{2}U_{1}\Phi ^{2}_{3} + \beta _{3}U_{2}\Phi _{3}(U_{1}y_{4} +
U_{2}z_{3}) +\\
&+ \beta _{4}U^{2}_{2}(U_{1}y_{8}+ U_{2}z_{7}) +
\beta _{5}U_{2}U_{3}(U_{2}y_{6}+ U_{3}y_{4}) +
\beta _{6}U^{2}_{3}(U_{1}y_{4}+ U_{2}z_{3}) +\\
&+ \beta _{7}U_{1}U^{2}_{2}y^{2}_{4}.
\end{align*}
Acting by the operation $S_{5,5}$ on our equality an using the
relation
$S_{5,5}y_{11} \equiv 0 \mod{(2z_{1})}$,
we obtain that $\beta _{2} \equiv 0 \mod{2}$.
Analogous use of the operation $S_{2,2,2,2}$, together with
the relations
\[
S_{2,2,2,2}y_{11} \equiv  z_{3} \mod{(2z_{3})}, \ S_{2,2,2}y_{6}
\equiv 1 \mod{2}, \ S_{2,2,2,2}\Phi _{5} = U_{2},
\]
gives the formula
\[
d_{3}(z_{3}) = U_{1}U^{2}_{2} = U_{1}U^{2}_{2} +
(\beta _{3} + \beta _{5} + \beta _{6} + \beta _{7})U_{1}U^{2}_{2}.
\]
Hence,
$\beta _{3} + \beta _{5} + \beta _{6} + \beta _{7} \equiv 0 \mod{2}$.
Let us use the operation $S_{4,4}$.
From the relations: $S_{4,4}y_{11} \equiv  z_{3} \mod{(2z_{3})}$
and $S_{4,4}\Phi _{5} = U_{2}$ we obtain the equality:
$d_{3}(z_{3}) = U_{1}U^{2}_{2} = U_{1}U^{2}_{2} + \beta _{4}U_{1}U^{2}_{2}$,
hence, it follows that
$\beta _{4} \equiv 0 \mod{2}$.
Let us apply the operation $S_{7}$ to our equality. Because of the
relation $S_{7}y_{11} \equiv 0 \mod{(2z_{4},2y_{4})}$, we obtain:
$\beta _{1} \equiv 0 \mod{2}$. From the following property of the
operation $S_{6}$: $S_{6}y_{11} \equiv z_{5} \mod{(2z_{5})}$,
we obtain that
$$d_{3}(z_{5}) = U_{1}U_{2}U_{3} = U_{1}U_{2}U_{3} +
\beta _{5}U_{1}U_{2}U_{3}.$$
Hence,
$\beta _{5} \equiv 0 \mod{2}$.
From the following property of the operation
$S_{5}$: $S_{5}y_{11} \equiv 0 \mod{(2y_{6},2z_{6})}$
we obtain that
$\beta _{3}U_{1}U_{2}(U_{1}y_{4} + U_{2}z_{3}) = 0$.
Hence, $\beta _{3} \equiv 0 \mod{2}$.
Finally using the operation $S_{2,2}$, we obtain:
$S_{2,2}y_{11} \equiv  y_{7} \mod{(2y_{7},2z_{7})}$,
this gives:
\[
d_{3}(y_{7}) = U_{1}U_{2}\Phi _{3} = U_{1}U_{2}\Phi _{3} +
\beta _{6}U_{1}U^{2}_{3} + \beta _{6}U^{2}_{2}(U_{1}y_{4} + U_{2}z_{3}).
\]
Hence, $\beta _{6} \equiv 0 \mod{2}$.
From the previous relations we obtain now
$\beta _{7} \equiv 0 \mod{2}$.
So, finally we get:
\[
d_{3}(y_{11}) = U_{1}U_{2}\Phi _{5}.
\]

25. There are the following equalities in the MASS:
$h_{0}e_{4}a_{7} + h_{0}e_{8}a_{3} = a_{2}b_{9} + a_{4}b_{7}$,
and $h_{0}e_{8}a_{3} + h_{0}e_{10}a_{1} = a_{4}b_{7} + a_{5}b_{6}$.
Using the relations 4, 6, 9 from the work \cite{V2} and relations of
the items 3, 9, 16 of the present work, we obtain that there is the
equality in the Adams-Novikov spectral sequence:
\begin{align*}
2y_{8}z_{3} + 2(y^{*}_{10} + y_{6}y_{4})z_{1} = z_{4}y_{7} +
z_{5}z_{6} + X^{2}_{44}(\widehat{y_{6}z_{5}}, \widehat{y_{8}z_{3}},
\widehat{z_{1}(y^{*}_{10}+y_{6}y_{4})}, \widehat{{z}_{4}y_{7}} ).\\
2y_{4}z_{7} + 2y_{8}z_{3} = z_{2}y_{9} + z_{4}y_{7} +
X^{1}_{44}(\widehat{y_{4}z_{7}}, \widehat{y_{8}z_{3}},
\widehat{z_{2}y_{9}}, \widehat{z_{4}y_{7}} ).
\end{align*}
Here we have: $X^{i}_{44} \in  F^{6}(E^{0,44}_{2})$,
$d_{3}(X^{i}_{44}) = 0$, and the symbol $\widehat{\quad}$ over
an element means that this element does not appear in the given
expression.

26. From the relation $h_{0}b_{11} = a_{3}a_{8} + a_{2}a_{9}$ in the
MASS, the action of the operation and the relation~9 from the work
[3] it follows that we can choose the element $y_{11}$ so that the
relation is fulfilled:
\[
2y_{11} = z_{2}z_{9} + 3z_{3}z_{8} + X_{44}(\widehat{y_{11}},
\widehat{z_{2}z_{9}},\widehat{z_{3}z_{8}}).
\]
Here: $X_{44} \in  F^{6}(E^{0,44}_{2})$, $d_{3}(X_{44}) = 0,$ and
the symbol $\widehat{\quad}$ has the same meaning as earlier.

27. From the relation:
$u_{1}(a_{1}e_{10} + a_{3}e_{8} + a_{7}e_{4}) = \varphi _{3}b_{6}$,
valid in the MASS, the equality follows:
\begin{align*}
U_{1}(z_{1}y^{*}_{10} + z_{1}y_{4}y_{6} + z_{3}y_{8} + z_{7}y_{4}) =
\Phi _{3}y_{6} + U_{1}[\beta _{1}z_{1}z_{10} +\\
+ \beta _{2}z_{2}z_{9} + \beta _{3}z_{2}y_{9} +
\beta _{4}z^{2}_{1}z_{2}z_{7} + \beta _{5}z_{4}z_{7} +
\beta _{6}z_{2}(z_{1}y_{8} +\\
+ z_{2}z_{7}) + \beta _{7}z_{4}(z_{1}y_{6} + z_{3}z_{4}) +
\beta _{8}z^{2}_{1}z_{4}(z_{1}y_{4} + z_{2}z_{3}) +\\
+ \beta _{9}z^{2}_{1}z_{2}(z_{1}y_{6} + z_{3}z_{4}) +
\beta _{10}z^{3}_{2}(z_{1}y_{4} + z_{2}z_{3}) +
\beta _{11}z^{4}_{1}\times \\
\times z_{2}(z_{1}y_{4} + z_{2}z_{3}) +
\beta _{12}z_{6}(z_{1}y_{4} + z_{2}z_{3}) +
\beta _{13}z_{3}(z_{2}y_{6}+\\
+ z_{4}y_{4}) + \beta _{14}z_{3}(z_{2}y_{6} + z_{3}z_{5}) +
\beta _{15}z^{2}_{1}z_{2}y_{7} + \beta _{16}z^{3}_{1}z_{8}+\\
+ \beta _{17}z^{3}_{1}z_{2}z^{2}_{3} +
\beta _{18}z^{3}_{1}z^{2}_{2}z_{4} + \beta _{19}z^{7}_{1}z_{4} +
\beta _{20}z^{5}_{1}z_{6} + \beta _{21}\times \\
\times z^{5}_{1}z^{3}_{2} + \beta _{22}z^{9}_{1}z_{2} +
\beta _{23}z_{1}z_{2}z_{3}z_{5}+ \beta _{24}z_{1}z_{2}y^{2}_{4} +
\beta _{25}z_{4}z_{7}+\\
+ \beta _{26}z_{2}z^{3}_{3}] + \beta _{27}U_{2}z_{3}z_{7} +
\beta _{28}U_{2}z_{3}y_{7}.
\end{align*}
Multiplying this equality by $U_{1}$, and the equality from the item~22
by $z_{1}$ and add them. Using the relations:
$U^{2}_{1}z_{3} = U^{2}_{2}z_{1}$, $U^{2}_{1}z_{7} = U^{2}_{3}z_{1}$,
$U_{1}y_{6} = \Phi _{3}z_{1}$, $U_{1}z_{2} = U_{2}z_{1}$,
we obtain, that
$\beta _{2},\beta _{3},\beta _{5},\beta _{13} \equiv 1 \mod{2}$,
and for all the other $i$ the relation is valid:
$\beta _{i} \equiv 0 \mod{2}$.
So, finally we get:
\begin{align*}
U_{1}(z_{1}y^{*}_{10} + z_{1}y_{4}y_{6} + z_{3}y_{8} + z_{7}y_{4}) +
U_{1}z_{4}z_{7} + U_{1}z_{2}z_{9} +\\
+ U_{1}z_{2}y_{9} + U_{1}z_{3}(z_{2}y_{6} + z_{4}y_{4}) = \Phi _{3}y_{6}.
\end{align*}

28. Let us consider the product $U_{1}y_{11}$. From the action of
the differential $d_{3}$:
$d_{3}(U_{1}y_{11}) = U^{2}_{1}U_{2}\Phi _{5}$
and the relation valid in the MASS:
$u_{1}(b_{11} + a_{5}c_{6}) = \varphi _{5}a_{2}$,
it follows that one can choose $y_{11}$ so that the following formula
holds:
\[
U_{1}y_{11} = \Phi _{5}z_{2} + \beta _{1}U_{2}z_{3}z_{7} +
\beta _{2}U_{2}z_{3}y_{7} + \beta _{3}U_{3}y^{2}_{4} +
\beta _{4}U_{1}z_{3}z_{8}.
\]
From the relation of the item~25 it follows that
$S_{6}y_{11} \equiv  0 \mod{(2z_{2}z_{3})}$.
Hence, if we apply the operation $S_{6}$ to our equality we obtain:
$\beta _{4}U_{1}z_{2}z_{3} = 0$,
hence, $\beta _{4} \equiv 0 \mod{2}$.
Let us apply the operation $S_{2,2,2,2}$ to our equality, then we have:
$U_{1}z_{3} = U_{2}z_{2} + \beta _{3}U_{3}$.
Hence, $\beta _{3} \equiv 0 \mod{2}$.
Now let us act by the operation $S_{3,3}$ on our equality. We shall
have: $\beta _{1}U_{2}z_{1}z_{3} = 0$, hence,
$\beta _{1} \equiv 0 \mod{2}$. Finally, if we apply the operation
$S_{5}$, then we get the relation $\beta _{2}U_{2}z_{2}z_{3}$ = 0,
in other words: $\beta _{2} \equiv 0 \mod{2}$.
Hence, it is possible to choose the element $y_{11}$ in such a way
that the relation holds:
\[
U_{1}y_{11} = \Phi _{5}z_{2}.
\]

29. Let us consider the product $U_{2}z_{10}$. The action of the
differential $d_{3}$ on this element in the following:
$d_{3}(U_{2}z_{10}) = U^{2}_{1}U_{2}\Phi _{5}$.
Also we have the formula in the MASS:
$u_{2}(b_{10} + a_{4}c_{6}) = \varphi _{5}a_{2}$.
Hence by the choice of the element $z_{10}$ we can obtain the
following equality:
$U_{2}z_{10} = \Phi _{5}z_{2} + X(\beta _{i})$,
where $X(\beta _{i})$ has the decomposition:
\begin{align*}
X(\beta _{i}) &= U_{1}[\beta _{1}z^{2}_{1}z_{2}z_{7} +
\beta _{2}z_{1}z_{10} + \beta _{3}z_{2}(z_{111}y_{8} + z_{2}z_{7}) +
\beta _{4}z_{4}(z_{1}y_{6} + z_{3}z_{4}) +\\
&+ \beta _{5}z^{2}_{1}z_{4}(z_{1}y_{4} + z_{2}z_{3}) +
\beta _{6}z^{2}_{1}z_{2}(z_{1}y_{6}+ z_{2}z_{3})+
\beta _{7}z^{3}_{2}(z_{1}y_{4}+ z_{2}z_{3}) +\\
&+ \beta _{8}z^{4}_{1}z_{2}(z_{1}y_{4}+ z_{2}z_{3})+
\beta _{9}z_{6}(z_{1}y_{4} + z_{2}z_{3}) +
\beta _{10}z_{3}(z_{2}y_{6}+ z_{4}y_{4}) +\\
&+ \beta _{11}z_{3}(z_{2}y_{6}+ z_{3}z_{5}) +
\beta _{12}z^{2}_{1}z_{2}y_{7} +
\beta _{13}z^{3}_{1}z_{2}z^{2}_{3} + \beta _{14}z^{3}_{1}z_{8} +\\
&+ \beta _{15}z^{7}_{1}z_{4} + \beta _{16}z^{3}_{1}z^{2}_{2}z_{4} +
\beta _{17}z^{5}_{1}z_{6} + \beta _{18}z^{5}_{1}z^{3}_{2} +
\beta _{19}z^{9}_{1}z_{2} +\\
&+ \beta _{20}z_{4}z_{7} + \beta _{21}z_{1}z_{2}z_{3}z_{5} +
\beta _{22}z_{1}z_{2}y^{2}_{4} + \beta _{23}z_{2}y_{9} +
\beta _{24}z_{2}z_{9} +\\
&+ \beta _{25}z_{2}z^{3}_{3}] + \beta _{26}U_{2}z_{3}z_{7} +
\beta _{27}U_{2}z_{3}y_{7}.
\end{align*}
Multiplying the expression $U_{2}z_{10} = \Phi _{5}z_{2} + X(\beta _{i})$
by $U_{1}$, and using the relations:
$U_{1}z_{10} = \Phi _{5}z_{1}$ and $U_{1}z_{2} = U_{2}z_{1}$
we get the following: $U_{1}X(\beta _{i}) = 0$.
Hence it follows that for all
$i = 1,2,\ldots ,27 $, $\beta _{i} \equiv  0 \mod{2}$.
So, finally we obtain:
\[
U_{2}z_{10} = \Phi _{5}z_{2}.
\]

30. Let us consider the expression: $U_{1}z_{11}$. Having in mind the
action of the differential $d_{3}$:
$d_{3}(U_{1}z_{11}) = U^{2}_{1}U_{3}U_{4}$
and the relation in the MASS: $u_{1}a_{11} = u_{3}a_{8}$,
we get the equality:
\[
U_{1}z_{11} = U_{3}z_{8} + \beta _{1}U_{2}z_{3}z_{7} +
\beta _{2}U_{2}z_{3}y_{7} + \beta _{3}U_{3}y^{2}_{4}.
\]
Let us act by the operation $S_{2,2,2,2}$ on this relation, we
get $\beta _{3}U_{3} = 0$, hence, $\beta _{3} \equiv 0 \mod{2}$.
Acting by the operation $S_{3,3}$, we get:
$U_{1}z_{5} = U_{3}z_{2} + \beta _{1}U_{2}z_{1}z_{3}$, so,
$\beta _{1} \equiv 0 \mod{2}$. Finally, using the operation $S_{5}$,
we get the expression: $\beta _{2} \equiv 0 \mod{2}$.
The final form of the equality is the following:
\[
U_{1}z_{11} = U_{3}z_{8}.
\]

31. Let us consider the product: $U_{3}z_{8}$. Because of the fact
that $d_{3}(U_{3}z_{8}) = U^{2}_{1}U_{3}U_{4}$ and in the MASS
we have the relation: $u_{3}a_{8} = u_{4}a_{4}$, it follows that
it is possible to choose the element $z_{8}$ to fulfill the
equality:
\[
U_{3}z_{8} = U_{4}z_{4} + X(\beta _{i}),
\]
where the expression $X(\beta _{i})$ is defined in the item~29.
Let us multiply this expression by $U_{1}$ and use the relations:
$U_{1}z_{8} = U_{4}z_{1}$ and $U_{1}z_{4} = U_{3}z_{1}$. Then we get:
$U_{1}X(\beta _{i}) = 0$. Hence, $\beta _{i} \equiv 0 \mod{2}$.
So, we have:
\[
U_{3}z_{8} = U_{4}z_{4}.
\]

32. An element $z_{12} \in  E^{0,40}_{2}$ we choose to fulfill the
condition:
$\pi ^{2}_{2}(z_{12}) = b_{12} + a_{2}c_{10} + a_{4}c^{2}_{4}$.
From the condition:
$$S_{2}(z_{12} + \text{any element of this dimension not equal to} \
 z_{12}) = z_{10} + ...$$
it follows that the conditions are fulfilled:
$d_{3}(z_{12}) \neq  0,$
$d_{3}(z_{12} + \ldots ) \neq  0.$
Hence, we can choose $z_{12}$ to satisfy the conditions:
\begin{align*}
d_{3}(z_{12}) &= U^{2}_{1}\Phi _{6} + \beta _{1}U^{2}_{2}\Phi _{5} +
\beta _{2}U_{2}\Phi ^{2}_{3} + \beta _{3}U_{2}U_{3}U_{4} +
\beta _{4}U^{2}_{3}\Phi _{3} + \beta _{5}U^{3}_{2}y^{2}_{4} +\\
&+ \beta _{6}U^{2}_{2}(U_{2}y_{8} + U_{3}y_{6}) +
\beta _{7}U_{2}U_{3}(U_{2}y_{6} + U_{3}y_{4}) +
\beta _{8}U^{2}_{1}U_{3}y^{2}_{4} +\\
&+ \beta _{12}U_{1}U_{4}(U_{1}y_{4}+ U_{2}z_{3})+
\beta _{9}U_{1}U_{3}(U_{1}y_{8}+ U_{2}z_{7})+
\beta _{10}U_{1}\Phi _{3}(U_{1}y_{6}+\\
&+ U_{3}z_{3}) + \beta _{11}U_{1}U_{2}(U_{1}(y_{10} + y^{*}_{10}) +
U_{2}(z_{9} + y_{9})).
\end{align*}
Applying the operation $S_{10}$ to our equality and having in mind
the relation $S_{10}(z_{12}) \equiv  z_{2} \mod{(2z_{2},z^{2}_{1})}$,
we obtain: $d_{3}(z_{2}) = \beta _{11}U^{2}_{1}U_{2} + U^{2}_{1}U_{2}$.
Hence, $\beta _{11}\equiv 0 \mod{2}$. Let us use the operation $S_{5,5}$.
We have the relations:
$S_{5,5}z_{12} \equiv z_{2} \mod{(2z_{2},z^{2}_{1})}$ and
$S_{5,5}\Phi _{6} = U_{2}$.
So, we get the equality: $\beta _{2} \equiv 0 \mod{2}$. From the
conditions: $S_{9}z_{12} \equiv 0 \mod{(2z_{3})}$ and
$S_{9}\Phi _{5} = U_{1}$ it follows that if we apply the operation
$S_{9}$ to our equality we obtain the relation:
$\beta _{1} \equiv 0 \mod{2}$. Using the operation $S_{4,4}$ and
having in mind the facts:
$S_{4,4}z_{12} \equiv 0 \mod{(2z_{4},2y_{4})}$ and
$S_{4,4}\Phi _{6} = 0$, we obtain the relation:
$0 = \beta _{6}U^{3}_{2} + \beta _{9}U^{2}_{1}U_{3}$. Hence,
$\beta _{6}$,$\beta _{9} \equiv 0 \mod{2}$. Let us the operation
$S_{2,2,2,2}$: $S_{2,2,2,2}z_{12} \equiv  z_{4} \mod{(2z_{4},2y_{4})}$,
$S_{2,2,2,2}\Phi _{6} = U_{3}$. Hence,
\[
d_{3}(z_{4}) = U^{2}_{1}U_{3} = U^{2}_{1}U_{3} +
(\beta _{4}+ \beta _{5})U^{3}_{2} +
(\beta _{8} + \beta _{10})U^{2}_{1}U_{3}.
\]
So, we obtain: $\beta _{4} \equiv \beta_{5} \mod{2}$,
$\beta _{8} \equiv  \beta _{10} \mod{2}$. Let us act by the operation
$S_{7}$ on our equality, we obtain:
$0 = \beta _{3}U_{1}U_{2}U_{3} + \beta _{12}U_{1}(U_{1}y_{4} +
U_{2}z_{3})$. Hence, $\beta _{3}$,$\beta _{12}\equiv 0 \mod{2}$.
Using the operation $S_{6}$ and the relation
$S_{6}z_{12}\equiv  z_{6} \mod{(2z_{6},2y_{6})}$,
gives us the relation:
\[
d_{3}(z_{6}) = U^{2}_{1}\Phi _{3} = \beta _{7}U^{2}_{2}U_{3} +
(1 + \beta _{10})U^{2}_{1}\Phi _{3},
\]
from which it follows that $\beta _{7}$,$\beta _{10} \equiv 0 \mod{2}$,
and from the previous relations we get: $\beta _{8} \equiv 0 \mod{2}$.
Let us use the operation $S_{5}$, we obtain that
$0 = \beta _{4}U_{1}U^{2}_{3}$, i.e. $\beta _{4} \equiv  0 \mod{2}$,
and so $\beta _{5} \equiv 0 \mod{2}$. Finally we obtain:
\[
d_{3}(z_{12}) = U^{2}_{1}\Phi _{6}.
\]

33. An element $y_{12} \in  E^{0,48}_{2}$ we choose so that
$\pi ^{2}_{0}(y_{12}) = c_{12}$. From the relation in the MASS:
$S_{2}c_{12} = c^{2}_{5} + c^{2}_{2}c_{6}$, we obtain:
$$S_{2}(y_{12}+ \text{any element of this dimension not equal to} \
y_{12}) = y^{*}_{10} + \ldots .$$
Hence, $d_{3}(y_{12}) \neq 0$, $d_{3}(y_{12} + \ldots ) \neq  0$.
Choose the element $y_{12}$ so that:
\begin{multline*}
d_{3}(y_{12}) = \beta _{1}U_{2}\Phi ^{2}_{3} +
\beta _{2}U^{2}_{2}\Phi _{5} + \beta _{3}U_{2}U_{3}U_{4} +
\beta _{4}U^{2}_{3}\Phi _{3} + \beta _{5}U^{3}_{2}y^{2}_{4} +\\
+ \beta _{6}U^{2}_{2}(U_{2}y_{8} + U_{3}y_{6}) +
\beta _{7}U_{2}U_{3}(U_{2}y_{6} + U_{3}y_{4}).
\end{multline*}
Let us apply the operation $S_{8}$ to this equality, we have:
$S_{8}y_{12}\equiv 0 \mod{(2y_{4})}$, so $0 = \beta _{2}U^{3}_{2}$.
Hence: $\beta _{2} \equiv 0 \mod{2}$. Let us apply the operation
$S_{4,4}$. From the condition $S_{4,4}y_{12} \equiv 0 \mod{(2y_{4})}$,
we obtain $0 = \beta _{1}U^{3}_{2} +  \beta _{6}U^{3}_{2}$. Hence,
$\beta _{1} \equiv \beta _{6} \mod{2}$. Let us act by the operation
$S_{2,2,2,2}$ on our equality, we have:
$S_{2,2,2,2}y_{12} \equiv y_{4} \mod{(2y_{4})}$.
Hence, we have:
$d_{3}(y_{4}) = U^{3}_{2} = \beta _{1}U^{3}_{2} +
\beta _{4}U^{3}_{2} + \beta _{5}U^{3}_{2} + \beta _{6}U^{3}_{2}$.
So:
$\beta _{1}$ + $\beta _{4}$ + $\beta _{5}$ + $\beta _{6} \equiv 1
\mod{2}$.
Let us apply the operation $S_{6}$, then we have:
$S_{6}y_{12} \equiv  y_{6} \mod{(2y_{6})}$. So it will be:
$d_{3}(y_{6}) = U^{2}_{2}U_{3} = (\beta _{3} + \beta _{6} +
\beta _{7})U^{2}_{2}U_{3}$, hence, $\beta _{3}$ + $\beta _{6}$ +
$\beta _{7} \equiv 1 \mod{2}$. Let us apply the operation $S_{3,3}$,
then we obtain: $S_{3,3}y_{12} \equiv y_{6} \mod{(2y_{6})}$.
Hence, we have:
$d_{3}(y_{6}) = U^{2}_{2}U_{3} = \beta _{3}U^{2}_{2}U_{3} +
\beta _{4}U^{2}_{1}\Phi _{3}$. So, $\beta _{3} \equiv 1
\mod{2}$.
Let us apply the operation $S_{2,2,2}$, then we obtain:
$S_{2,2,2}y_{12} \equiv  y_{6} \mod{(2y_{6})}$, and so
$d_{3}(y_{6}) = U^{2}_{2}U_{3} = \beta _{4}U^{2}_{2}U_{3} +
\beta _{6}U^{2}_{2}U_{3}$. Hence, $\beta _{4}$ +  $\beta _{6}
\equiv 1 \mod{2}$. Let us act by the operation $S_{2,2}$ on our
equality. From the relation
$S_{2,2}y_{12} \equiv  y_{8} \mod{(2y_{8})}$ we get:
$$d_{3}(y_{8}) = U_{2}U^{2}_{3} = \beta _{1}U_{2}U^{2}_{3} +
U^{2}_{2}\Phi _{3} + \beta _{4}U^{2}_{2}\Phi _{3} +
\beta _{4}U_{2}U^{2}_{3} + \beta _{7}U_{2}U^{2}_{3}.$$
Hence, $\beta _{4} \equiv 1 \mod{2}$,
$\beta _{1} \equiv  \beta _{7} \mod{2}$. So,
$\beta _{1},\beta _{5},\beta _{6},\beta _{7} \equiv 0 \mod{2}$.
Finally we obtain:
\[
d_{3}(y_{12}) = U_{2}U_{3}U_{4} + U^{2}_{3}\Phi _{3}.
\]

34. There is a formula in the MASS:
$b^{2}_{6} = a^{2}_{2}e_{8} + a_{1}a_{7}e_{4} + a^{2}_{1}e_{10}$.
Hence in the Adams-Novikov spectral sequence we have a relation:
\[
z^{2}_{6} = (1 + 2\beta _{1})z^{2}_{2}y_{8} +
(1 + 2\beta _{2})z_{1}z_{7}y_{4} +
(1 + 2\beta _{3})z^{2}_{1}(y^{*}_{10} + y_{4}y_{6}) + X^{1}_{48}.
\]
Here $X^{1}_{48} \in  F^{6}(E^{0,48}_{2})$, $d_{3}(X^{1}_{48}) = 0$.

35. In the MASS there is the relation:
$h_{0}b_{12} = a_{1}a_{11} + a_{4}a_{8}$. Hence, we can change the
choice of $z_{12}$, not changing the properties, to satisfy the
relation:
\[
2z_{12} = z_{1}z_{11} + 3z_{4}z_{8} + X^{2}_{48}(\widehat{y_{12}},
\widehat{z_{1}z_{11}}, \widehat{z_{4}z_{8}}),
\]
where $X^{2}_{48} \in  F^{6}(E^{0,48}_{2})$, $d_{3}(X^{2}_{48}) = 0,$
and the meaning of the symbol $\widehat{\quad}$ is the same as in the
item~25.

36. Consider the expression $U_{1}z_{12}$. Because of the formula
$d_{3}(U_{1}z_{12}) = U^{3}_{1}\Phi _{6}$ and the equality in the MASS:
$u_{1}(b_{12} + a_{2}c_{10} + a_{4}e^{2}_{4})$, it is possible to
change the choice of $z_{12}$, conserving its properties to satisfy
the relation:
\[
U_{1}z_{12} = \Phi _{6}z_{1} + \beta _{1}U_{1}z_{4}z_{8}.
\]
From the relation of the item~35 it follows that
$S_{7}z_{12} \equiv 0 \mod{(2z_{1}z_{4})}$, hence applying the operation
$S_{7}$, we arrive to relation: $\beta _{1} \equiv 0 \mod{2}$. So,
\[
U_{1}z_{12} = \Phi _{6}z_{1}.
\]

37. Let us consider the expression $U_{2}y_{11}$. From the action of
$d_{3}$ : $d_{3}(U_{2}y_{11}) = U_{1}U^{2}_{2}\Phi _{5}$ and the relation
in the MASS: $u_{2}(b_{11} + a_{5}c_{6}) = \varphi _{5}z_{3}$, it
follows that the equality is fulfilled:
\[
U_{2}y_{11} = \Phi _{5}z_{3} + Y_{49}(\beta _{i}),
\]
where $Y_{49} \in  F^{4}(E^{1,49}_{2})$ and $d_{3}(Y_{49}) = 0$.
Multiplying this expression by $U_{1}$, and using the relations:
$U_{1}y_{11} = \Phi _{5}z_{2}$ and $U_{1}z_{3} = U_{2}z_{2}$, we
obtain: $Y_{49}U_{1} = 0$, that gives after the consideration the
relations $\beta _{i} \equiv 0 \mod{2}$ for all $i$.
Hence,
\[
U_{2}y_{11} = \Phi _{5}z_{3}.
\]

38. Let us consider the expression $U_{2}z_{11}$. From the
condition $d_{3}(U_{2}z_{11}) = U_{1}U_{2}U_{3}U_{4}$
and relation in the MASS: $u_{2}a_{11} = u_{3}a_{9} = u_{4}a_{5}$ it
follows that:
\[
U_{2}z_{11} = U_{3}z_{9} + Y_{49}(\beta _{i}), \quad U_{3}z_{9} =
U_{4}z_{5} + Y(\beta _{i}).
\]
where $Y(\beta _{i})$ has the same sense as in the item~36. Multiplying
both last relations by $U_{1}$ and using the equalities:
\[
U_{1}z_{11} = U_{3}z_{8} = U_{4}z_{4}, \quad U_{1}z_{9} = U_{2}z_{8},
\quad U_{1}z_{5} = U_{2}z_{4},
\]
we obtain that in both relations coefficients $\beta _{i} \equiv 0 \mod{2}$.
Hence,
\[
U_{2}z_{11} = U_{3}z_{9} = U_{4}z_{5}.
\]

39. Let us consider the expression $U_{3}y_{9}$. Because of the formula
$d_{3}(U_{3}y_{9}) = U_{1}U^{2}_{3}\Phi _{3}$ and the equality in the
MASS: $u_{3}b_{9} = \varphi _{3}a_{7}$ we get the relation:
$U_{3}y_{9} = \Phi _{3}z_{7} + Y_{49}(\beta _{i})$, where $Y(\beta _{i})$
has the same sense as in the item~37. Multiply this expression by $U_{1}$
and use the relations $U_{1}y_{9} = \Phi _{3}z_{4}$ and
$U_{1}z_{7} = U_{3}z_{4}$. We have for all $\beta _{i}$:
$\beta _{i} \equiv 0 \mod{2}$. The final relation is the following:
\[
U_{3}y_{9} = \Phi _{3}z_{7}.
\]

40. From the relation in the MASS
$u_{2}(a_{1}e_{10} + a_{3}e_{8} + a_{7}e_{4}) = \varphi _{3}b_{7}$
the relation in the Adams-Novikov spectral sequence follows:
\begin{multline*}
U_{2}(z_{1}y^{*}_{10} + z_{3}y_{8} + z_{1}y_{4}y_{6} + z_{7}y_{4}) =
\Phi _{3}y_{7} + \beta _{1}U_{2}z_{1}z_{10} +
\beta _{2}U_{2}z_{2}z_{9} +\\
+ \beta _{3}U_{2}z_{2}y_{9} + \beta _{4}U_{2}z_{3}(z_{2}y_{6} +
z_{4}y_{4}) + Y^{*}_{49}(\beta _{i}).
\end{multline*}
In this expression $Y(\beta _{i}) \in  F^{4}(E^{1,49}_{2})$ denotes
a linear combination of various generators of the given cell of the
Adams-Novikov spectral sequence such that the filtration of these
generators is not less than 4 and there do not participate the
monomials that are already mentioned. Here we consider all 
coefficients $\beta _{i}$ with $5 \le  i \le 37$. Let us multiply
this expression by $U_{1}$ and use the relation of the item~26. We
obtain:
$\beta _{1},\beta _{2},\beta _{3},\beta _{4} \equiv 1 \mod{2}$,
$\beta _{i} \equiv 0 \mod{2}$ for all $5 \le i \le 37$.
So, we get the relation:
\[
\Phi _{3}y_{7} = U_{2}(z_{1}(y^{*}_{10} + z_{10} + y_{4}y_{6}) +
z_{2}(y_{9} + z_{9}) + z_{3}(y_{8} + y_{6}z_{2} + y_{4}z_{4}) +
z_{7}y_{4}).
\]

41. Let us choose the element $\Omega _{1} \in  E^{1,49}_{2}$ so that
$\pi ^{2}_{0}(\Omega _{1}) = \omega _{1}$. For any cycle $x$ of the
differential $d_{3}$, lying in $E^{4,48}_{2} = E^{4,48}_{3}$ of the
Adams-Novikov spectral sequence, there exists an element
$\tilde{x} \in  F^{2}(E^{1,49}_{2})$ such that: $d_{3}(\tilde{x}) = x$.
Hence, the element $\Omega _{1}$ can be chosen in such a way that
$d_{3}(\Omega _{1}) = 0$. Because in the term $E^{*,*}_{4}$ of the
Adams-Novikov spectral sequence for all $s > 4$ the cells
$E^{s,48}_{4}$ consist of zeros all higher differentials map the
element $\Omega _{1}$ to zero. Hence, $\Omega _{1}$ lives to infinity
and define an indecomposable element $\Omega _{1} \in MSp_{49}$. All
the cells $E^{s,49}_{4}$ of the Adams-Novikov spectral sequence for
$s > 1$ consist of zeros, so there is no extension problem for the
cell $E^{1,49}_{\infty }$. Hence, the order of the element
$\Omega _{1} \in MSp_{49}$ is equal to two.

42. There is the relation in the MASS
\[
u_{1}\omega _{1} = u_{2}(\varphi _{6} + u_{2}c_{10}+ u_{4}e^{2}_{4}) +
u_{3}(\varphi _{5} + u_{3}c_{6}) +u_{4}\varphi _{3}.
\]
Hence it is possible to choose $\Omega _{1}$ in such a way that the
equality is satisfied:
\begin{multline*}
U_{1}\Omega _{1}= U_{2}\Phi _{6}+ U_{3}\Phi _{5}+ U_{4}\Phi _{3}+
U^{2}_{2}y_{10}+ U^{2}_{3}y_{6}+\\
+ \beta _{1}U^{2}_{2}z^{2}_{5} + \beta _{2}U^{2}_{2}z_{3}y_{7}+
U_{2}U_{3}y^{2}_{4}.
\end{multline*}
Let us apply the operation $S_{3,3}$ to this equality. Then we get
the relation:
$U_{1}S_{3,3}\Omega _{1} = U^{2}_{2}S_{3,3}y_{10} + U^{2}_{1}y_{6} +
\beta _{1}U^{2}_{2}z^{2}_{3}$ (because of relation
$U_{1}z_{5} = U_{3}z_{2}$ we have
$S_{3}z_{5} \equiv z_{2} \mod{(2z_{2},2z^{2}_{1})}$, and because of
relation $U_{1}y_{7} = \Phi _{3}z_{2}$ we have
$S_{3,3}y_{7} \equiv  0 \mod{(2z_{1}) )}$. So,
$S_{3,3}y_{10} \equiv  \sigma z^{2}_{2} + z_{4}
\mod{(2z^{2}_{2},2z_{4},2z^{2}_{1}z_{2},z_{1}z_{3},z^{4}_{1})}$.
Let us choose $\Omega _{1}$ so that
$S_{3,3}\Omega _{1} \equiv  \sigma U_{1}z^{2}_{3}
\mod{(U_{1}z^{2}_{1}z^{2}_{2},U_{1}z^{6}_{1})}$. Then we have
$\beta _{1} \equiv 0 \mod{2}$. Let us apply the operation $S_{5}$:
$S_{5}y_{7} \equiv  z_{2} \mod{(2z_{2},z^{2}_{1})}$ (corollary of
the condition
$U_{1}y_{7} = \Phi _{2}z_{2})$, $S_{5}y_{10} \equiv
\gamma z_{2}z_{3} \mod{(2z_{2}z_{3})}$, $S_{5}y_{6} \equiv  0
\mod{(2z_{1})}$.
We obtain: $U_{1}S_{5}\Omega _{1} = U_{1}U_{4} +
U_{1}U_{2}(\gamma  + \beta _{2})z^{2}_{3}$. Choose $\Omega _{1}$
so that
$S_{5}\Omega _{1} \equiv  U_{4} + \gamma U_{2}z^{2}_{3}
\mod{(U_{2}2z^{2}_{3})}$. Then $\beta _{2} \equiv 0 \mod{2}$.
We have:
\[
U_{1}\Omega _{1} = U_{2}\Phi _{6} + U_{3}\Phi _{5} + U_{4}\Phi _{3} +
U^{2}_{2}y_{10} + U^{2}_{3}y_{6} + U_{2}U_{3}y^{2}_{4}.
\]

43. We choose an element $z_{13} \in  E^{0,52}_{2}$ to satisfy the
relation: $\pi ^{2}_{2}(z_{13}) = a_{13}$. From the condition
$$S_{9}(z_{13} + \text{any element of this dimension not equal to}
\ z_{13}) = z_{4} + \text{decomposables}$$
we obtain:
$d_{3}(z_{13}) \neq  0$, $d_{3}(z_{13} + \ldots ) \neq  0$.
It is possible to choose $z_{13}$ so that
\begin{align*}
d_{3}(z_{13}) &= U^{2}_{1}\Omega _{1} +
\beta _{1}U_{1}U_{2}\Phi _{6} + \beta _{2}U_{1}U_{3}\Phi _{5} +
\beta _{4}U_{1}U_{4}\Phi _{3} +
\beta _{5}U_{1}U_{2}U_{3}y^{2}_{4} +\\
&+ \beta _{6}U^{2}_{1}y^{2}_{6} + \beta _{7}U_{1}U_{3}(U_{2}y_{8} +
U_{3}y_{6}) + \beta _{8}U_{1}U_{2}(U_{2}y_{10} + U_{3}y_{8}) +\\
&+ \beta _{19}U^{2}_{1}(U_{1}y_{12} + U_{2}z_{11} + U_{3}y_{8}) +
\beta _{20}U^{2}_{1}y^{2}_{4}(U_{1}y_{4} + U_{2}z_{3}) +\\
&+ \beta _{9}U^{2}_{2}(U_{1}(y^{*}_{10} + y_{10}) +
U_{2}(y_{9} + z_{9}))+ \beta _{10}U^{2}_{2}(U_{1}y_{10} +
U_{3}z_{7}) +\\
&+ \beta _{11}U_{2}U_{3}(U_{2}y_{8} + U_{2}z_{7}) +
\beta _{12}U^{2}_{3}(U_{1}y_{6} + U_{3}z_{3}) +\\
&+ \beta _{13}U_{3}\Phi _{3}(U_{1}y_{4} + U_{2}z_{3}) +
\beta _{14}U_{2}\Phi _{3}(U_{1}y_{6} + U_{3}z_{3}) +\\
&+ \beta _{15}U_{2}U_{4}(U_{1}y_{4} + U_{2}z_{3}) +
\beta _{16}U_{1}\Phi _{3}(U_{2}y_{6} + U_{3}y_{4}) +\\
&+ \beta _{17}U_{1}U_{2}(U_{2}(y^{*}_{10}+ y_{10})+
\Phi _{3}y_{6}+ U_{4}y_{4}) +\\
&+ \beta _{18}U^{2}_{2}(U_{1}y_{6}y_{4} + U_{2}(y_{6}z_{3} +
y_{4}z_{5})) +\\
&+ \beta _{21}U^{2}_{1}(U_{1}y_{8}y_{4} + U_{2}y_{8}z_{3} +
U_{3}y_{4}z_{5}).
\end{align*}
Applying the operation $S_{12}$ and having in mind that
$S_{12}z_{13} \equiv 0 \mod{(2z_{1})}$, $S_{12}\Omega _{1} = 0$,
we obtain the relation: $\beta _{19} \equiv 0 \mod{2}$. Using the
operation $S_{6,6}$ and relations
$S_{6,6}z_{13} \equiv 0 \mod{(2z_{1})}$, $S_{6,6}\Omega _{1} = 0$,
we get: $\beta _{6} \equiv 0 \mod{2}$. Let us consider the action
of the operation $S_{2,2,2,2,2,2}$. There are relations:
$S_{2,2,2,2,2,2}z_{13} \equiv  0 \mod{(2z_{1})}$,
$S_{2,2,2,2,2,2}\Omega _{1} = 0$, hence,
$\beta _{20} \equiv 0 \mod{2}$. Now we apply the operation $S_{11}$.
Because of the formulae $S_{11}z_{13} \equiv  z_{2}
\mod{(2z_{2},z^{2}_{1})}$, $S_{11}\Omega _{1} \equiv U_{1}$,
we have the relation:
$d_{3}(z_{2}) = (1 + \beta _{1})U^{2}_{1}U_{2} = U^{2}_{1}U_{2}$,
or,
$\beta _{1} \equiv 0 \mod{2}$. Let us use the operation $S_{10}$.
Because of relations:
$S_{10}z_{13} \equiv  0 \mod{(2z_{3},z_{1}z_{2},2z^{3}_{1})}$,
$S_{10}\Omega _{1} \equiv  0 \mod{(U_{1}z^{2}_{1})}$,
we obtain that
$\beta _{8} + \beta _{9} + \beta _{10} + \beta _{17} \equiv 0 \mod{2}$.
Let us consider the action of the operation $S_{5,5}$ on our equality.
We have $S_{5,5}z_{13} \equiv 0 \mod{(2z_{3},z_{1}z_{2},2z^{2}_{1})}$
and $S_{5,5}\Omega _{1} \equiv 0 \mod{(U_{1}z^{2}_{1})}$. This gives the
relation: $\beta _{8}$ + $\beta _{10} \equiv 0 \mod{2}$.
Using the relations:
$S_{9}z_{13} \equiv  z_{4}
\mod{(2z_{4},z_{1}z_{3},2y_{4},z^{2}_{2},2z^{2}_{1}z_{2},z^{4}_{1})}$
and $S_{9}\Omega _{1} \equiv  U_{3} \mod{(U_{2}z^{2}_{1})}$, and
applying the operation $S_{9}$ to the equality under the consideration
we get the following:
$d_{3}(z_{4}) = (1 + \beta _{2})U^{2}_{1}U_{3}=U^{2}_{1}U_{3},$
so, $\beta _{2} \equiv  0 \mod{2}$. Let us consider the action of the
operation $S_{4,4}$, we have:
\begin{align*}
S_{4,4}z_{13} &\equiv  0
\mod{(2z_{5},z_{1}z_{4},z_{2}z_{3},2z^{2}_{1}z_{3},2z_{1}z^{2}_{2},
2z^{5}_{1},2z_{1}y_{4},2z^{3}_{1}z_{2})},\\
S_{4,4}\Omega _{1} &\equiv  0 \mod{(U_{1}z^{4}_{1},U_{1}z^{2}_{2})}.
\end{align*}
It follows from these relations that:
\[
\beta _{4}U_{1}U_{2}U_{3}+ \beta _{7}U_{1}U_{2}U_{3} +
\beta _{8}U_{1}U_{2}U_{3} + \beta _{11}U_{1}U_{2}U_{3} +
\beta _{21}U^{2}_{1}(U_{1}y_{4} + U_{2}z_{3}) \equiv  0.
\]
Hence, $\beta _{21} \equiv 0 \mod{2}$, $\beta _{4} + \beta _{7} +
\beta _{8} + \beta _{11} \equiv  0 \mod{2}$.
Let us apply the operation $S_{2,2,2,2}$. We obtain the following
equalities:
\begin{align*}
S_{2,2,2,2}z_{13} &\equiv  0
\mod{(2z_{5},z_{1}z_{4},z_{2}z_{3},2z^{2}_{1}z_{3},2z_{1}z^{2}_{2},
2z^{5}_{1},2z_{1}y_{4},2z^{3}_{1}z_{2})},\\
S_{2,2,2,2}\Omega _{1} &\equiv 0 \mod{(U_{1}z^{2}_{2},U_{2}z^{4}_{1})}.
\end{align*}
From these conditions it follows that:
$\beta _{5}U_{1}U_{2}U_{3} + \beta _{7}U_{1}U_{2}U_{3} +
\beta _{14}U_{1}U_{2}U_{3} \equiv 0$. Hence,
$\beta _{5}$ + $\beta _{7}$ + $\beta _{14} \equiv 0 \mod{2}$.
Let us consider the action of the operation $S_{7}$, we have:
\begin{align*}
S_{7}z_{13} \equiv  z_{6}
\mod{&(2z_{6},2y_{6},z_{1}z_{5},z_{2}z_{4},z^{2}_{3},2z^{2}_{1}z_{4},
z^{2}_{1}z^{2}_{2},z^{3}_{1}z_{3},z^{6}_{1},2z^{2}_{1}y_{4},\\
&2z_{1}z_{2}z_{3},2z_{2}y_{4},2z^{4}_{1}z_{2})},\\
S_{7}\Omega _{1} \equiv  \Phi _{3} \mod{&(U_{3}z^{2}_{1},
U_{2}z^{2}_{2},U_{2}z^{4}_{1})}.
\end{align*}
Hence,
$d_{3}(z_{6}) = U^{2}_{1}\Phi _{3} \equiv
(1 + \beta _{4})U^{2}_{1}\phi _{3} +
(\beta _{15} + \beta _{17})U_{1}U_{2}(U_{1}y_{4} + U_{2}z_{3})$.
So,
$\beta _{4} \equiv  0 \mod{2}$, $\beta _{15} + \beta _{17} \equiv 0
\mod{2}$.
Let us apply the operation $S_{6}$. There are the following relations:
$S_{6}z_{13} \equiv 0 \mod{(2z_{7},2y_{7},\ldots.)}$,
$S_{6}\Omega _{1} \equiv 0 \mod{(U_{1}z^{6}_{1},\ldots.)}$.
It follows from these conditions:
\begin{align*}
\beta _{7}U_{1}U^{2}_{3} + \beta _{9}U^{2}_{2}(U_{1}y_{4} +
U_{2}z_{3}) + \beta _{12}U_{1}U^{2}_{3} +
\beta _{14}U_{1}U_{2}\Phi _{3} +\\
+ \beta _{15}U^{2}_{2}(U_{1}y_{4}+ U_{2}z_{3}) +
\beta _{16}U_{1}U_{2}\Phi _{3} + \beta _{17}U_{1}U_{2}\Phi _{3} +\\
+ \beta _{18}U^{2}_{2}(U_{1}y_{4}+ U_{2}z_{3}) \equiv  0
\mod{(U_{1}z^{6}_{1},\ldots .)}.
\end{align*}
So, we have the following: $\beta _{7} + \beta _{12} \equiv 0 \mod{2}$,
$\beta _{14} + \beta _{16} + \beta _{17} \equiv  0 \mod{2}$,
$\beta _{9} + \beta _{15} + \beta _{18} \equiv 0 \mod{2}$.
From these and previous relations it follows that
$\beta _{18} \equiv 0 \mod{2}$.
Applying $S_{2,2,2}$ and using the relations:
$S_{2,2,2}z_{13} \equiv  0 \mod{(2z_{7},2y_{7},\ldots )}$,
$S_{2,2,2}\Omega _{1} \equiv 0 \mod{((U_{1}z^{6}_{1}, \ldots )}$,
we obtain:
\begin{multline*}
0 \equiv (\beta _{7} + \beta _{12} + \beta _{13})U_{1}U^{2}_{3} +
(\beta _{13}+ \beta _{14} + \beta _{15} +
\beta _{16})U_{1}U_{2}\Phi _{3} +\\
+ (\beta _{9} + \beta _{10} + \beta _{11} + \beta _{12} +
\beta _{13} + \beta _{14})U^{2}_{2}(U_{1}y_{4} + U_{2}z_{3}).
\end{multline*}
Hence,
$\beta _{9} + \beta _{10} + \beta _{11} + \beta _{12} + \beta _{13} +
\beta _{14} \equiv 0 \mod{2}$,
$\beta _{13} + \beta _{14} + \beta _{15} + \beta _{16} \equiv  0
\mod{2}$, $\beta _{7} + \beta _{12} + \beta _{13} \equiv  0 \mod{2}$.
From these and previous conditions it follows that:
$\beta _{13} \equiv 0 \mod{2}$, $\beta _{16} \equiv  0 \mod{ 2}$.
From the action of the operation $S_{2,2}$:
$S_{2,2}z_{13} \equiv  0 \mod{(F^{2}E^{0,36}_{2})}$,
$S_{2,2}\Omega _{1} \equiv  0 \mod{(F^{2}E^{1,35}_{2})}$,
we have:
\begin{multline*}
(\beta _{8} + \beta _{17})U_{1}U_{2}(U_{2}y_{6} + U_{3}y_{4}) +
(\beta _{10} + \beta _{12} + \beta _{14})U^{2}_{2}(U_{1}y_{6} +
U_{3}z_{3}) +\\
+ (\beta _{11} + \beta _{14})U_{2}U_{3}(U_{1}y_{4} + U_{2}z_{3}) +
(\beta _{15} + \beta _{17})U_{1}U_{2}U_{4} \equiv  0
\mod{((U_{1}z^{8}_{1}, \ldots)}.
\end{multline*}
Hence,
$\beta _{8} + \beta _{17} \equiv  \beta _{10} + \beta _{12} +
\beta _{14} \equiv  \beta _{11} + \beta _{14} \mod{2}$,
$\beta _{15} \equiv  \beta _{17} \mod{2}$. From this relation and from
previous ones we obtain:
$\beta _{12} \equiv  \beta _{7} \equiv 0 \mod{2}$,
$\beta _{10} \equiv  \beta _{11} \mod{2}$,
$\beta _{5} \equiv  \beta _{14}  \mod{2}$.
Finally using the operation $S_{1,1,1,1}$ and relations:
$S_{1,1,1,1}z_{13} \equiv  z_{1}y^{2}_{4} \mod{(F^{4}E^{0,36}_{2})}$,
$S_{1,1,1,1}\Omega _{1} \equiv U_{1}y^{2}_{4} \mod{(F^{2}E^{1,35}_{2})}$
we obtain:
\begin{align*}
d_{3}(z_{1}y^{2}_{4}) &= U^{3}_{1}y^{2}_{4} \equiv
(1 + \beta _{5} + \beta _{8} + \beta _{9} + \beta _{10} +
\beta _{17})U^{3}_{1}y^{2}_{4} +\\
&+ (\beta _{8} + \beta _{17})U_{1}U_{2}(U_{2}y_{6} + U_{3}y_{4}) +
(\beta _{9} + \beta _{10})U^{2}_{2}(U_{1}y_{6} + U_{3}z_{3})+\\
&+ (\beta _{11}+ \beta _{15})U_{2}U_{3}(U_{1}y_{4} + U_{2}z_{3}) +
(\beta _{10}+ \beta _{11})U^{2}_{1}(U_{1}y_{8} +\\
&+ U_{2}z_{7}) \mod{(F^{2}E^{1,35}_{2})}.
\end{align*}
Hence,
$\beta _{5} + \beta _{8} + \beta _{9} + \beta _{10} + \beta _{17}
\equiv 0 \mod{2}$, $\beta _{8} + \beta _{17}
\equiv  \beta _{9} + \beta _{10} \equiv  \beta _{11} + \beta _{15}
\mod{2}$, $\beta _{10} \equiv  \beta _{11} \mod{2}$.
From the previous equalities we determine step by step
the meanings of the rest $\beta _{i}$. All of them are nonzero.
So, we have:
\[
d_{3}(z_{13}) = U^{2}_{1}\Omega _{1}.
\]

44. An element $y_{13} \in  E^{0,52}_{2}$ we choose so that
$\pi ^{2}_{2}(y_{13}) = b_{13}$. From the condition:
$$S_{6}(y_{13} + \text{any element of this dimension non-equal to} \
y_{13}) = y_{9} + \ldots $$
it follows that $d_{3}(y_{13}) \neq  0$, $d_{3}(y_{13} + \ldots ) \neq  0$.
Hence we can choose $y_{13}$ so that:
\[
d_{3}(y_{13}) = U_{1}U_{4}\Phi _{3} + \beta _{1}U_{1}U_{2}\Phi _{6} +
\beta _{2}U_{1}U_{3}\Phi _{5} + \beta _{3}U^{2}_{1}\Omega _{1} + \ldots ,
\]
where the dots denote summands starting with the coefficient
$\beta _{5}$ as in the item~43. Using the same operations as in the
item~42, we get analogous relations for the coefficients $\beta _{i}$
excluding the following relations. Using the operation $S_{11}$ gives
the relation: $\beta _{1} \equiv  \beta _{3} \mod{2}$, the operation
$S_{10}$ gives in addition to the equality of the item~43:
$\beta _{1} \equiv  \beta _{2} \mod{2}$. As a result of the action of
the rest of the operations there will be the same values a in the
item~43. The same way we get:
\[
d_{3}(y_{13}) = U_{1}U_{4}\Phi _{3}.
\]

45. An element $y^{*}_{13} \in  E^{0,52}_{2}$ we choose so that
$\pi ^{2}_{2}(y^{*}_{13}) = f_{13} + a_{5}e_{8} + a_{3}c_{10}$.
From the condition:
$$S_{8}(y^{*}_{13} + \text{any element of this dimension non-equal to}
\ y^{*}_{13}) = z_{5} + \ldots$$
we get, that
$d_{3}(y^{*}_{13})\neq 0$, $d_{3}(y^{*}_{13}+\ldots )\neq 0$.
Considerations analogous to those given in the items~42 and 43, show
that it is possible to choose the element $y^{*}_{13}$, so that
the equality is fulfilled:
\[
d_{3}(y^{*}_{13}) = U_{1}U_{2}\Phi _{6}.
\]

The results of the calculations are given in Tables~15~-~17. In
Table~17 the action of the differential $d_{3}$ is given for the
generators of dimension not bigger than 52 in the term
$E_{3} \simeq  E_{2} \simeq  {\rm Ext}_{A}(BP^{*}(MSp),BP^{*})$
of the Adams-Novikov spectral sequence for the spectrum $MSp$, and
in Table~16 there are given relation between generators in the given
dimensions (part of the ''integer" are given only modulo filtration
corresponding to the MASS), which continue the analogous Table from
the work \cite{V2}. Vertical lines in Table~15 denote the multiplication by
the element $U_{1}$ in the term $E_{2}$. Because of the fact
$\pi _{*}(MSp)\otimes {\Bbb Z}_{(p)} \cong
{\Bbb Z}_{(p)}[z_{1},\ldots ,z_{k},\ldots ]$
for all $p >$ 2, Tables~15~-~17 describe the structure of the initial term
and the action of the differential $d_{3}$ in the integer case of the
Adams-Novikov spectral sequence. In this case we mast consider that
the zero line consists of free abelian groups and not of free
${\Bbb Z}_{(2)}$-modules. From Tables~15 and 17 it is possible to see
that in the considering dimensions $E_{4} \simeq  E_{\infty }$ and
the extension problem from $E_{\infty }$ to $\pi _{*}(MSp)$ is trivial.
So, we get a description of the symplectic cobordism ring (not
complete because not all the relation among ''free" generators are
known) up to dimension 52. Table~18 describes the ring $\pi _{*}(MSp)$
as a subring of ${\rm Ext}_{A}U(MU^{*}(MSp),MU^{*})$. To diminish the
volume there is a familiarity in description which we hope does not
lead to ambiguity, for example, many evident relations are not
shown.

\section[Differential $d_{3}$ of the Adams-Novikov s.s.]
{Computations of the action of the differential
$d_{3}$ of the Adams-Novikov spectral sequence}

The aim of the present section is a calculation of the differential
$d_{3}$ on the generators $y_{26}, y^{*}_{26} \in  E^{0,104}_{2}$
of the Adams-Novikov spectral sequence modulo $\theta _{1}$
and elements having $F$-filtration (corresponding to MASS) strictly
greater than zero, $E^{*,*}_{2} \cong
{\rm Ext}_{A}(BP^{*}(MSp),BP^{*})$,  $A= A^{BP}$.
These spectral sequences were described in particular in the
works of the second author \cite{V1, V2, V3, V4}.
The term $E_{\infty }$ of the MASS up to dimension 106 is described in
Chapter~2, and the action of Landweber-Novikov operations on
the generators $c_{j}$ of the MASS is described in Chapter~1.

1. All subsequent calculations are done modulo 2, and images of the
differential $d_{3}$ are expressions considered modulo triple
products which either contain $\theta _{1}$, or  have $F$-filtration
(corresponding to MASS) strictly greater than zero. The scheme of
the calculations will be the same for all dimensions: we write down
an image of the differential $d_{3}$ on a given element as a linear
combination of generators in given dimension with unknown coefficients.
Applying subsequently various Landweber-Novikov operations we define
all the coefficients. All notations dealing with dimensions less
than 32 coincide with the notations of mentioned works of the second
author. The only exclusions are the changes: $V_{i}$ by $U_{i}$ and
$c_{2i-1}$ by $h^{2}_{i}$. This  is true for notations in MASS, in
the Adams-Novikov spectral sequence as well as in the symplectic
cobordism ring $MSp_{*}$. By $\pi ^{2}_{0}(x)$ we denote the
projection of an element $x \in  E^{0,*}_{2}$ of the Adams-Novikov
spectral sequence into the term $E_{\infty }$ of the MASS.
\begin{remarka}
Our notation $y_{6}$ corresponds to $z_{6}$ from the works of
the second author.
\end{remarka}

2. Choose elements $y_{10},y^{*}_{10}\in  E^{0,40}_{2}$ so that
$\pi ^{2}_{0}(y_{10}) = c_{10}, \pi ^{2}_{0}(y^{*}_{10}) =
c_{10}+ c^{2}_{5}+ c^{2}_{2}c_{6}$ and also that:
\begin{equation}\label{eq:c32fi}
d_{3}(y_{10}) = \beta _{1}U^{2}_{2}U_{4} +
\beta _{2}U_{2}U_{3}\Phi _{3} + \beta _{3}U^{3}_{3} +
\beta _{4}U^{2}_{2}\tau_{3}.
\end{equation}
In the analogous equality for $y^{*}_{10}$, $\beta _{i}$ are changed by
$\beta ^{*}_{i}$. To determine coefficients $\beta _{i}$ let us use
Landweber-Novikov operations $S_{\omega }$.

\medskip
\begin{longtable}[c]{|c|c|c|c|c|}
\caption{Calculation of the action of differential $d_3$ on
$y_{10}$, $y_{10}^{*}$}\label{t1:LongTable}                                     \\  \hline
Operation    &  & $S_{\omega }$ & $S_{\omega }$ on & Relation among \\
$S_{\omega }$ & $y_{j}$ & on & the right part & the elements $\beta _{i}$.            \\
    &        & $y_{j}$      & of the formula (\ref{eq:c32fi}) & \\  \hline
             \endfirsthead  \hline
\multicolumn{5}{|c|}{\slshape (continuation)}            \\  \hline
                                                                                                          \endhead                      \hline
\multicolumn{5}{|r|}{\slshape to be continued }                                                                                 \\  \hline
                                                                                                          \endfoot                      \hline
                                                                                                          \endlastfoot
  $S_{6}$    &$y_{10}$    &          $0$         &$(\beta _{1} + \beta _{4})U^{3}_{2}$        &$\beta _{1} + \beta _{4} = 0$        \\  \hline
  $S_{6}$    &$y^{*}_{10}$&         $y_{4}$      &$U^{3}_{2}(\beta ^{*}_{1} + \beta ^{*}_{4})$&$\beta ^{*}_{1} + \beta ^{*}_{4} = 1$\\  \hline
  $S_{3,3}$  &$y_{10}$    &          $0$         &$\beta _{1}U^{3}_{2}$                       &$\beta _{1} = 0$                     \\  \hline
  $S_{3,3}$  &$y^{*}_{10}$&         $y_{4}$      &$\beta ^{*}_{1}U^{3}_{2}$                   &$\beta ^{*}_{1} = 1$                 \\  \hline
 $S_{2,2,2}$ &$y_{10}$    &         $y_{4}$      &$\beta _{3}U^{3}_{2}$                       &$\beta _{3} = 1$                     \\  \hline
 $S_{2,2,2}$ &$y^{*}_{10}$&         $y_{4}$      &$\beta ^{*}_{3}U^{3}_{2}$                   &$\beta ^{*}_{3} = 1$                 \\  \hline
  $S_{4}$    &$y_{10}$    &          $0$         &$\beta _{2}U^{2}_{2}U_{3}$                  &$\beta _{2} = 0$                     \\  \hline
  $S_{4}$    &$y^{*}_{10}$&          $0$         &$(1 + \beta ^{*}_{2})U^{2}_{2}U_{3}$        &$\beta ^{*}_{2} = 1$                 \\  \hline
\end{longtable}

It follows from this table that
\[
d_{3}(y^{*}_{10}) = U^{2}_{2}U_{4} + U_{2}U_{3}\Phi _{3} + U^{3}_{3},
\]
\[
d_{3}(y_{10}) = U^{3}_{3}.
\]

3. Choose an element $y_{12} \in  E^{0,48}_{2}$ so that
$\pi ^{2}_{0}(y_{12}) = c_{12}$ and also that:
\addtocounter{equation}{1}
\begin{align*}\label{eq:c32fii}
d_{3}(y_{12}) &= \beta _{1}U^{2}_{2}\Phi _{5} +
\beta _{2}U_{2}U_{3}U_{4} + \beta _{3}U^{2}_{3}\Phi _{3} +
\beta _{4}U_{2}\Phi ^{2}_{3} +  \\
&+ \beta _{5}U^{2}_{2}\tau_{5} + \beta _{6}U^{3}_{2}y^{2}_{4}. \tag{\theequation}
\end{align*}
To determine the coefficients $\beta _{i}$ we use operations $S_{\omega }$.

\medskip
\begin{longtable}[c]{|c|c|c|c|c|}
\caption{Calculation of the action of the differential $d_3$ on
 $y_{12}$}\label{t2:LongTable}                                    \\ \hline
Operation & &$S_{\omega }$ & $S_{\omega }$ on & Relation among \\
$S_{\omega }$&$y_{j}$ & on & the right part & the elements $\beta _{i}$ \\
& & $y_{j}$ & of the formula (\ref{eq:c32fii}) & \\  \hline
\endfirsthead   \hline
\multicolumn{5}{|c|}{\slshape (continuation)} \\ \hline
\endhead \hline
\multicolumn{5}{|r|}{\slshape to be continued} \\ \hline
\endfoot                      \hline
\endlastfoot
$S_{8}$ &$y_{12}$ &$0$ &$\beta _{1}U^{3}_{2}$ &$\beta _{1} = 0$ \\ \hline
$S_{4,4}$ &$y_{12}$ &$0$ &$(\beta _{4} + \beta _{5})U^{3}_{2}$ &$\beta _{4} = \beta _{5}$ \\  \hline
$S_{6}$      &$y_{12}$    &$y_{6}$ &$(\beta _{2} + \beta _{5})U^{2}_{2}U_{3}$ &$\beta _{2} + \beta _{5} = 1$ \\ \hline
$S_{3,3}$ &$y_{12}$ &$y_{6}$ &$\beta _{2}U^{2}_{2}U_{3}$ &$\beta _{2} = 1$ \\ \hline
$S_{2,2,2,2}$&$y_{12}$    &$y_{4}$ &$(\beta _{3} + \beta _{6})U^{3}_{2}$ &$\beta _{3} + \beta _{6} = 1$ \\ \hline
$S_{4}$      &$y_{12}$    &$y^{2}_{4}$ &$(\beta _{3} + 1)U_{2}U^{2}_{3}$ &$\beta _{3} = 1$ \\ \hline
\end{longtable}

Solving the system of equations on the coefficients $\beta_{i}$ we
come to the formula:
\[
d_{3}(y_{12}) = U_{2}U_{3}U_{4} + U^{2}_{3}\Phi _{3}.
\]

4. Choose an element $y_{14} \in  E^{0,56}_{2}$ so that
$\pi ^{2}_{0}(y_{14}) = c_{14}$ and also that:
\addtocounter{equation}{1}
\begin{align*}\label{eq:c32fiii}
d_{3}(y_{14}) &= \beta _{1}U^{2}_{2}\Phi _{6} +
\beta _{2}U_{2}U_{3}\Phi _{5} + \beta _{3}U_{2}\Phi _{3}U_{4} +
\beta _{4}U^{2}_{3}U_{4} +  \\
&+ \beta _{5}U_{3}\Phi ^{2}_{3} + \beta _{6}U^{2}_{2}\tau^{*}_{7} +
\beta _{7}U^{2}_{3}\tau_{3} + \beta _{8}U_{2}\Phi _{3}\tau_{3} +
\beta _{9}U^{2}_{2}U_{3}y^{2}_{4}. \tag{\theequation}
\end{align*}
To determine the coefficients $\beta _{i}$ we again use operations
$S_{\omega }$.

\medskip
\begin{longtable}[c]{|c|c|c|c|c|}
\caption{Calculation of the action of the differential $d_3$ on
$y_{14}$             }\label{t3:LongTable} \\ \hline
Operation & &$S_{\omega }$   &  $S_{\omega }$ on & Relation among \\
$S_{\omega }$ & $y_{j}$ & on & of the right part & the elements $\beta _{i}$ \\
 & & $y_{j}$ & of the formula (\ref{eq:c32fiii}) & \\ \hline
 \endfirsthead                 \hline
\multicolumn{5}{|c|}{\slshape (continuation)} \\ \hline
\endhead                      \hline
\multicolumn{5}{|r|}{\slshape continuation follows } \\  \hline
\endfoot                      \hline
\endlastfoot
$S_{10}$     &$y_{14}$    &$y_{4}$ &$(\beta _{1}+\beta _{6})U^{3}_{2}$ &$\beta _{1} + \beta _{6} = 1$ \\ \hline
$S_{5,5}$    &$y_{14}$    &$y_{4}$ &$ \beta _{1}U^{3}_{2}$ &$\beta _{1} = 1$ \\ \hline
$S_{8}$      &$y_{14}$    &$0$ &$(\beta _{1}+\beta _{2})U^{2}_{2}U_{3}$ &$\beta _{2} = 1$ \\ \hline
$S_{4,4}$ &$y_{14}$ &$0$ &$(\beta _{2}+\beta _{3}+\beta _{5})U^{2}_{2}U_{3}$&$\beta _{3} + \beta _{5} = 1$ \\ \hline
$S_{6}$ &$y_{14}$ &$y_{8}$ &$(\beta _{3}+\beta _{8}+1)U^{2}_{2}\Phi _{3}+$ &$\beta _{3}+\beta _{8}=1$ \\
 & & &$+(1+\beta _{4}+\beta _{7})U_{2}U^{2}_{3}$ &$\beta _{4}\beta _{7}=0$ \\  \hline
$S_{3,3}$    &$y_{14}$    &$0$ &$(1+\beta _{3})U^{2}_{2}\Phi _{3}+$ &$\beta _{3}=1, \beta _{4}=0$   \\
 & & &$+ \beta _{4}U_{2}U^{2}_{3}$ & \\  \hline
$S_{2}$ &$y_{14}$ &$y_{4}y_{8}$ &$U_{2}\Phi ^{2}_{3}+\beta _{9}U^{3}_{2}y^{2}_{4}$ &$\beta _{9} = 0$ \\ \hline
\end{longtable}

Solving the system of equations on the coefficients $\beta_{i}$ we
come to the formula:
\[
d_{3}(y_{14}) = U^{2}_{2}\Phi _{6} + U_{2}U_{3}\Phi _{5} +
U_{2}\Phi _{3}U_{4}.
\]

5. Let us choose $y_{16} \in  E^{0,64}_{2}$ so that
$\pi ^{2}_{0}(y_{16}) = c^{2}_{8}$ and also that:
\addtocounter{equation}{1}
\begin{align*}\label{eq:c32fiv}
d_{3}(y_{16}) &= \beta _{1}U^{2}_{2}\Phi _{7} +
\beta _{2}U_{2}U_{3}\Phi _{6} + \beta _{3}U_{2}\Phi _{3}\Phi _{5} +
\beta _{4}U^{2}_{3}\Phi _{5} + \beta _{5}U_{2}U^{2}_{4} +\\
&+ \beta _{11}U_{2}U_{4}(U_{2}y_{6} + U_{3}y_{4}) +
\beta _{12}U_{2}U^{2}_{3}y^{2}_{4} + \beta _{13}U^{3}_{2}y^{2}_{6} +
\beta _{14}U^{2}_{2}\Phi _{3}y^{2}_{4} +\\
&+ \beta _{6}U_{3}\Phi _{3}U_{4} + \beta _{7}\Phi ^{3}_{3} +
\beta _{8}U_{2}U_{3}(U_{2}(y_{10} + y^{*}_{10}) + \Phi _{3}y_{6} +
U_{4}y_{4}) + \tag{\theequation} \\
&+ \beta _{9}U_{3}\Phi _{3}(U_{2}y_{6} + U_{3}y_{4})+
\beta _{10}U^{2}_{3}(U_{2}y_{6} + U_{3}y_{4}).
\end{align*}

To determine the coefficients $\beta _{i}$ we again use operations
$S_{\omega }$.

\medskip
\begin{longtable}[c]{|c|c|c|c|c|}
\caption{Calculation of the action of the differential $d_3$ on
$y_{16}$             }\label{t3:LongTable} \\ \hline
Operation & &$S_{\omega }$   &  $S_{\omega }$ on & Relation among \\
$S_{\omega }$ & $y_{j}$ & on & of the right part & the elements $\beta _{i}$ \\
 & & $y_{j}$ & of the formula (\ref{eq:c32fiv}) & \\ \hline
 \endfirsthead                 \hline
\multicolumn{5}{|c|}{\slshape (continuation)} \\ \hline
\endhead                      \hline
\multicolumn{5}{|r|}{\slshape continuation follows } \\  \hline
\endfoot                      \hline
\endlastfoot
$S_{12}$     & $y_{16}$   &$0$                   &$ \beta _{1}U^{3}_{2}$                      &$\beta _{1} = 0$                     \\  \hline
$S_{10}$     & $y_{16}$   &$0$                   &$(\beta _{2}+\beta _{8})U^{2}_{2}U_{3}$     &$\beta _{2} = \beta _{8}$            \\  \hline
$S_{5,5}$    & $y_{16}$   &$0$                   &$\beta _{2}U^{2}_{2}U_{3}$                  &$\beta _{2} = 0$                     \\  \hline
$S_{6,6}$    & $y_{16}$   &$y_{4}$               &$(\beta _{5}+\beta _{11}+\beta _{13})U^{3}_{2}$&$\beta _{5}+\beta _{11}+\beta _{13}=1$ \\  \hline
$S_{3,3,3,3}$& $y_{16}$   &$y_{4}$               &$\beta _{5}U^{3}_{2}$                       &$\beta _{5} = 1$                     \\  \hline
$S_{4,4,4}$  & $y_{16}$   &$0$                   &$(\beta _{3}+\beta _{7})U^{3}_{2}$          &$\beta _{3} = \beta _{7}$            \\  \hline
$S_{8}$      & $y_{16}$   &$0$                   &$\beta _{3}U^{2}_{2}\Phi _{3}+ \beta _{4}U_{2}U^{2}_{3}$&$\beta _{3}= 0, \beta _{4}= 0$  \\  \hline
$S_{4,4}$    & $y_{16}$   &$y_{8}$               &$(\beta _{5}+\beta _{6}+\beta _{10})U_{2}U^{2}_{3}$&$\beta _{5}+\beta _{6}+\beta _{10}=1$ \\  \hline
$S_{6}$      & $y_{16}$   &$0$                   &$(\beta _{6}+\beta _{9})U_{2}U_{3}\Phi _{3}+$&$\beta _{6}=\beta _{9}, \beta _{10}=0$  \\
             &            &                      &$+\beta _{10}U^{3}_{3}+\beta _{11}U^{2}_{2}U_{4}$&$\beta _{11}=0$                 \\  \hline
$S_{4}$      & $y_{16}$   &$0$                   &$\beta _{14}U^{3}_{2}y_{4}$                 &$\beta _{14}=0$                      \\  \hline
$S_{2,2}$    & $y_{16}$   &$y_{4}y_{8}$          &$U_{2}\Phi ^{2}_{3} + \beta _{12}U^{3}_{2}y^{2}_{4}$&$\beta _{12}=0$              \\  \hline
\end{longtable}

Solving the linear system we obtain:
\[
d_{3}(y_{16}) = U_{2}U^{2}_{4}.
\]

6. Elements $y_{18},y^{*}_{18} \in  E^{0,72}_{2}$ we choose so that
$\pi ^{2}_{0}(y_{18})=c_{18}$,
$\pi ^{2}_{0}(y^{*}_{18}) = c^{2}_{9} + c^{2}_{2}c_{14} +
c^{2}_{4}(c_{10} + c^{2}_{5}) + c^{3}_{6}$, and also that:
\addtocounter{equation}{1}
\begin{align*}\label{eq:c32fv}
d_{3}(y_{18}) &= \beta _{1}U^{2}_{2}U_{5} +
\beta _{2}U_{2}U_{3}\Phi _{7} + \beta _{3}U_{2}\Phi _{3}\Phi _{6} +
\beta _{4}U_{2}U_{4}\Phi _{5} + \beta _{5}U^{2}_{3}\Phi _{6} +\\
&+ \beta _{7}U_{3}U^{2}_{4} + \beta _{8}\Phi ^{2}_{3}U_{4} +
\beta _{9}U^{2}_{2}(U_{2}y_{14} + U_{3}y_{4}y_{8} +
\Phi _{3}(y_{10}+ y^{*}_{10})+\\
&+ \Phi _{5}y_{6}+ \Phi _{6}y_{4}) +
\beta _{10}U_{2}\Phi _{3}(U_{2}(y_{10}+y^{*}_{10})+
\Phi _{3}y_{6}+ U_{4}y_{4})+\\
&+ \beta _{11}U^{2}_{3}(U_{2}y_{6}+U_{3}y_{4})+
\beta _{12}U^{2}_{3}(U_{2}(y_{10}+y^{*}_{10})+
\Phi _{3}y_{6}+U_{4}y_{4})+ \tag{\theequation} \\
&+ \beta _{13}U_{3}\Phi _{3}(U_{2}y_{8}+ U_{3}y_{6}) +
\beta _{14}U_{2}\Phi _{5}(U_{2}y_{6}+ U_{3}y_{4})+\\
&+ \beta _{15}U_{3}U_{4}(U_{2}y_{6}+ U_{3}y_{4}) +
\beta _{16}\Phi ^{2}_{3}(U_{2}y_{6}+ U_{3}y_{4}) +\\
&+ \beta _{17}U^{2}_{2}U_{3}y^{2}_{6}+
\beta _{18}U^{2}_{2}U_{4}y^{2}_{4} +
\beta _{19}U_{2}U_{3}\Phi _{3}y^{2}_{4} +
\beta _{20}U^{3}_{3}y^{2}_{4} +\\
&+ \beta _{21}U^{2}_{2}(U_{2}y_{6} + U_{3}y_{4})y^{2}_{4}+
\beta _{6}U_{3}\Phi _{3}\Phi _{5}.
\end{align*}
In analogous equality for $y^{*}_{18}$, coefficients $\beta _{i}$
are changed by $\beta ^{*}_{i}$).

To determine the coefficients $\beta _{i}$ we  use operations
$S_{\omega }$.

\medskip
\begin{longtable}[c]{|c|c|c|c|c|}
\caption{Calculation of the action of the differential $d_3$ on
$y_{18}$,$y^{*}_{18}$       }\label{t3:LongTable} \\ \hline
Operation & &$S_{\omega }$   &  $S_{\omega }$ on & Relation among \\
$S_{\omega }$ & $y_{j}$ & on & of the right part & the elements $\beta _{i}$ \\
 & & $y_{j}$ & of the formula (\ref{eq:c32fv}) & \\ \hline
                                                                                                         \endfirsthead                 \hline
\multicolumn{5}{|c|}{\slshape (continuation)}                                                                                        \\  \hline
                                                                                                          \endhead                      \hline
\multicolumn{5}{|r|}{\slshape to be continued }                                                                                 \\  \hline
                                                                                                          \endfoot                      \hline
                                                                                                          \endlastfoot
$S_{14}$     &$y_{18}$    &$0$                   &$(\beta _{1}+\beta _{9})U^{3}_{2}$          &$\beta _{1} = \beta _{9}$            \\  \hline
$S_{14}$     &$y^{*}_{18}$&$y_{4}$               &$(\beta ^{*}_{1}+\beta ^{*}_{9})U^{3}_{2}$  &$\beta ^{*}_{1} + \beta ^{*}_{9} = 1$\\  \hline
$S_{7,7}$    &$y_{18}$    &$0$                   &$\beta _{1}U^{3}_{2}$                       &$\beta _{1} = 0$                     \\  \hline
$S_{7,7}$    &$y^{*}_{18}$&$y_{4}$               &$\beta _{1}U^{3}_{2}$                       &$\beta _{1} = 1$                     \\  \hline
$S_{12}$     &$y_{18}$    &$0$                   &$\beta _{2}U^{3}_{2}$                       &$\beta _{2} = 0$                     \\  \hline
$S_{12}$     &$y^{*}_{18}$&$0$                   &$(\beta ^{*}_{1}+\beta ^{*}_{2})U^{2}_{2}U_{3}$&$\beta ^{*}_{2} = 1$              \\  \hline
$S_{6,6}$    &$y_{18}$    &$y_{6}$               &$(\beta _{4}+\beta _{7}+\beta _{14}+\beta _{15}+$&$\beta _{4}+\beta _{7}+\beta _{14}+\beta _{15}+$\\
             &            &                      &$+\beta _{17})U^{2}_{2}U_{3}$               &$+\beta _{17}=1$                     \\  \hline
$S_{6,6}$    &$y^{*}_{18}$&$y_{6}$               &$(\beta ^{*}_{4}+\beta ^{*}_{7}+\beta ^{*}_{14}+\beta ^{*}_{15}+$&$\beta ^{*}_{4}+\beta ^{*}_{7}+\beta ^{*}_{14}+\beta ^{*}_{15}$+\\
             &            &                      &$+\beta_{17}U^{2}_{2}U_{3}$                 &$+\beta^{*}_{17}=0$                  \\  \hline
$S_{3,3,3,3}$&$y_{18}$    &$y_{6}$               &$\beta _{7}U^{2}_{2}U_{3}$                  &$\beta _{7} = 1$                     \\  \hline
$S_{3,3,3,3}$&$y^{*}_{18}$&$0$                   &$(\beta ^{*}_{7}+1)U^{2}_{2}U_{3}$          &$\beta ^{*}_{7} = 1$                 \\  \hline
$S_{10}$     &$y_{18}$    &$0$                   &$(\beta _{5}+\beta _{11}+\beta _{12})U_{2}U^{2}_{3}+$&$\beta _{5}+\beta _{11}+\beta _{12}=0$\\
             &            &                      &$+(\beta _{3}+\beta _{10})U^{2}_{2}\Phi _{3}$&$\beta _{3} = \beta _{10}$          \\  \hline
$S_{10}$     &$y^{*}_{18}$&$y_{8}$               &$(1+\beta ^{*}_{5}+\beta ^{*}_{11}+\beta ^{*}_{12})U_{2}U^{2}_{3}$&$\beta ^{*}_{5}+\beta ^{*}_{11}+\beta ^{*}_{12}=1$\\
             &            &                      &$+(1+\beta ^{*}_{3}+\beta ^{*}_{10})U^{2}_{2}\Phi _{3}$&$\beta ^{*}_{3} + \beta ^{*}_{10}= 1$\\  \hline
$S_{5,5}$    &$y_{18}$    &$0$                   &$(\beta _{5}+\beta _{11} )U_{2}U^{2}_{3}+$  &$\beta _{5} = \beta _{11}$\\
             &            &                      &$+ \beta _{3}U^{2}_{2}\Phi _{3}$            &$\beta _{3} = 0$                     \\  \hline
$S_{5,5}$    &$y^{*}_{18}$&$y_{8}$               &$(1+\beta ^{*}_{5}+\beta ^{*}_{11})U_{2}U^{2}_{3}+$&$\beta ^{*}_{5} = \beta ^{*}_{11}$\\
             &            &                      &$+ (1+\beta ^{*}_{3})U^{2}_{2}\Phi _{3}$    &$\beta ^{*}_{3} = 1$                 \\  \hline
$S_{8}$      &$y_{18}$    &$0$                   &$\beta _{4}U^{2}_{2}U_{4}+\beta _{5}U^{3}_{3}+$&$\beta _{4},\beta _{5},\beta _{6}= 0$ \\
             &            &                      &$+ \beta _{6}U_{2}U_{3}\Phi _{3}$           &                                     \\  \hline
$S_{8}$      &$y^{*}_{18}$&$0$                   &$(\beta ^{*}_{4}+1)U^{2}_{2}U_{4}+\beta ^{*}_{5}U^{3}_{3}+$&$\beta ^{*}_{5},\beta ^{*}_{6}=0$\\
             &            &                      &$+ \beta ^{*}_{6}U_{2}U_{3}\Phi _{3}$       &$\beta ^{*}_{4} = 1$                 \\  \hline
$S_{4,4}$    &$y_{18}$    &$y_{10}$              &$(1+\beta _{11})U^{3}_{3}+\beta _{8}U^{2}_{2}U_{4}+$&$\beta _{11},\beta _{8},\beta _{13},\beta _{14}=0$\\
             &            &                      &$+\beta _{13}U_{2}U_{3}\Phi _{3}+\beta _{14}U^{2}_{2}\tau_{3}$&                   \\  \hline
$S_{4,4}$    &$y^{*}_{18}$&$y_{10}$              &$(1+\beta ^{*}_{11})U^{3}_{3}+\beta ^{*}_{8}U^{2}_{2}U_{4}+$&$\beta _{11},\beta _{8},\beta _{13},\beta _{14}=0$ \\
             &            &                      &$+\beta ^{*}_{13}U_{2}U_{3}\Phi _{3}+\beta ^{*}_{14}U^{2}_{2}\tau_{3}$&           \\  \hline
$S_{6}$      &$y_{18}$    &$0$                   &$\beta _{15}U_{2}U_{3}U_{4}+\beta _{16}U_{2}\Phi ^{2}_{3}+$&$\beta _{15},\beta _{16}=0$ \\
             &            &                      &$+\beta _{18}U^{3}_{2}y^{2}_{4}+\beta _{21}U^{3}_{2}y^{2}_{4}+$&$\beta _{18} = \beta _{21}$ \\
             &            &                      &$+ \beta _{15}U_{2}U_{3}\tau_{3}$           &                                     \\  \hline
$S_{6}$      &$y^{*}_{18}$&$y_{4}y_{8}$          &$\beta ^{*}_{15}U_{2}U_{3}U_{4}+\beta ^{*}_{16}U_{2}\Phi ^{2}_{3}+$&$\beta ^{*}_{15},\beta ^{*}_{16}=0$\\
             &            &                      &$+\beta ^{*}_{18}U^{3}_{2}y^{2}_{4}+\beta ^{*}_{21}U^{3}_{2}y^{2}_{4}+$&$\beta ^{*}_{18} = \beta ^{*}_{21}$\\
             &            &                      &$+ \beta ^{*}_{15}U_{2}U_{3}\tau_{3}+U_{2}\Phi ^{2}_{3}$ &                        \\  \hline
$S_{3,3}$    &$y_{18}$    &$0$                   &$\beta _{18}U^{3}_{2}y^{2}_{4}$             &$\beta _{18} = 0$                    \\  \hline
$S_{3,3}$    &$y^{*}_{18}$&$y_{4}y_{8}$          &$\beta ^{*}_{18}U^{3}_{2}y^{2}_{4} + U_{2}\Phi ^{2}_{3}$&$\beta ^{*}_{18} = 0$    \\  \hline
$S_{4}$      &$y_{18}$    &$0$                   &$\beta _{19}U^{2}_{2}U_{3}y^{2}_{4}$        &$\beta _{19} = 0$                    \\  \hline
$S_{4}$      &$y^{*}_{18}$&$0$                   &$\beta ^{*}_{19}U^{2}_{2}U_{3}y^{2}_{4}$    &$\beta ^{*}_{19} = 0$                \\  \hline
$S_{2,2,2}$  &$y_{18}$    &$y_{4}y_{8}$          &$U_{2}\Phi ^{2}_{3} + \beta _{20}U^{3}_{2}y^{2}_{4}$&$\beta _{20} = 0$            \\  \hline
$S_{2,2,2}$  &$y^{*}_{18}$&$0$                   &$\beta ^{*}_{20}U^{3}_{2}y^{2}_{4}$         &$\beta ^{*}_{20} = 0$                \\  \hline
\end{longtable}

Solving the linear system we get the following relations:
\[
d_{3}(y^{*}_{18}) = U_{3}U^{2}_{4} + U^{2}_{2}U_{5} +
U_{2}U_{3}\Phi _{7} + U_{2}\Phi _{3}\Phi _{6} + U_{2}U_{4}\Phi _{5},
\]
\[
d_{3}(y_{18}) = U_{3}U^{2}_{4}.
\]

7. Choose the element $y_{20}$ so that $\pi ^{2}_{0}(y_{20}$) = $c_{20}$
and so that:
\addtocounter{equation}{1}
\begin{align*}\label{eq:c32fvi}
d_{3}(y_{20}) &= \beta _{1}U^{2}_{2}\Phi _{9}+
\beta _{2}U_{2}U_{3}U_{5}+ \beta _{3}U_{2}\Phi _{3}\Phi _{7}+
\beta _{4}U_{2}U_{4}\Phi _{6}+ \beta _{5}U_{2}\Phi ^{2}_{5}+ \\
&+ \beta _{6}U^{2}_{3}\Phi _{7}+ \beta _{7}U_{3}\Phi _{3}\Phi _{6}+
\beta _{8}U_{3}U_{4}\Phi _{5}+ \beta _{9}\Phi ^{2}_{3}\Phi _{5}+\\
&+ \beta _{10}\Phi _{3}U^{2}_{4}+\beta _{11}U_{2}U_{3}[U_{2}y_{14}+
U_{3}y_{4}y_{8}+ \Phi _{3}(y_{10}+ y^{*}_{10})+\\
&+ \Phi _{5}y_{6}+ \Phi _{6}y_{4}] + \beta _{12}U^{2}_{3}[U_{2}y_{12}+
\Phi _{3}y_{8}+ U_{4}y_{6}]+\\
&+ \beta _{13}U_{2}U_{4}[U_{2}(y_{10}+ y^{*}_{10})+
\Phi _{3}y_{6}+U_{4}y_{4}]+\beta _{14}U_{3}\Phi _{3}[U_{2}(y_{10}+y^{*}_{10})+\\
&+ \Phi _{3}y_{6}+ U_{4}y_{4}]+\beta _{15}U_{3}\Phi _{3}[U_{2}y_{10}+
U_{3}y_{8}]+\beta _{16}U_{3}U_{4}[U_{2}y_{8} + \\
&+ U_{3}y_{6}] +\beta _{17}\Phi ^{2}_{3}[U_{2}y_{8}+ U_{3}y_{6}]+
\beta _{18}U_{2}\Phi _{6}[U_{2}y_{6}+ \tag{\theequation}\\
&+ U_{3}y_{4}] + \beta _{19}U_{3}\Phi _{5}[U_{2}y_{6}+ U_{3}y_{4}]+
\beta _{20}U_{4}\Phi _{3}[U_{2}y_{6}+ U_{3}y_{4}]+\\
&+ \beta _{21}U_{2}[U_{2}y_{6}+U_{3}y_{4}][U_{2}(y_{10}+y^{*}_{10})+
\Phi _{3}y_{6}+U_{4}y_{4}]+\\
&+ \beta _{22}U^{3}_{2}y^{2}_{8}+
\beta _{23}U^{2}_{2}\Phi _{5}y^{2}_{4}+
\beta _{24}U^{2}_{2}\Phi _{3}y^{2}_{6}+\\
&+ \beta _{25}U_{2}U_{3}U_{4}y^{2}_{4}+
\beta _{26}U_{2}\Phi ^{2}_{3}y^{2}_{4}+
\beta _{27}U^{2}_{3}\Phi _{3}y^{2}_{4}+\\
&+ \beta _{28}U^{2}_{2}[U_{2}y_{8}+U_{3}y_{6}]y^{2}_{4}+
\beta _{29}U_{2}U^{2}_{3}y^{2}_{6}+\beta _{30}U^{3}_{2}y^{4}_{4}.
\end{align*}
To determine the coefficients $\beta _{i}$ we  use operations
$S_{\omega }$.

\medskip
\begin{longtable}[c]{|c|c|c|c|c|}
\caption{Calculation of the action of the differential $d_3$ on
$y_{20}$       }\label{t3:LongTable} \\ \hline
Operation & &$S_{\omega }$   &  $S_{\omega }$ on & Relation among \\
$S_{\omega }$ & $y_{j}$ & on & of the right part & the elements $\beta _{i}$ \\
 & & $y_{j}$ & of the formula (\ref{eq:c32fvi}) & \\ \hline
                                                                                                          \endfirsthead                 \hline
\multicolumn{5}{|c|}{\slshape (continuation)}                                                                                        \\  \hline
                                                                                                          \endhead                      \hline
\multicolumn{5}{|r|}{\slshape to be continued }                                                                                 \\  \hline
                                                                                                          \endfoot                      \hline
                                                                                                          \endlastfoot
$S_{16}$     &$y_{20}$    & $0$                  &$\beta _{1}U^{3}_{2}$                       &$\beta _{1} = 0$                     \\  \hline
$S_{14}$     &$y_{20}$    & $y_{6}$              &$(\beta _{2}+\beta _{11})U^{2}_{2}U_{3}$    &$1 + \beta _{2} = \beta _{11}$       \\  \hline
$S_{7,7}$    &$y_{20}$    & $y_{6}$              &$\beta _{2}U^{2}_{2}U_{3}$                  &$\beta _{2} = 1$                     \\  \hline
$S_{12}$     &$y_{20}$    & $y_{8}$              &$(1+\beta _{6}+\beta _{12})U_{2}U^{2}_{3}+$ &$\beta _{6}=\beta _{12}, \beta _{3}= 0$ \\
             &            &                      &$+ \beta _{3}U^{2}_{2}\Phi _{3}$            &                                     \\  \hline
$S_{4,4,4}$  &$y_{20}$    & $y_{8}$              &$(\beta_{6}+\beta_{8}+\beta_{10}+\beta_{15}+$&$(\beta _{6}+\beta _{8}+\beta _{10}+\beta _{15}+$ \\
             &            &                      &$+\beta_{16})U_{2}U^{2}_{3}+\beta_{9}U^{2}_{2}\Phi_{3}$&$+\beta_{16})=1,\beta_{9}= 0$ \\  \hline
$S_{8,8}$    &$y_{20}$    & $0$                  &$\beta _{5}U^{3}_{2}$                       &$\beta _{5} = 0$                     \\  \hline
$S_{4,4,4,4}$&$y_{20}$    & $0$                  &$(\beta _{9}+\beta_{17}+\beta_{22})U^{3}_{2}$&$\beta _{9}+\beta _{17}+\beta _{22}= 0$ \\  \hline
$S_{3,3,3,3}$&$y_{20}$    & $y_{8}$              &$U_{2}U^{2}_{3}+(\beta_{4}+$                &$\beta _{4} = \beta _{10}$           \\
             &            &                      &$+\beta_{10})U^{2}_{2}\Phi_{3}$             &                                     \\  \hline
$S_{6,6}$    &$y_{20}$    & $0$                  &$(\beta _{6}+\beta _{8}+\beta _{12}+\beta _{16}+$&$\beta _{6}+ \beta _{8}+ \beta _{16}+ \beta _{19}+$ \\
             &            &                      &$+\beta _{19}+\beta _{29})U_{2}U^{2}_{3}+$  &$+ \beta _{12}+ \beta _{29}= 0$      \\
             &            &                      &$+(\beta _{4}+\beta _{10}+\beta _{13}+$     &$\beta _{4}+ \beta _{10}+\beta _{13}+\beta _{18}+$ \\
             &            &                      &$+\beta _{18}+\beta _{20}+\beta _{21}+$     &$+ \beta _{20}+ \beta _{21}=\beta _{24}$ \\
             &            &                      &$\beta _{24})U^{2}_{2}\Phi _{3}$            &                                     \\  \hline
$S_{10}$     &$y_{20}$    &$y_{10}+$             &$(1+\beta _{7}+\beta _{14}+\beta _{15})\times$&$\beta _{7}+ \beta _{14}+ \beta _{15}= 0$ \\
             &            &$+y^{*}_{10}$         &$\times U_{2}U_{3}\Phi _{3}+(\beta _{4}+$   &$\beta _{4}+ \beta _{13} = 1$        \\
             &            &                      &$+\beta _{13})U^{2}_{2}U_{4}+\beta _{6}U^{3}_{3}+$&$\beta _{6} = 0$               \\
             &            &                      &$+(\beta _{21}+\beta _{18})U^{2}_{2}\tau_{3}$&$\beta _{21} = \beta _{18}$         \\  \hline
$S_{5,5}$    &$y_{20}$    &$y_{10}+$             &$(1+\beta _{7}+\beta _{15})U_{2}U_{3}\times$&$\beta _{7} + \beta _{15} = 0$       \\
             &            &$+y^{*}_{10}$         &$\times \Phi _{3}+\beta _{4}U^{2}_{2}U_{4}+$&$\beta _{4} = 1$                     \\
             &            &                      &$+\beta _{18}U^{2}_{2}\tau_{3}$             &$\beta _{18} = 0$                    \\  \hline
$S_{8}$      &$y_{20}$    &$0$                   &$\beta _{7}U^{2}_{3}\Phi _{3}+\beta _{23}U^{2}_{2} \times$ & $\beta _{7} = 0 \beta _{23}= 0$ \\
             &            &                      &$\times U_{3}y^{2}_{4}+\beta _{19}U_{2}U_{3}\tau_{3}+$& $\beta _{19}= 0 \beta _{8} = 0$ \\
             &            &                      &$+\beta _{8}U_{2}U_{3}U_{4}$                &                                     \\  \hline
$S_{4,4}$    &$y_{20}$    &$y_{12}$              &$U_{2}U_{3}U_{4}+ \beta _{17}U_{2}\Phi ^{2}_{3}+$ & $\beta _{17} = \beta _{20}$   \\
             &            &                      &$+U^{2}_{3}\Phi _{3}+ \beta _{20}U_{2}U_{3}\tau_{3}+$ & $\beta _{26} = 0$         \\
             &            &                      &$+\beta _{17}U^{2}_{2}\tau_{5}+\beta _{26}U^{3}_{2}y^{2}_{4}$ &                   \\  \hline
$S_{6}$      &$y_{20}$    &$y_{14}+$             &$U^{2}_{2}\Phi _{6}+ \beta _{20}U_{2}\Phi _{3}\tau_{3}+$&$\beta _{17},\beta _{20} = 0$        \\
             &            &$+y_{6}y^{2}_{4}$     &$+U_{2}U_{3}\Phi _{5}+ \beta _{17}U_{3}\Phi ^{2}_{3}+$&$\beta _{25}+\beta _{28}=1$\\
             &            &                      &$+(\beta _{25}+\beta _{28})U^{2}_{2}U_{3}y^{2}_{4}+$&                             \\
             &            &                      &$+(1 + \beta _{20})U_{2}\Phi _{3}U_{4}$     &                                     \\  \hline
$S_{3,3}$    &$y_{20}$    &$y_{14}+$             &$U^{2}_{2}\Phi _{6}+ \beta _{25}U^{2}_{2}U_{3}y^{2}_{4}+$&$\beta _{25} = 1$                   \\
             &            &$+y_{6}y^{2}_{4}$     &$+U_{2}U_{3}\Phi _{5}+ U_{2}\Phi _{3}U_{4}$ &                                     \\  \hline
$S_{4}$      &$y_{20}$    &$y_{8}y^{2}_{4}$      &$(1 + \beta _{27})U_{2}U^{2}_{3}y^{2}_{4}$  &$\beta _{27} = 0$                    \\  \hline
\end{longtable}

Finally applying the operation $S_{2,2,2,2,2,2,2,2}$ to the equality (vi)
we obtain 0 = $\beta _{30}U^{3}_{2}$, from where it follows
$\beta _{30} = 0$. Solving the linear system on the coefficients $\beta_i$
we obtain:
\[
d_{3}(y_{20}) = U_{2}U_{3}U_{5}+ U_{2}U_{4}\Phi _{6}+
\Phi _{3}U^{2}_{4}+ U_{2}U_{3}U_{4}y^{2}_{4}.
\]

8. Choose elements $y_{22},y^{*}_{22} \in  E^{0,88}_{2}$ so that
$\pi ^{2}_{0}(y_{22})= c_{22}$,
\[
\pi ^{2}_{0}(y^{*}_{22})= c^{2}_{11}+ c_{14}(c^{2}_{4}+ c^{4}_{2})+
c_{10}(c_{12}+ c^{2}_{2}c^{2}_{4}+ c^{6}_{2})+
c_{6}(c^{4}_{2}c^{2}_{4}+c^{2}_{8}+c^{8}_{2}),
\]
and so that:
\addtocounter{equation}{1}
\begin{align*}\label{eq:c32fvii}
d_{3}(y_{22}) &= \beta _{1}U^{2}_{2}\Phi _{10} +
\beta _{2}U_{2}U_{3}\Phi _{9} + \beta _{3}U_{2}\Phi _{3}U_{5} +
\beta _{4}U_{2}U_{4}\Phi _{7} + \beta _{5}U_{2}\Phi _{5}\Phi _{6} +\\
&+ \beta _{6}U^{2}_{3}U_{5} + \beta _{7}U_{3}\Phi _{3}\Phi _{7} +
\beta _{8}U_{3}U_{4}\Phi _{6} + \beta _{9}U_{3}\Phi ^{2}_{5} +
\beta _{10}\Phi ^{2}_{3}\Phi _{6} +\\
&+ \beta _{11}\Phi _{3}U_{4}\Phi _{5} + \beta _{12}U^{3}_{4} +
\beta _{13}U^{2}_{2}[U_{2}(y_{18}+ y^{*}_{18}) + \Phi _{3}y_{14}+
U_{4}y_{4}y_{8} +\\
&+ \Phi _{5}(y_{10} + y^{*}_{10}) + \Phi _{7}y_{6} + U_{5}y_{4}] +
\beta _{14}U_{2}\Phi _{3}[U_{2}y_{14}+ U_{3}y_{4}y_{8}+\\
&+ \Phi _{3}(y_{10}+ y^{*}_{10})+ \Phi _{5}y_{6}+ \Phi _{6}y_{4}] +
\beta _{15}U^{2}_{3}[U_{3}y_{12}+\Phi _{3}y_{10}+ U_{4}y_{8}] +\\
&+ \beta _{16}U_{2}U_{4}[U_{2}y_{12}+U_{3}(y_{10}+ y^{*}_{10})]+
\beta _{17}U_{3}\Phi _{3}[U_{2}y_{12}+ U_{3}(y_{10}+y^{*}_{10})]+\\
&+ \beta _{18}U_{2}\Phi _{5}[U_{2}(y_{10}+ y^{*}_{10}) +
\Phi _{3}y_{6}+ U_{4}y_{4}]+ \beta _{19}U_{3}U_{4}[U_{2}y_{10}+
U_{3}y_{8}]+\\
&+ \beta _{20}U_{3}U_{4}[U_{2}(y_{10}+y^{*}_{10}) + \Phi _{3}y_{6} +
U_{4}y_{4}]+ \beta _{21}\Phi ^{2}_{3}[U_{2}y_{10}+ U_{3}y_{8}]+\\
&+ \beta _{22}\Phi ^{2}_{3}[U_{2}(y_{10}+ y^{*}_{10}) + \Phi _{3}y_{6}+
U_{4}y_{4}]+ \beta _{23}\Phi _{3}U_{4}[U_{2}y_{10}+ U_{3}y_{8}] \tag{\theequation}\\
&+ \beta _{24}U_{2}\Phi _{7}[U_{2}y_{6}+ U_{3}y_{4}]+
\beta _{25}U_{3}\Phi _{6}[U_{2}y_{6}+U_{3}y_{4}]+
\beta _{26}\Phi _{3}\Phi _{5}[U_{2}y_{6}+\\
&+ U_{3}y_{4}]+\beta _{27}U^{2}_{4}[U_{2}y_{6}+ U_{3}y_{4}]+
\beta _{28}U_{2}[U_{2}y_{6}+ U_{3}y_{4}][U_{2}y_{12}+
\Phi _{3}y_{8}+\\
&+ U_{4}y_{6}]+\beta _{29}U_{3}[U_{2}y_{6}+ U_{3}y_{4}][U_{2}y_{10}+
U_{3}y_{8}] + \beta _{30}U_{3}[U_{2}y_{6}+ U_{3}y_{4}]\times \\
&\times [U_{2}(y_{10}+y^{*}_{10})+ \Phi _{3}y_{6}+ U_{4}y_{4}]+
\beta _{31}\Phi _{3}[U_{2}y_{6}+U_{3}y_{4}][U_{2}y_{8}+ U_{3}y_{6}]+\\
&+ \beta _{32}U^{2}_{2}U_{3}y^{2}_{8} +
\beta _{33}U^{2}_{2}U_{4}y^{2}_{6}+
\beta _{34}U_{2}U_{3}\Phi _{3}y^{2}_{6}+ \beta _{35}U^{3}_{3}y^{2}_{6}+
\beta _{36}U^{2}_{2}[U_{2}y_{6}+\\
&+ U_{3}y_{4}]+ \beta _{37}U^{2}_{2}\Phi _{6}y^{2}_{4}+
\beta _{38}U_{2}U_{3}\Phi _{5}y^{2}_{4}+
\beta _{39}U_{2}\Phi _{3}U_{4}y^{2}_{4}+
\beta _{40}U^{2}_{3}U_{4}y^{2}_{4}+\\
&+ \beta _{41}U_{3}\Phi ^{2}_{3}y^{2}_{4}+
\beta _{42}U^{2}_{2}[U_{2}(y_{10}+ y^{*}_{10})+
\Phi _{3}y_{6}+ U_{4}y_{4}]y^{2}_{4}+\\
&+ \beta _{43}U^{2}_{2}U_{3}y^{4}_{4}+ \beta _{44}U^{2}_{3}[U_{2}y_{6}+
U_{3}y_{4}]y^{2}_{4}.
\end{align*}
For the image of the differential $d_{3}$ on the element $y^{*}_{22}$
there is analogous equality with the change of $\beta _{i}$ on
$\beta ^{*}_{i}$). Let us determine the unknown coefficients
applying operations $S_{\omega }$ to equality (vii). For almost all
operations except $S_{8,8}$, $S_{4,4,4,4}$, $S_{6,6}$, $S_{4,4}$ and
$S_{2,2,2,2,2,2,2,2}$ their action on $y_{22}$ and $y^{*}_{22}$ is
the same. Analogously the action of almost all operations except mentioned
on the right part of the formula (vii). Because of this fact only the
action of these operations on the element $y_{22}$ will be given in the
Table and for the relations we give only for $\beta _{i}$, except the
mentioned operations where two cases are given.

\medskip
\begin{longtable}[c]{|c|c|c|c|c|}
\caption{Calculation of the action of the differential $d_3$ on
$y_{22}$ and $y^{*}_{22}$       }\label{t3:LongTable} \\ \hline
Operation & &$S_{\omega }$   &  $S_{\omega }$ on & Relation among \\
$S_{\omega }$ & $y_{j}$ & on & of the right part & the elements $\beta _{i}$ \\
 & & $y_{j}$ & of the formula (\ref{eq:c32fvii}) & \\ \hline
                                                                                                          \endfirsthead                 \hline
\multicolumn{5}{|c|}{\slshape (continuation)}                                                                                        \\  \hline
                                                                                                          \endhead                      \hline
\multicolumn{5}{|r|}{\slshape to be continued }                                                                                 \\  \hline
                                                                                                          \endfoot                      \hline
                                                                                                          \endlastfoot
$S_{18}$     &$y_{22}$    &$0$                   &$(\beta _{1}+ \beta _{13})U^{3}_{2}$        &$\beta _{1} = \beta _{13}$           \\  \hline
$S_{9,9}$    &$y_{22}$    &$0$                   &$\beta _{1}U^{3}_{2}$                       &$\beta _{1} = 0$                     \\  \hline
$S_{6,6,6}$  &$y_{22}$    &$y_{4}$               &$ (\beta _{4}+ \beta _{12}+ \beta _{16}+$   &$\beta _{4}+ \beta _{12}+ \beta _{16}+$ \\
             &            &                      &$+\beta _{27}+ \beta _{28}+ \beta _{33}+$   &$+\beta _{27}+ \beta _{28}+ \beta _{33}+$ \\
             &            &                      &$+ \beta _{24}+ \beta _{36})U^{3}_{2}$      &$+ \beta _{24}+ \beta _{36} = 1$     \\  \hline
$S_{16}$     &$y_{22}$    &$0$                   &$\beta _{2}U^{3}_{2}$                       &$\beta _{2} = 0$                     \\  \hline
$S_{8,8}$    &$y_{22}$    &$0$                   &$(\beta _{2}+ \beta _{5}+ \beta _{9})U^{2}_{2}U_{3}$ &$\beta _{2}+ \beta _{5}+ \beta _{9} = 0$ \\  \hline
$S_{8,8}$    &$y^{*}_{22}$&$y_{6}$               &$(\beta ^{*}_{2}+ \beta ^{*}_{5}+ \beta ^{*}_{9})U^{2}_{2}U_{3}$ &$\beta ^{*}_{2}+ \beta ^{*}_{5}+ \beta ^{*}_{9} = 1$ \\  \hline
$S_{7,7}$    &$y_{22}$    &$y_{8}$               &$U_{2}U^{2}_{3}+ \beta _{3}U^{2}_{2}\Phi _{3}$ &$\beta _{3} = 0$                  \\  \hline
$S_{4,4,4,4}$&$y_{22}$    &$0$                   &$(\beta _{2}+ \beta _{4}+ \beta _{7}+ \beta _{9}+$ &$\beta _{2}+ \beta _{4}+ \beta _{7}+ \beta _{9}+$ \\
             &            &                      &$+ \beta _{11}+ \beta _{16}+ \beta _{17}+$  &$+ \beta _{11}+ \beta _{16}+ \beta _{17}+$ \\
             &            &                      &$+\beta _{21}+\beta _{23}+\beta _{32})U^{2}_{2}U_{3}$ &$+\beta _{21}+\beta _{23}+\beta _{32} = 0$ \\  \hline
$S_{4,4,4,4}$&$y^{*}_{22}$&$y_{6}$               &$(\beta ^{*}_{2}+ \beta ^{*}_{4}+ \beta ^{*}_{7}+ \beta ^{*}_{9}+$ &$(\beta ^{*}_{2}+ \beta ^{*}_{4}+ \beta ^{*}_{7}+ \beta ^{*}_{9}+$ \\
             &            &                      &$+ \beta ^{*}_{11}+ \beta ^{*}_{16}+ \beta ^{*}_{17}+$ &$+ \beta ^{*}_{11}+ \beta ^{*}_{16}+ \beta ^{*}_{17}+$ \\
             &            &                      &$+\beta ^{*}_{21}+\beta ^{*}_{23}+\beta ^{*}_{32})U^{2}_{2}U_{3}$ &$+\beta ^{*}_{21}+\beta ^{*}_{23}+\beta ^{*}_{32} = 1$ \\  \hline
$S_{14}$     &$y_{22}$    &$y_{8}$               &$(\beta _{2}+ \beta _{6})U_{2}U^{2}_{3} +$  &$\beta _{2} = \beta _{6}$            \\
             &            &                      &$+ (\beta _{3}+ \beta _{14})U^{2}_{2}\Phi _{3}$ &$\beta _{3} = \beta _{14}$       \\  \hline
$S_{12}$     &$y_{22}$    &$y_{10}$              &$(\beta _{4}+ \beta _{16})U^{2}_{2}U_{4}+$  &$\beta _{4} = \beta _{16}$           \\
             &            &                      &$+ (\beta _{2}+ \beta _{3}+ \beta _{7} +$   &$\beta _{2} + \beta _{3} =$          \\
             &            &                      &$+ \beta _{17})U_{2}U_{3}\Phi _{3} +$       &$ =\beta _{7}+ \beta _{17}$          \\
             &            &                      &$+ (\beta _{6}+ \beta _{15})U^{3}_{3} +$    &$\beta _{6} = \beta _{15}$           \\
             &            &                      &$+ (\beta_{24}+\beta_{28})U^{2}_{2}\tau_{3}$&$\beta _{24}= \beta _{28}$           \\  \hline
$S_{3,3,3,3}$&$y_{22}$    &$y^{*}_{10}$          &$U^{2}_{2}U_{4}+\beta_{8}U_{2}U_{3}\Phi_{3}+$&$\beta _{8} = 1$                    \\
             &            &                      &$+(\beta_{16}+\beta_{27})U^{2}_{2}\tau_{3}+$&$\beta _{16}= \beta _{27}$           \\
             &            &                      &$+U^{3}_{3}$                                &                                     \\  \hline
$S_{6,6}$    &$y_{22}$    &$y_{10} +$            &$(\beta_{12}+\beta_{24}+\beta_{28}+$        &$\beta _{12}+\beta _{24}+\beta _{28}+$\\
             &            &$+y^{*}_{10}$         &$+\beta_{33})U^{2}_{2}U_{4}+(\beta _{5}+$   &$+\beta _{33}=1, \beta _{5}+$        \\
             &            &                      &$+\beta _{7}+\beta _{8}+\beta _{11}+$       &$+\beta _{7}+\beta _{8}+\beta _{11}+$\\
             &            &                      &$+\beta _{18}+\beta _{20}+\beta _{23}+$     &$+\beta _{18}+\beta _{20}+\beta _{23}+$\\
             &            &                      &$+\beta _{25}+\beta _{26}+\beta _{30}+$     &$+\beta _{25}+\beta _{26}+\beta _{30}+$ \\
             &            &                      &$+\beta _{31}+\beta_{34})U_{2}U_{3}\Phi_{3}+$&$+\beta _{31}+\beta _{34} = 1$      \\
             &            &                      &$+(\beta_{9}+\beta_{35})U^{3}_{3}+$         &$\beta _{9} = \beta _{35}$           \\
             &            &                      &$+(\beta_{16}+\beta_{24}+\beta _{27}+$      &$\beta _{16}+\beta _{24}+\beta _{27}+$\\
             &            &                      &$+\beta _{28}+\beta _{36})U^{2}_{2}\tau_{3}$&$+\beta _{28}+\beta _{36}=0$         \\  \hline
$S_{6,6}$    &$y^{*}_{22}$&$y_{10} +$            &$(\beta^{*}_{12}+\beta^{*}_{24}+\beta^{*}_{28}+$        &$\beta^{*} _{12}+\beta^{*} _{24}+\beta^{*} _{28}+$\\
             &            &$+y^{*}_{10}$         &$+\beta^{*}_{33})U^{2}_{2}U_{4}+(\beta^{*} _{5}+$   &$+\beta^{*} _{33}=1, \beta^{*} _{5}+$        \\
             &            &                      &$+\beta^{*} _{7}+\beta^{*} _{8}+\beta^{*} _{11}+$       &$+\beta^{*} _{7}+\beta^{*} _{8}+\beta^{*} _{11}+$\\
             &            &                      &$+\beta^{*} _{18}+\beta^{*} _{20}+\beta^{*} _{23}+$     &$+\beta^{*} _{18}+\beta^{*} _{20}+\beta^{*} _{23}+$\\
             &            &                      &$+\beta^{*} _{25}+\beta^{*} _{26}+\beta^{*} _{30}+$     &$+\beta^{*} _{25}+\beta^{*} _{26}+\beta^{*} _{30}+$ \\
             &            &                      &$+\beta^{*} _{31}+\beta^{*}_{34})U_{2}U_{3}\Phi_{3}+$&$+\beta^{*} _{31}+\beta^{*} _{34} = 1$      \\
             &            &                      &$+(\beta^{*}_{9}+\beta^{*}_{35}+1)U^{3}_{3}+$       &$\beta^{*} _{9} = \beta^{*} _{35}+1$         \\
             &            &                      &$+(\beta^{*}_{16}+\beta^{*}_{24}+\beta^{*} _{27}+$      &$\beta^{*} _{16}+\beta^{*} _{24}+\beta^{*} _{27}+$\\
             &            &                      &$+\beta^{*} _{28}+\beta^{*} _{36})U^{2}_{2}\tau_{3}$&$+\beta^{*} _{28}+\beta^{*} _{36}=0$         \\  \hline
$S_{4,4,4}$  &$y_{22}$    &$y_{10}$              &$(\beta _{4}+\beta _{5}+\beta _{10}+$       &$\beta _{4}+\beta _{5}+\beta _{10}+$ \\
             &            &                      &$+\beta _{11}+\beta _{16}+\beta _{23})\times$&$+\beta _{11}+\beta _{16}+\beta _{23}=0$ \\
             &            &                      &$\times U^{2}_{2}U_{4}+(\beta _{4}+\beta _{17}+$&$ \beta _{19} = 0$               \\
             &            &                      &$+\beta _{23})U_{2}U_{3}\Phi _{3}+$         &$\beta_{4}+\beta_{17}+\beta_{23}=0$  \\
             &            &                      &$+(\beta _{19}+1)U^{3}_{3}+(\beta _{18}+$   &$\beta _{18}+\beta _{22}+\beta _{24}+$ \\
             &            &                      &$+\beta _{22}+\beta _{24}+\beta _{26}+$     &$+\beta _{26}+\beta _{37}=0$         \\
             &            &                      &$+\beta _{37})U^{2}_{2}\tau_{3}$            &                                     \\  \hline
$S_{10}$     &$y_{22}$    &$y_{12}$              &$(\beta _{4}+\beta _{16}+\beta _{19}+$      &$\beta _{4}+\beta _{16}+\beta _{19}+$\\
             &            &                      &$+\beta _{20}+1)U_{2}U_{3}U_{4}+$           &$+\beta _{20} = 0 $                 \\
             &            &                      &$+(\beta _{7}+1+\beta_{17})U^{2}_{3}\Phi_{3}+$&$\beta _{7}=\beta _{17}$            \\
             &            &                      &$+(\beta_{5}+\beta_{18})U^{2}_{2}\Phi_{5}+$ &$\beta _{5} = \beta _{18}$           \\
             &            &                      &$+U_{2}\Phi^{2}_{3}(\beta_{10}+\beta_{21}+$ &$\beta_{10}+\beta_{21}+\beta_{22}=0$ \\
             &            &                      &$+\beta_{22})+(\beta _{24}+\beta _{25}+$    &$\beta _{24}+\beta_{25}+\beta _{29}+$\\
             &            &                      &$+\beta _{29}+\beta_{30})U_{2}U_{3}\tau_{3}+$&$+\beta _{30}=0$                    \\
             &            &                      &$+(\beta_{37}+\beta_{42})U^{3}_{2}y^{2}_{4}$&$\beta _{37} = \beta _{42}$          \\  \hline
$S_{5,5}$    &$y_{22}$    &$y_{12}$              &$(1+\beta _{19})U_{2}U_{3}U_{4}+$           &$\beta _{5} = 0$                     \\
             &            &                      &$+U^{2}_{3}\Phi_{3}+\beta_{5}U^{2}_{2}\Phi_{5}+$&$\beta _{19} = 0$                \\
             &            &                      &$+(\beta_{10}+\beta_{21})U_{2}\Phi^{2}_{3}+$&$\beta _{10} = \beta _{21}$          \\
             &            &                      &$+(\beta_{25}+\beta_{29})U_{2}U_{3}\tau_{3}+$&$\beta _{37} = 0$                   \\
             &            &                      &$ + \beta _{37}U^{3}_{2}y^{2}_{4} +$        &$\beta _{25} = \beta _{29}$          \\  \hline
$S_{8}$      &$y_{22}$    &$0$                   &$\beta _{5}U^{2}_{2}\Phi _{6}+(\beta _{11}+$&$\beta _{11} = \beta _{4}$           \\
             &            &                      &$+\beta _{4})U_{2}U_{4}\Phi_{3}+(\beta_{7}+$&$\beta_{18} = 0$                     \\
             &            &                      &$+\beta_{10})U_{3}\Phi^{2}_{3}+\beta_{6}\times$&$\beta _{10} = \beta _{7}$        \\
             &            &                      &$\times U_{2}U_{3}\Phi_{5}+\beta_{18}U^{2}_{2}\tau^{*}_{7}+$&$\beta_{24}=\beta_{26}$ \\
             &            &                      &$+(\beta_{24}+\beta_{26})U_{2}\Phi_{3}\tau_{3}+$&$\beta _{37}= \beta _{38}$       \\
             &            &                      &$+(\beta_{37}+\beta_{38})U^{2}_{2}U_{3}y^{2}_{4}+$&$\beta_{5} = 0 $               \\
             &            &                      &$+\beta _{25}U^{2}_{3}\tau_{3}+$            &$\beta _{6} = 0 \beta _{25} = 0$     \\  \hline
$S_{4,4}$    &$y_{22}$    &$0$                   &$\beta_{10}U^{2}_{2}\Phi_{6}+\beta_{21}U^{2}_{2}\tau_{7}+$&                       \\
             &            &                      &$(\beta _{7}+\beta _{21})U_{3}\Phi^{2}_{3}+$&                                     \\
             &            &                      &$ +\beta _{23}U_{2}U_{4}\Phi _{3}+$         &$\beta _{7} = \beta _{41}$           \\
             &            &                      &$+\beta _{7}U_{2}U_{3}\Phi _{5}+$           &                                     \\
             &            &                      &$+\beta _{23}U_{2}U_{3}\tau_{5}+$           &$\beta _{31} = 0$                    \\
             &            &                      &$+\beta _{41}U^{2}_{2}U_{3}y^{2}_{4}+$      &                                     \\
             &            &                      &$+\beta _{31}U_{2}\Phi _{3}\tau_{3}$        &                                     \\  \hline
$S_{6}$      &$y_{22}$    &$y_{16}+$             &$(\beta _{4}+\beta_{24})U^{2}_{2}\Phi _{7}+$&$\beta _{4} + \beta _{24} =$         \\
             &            &$+y_{6}y_{10}$        &$+(\beta _{25}+1)U_{2}U_{3}\Phi _{6}+ $     &$= \beta _{11}+ \beta _{26} =$       \\
             &            &                      &$+U_{2}\Phi_{3}\Phi_{5}(\beta_{11}+\beta_{26})+$&$= \beta _{16} = \beta _{39}$    \\
             &            &                      &$+(1+\beta _{27}+\beta _{4})U_{2}U^{2}_{4}+$&                                     \\
             &            &                      &$+U_{3}\Phi _{3}U_{4}(1+\beta _{7}+$        &$\beta _{40} + \beta _{44} = 1$      \\
             &            &                      &$+\beta_{11}+\beta_{23})+U^{2}_{3}\Phi_{5}+$&$\beta _{17}+ \beta _{26} +$         \\
             &            &                      &$+\beta _{10}\Phi ^{3}_{3}+\beta _{30}U_{2}U_{3}\tau^{*}_{7}+$&$+ \beta _{30} = \beta _{10}$ \\
             &            &                      &$+\beta _{16}U^{2}_{2}\tau^{*}_{9}+\beta _{28}U^{2}_{2}\tau_{9}+$&$\beta _{33} = 0$ \\
             &            &                      &$+\beta _{33}U^{3}_{2}y^{2}_{6}+U_{2}U_{4}\tau_{3}\times $&$\beta _{28} = 0$      \\
             &            &                      &$\times (\beta_{16}+\beta_{24}+\beta _{28})+$&                                    \\
             &            &                      &$+\beta_{23}U_{2}\Phi_{3}\tau_{5}+(\beta_{17}+$&                                  \\
             &            &                      &$+\beta_{26}+\beta_{30})U_{3}\Phi_{3}\tau_{3}+$&                                  \\
             &            &                      &$+(\beta_{40}+\beta_{44})U_{2}U^{2}_{3}y^{2}_{4}+$&                               \\
             &            &                      &$+\beta _{39}U^{2}_{2}\Phi _{3}y^{2}_{4}$   &                                     \\
             &            &                      &$+\beta _{29}U_{2}U_{3}\tau_{7}$            &                                     \\  \hline
$S_{3,3}$    &$y_{22}$    &$y_{16}+$             &$U_{2}U_{3}\Phi _{6}+U^{2}_{3}\Phi _{5}+$   &$\beta _{40} = 1$                    \\
             &            &$+y_{6}y_{10}$        &$+U_{3}\Phi_{3}U_{4}+\beta_{10}\Phi^{3}_{3}+$&                                    \\
             &            &                      &$+\beta _{23}U_{2}\Phi _{3}\tau_{5}+$       &                                     \\
             &            &                      &$+\beta _{17}U_{3}\Phi _{3}\tau_{3}+$       &                                     \\
             &            &                      &$++\beta_{40}U_{2}U^{2}_{3}y^{2}_{4}$       &                                     \\
             &            &                      &$+U_{2}U^{2}_{4}$                           &                                     \\  \hline
$S_{4}$      &$y_{22}$    &$y_{10}y^{2}_{4}$     &$\beta _{7}(U_{2}U_{3}\Phi _{7}+$           &$\beta _{7} = 0$                     \\
             &            &                      &$+U_{3}\Phi_{3}\Phi_{5})+U^{3}_{3}y^{2}_{4}+$&$\beta _{17}= 0$                    \\
             &            &                      &$+\beta _{10}\Phi ^{2}_{3}U_{4}+$           &$\beta _{10}= 0$                     \\
             &            &                      &$+\beta _{23}(U_{2}U_{4}+$                  &$\beta _{23}= 0$                     \\
             &            &                      &$+U_{3}\Phi _{3})\tau_{5}+$                 &                                     \\
             &            &                      &$\beta _{17}U_{2}U_{3}\tau^{*}_{9}$         &                                     \\  \hline
\end{longtable}

Applying the operation $S_{2,2,2,2,2,2,2,2}$ to the considering
equality we get the relation: $\beta _{43}U^{2}_{2}U_{3} = 0$, or
$\beta _{43}= 0$. Applying the operation $S_{3,3,3,3,3,3}$ we get
the relation: $d_{3}y_{4}=\beta _{12}U^{3}_{2}$, or $\beta _{12}= 1$.
Solving the corresponding linear system we have:
\begin{align*}
d_{3}(y_{22}) &= U^{2}_{3}U_{5} + U_{3}U_{4}\Phi _{6} + U^{3}_{4} +
U^{2}_{3}U_{4}y^{2}_{4},\\
d_{3}(y^{*}_{22}) &= U^{2}_{3}U_{5} + U_{3}U_{4}\Phi _{6} +
U^{3}_{4} + U^{2}_{3}U_{4}y^{2}_{4} + U_{3}\Phi ^{2}_{5}.
\end{align*}
9. Choose elements $y_{26},y^{*}_{26} \in  E^{0,104}_{2}$ so that
$\pi^{2}_{0}(y_{26})= c_{26} + c_{10}c^{8}_{2}$,
$\pi ^{2}_{0}(y^{*}_{26}) = c^{2}_{13}+ c^{2}_{2}c^{2}_{11} +
c^{2}_{4}c^{2}_{9} + c^{2}_{5}c^{2}_{8}$
and also such that:
\addtocounter{equation}{1}
\begin{align*}\label{eq:c32fviii}
d_{3}(y_{26}) &= \beta _{1}U^{2}_{2}\Phi _{12} +
\beta _{2}U_{2}U_{3}\Phi _{11} + \beta _{3}U_{2}\Phi _{3}\Phi _{10} +
\beta _{4}U^{2}_{3}\Phi _{10}+ \beta _{5}U_{2}U_{4}\Phi _{9}+\\
&+ \beta _{6}U_{3}\Phi _{3}\Phi _{9}+ \beta _{7}U_{2}\Phi _{5}U_{5}+
\beta _{8}U_{3}U_{4}U_{5}+ \beta _{9}\Phi ^{2}_{3}U_{5}+
\beta _{10}U_{2}\Phi _{6}\Phi _{7}+ \\
&+ \beta _{11}U_{3}\Phi _{5}\Phi _{7}+
\beta _{12}\Phi _{3}U_{4}\Phi _{7}+ \beta _{13}U_{3}\Phi ^{2}_{6}+
\beta _{14}\Phi _{3}\Phi _{5}\Phi _{6}+ \beta _{15}U^{2}_{4}\Phi _{6} + \\
&+ \beta _{16}U_{4}\Phi ^{2}_{5} +
\beta _{17}U_{2}\Phi _{3}[U_{2}(y_{18}+ y^{*}_{18})+
\Phi _{3}y_{14}+ U_{4}y_{4}y_{8}+ \Phi _{5}(y_{10}+ \\
&+ y^{*}_{10})+ \Phi _{7}y_{6}+ U_{5}y_{4}]+
\beta _{19}U^{2}_{3}[U_{2}(y_{18}+ y^{*}_{18})+
\Phi _{3}y_{14}+ U_{4}y_{4}y_{8}+ \\
&+ \Phi _{5}(y_{10}+ y^{*}_{10})+ \Phi _{7}y_{6}+ U_{5}y_{4}] +
\beta _{18}U^{2}_{3}[U_{2}y_{18}+ U_{3}y_{16}]+ \\
&+ \beta _{20}U_{2}\Phi _{5}[U_{2}y_{14}+ U_{3}y_{4}y_{8}+
\Phi _{3}(y_{10}+ y^{*}_{10}) + \Phi _{5}y_{6}+ \Phi _{6}y_{4}]+ \\
&+ \beta _{21}U_{3}\Phi _{3}[U_{2}(y_{10}y_{6}+ y_{8}y^{2}_{4})+
U_{3}y_{14}]+ \beta _{22}U_{3}U_{4}[U_{3}y_{12}+ \Phi _{3}y_{10}+\\
&+ U_{4}y_{8}]+ \beta _{23}U_{3}U_{4}[U_{2}y_{14}+ U_{3}y_{4}y_{8} +
\Phi _{3}(y_{10}+ y^{*}_{10}) + \Phi _{5}y_{6} +\Phi _{6}y_{4}] +\\
&+ \beta _{24}U_{3}\Phi _{6}[U_{2}y_{10}+ U_{3}y_{8}] +
\beta _{25}U_{3}\Phi _{6}[U_{2}(y_{10}+ y^{*}_{10})+ \Phi _{3}y_{6} +\\
&+ U_{4}y_{4}] + \beta _{26}\Phi _{3}\Phi _{5}[U_{2}y_{10}+
U_{3}y_{8}] + \beta _{27}\Phi _{3}\Phi _{5}[U_{2}(y_{10}+ y^{*}_{10})+ \\
&+ \Phi _{3}y_{6} +U_{4}y_{4}]+ \beta _{28}U^{2}_{4}[U_{2}y_{10}+
U_{3}y_{8}]+ \beta _{29}U^{2}_{4}[U_{2}(y_{10}+ y^{*}_{10})+ \\
&+ \Phi _{3}y_{6}+  U_{4}y_{4}]+
\beta _{30}\Phi _{3}\Phi _{6}[U_{2}y_{8}+U_{3}y_{6}]+
\beta _{31}U_{4}\Phi _{5}[U_{2}y_{8}+U_{3}y_{6}]+\\
&+ \beta _{32}U_{2}\Phi _{9}[U_{2}y_{6}+ U_{3}y_{4}]+
\beta _{33}U_{3}U_{5}[U_{2}y_{6}+  U_{3}y_{4}]+
\beta _{34}U_{4}\Phi _{6}[U_{2}y_{6}+\\
&+ U_{3}y_{4}]+ \beta _{35}U_{2}[U_{2}y_{6}+
U_{3}y_{4}][U_{2}y_{16} +  U_{3}(y_{14}+ y_{4}y_{10}+
y_{6}(y_{8}+ y^{2}_{4})) + \\
&+ \Phi _{3}y_{12}+ U_{4}(y_{10}+ y^{*}_{10})] +
\beta _{36}U_{2}(U_{2}y_{8}+  U_{3}y_{6})[U_{2}y_{14}+
U_{3}y_{4}y_{8}+ \\
&+ \Phi _{3}(y_{10}+ y^{*}_{10})+ \Phi _{5}y_{6}+ \Phi _{6}y_{4}]+
\beta _{37}U_{2}[U_{2}y_{10}+  U_{3}y_{8}][U_{2}y_{12}+ \\
&+ U_{3}(y_{10}+ y^{*}_{10})]+
\beta _{38}U_{2}[U_{2}(y_{10}+y^{*}_{10})+ \Phi _{3}y_{6}+
U_{4}y_{4}][U_{2}y_{12}+  \\
&+ U_{3}(y_{10}+ y^{*}_{10})]+ \beta _{39}U_{3}[U_{2}y_{6}+
U_{3}y_{4}][U_{3}y_{12} + \Phi _{3}y_{10}+ U_{4}y_{8}]+ \\
&+ \beta _{40}U_{3}[U_{2}y_{6}+ U_{3}y_{4}][U_{2}y_{14}+
U_{3}y_{4}y_{8}+ \Phi _{3}(y_{10}+y^{*}_{10})+ \Phi _{5}y_{6}+ \\
&+ \Phi _{6}y_{4}] +
\beta _{41}U_{3}[U_{2}y_{8}+ U_{3}y_{6}][U_{2}y_{12}+
\Phi _{3}y_{8}+ U_{4}y_{6}] + \tag{\theequation}\\
&+ \beta _{42}U_{3}[U_{2}y_{8} + U_{3}y_{6}][U_{2}y_{12} +
U_{3}(y_{10}+ y^{*}_{10})] + \beta _{43}\Phi _{3}[U_{2}y_{6} +
U_{3}y_{4}] \times \\
&\times [U_{2}y_{12}+ \Phi _{3}y_{8}+ U_{4}y_{6}]+
\beta _{44}\Phi _{3}[U_{2}y_{6}+  U_{3}y_{4}][U_{2}y_{12}+
U_{3}(y_{10}+ \\
&+ y^{*}_{10})]+ \beta _{45}\Phi _{3}[U_{2}y_{8}+
U_{3}y_{6}][U_{2}y_{10}+ U_{3}y_{8}]+
\beta _{46}\Phi _{3}[U_{2}(y_{10}+ y^{*}_{10}) + \\
&+ \Phi _{3}y_{6} + U_{4}y_{4}][U_{2}y_{8}+ U_{3}y_{6}] +
\beta _{47}U_{4}[U_{2}y_{6}+ U_{3}y_{4}][U_{2}y_{10}+ U_{3}y_{8}] + \\
&+ \beta _{48}U_{4}[U_{2}y_{6}+ U_{3}y_{4}][U_{2}(y_{10}+ y^{*}_{10}) +
\Phi _{3}y_{6}+ U_{4}y_{4}] + \beta _{49}\Phi _{5}[U_{2}y_{6}+ \\
&+ U_{3}y_{4}][U_{3}y_{8}+ U_{4}y_{6}]+
\beta _{50}U^{2}_{2}U_{3}y^{2}_{10}+
\beta _{51}U^{2}_{2}U_{3}(y^{*}_{10})^{2}+
\beta _{52}U^{2}_{2}U_{4}y^{2}_{8}+ \\
&+ \beta _{53}U_{2}U_{3}\Phi _{3}y^{2}_{8}+
\beta _{54}U^{3}_{3}y^{2}_{8}+ \beta _{55}U^{2}_{2}[U_{2}y_{6}+
U_{3}y_{4}]y^{2}_{8}+ \beta _{56}U_{2}U_{3}\Phi _{5}y^{2}_{6} +\\
&+ \beta _{57}U_{2}\Phi _{3}U_{4}y^{2}_{6}+
\beta _{58}U^{2}_{3}U_{4}y^{2}_{6} +
\beta _{59}U^{2}_{2}[U_{2}(y_{10}+ y^{*}_{10})+
\Phi _{3}y_{6}+ U_{4}y_{4}]\times \\
&\times y^{2}_{6} + \beta _{60}U^{2}_{2}y^{6}_{6}[U_{2}y_{10}+
U_{3}y_{8}] + \beta _{61}U_{2}U_{3}[U_{2}y_{8}+ U_{3}y_{6}] +
\beta _{62}U_{2}\Phi _{3}\times \\
&\times [U_{2}y_{6}+ U_{3}y_{4}]y^{2}_{6}+
\beta _{63}U^{2}_{3}[U_{2}y_{6}+ U_{3}y_{4}] +
\beta _{64}U^{2}_{2}U_{3}y^{2}_{4}y^{2}_{6} +
\beta _{65}U^{2}_{2}U_{5}\times\\
&\times y^{2}_{4} + \beta _{66}U_{2}U_{4}\Phi _{5}y^{2}_{4} +
\beta _{67}U^{2}_{3}\Phi _{6}y^{2}_{4}+
\beta _{68}U_{3}\Phi _{3}\Phi _{5}+ \beta _{69}U_{3}U^{2}_{4}y^{2}_{4}+\\
&+ \beta _{70}U^{2}_{2}[U_{2}y_{14} + U_{3}y_{4}y_{8} +
\Phi _{3}(y_{10}+ y^{*}_{10})+ \Phi _{5}y_{6}+
\Phi _{6}y_{4}]y^{2}_{4} + \\
&+ \beta _{71}U_{2}U_{3}[U_{2}y_{12}+ U_{3}(y_{10}+
y^{*}_{10})]y^{2}_{4} + \beta _{72}U_{2}\Phi _{3}[U_{2}y_{10}+
U_{3}y_{8}]y^{2}_{4}+\\
&+ \beta _{73}U_{2}\Phi _{3}[U_{2}(y_{10}+ y^{*}_{10})+
\Phi _{3}y_{6}+ U_{4}y_{4}]y^{2}_{4}+ \beta _{74}U^{2}_{3}[U_{2}y_{10}+
U_{3}y_{8}]y^{2}_{4}+ \\
&+ \beta _{75}U^{2}_{3}[U_{2}(y_{10}+ y^{*}_{10})+ \Phi _{3}y_{6}+
U_{4}y_{4}]+ \beta _{76}U_{2}U_{4}[U_{2}y_{8}+ U_{3}y_{6}]y^{2}_{4}+ \\
&+ \beta _{77}U_{3}\Phi _{3}[U_{2}y_{8}+ U_{3}y_{6}]y^{2}_{4}+
\beta _{78}U_{2}\Phi _{5}[U_{2}y_{6}+U_{3}y_{4}]y^{2}_{4}+
\beta _{79}U_{3}U_{4}\times \\
&\times [U_{2}y_{6}+ U_{3}y_{8}]y^{2}_{4}+
\beta _{80}U^{2}_{2}U_{4}y^{4}_{4}+
\beta _{81}U_{2}U_{3}\Phi _{3}y^{4}_{4}+
\beta _{82}U^{3}_{3}y^{4}_{4} + \beta _{83}U^{2}_{2}\times\\
&\times [U_{2}y_{6} + U_{3}y_{4}]y^{4}_{4} +
\beta _{84}U_{2}[U_{2}y_{6} + U_{3}y_{4}][U_{2}y_{8} +
U_{3}y_{6}]y^{2}_{4} + \beta _{85}U_{2}\Omega ^{2}_{1}.
\end{align*}
To determine the coefficients $\beta _{i}$ we use operations
$S_{\omega }$. When the actions of the operation $S_{\omega }$ on
$y_{26}$ and $y^{*}_{26}$ coincide then in the table this action
is described only for $y_{26}$.

9. Choose elements $y_{26},y^{*}_{26} \in  E^{0,104}_{2}$ so that
$\pi^{2}_{0}(y_{26})= c_{26} + c_{10}c^{8}_{2}$,
$\pi ^{2}_{0}(y^{*}_{26}) = c^{2}_{13}+ c^{2}_{2}c^{2}_{11} +
c^{2}_{4}c^{2}_{9} + c^{2}_{5}c^{2}_{8}$
and also such that:
\addtocounter{equation}{1}
\begin{align*}\label{eq:c32fviii}
d_{3}(y_{26}) &= \beta _{1}U^{2}_{2}\Phi _{12} +
\beta _{2}U_{2}U_{3}\Phi _{11} + \beta _{3}U_{2}\Phi _{3}\Phi _{10} +
\beta _{4}U^{2}_{3}\Phi _{10}+ \beta _{5}U_{2}U_{4}\Phi _{9}+\\
&+ \beta _{6}U_{3}\Phi _{3}\Phi _{9}+ \beta _{7}U_{2}\Phi _{5}U_{5}+
\beta _{8}U_{3}U_{4}U_{5}+ \beta _{9}\Phi ^{2}_{3}U_{5}+
\beta _{10}U_{2}\Phi _{6}\Phi _{7}+ \\
&+ \beta _{11}U_{3}\Phi _{5}\Phi _{7}+
\beta _{12}\Phi _{3}U_{4}\Phi _{7}+ \beta _{13}U_{3}\Phi ^{2}_{6}+
\beta _{14}\Phi _{3}\Phi _{5}\Phi _{6}+ \beta _{15}U^{2}_{4}\Phi _{6} + \\
&+ \beta _{16}U_{4}\Phi ^{2}_{5} +
\beta _{17}U_{2}\Phi _{3}[U_{2}(y_{18}+ y^{*}_{18})+
\Phi _{3}y_{14}+ U_{4}y_{4}y_{8}+ \Phi _{5}(y_{10}+ \\
&+ y^{*}_{10})+ \Phi _{7}y_{6}+ U_{5}y_{4}]+
\beta _{19}U^{2}_{3}[U_{2}(y_{18}+ y^{*}_{18})+
\Phi _{3}y_{14}+ U_{4}y_{4}y_{8}+ \\
&+ \Phi _{5}(y_{10}+ y^{*}_{10})+ \Phi _{7}y_{6}+ U_{5}y_{4}] +
\beta _{18}U^{2}_{3}[U_{2}y_{18}+ U_{3}y_{16}]+ \\
&+ \beta _{20}U_{2}\Phi _{5}[U_{2}y_{14}+ U_{3}y_{4}y_{8}+
\Phi _{3}(y_{10}+ y^{*}_{10}) + \Phi _{5}y_{6}+ \Phi _{6}y_{4}]+ \\
&+ \beta _{21}U_{3}\Phi _{3}[U_{2}(y_{10}y_{6}+ y_{8}y^{2}_{4})+
U_{3}y_{14}]+ \beta _{22}U_{3}U_{4}[U_{3}y_{12}+ \Phi _{3}y_{10}+\\
&+ U_{4}y_{8}]+ \beta _{23}U_{3}U_{4}[U_{2}y_{14}+ U_{3}y_{4}y_{8} +
\Phi _{3}(y_{10}+ y^{*}_{10}) + \Phi _{5}y_{6} +\Phi _{6}y_{4}] +\\
&+ \beta _{24}U_{3}\Phi _{6}[U_{2}y_{10}+ U_{3}y_{8}] +
\beta _{25}U_{3}\Phi _{6}[U_{2}(y_{10}+ y^{*}_{10})+ \Phi _{3}y_{6} +\\
&+ U_{4}y_{4}] + \beta _{26}\Phi _{3}\Phi _{5}[U_{2}y_{10}+
U_{3}y_{8}] + \beta _{27}\Phi _{3}\Phi _{5}[U_{2}(y_{10}+ y^{*}_{10})+ \\
&+ \Phi _{3}y_{6} +U_{4}y_{4}]+ \beta _{28}U^{2}_{4}[U_{2}y_{10}+
U_{3}y_{8}]+ \beta _{29}U^{2}_{4}[U_{2}(y_{10}+ y^{*}_{10})+ \\
&+ \Phi _{3}y_{6}+  U_{4}y_{4}]+
\beta _{30}\Phi _{3}\Phi _{6}[U_{2}y_{8}+U_{3}y_{6}]+
\beta _{31}U_{4}\Phi _{5}[U_{2}y_{8}+U_{3}y_{6}]+\\
&+ \beta _{32}U_{2}\Phi _{9}[U_{2}y_{6}+ U_{3}y_{4}]+
\beta _{33}U_{3}U_{5}[U_{2}y_{6}+  U_{3}y_{4}]+
\beta _{34}U_{4}\Phi _{6}[U_{2}y_{6}+\\
&+ U_{3}y_{4}]+ \beta _{35}U_{2}[U_{2}y_{6}+
U_{3}y_{4}][U_{2}y_{16} +  U_{3}(y_{14}+ y_{4}y_{10}+
y_{6}(y_{8}+ y^{2}_{4})) + \\
&+ \Phi _{3}y_{12}+ U_{4}(y_{10}+ y^{*}_{10})] +
\beta _{36}U_{2}(U_{2}y_{8}+  U_{3}y_{6})[U_{2}y_{14}+
U_{3}y_{4}y_{8}+ \\
&+ \Phi _{3}(y_{10}+ y^{*}_{10})+ \Phi _{5}y_{6}+ \Phi _{6}y_{4}]+
\beta _{37}U_{2}[U_{2}y_{10}+  U_{3}y_{8}][U_{2}y_{12}+ \\
&+ U_{3}(y_{10}+ y^{*}_{10})]+
\beta _{38}U_{2}[U_{2}(y_{10}+y^{*}_{10})+ \Phi _{3}y_{6}+
U_{4}y_{4}][U_{2}y_{12}+  \\
&+ U_{3}(y_{10}+ y^{*}_{10})]+ \beta _{39}U_{3}[U_{2}y_{6}+
U_{3}y_{4}][U_{3}y_{12} + \Phi _{3}y_{10}+ U_{4}y_{8}]+ \\
&+ \beta _{40}U_{3}[U_{2}y_{6}+ U_{3}y_{4}][U_{2}y_{14}+
U_{3}y_{4}y_{8}+ \Phi _{3}(y_{10}+y^{*}_{10})+ \Phi _{5}y_{6}+ \\
&+ \Phi _{6}y_{4}] +
\beta _{41}U_{3}[U_{2}y_{8}+ U_{3}y_{6}][U_{2}y_{12}+
\Phi _{3}y_{8}+ U_{4}y_{6}] + \tag{\theequation}\\
&+ \beta _{42}U_{3}[U_{2}y_{8} + U_{3}y_{6}][U_{2}y_{12} +
U_{3}(y_{10}+ y^{*}_{10})] + \beta _{43}\Phi _{3}[U_{2}y_{6} +
U_{3}y_{4}] \times \\
&\times [U_{2}y_{12}+ \Phi _{3}y_{8}+ U_{4}y_{6}]+
\beta _{44}\Phi _{3}[U_{2}y_{6}+  U_{3}y_{4}][U_{2}y_{12}+
U_{3}(y_{10}+ \\
&+ y^{*}_{10})]+ \beta _{45}\Phi _{3}[U_{2}y_{8}+
U_{3}y_{6}][U_{2}y_{10}+ U_{3}y_{8}]+
\beta _{46}\Phi _{3}[U_{2}(y_{10}+ y^{*}_{10}) + \\
&+ \Phi _{3}y_{6} + U_{4}y_{4}][U_{2}y_{8}+ U_{3}y_{6}] +
\beta _{47}U_{4}[U_{2}y_{6}+ U_{3}y_{4}][U_{2}y_{10}+ U_{3}y_{8}] + \\
&+ \beta _{48}U_{4}[U_{2}y_{6}+ U_{3}y_{4}][U_{2}(y_{10}+ y^{*}_{10}) +
\Phi _{3}y_{6}+ U_{4}y_{4}] + \beta _{49}\Phi _{5}[U_{2}y_{6}+ \\
&+ U_{3}y_{4}][U_{3}y_{8}+ U_{4}y_{6}]+
\beta _{50}U^{2}_{2}U_{3}y^{2}_{10}+
\beta _{51}U^{2}_{2}U_{3}(y^{*}_{10})^{2}+
\beta _{52}U^{2}_{2}U_{4}y^{2}_{8}+ \\
&+ \beta _{53}U_{2}U_{3}\Phi _{3}y^{2}_{8}+
\beta _{54}U^{3}_{3}y^{2}_{8}+ \beta _{55}U^{2}_{2}[U_{2}y_{6}+
U_{3}y_{4}]y^{2}_{8}+ \beta _{56}U_{2}U_{3}\Phi _{5}y^{2}_{6} +\\
&+ \beta _{57}U_{2}\Phi _{3}U_{4}y^{2}_{6}+
\beta _{58}U^{2}_{3}U_{4}y^{2}_{6} +
\beta _{59}U^{2}_{2}[U_{2}(y_{10}+ y^{*}_{10})+
\Phi _{3}y_{6}+ U_{4}y_{4}]\times \\
&\times y^{2}_{6} + \beta _{60}U^{2}_{2}y^{6}_{6}[U_{2}y_{10}+
U_{3}y_{8}] + \beta _{61}U_{2}U_{3}[U_{2}y_{8}+ U_{3}y_{6}] +
\beta _{62}U_{2}\Phi _{3}\times \\
&\times [U_{2}y_{6}+ U_{3}y_{4}]y^{2}_{6}+
\beta _{63}U^{2}_{3}[U_{2}y_{6}+ U_{3}y_{4}] +
\beta _{64}U^{2}_{2}U_{3}y^{2}_{4}y^{2}_{6} +
\beta _{65}U^{2}_{2}U_{5}\times\\
&\times y^{2}_{4} + \beta _{66}U_{2}U_{4}\Phi _{5}y^{2}_{4} +
\beta _{67}U^{2}_{3}\Phi _{6}y^{2}_{4}+
\beta _{68}U_{3}\Phi _{3}\Phi _{5}+ \beta _{69}U_{3}U^{2}_{4}y^{2}_{4}+\\
&+ \beta _{70}U^{2}_{2}[U_{2}y_{14} + U_{3}y_{4}y_{8} +
\Phi _{3}(y_{10}+ y^{*}_{10})+ \Phi _{5}y_{6}+
\Phi _{6}y_{4}]y^{2}_{4} + \\
&+ \beta _{71}U_{2}U_{3}[U_{2}y_{12}+ U_{3}(y_{10}+
y^{*}_{10})]y^{2}_{4} + \beta _{72}U_{2}\Phi _{3}[U_{2}y_{10}+
U_{3}y_{8}]y^{2}_{4}+\\
&+ \beta _{73}U_{2}\Phi _{3}[U_{2}(y_{10}+ y^{*}_{10})+
\Phi _{3}y_{6}+ U_{4}y_{4}]y^{2}_{4}+ \beta _{74}U^{2}_{3}[U_{2}y_{10}+
U_{3}y_{8}]y^{2}_{4}+ \\
&+ \beta _{75}U^{2}_{3}[U_{2}(y_{10}+ y^{*}_{10})+ \Phi _{3}y_{6}+
U_{4}y_{4}]+ \beta _{76}U_{2}U_{4}[U_{2}y_{8}+ U_{3}y_{6}]y^{2}_{4}+ \\
&+ \beta _{77}U_{3}\Phi _{3}[U_{2}y_{8}+ U_{3}y_{6}]y^{2}_{4}+
\beta _{78}U_{2}\Phi _{5}[U_{2}y_{6}+U_{3}y_{4}]y^{2}_{4}+
\beta _{79}U_{3}U_{4}\times \\
&\times [U_{2}y_{6}+ U_{3}y_{8}]y^{2}_{4}+
\beta _{80}U^{2}_{2}U_{4}y^{4}_{4}+
\beta _{81}U_{2}U_{3}\Phi _{3}y^{4}_{4}+
\beta _{82}U^{3}_{3}y^{4}_{4} + \beta _{83}U^{2}_{2}\times\\
&\times [U_{2}y_{6} + U_{3}y_{4}]y^{4}_{4} +
\beta _{84}U_{2}[U_{2}y_{6} + U_{3}y_{4}][U_{2}y_{8} +
U_{3}y_{6}]y^{2}_{4} + \beta _{85}U_{2}\Omega ^{2}_{1}.
\end{align*}
To determine the coefficients $\beta _{i}$ we use operations
$S_{\omega }$. When the actions of the operation $S_{\omega }$ on
$y_{26}$ and $y^{*}_{26}$ coincide then in the table this action
is described only for $y_{26}$.

\chapter[Elements of order 4 in ${ MSp_{*}}$]{Construction of the
elements of order four in ${ MSp_{*}}$}

Let us show that that the product
$U_{1}\Phi _{7}\Omega _{1} \in  E^{3,103}_{2} \simeq  E^{3,103}_{3}$
lives to infinity and not equal to zero in $E_{\infty}$ of the
Adams-Novikov spectral sequence. Each factor in this product is a
nonzero infinite cycle. So, it is sufficient to prove that there is no
element in $E^{0,104}_{2} \simeq  E^{0,104}_{3}$ that kills
$U_{1}\Phi _{7}\Omega _{1}$ by the action of the differential $d_{3}$,
because all higher differentials are equal to zero by the dimension
arguments.

From the conditions:
$S_{13}(U_{1}\Phi _{7}\Omega _{1}) = U^{2}_{1}\Omega _{1}$
and $d_{3}(z_{13}) = U^{2}_{1}\Omega _{1}$, $z_{13}$ has the
$F$-filtration, corresponding to MASS equal to 2, operations
$S_{\omega }$ conserve this filtration, it follows that the killing
element can't have the $F$-filtration greater than 2.
If there exists an element $X$ such that:
$X \in  F^{2}(E^{0,104}_{2})$, $X \not\in  F^{3}(E^{0,104}_{2})$
and $d_{3}(X) = U_{1}\Phi _{7}\Omega _{1}$, then for the projection
$x$ of this element in MASS there must be the relation
$S_{13}x = u_{1}\varphi _{7}\omega _{1}$ and $x \in
E^{0,1,106}_{\infty }$ in MASS. However from the description of
the $E_{\infty }$ it follows that there is no such an element in the
given dimension. Hence the only possible candidate for killing of this
product is a multiplicatively nondecomposable element of zero
filtration in $E^{0,104}_{2}$ of the Adams-Novikov spectral sequence.
There are two such elements: $y_{26}$ and $y^{*}_{26}$ and we proved two
following relations for them:
\[
d_{3}(y_{26}) \equiv  U_{3}\Phi ^{2}_{6} \mod{(F^{2}E^{3,103}_{3}
\cup (\text{elements containing}\  U_{1}) )},
\]
\[
d_{3}(y^{*}_{26}) \equiv  U_{2}\Omega ^{2}_{1}
\mod{(F^{2}E^{3,103}_{3} \cup (\text{elements containing} \ U_{1}) )}.
\]
Hence, if the element $y^{*}_{26}$ (for $y_{26}$ this is analogous)
kills $U_{1}\Phi _{7}\Omega _{1}$, then there must exist another
element $\eta$, different from $y^{*}_{26}$, and such that:
$d_{3}(\eta ) = U_{2}\Omega ^{2}_{1}$. Then from the action of the
operation $S_{11,11}$: $S_{11,11}(U_{2}\Omega ^{2}_{1}) = U^{3}_{2}$
and the action of the differential $d_{3}$, we conclude that
$\eta  = y_{4}y^{*}_{22} + \ldots $. However because of multiplicativity
of the action of $d_{3}$ and the formula
$d_{3}(y^{*}_{22}) = U^{2}_{3}U_{5} + \ldots $,
it follows the necessity of eliminating of the summand
$y_{4}U^{2}_{3}U_{5}$, which is deleted either  with the help of
$d_{3}(y_{4}y_{22})$, (and then there appears $U^{3}_{2}y_{22}$ which is
also undeletable), or again by $d_{3}(y_{4}y^{*}_{22})$. Hence such
$\eta$ does not exist. So, $U_{1}\Phi _{7}\Omega _{1}$ not equal to zero
in $E_{\infty }$ of the Adams-Novikov spectral sequence and hence in
the ring $MSp_{*}$.

Let us see the connection between the elements
$<\Phi _{7},2,\Omega _{1}>$ and $\theta _{1}\Phi _{7}\Omega _{1}$ of
the ring $MSp_{*}$.

\begin{teore} Let $m_{1},m_{2} \in \operatorname{Tors} MSp_{*}$ are 
elements of the order two. Then we have a relation:
\[
\theta _{1}m_{1}m_{2} \in  2<m_{1},2,m_{2}>
\]
where $<m_{1},2,m_{2}>$ denotes a triple Massey products in the theory
of symplectic cobordisms and $\theta _{1} \in MSp_{2}$ is the generator.
\begin{proof}
Let us consider inclusions of manifolds
$i_{j} : M^{n_{j}}_{j} \rightarrow  {\Bbb R}^{n_{j}+r_{j}}$, $j = 1, 2$
which together with the normal bundles of these inclusions define on
manifolds $M_{1}$ and $M_{2}$ some $(B,f)$-structures which represent
the elements $m_{1}$ and $m_{2}$ respectively. Let $MSp$-manifolds
$X_{j}, (j = 1,2)$ are included in ${\Bbb R}^{n_{j}+r_{j}+1}_{+}$
so that $\partial X_{j} \cong 2M_{j}$ and a collar neighborhood of
the boundary in $X_{j}$ has the form
$\partial X_{j}\times [0,1), \partial X_{j}\subset
{\Bbb R}^{n_j+r_j} \times {0}$ and $X_{j}$ meets the boundary
transversally. Let us consider the products together with injections
$M_{1}\times X_{2} \subset {\Bbb R}^{n_{1}+r_{1}+n_{2}+r_{2}+1}_{+}$
and
$M_{2}\times X_{1} \subset {\Bbb R}^{n_{1}+r_{1}+n_{2}+r_{2}+1}_{-}$
such that under restrictions of both injections on the neighborhoods
of each copy of $M_{1}\times M_{2}$ their images coincide.
Let us glue the products $M_{1}\times X_{2}$ and $X_{1}\times M_{2}$
by the given identification $\partial (M_{1}\times X_{2})$ with
$\partial (X_{1}\times M_{2})$. A manifold that we have obtained
$Y \cong  (M_{1}\times X_{2}\cup X_{1}\times M_{2})/(M_{1}\times
\partial X_{2} \cong  \partial X_{1}\times M_{2})$ with the
$MSp$-structure induced by this inclusion
$Y \subset  {\Bbb R}^{n_{1}+r_{1}+n_{2}+r_{2}+1}$ depicts an element from
$< m_{1}, 2, m_{2} >$ in the ring $MSp_{*}$. Let an interval $I = [0,1]$
be included in ${\Bbb R}^{1+r}$ with the trivial framing. Let us take two
copies of the manifold with a boundary
$Y\times I \subset {\Bbb R}^{n_{1}+r_{1}+n_{2}+r_{2}+r+2}$.
Let us glue the upper boundary $Y\times \{1\}$ of the given manifolds by
the following rule. We choose in the manifold $Y$ one of the copies
of $M_{1}\times M_{2}$ by the gluing along which the manifold $Y$ was
constructed. A normal boundary of $M_{1}\times M_{2}$ in $Y$ has the
form: $M_{1}\times M_{2}\times (-1,1)$. Let us identify two copies of
$Y\times \{1\}$ by the identical map out of the boundary
$M_{1}\times M_{2}\times (-1/2,1/2)$. We obtained a manifold whose
lower boundary consists of two copies of $Y$, upper boundary has the
form:
$M_{1}\times M_{2}\times S^{1}$, where $S^{1}$ is included in
${\Bbb R}^{2+r}$ in the form of the figure ''eight". With the help
of such inclusion the framing induced on $S^{1}$  gives
the element $\theta _{1}$ in $MSp_{1}$. So, we constructed a cobordism
between the elements representing the elements $\theta _{1}m_{1}m_{2}$
and $2<m_{1},2,m_{2}>$ respectively. The codimensions that we have chosen
are sufficiently large, so the result does not depend from the choice
of concrete inclusions.
\end{proof}
\end{teore}

Let $\Gamma _{1}$ be an element belonging to the Massey product
$<\Phi _{7},2,\Omega _{1}>$. From Theorem~4.1 it follows that
$2\Gamma _{1} = \theta _{1}\Phi _{7}\Omega _{1} \neq 0$, hence, the
element $\Gamma _{1}$ has the order four.

We denote by $\Gamma _{i}, (i = 3, 4, \ldots )$ a series of elements
belonging (for each $i$ respectively) to the following Massey product:
$<\Phi _{6+i},2,\Omega _{1}>, i = 3,4,\ldots $. We have
$S_{4}\Omega _{1} = 0$ and $S_{6}\Omega _{1} = 0$, because in the MASS
there are the formulas $S_{4}\omega _{1} = 0$ and $S_{6}\omega _{1}= 0$,
and  $MSp_{41}$ and $MSp_{37}$ consist of elements having projections
in the term $E_{2}$ with filtration of the MASS equal to zero.
For $i = 5, 6, \ldots$, $S_{2(i-1)}\Omega _{1} = 0$ by the dimensional
reasons. Also there are the formulas $S_{2(i-1)}\Phi _{6+i} = \Phi _{7}$,
so we obtain: $2\Gamma _{i} = \theta _{1}\Phi _{6+i}\Omega _{1} \neq  0$,
$i = 3, 4, \ldots $, and hence the elements of the series $\Gamma _{i}$
have the order four in MSp$_{*}$.
\clearpage

\chapter{Tables}
Table 9. The action of Landweber-Novikov operations on the Ray's
elements $\Phi _{i}$ (nonzero values)
\begin{align*}
&S_{k}\Phi _{m}= \Phi _{m-k}, m >k; S_{2m-1}\Phi _{m}= \theta _{1}; S_{k,k}\Phi _{m}= (m-k)\Phi _{m-k}, m >k; \\
&S_{k,k,k}\Phi _{m}= (m-k)\Phi _{m-3s}, k= 2s, m >3s. \\
&S_{2,2,2,2}\Phi _{5}= \Phi _{1}. \\
&S_{2,2,2,2}\Phi _{6}= \Phi _{2}, S_{2,2,2,2,2}\Phi _{6}= \Phi _{1}.\\
&S_{3,3,3,3}\Phi _{8}= \Phi _{2}.\\
&S_{2,2,2,2}\Phi _{9}= \Phi _{5}, S_{2,2,2,2,2}\Phi _{9}= \Phi _{4}, S_{2,2,2,2,2,2}\Phi _{9}= \Phi _{3}, S_{3,3,3,3}\Phi _{9}= \Phi _{3},\\
&S_{2,2,2,2,2,2,2}\Phi _{9}= \Phi _{2}, S_{4,4,4,4}\Phi _{9}= \Phi _{1}, S_{2,2,2,2,2,2,2,2}\Phi _{9}= \Phi _{1}.\\
&S_{2,2,2,2}\Phi _{10}= \Phi _{6}, S_{2,2,2,2,2}\Phi _{10}= \Phi _{5},S_{4,4,4,4}\Phi _{10}= \Phi _{2}, S_{2,2,2,2,2,2,2,2}\Phi _{10}= \Phi _{2},\\
&S_{2,2,2,2,2,2,2,2,2}\Phi _{10}= \Phi _{1}.\\
&S_{2,2,2,2,2,2,2,2}\Phi _{11}= \Phi _{3}, S_{2,2,2,2,2,2,2,2,2}\Phi _{11}= \Phi _{2}.\\
&S_{3,3,3,3}\Phi _{12}= \Phi _{6}, S_{2,2,2,2,2,2,2,2}\Phi _{12}= \Phi _{4}, S_{5,5,5,5}\Phi _{12}= \Phi _{2}, S_{3,3,3,3,3,3}\Phi _{12}= \Phi _{3}.\\
&S_{2,2,2,2,2,2,2,2,2}\Phi _{12}= \Phi _{3},\\
&S_{2,2,2,2}\Phi _{13}= \Phi _{9}, S_{2,2,2,2,2}\Phi _{13}= \Phi _{8}, S_{2,2,2,2,2,2}\Phi _{13}= \Phi _{7}, S_{3,3,3,3}\Phi _{13}= \Phi _{7},\\
&S_{2,2,2,2,2,2,2}\Phi _{13}= \Phi _{6}, S_{4,4,4,4}\Phi _{13}= \Phi _{5}, S_{5,5,5,5}\Phi _{13}= \Phi _{3}, S_{4,4,4,4,4}\Phi _{13}= \Phi _{3},\\
&S_{6,6,6,6}\Phi _{13}= \Phi _{1}, S_{4,4,4,4,4,4}\Phi _{13}= \Phi _{1}.\\
&S_{2,2,2,2}\Phi _{14}= \Phi _{2}, S_{2,2,2,2,2}\Phi _{14}= \Phi _{9}, S_{4,4,4,4}\Phi _{14}= \Phi _{6}, S_{4,4,4,4,4}\Phi _{14}= \Phi _{4},\\
&S_{6,6,6,6}\Phi _{14}= \Phi _{2}.
\end{align*}

Table 10. The action of Landweber-Novikov operations on the elements
$c_{i}$ of the MASS
\begin{align*}
&S_{2}c_{2}= 1, S_{1,1}c_{2}= 0.\\
&S_{2}c_{4}= c_{2}, S_{1,1}c_{4}= c_{2}, S_{4}c_{4}= 1, S_{2,2}c_{4}= 0.\\
&S_{1}c_{5}= c_{4}, S_{3}c_{5}= c_{2}, S_{5}c_{5}= 1.\\
&S_{2}c_{6}= c^{2}_{2}, S_{1,1}c_{6}= c^{2}_{2}, S_{6}c_{6}= 1, S_{3,3}c_{6}= 0, S_{2,2,2}c_{6}= 1.\\
&S_{2}c_{8}= c_{2}c_{4}+ c_{6}, S_{1,1}c_{8}= c_{2}c_{4}+ c_{6}, S_{4}c_{8}= c_{4}, S_{2,2}c_{8}= c^{2}_{2}, S_{6}c_{8}= c_{2}, \\
&S_{3,3}c_{8}= c_{2}, S_{2,2,2}c_{8}= 0, S_{8}c_{8}= 1, S_{4,4}c_{8}= 1, S_{2,2,2,2}c_{8}= 0.\\
\end{align*}

Table 10. (continuation)
\begin{align*}
&S_{1}c_{9}= c_{8}, S_{2}c_{9}=c_{2}c_{5}, S_{1,1}c_{9}= c_{2}c_{5}, S_{3}c_{9}= c_{6}, S_{4}c_{9}= c_{5},S_{5}c_{9} = 0, S_{2,2}c_{9}= 0,\\
&S_{7}c_{9}= c_{2}, S_{9}c_{9}= 1, S_{3,3,3}c_{9}= 0.\\
&S_{2}c_{10}= c^{2}_{4}+ c^{4}_{2}, S_{1,1}c_{10}= c^{2}_{4}+ c^{4}_{2}, S_{4}c_{10}= 0, S_{2,2}c_{10}= c_{6}, S_{3,3}c_{10}= 0,\\
&S_{2,2,2}c_{10}= c^{2}_{2}, S_{10}c_{10}= 1, S_{5,5}c_{10}= 1, S_{2,2,2,2,2}c_{10}= 0, S_{6}c_{10}= 0.\\
&S_{1}c_{11}=c_{10}, S_{2}c_{11}= c_{9}+ c_{4}c_{5}, S_{1,1}c_{11}= c_{9}+ c_{4}c_{5}, S_{3}c_{11}=c_{8}+c^{4}_{2}, \\
&S_{2,2}c_{11}= c_{2}c_{5}, S_{4}c_{11}= 0, S_{5}c_{11}= 0, S_{6}c_{11}= c_{5}, S_{3,3}c_{11}= c_{5}, S_{7}c_{11}= c_{4}, \\
&S_{3,3,3}c_{11}= c_{2}, S_{11}c_{11}= 1.\\
&S_{2}c_{12}= c^{2}_{5}+ c^{2}_{2}c_{6}, S_{1,1}c_{12}= c_{10}+ c^{2}_{5}+ c^{2}_{2}c_{6}, S_{4}c_{12}= c^{4}_{2}, S_{2,2}c_{12}= c^{2}_{4}+ c^{4}_{2},\\
&S_{6}c_{12}= c_{6}, S_{3,3}c_{12}= c_{6}, S_{2,2,2}c_{12}= c_{6}, S_{8}c_{12}= 0, S_{4,4}c_{12}= 0, S_{2,2,2,2}c_{12}= c^{2}_{2}, \\
&S_{6,6}c_{12}= 0,S_{12}c_{12}=1, S_{4,4,4}c_{12}= 1, S_{3,3,3,3}c_{12}= 0, S_{2,2,2,2,2,2}c_{12}= 0.\\
&S_{1}c_{13}= c_{12}+ c_{4}c_{8}, S_{2}c_{13}= c_{11}+ c_{2}c_{4}c_{5}+ c_{5}c_{6}, S_{1,1}c_{13}=c_{2}c_{4}c_{5} + c_{5}c_{6}, \\
&S_{3}c_{13}= c_{2}c_{8}+ c_{10}+ c^{2}_{5}+ c^{2}_{2}c_{6}, S_{4}c_{13}= c_{9}, S_{2,2}c_{13}= c_{9} + c_{4}c_{5}+ c^{2}_{2}c_{5},\\
&S_{5}c_{13}= c^{4}_{2}+ c^{2}_{4}, S_{6}c_{13}= 0,S_{3,3}c_{13}= 0,S_{2,2,2}c_{13} = c_{2}c_{5}, S_{7}c_{13}= c_{2}c_{4}+ c_{6}, \\
&S_{8}c_{13}= c_{5}, S_{4,4}c_{13}= 0,S_{3,3,3}c_{13}=0, S_{2,2,2,2}c_{13}= 0, S_{13}c_{13}= 1.\\
&S_{2}c_{14}= c^{2}_{6}+ c^{2}_{2}c^{2}_{4}, S_{1,1}c_{14}= c^{2}_{6}+ c^{2}_{2}c^{2}_{4}, S_{4}c_{14} = 0, S_{2,2}c_{14}= 0, S_{6}c_{14}= c^{2}_{4},\\
&S_{3,3}c_{14}= 0, S_{2,2,2}c_{14}= c^{2}_{4}, S_{8}c_{14}= 0, S_{4,4}c_{14}= 0, S_{2,2,2,2}c_{14}= c_{6}, S_{10}c_{14}= c^{2}_{2}, \\
&S_{5,5}c_{14}= c^{2}_{2}, S_{2,2,2,2,2}c_{14}= 0, S_{14}c_{14}= 1, S_{7,7}c_{14}= 0, S_{2,2,2,2,2,2,2}c_{14}= 0.\\
&S_{2}c_{16}= c_{14}+ c_{2}c_{12}+ c_{4}c_{10}+ c_{6}c_{8}+ c_{2}c_{4}c_{8}+ c^{2}_{2}c_{10} + c_{4}c^{2}_{5} + c^{2}_{2}c_{4}c_{6} + c^{5}_{2}c_{4} + \\
& \qquad \; \: + c_{2}c^{3}_{4},\\
&S_{1,1}c_{16}= c_{14}+ c_{2}c_{12}+ c_{4}c_{10}+ c_{6}c_{8}+ c_{2}c_{4}c_{8}+ c^{2}_{2}c_{10} + c_{4}c^{2}_{5} + c^{2}_{2}c_{4}c_{6} + c^{5}_{2}c_{4} +\\
& \qquad \quad + c_{2}c^{3}_{4}, \\
&S_{4}c_{16}= c_{12}+ c_{2}c_{10}+ c_{4}c_{8}+ c_{4}c^{4}_{2}, S_{2,2}c_{16}= c^{2}_{6}+ c^{2}_{2}c^{2}_{4}, S_{2,2,2}c_{16}= 0,\\
&S_{6}c_{16}= c_{2}c_{8}+ c_{10}+ c^{2}_{5}+ c^{2}_{2}c_{6}+ c_{4}c_{6}, S_{3,3}c_{16}= c_{2}c_{8}+ c_{10}+ c^{2}_{5}+ c^{2}_{2}c_{6}+ c_{4}c_{6},\\
&S_{8}c_{16}= c^{2}_{4}+ c_{8}, S_{4,4}c_{16}= c^{2}_{4}, S_{2,2,2,2}c_{16}= 0, S_{3,3,3}c_{16}= 0, S_{10}c_{16}= c_{2}c_{4}+ c_{6},\\
&S_{5,5}c_{16}= c_{2}c_{4}+ c_{6}, S_{2,2,2,2,2}c_{16}= 0, S_{12}c_{16}= c_{4},S_{6,6}c_{16}=c^{2}_{2}, S_{4,4,4}c_{16}= 0, \\
&S_{3,3,3,3}c_{16}= c_{4}+ c^{2}_{2}, S_{2,2,2,2,2,2}c_{16}= 0, S_{14}c_{16}= c_{2}, S_{7,7}c_{16}= c_{2}, \\
&S_{2,2,2,2,2,2,2}c_{16}= 0, S_{16}c_{16}= 1, S_{8,8}c_{16}= 1, S_{4,4,4,4}c_{16}= 0,\\
&S_{2,2,2,2,2,2,2,2}c_{16}= 0.\\
&S_{2}c_{17}= c_{2}c_{13}+ c_{5}c_{10}+ c_{6}c_{9}+ c^{2}_{2}c_{11}+ c^{3}_{5}+ c^{2}_{2}c_{5}c_{6}+ c_{2}c_{5}c^{2}_{4}+ c_{5}c^{5}_{2},\\
&S_{1,1}c_{17}= c_{2}c_{13}+ c_{5}c_{10}+ c_{6}c_{9}+ c^{2}_{2}c_{11}+ c^{3}_{5}+ c^{2}_{2}c_{5}c_{6}+c_{2}c_{5}c^{2}_{4}+c_{5}c^{5}_{2},\\
&S_{1}c_{17}= c_{16}, S_{3}c_{17}= c_{14}, S_{4}c_{17}= c_{13}+ c_{2}c_{11}+ c_{5}c^{4}_{2}, S_{2,2}c_{17}=0, S_{5}c_{17}= 0, \\
&S_{6}c_{17}= c_{2}c_{9}+ c_{5}c_{6}, S_{3,3}c_{17}= c_{2}c_{9}+ c_{5}c_{6}, S_{2,2,2}c_{17}= 0, S_{7}c_{17}= c_{10}+ c^{2}_{5}+ c^{2}_{2}c_{6}\\
&S_{8}c_{17}= c_{9}, S_{4,4}c_{17}= 0, S_{2,2,2,2}c_{17}= 0, S_{9}c_{17}= c^{2}_{4}, S_{3,3,3}c_{17}= 0, S_{10}c_{17}= c_{2}c_{5},\\
&S_{5,5}c_{17}= c_{2}c_{5}, S_{2,2,2,2,2}c_{17}= 0, S_{11}c_{17}= c_{6}, S_{12}c_{17}= c_{5}, S_{6,6}c_{17}= 0, \\
&S_{4,4,4}c_{17}=0, S_{3,3,3,3}c_{17}= c_{5},S_{2,2,2,2,2,2}c_{17}=0, S_{15}c_{17}= c_{2}, S_{5,5,5}c_{17}= c_{2}, \\
\end{align*}

Table 10. (continuation)
\begin{align*}
&S_{3,3,3,3,3}c_{17}= c_{2}, S_{17}c_{17}= 1.\\
&S_{2}c_{18}= c^{2}_{8}, S_{1,1}c_{18}= c^{2}_{8}, S_{4}c_{18}= 0, S_{2,2}c_{18}= c_{6}(c^{2}_{4}+ c^{4}_{2}) + c_{14}+ c^{2}_{2}c_{10}, \\
&S_{6}c_{18}= 0, S_{3,3}c_{18}= 0, S_{2,2,2}c_{18}= c^{2}_{6}+ c^{2}_{2}c^{2}_{4}, S_{8}c_{18}= 0, S_{4,4}c_{18}= c_{10},
&S_{2,2,2,2}c_{18}= 0, S_{10}c_{18}= c^{4}_{2}, S_{5,5}c_{18}= c^{4}_{2}, S_{2,2,2,2,2}c_{18}= 0, S_{12}c_{18}= 0,\\
&S_{6,6}c_{18}= c_{6}, S_{4,4,4}c_{18}= 0, S_{3,3,3,3}c_{18}= c_{6}, S_{14}c_{18}= 0, S_{7,7}c_{18}= 0,\\
&S_{2,2,2,2,2,2}c_{18}= 0, S_{2,2,2,2,2,2,2}c_{18}= 0, S_{18}c_{18}= 1, S_{9,9}c_{18}= 1, S_{6,6,6}c_{18}= 0,\\
&S_{3,3,3,3,3,3}c_{18}= 0, S_{2,2,2,2,2,2,2,2,2}c_{18}= 0.\\
&S_{1}c_{19}= c_{18}, S_{2}c_{19}= c_{17}+ c_{9}c^{2}_{4}+ c_{11}c_{6}+ c_{13}c_{4}+ c_{5}c^{3}_{4}+ c_{5}c_{12}+ c_{4}c_{5}c^{4}_{2},\\
&S_{1,1}c_{19}= c_{17}+ c_{9}c^{2}_{4}+ c_{11}c_{6}+ c_{13}c_{4}+ c_{5}c^{3}_{4}+ c_{5}c_{12}+ c_{4}c_{5}c^{4}_{2},\\
&S_{3}c_{19}= c_{16}+ c^{2}_{8}+ c^{4}_{4}+ c_{6}c_{10}+ c^{2}_{4}c^{4}_{2}, S_{4}c_{19}= c_{5}c_{10}+ c_{4}c_{11},\\
&S_{2,2}c_{19}= c_{2}c_{13}+ c_{5}(c_{10}+ c^{2}_{5}+ c^{2}_{2}c_{6})+ c_{2}c_{5}(c^{2}_{4}+ c^{4}_{2})+ c_{6}c_{9}+c^{2}_{2}c_{11}, S_{5}c_{19}=0,\\
&S_{6}c_{19}= c_{13}+ c_{4}c_{9},S_{2,2,2}c_{19}=0, S_{3,3}c_{19}= c_{13}+ c_{4}c_{9}, S_{7}c_{19}= c_{12}, S_{8}c_{19}= c_{11},\\
&S_{4,4}c_{19}= 0,S_{2,2,2,2}c_{19}=0,S_{9}c_{19}= c_{10},S_{3,3,3}c_{19}=c_{10}+ c^{2}_{5}+ c^{2}_{2}c_{6}+c_{2}c_{8}+c_{4}c_{6}\\
&S_{10}c_{19}= c_{4}c_{5}, S_{5,5}c_{19}= c_{4}c_{5}, S_{2,2,2,2,2}c_{19}= 0, S_{12}c_{19}= 0, S_{11}c_{19}= c^{2}_{4}+ c^{4}_{2},\\
&S_{6,6}c_{19}= 0, S_{3,3,3,3}c_{19}= 0, S_{2,2,2,2,2,2}c_{19}= 0, S_{13}c_{19}= 0, S_{14}c_{19}= c_{5}, \\
&S_{7,7}c_{19}= c_{5}, S_{2,2,2,2,2,2,2}c_{19}= 0, S_{15}c_{19}= c_{4}, S_{3,3,3,3,3}c_{19}= c_{4}, S_{5,5,5}c_{19}= c_{4}, \\
&S_{17}c_{19}= 0, S_{19}c_{19}= 1.\\
&S_{2}c_{20}= c^{2}_{9}+ c^{2}_{2}c_{14}+ c^{2}_{4}(c_{10}+ c^{2}_{5})+ c^{4}_{2}(c_{10}+ c^{2}_{5}+ c^{2}_{2}c_{6})+ c^{3}_{6},\\
&S_{1,1}c_{20}= c_{18}+ c^{2}_{9}+ c^{2}_{2}c_{14}+ c^{2}_{4}(c_{10}+ c^{2}_{5})+ c^{4}_{2}(c_{10}+ c^{2}_{5}+ c^{2}_{2}c_{6})+c^{3}_{6},\\
&S_{4}c_{20}= c^{4}_{4}+ c^{4}_{2}c^{2}_{4}, S_{2,2}c_{20}= c^{2}_{8}+ c^{2}_{2}c_{12}+ c_{6}(c_{10}+ c^{2}_{5}+ c^{2}_{2}c_{6}), S_{6}c_{20}= c_{14}+ c^{4}_{2}c_{6}\\
&S_{3,3}c_{20}= c_{14}+ c^{4}_{2}c_{6}, S_{2,2,2}c_{20}= c_{14}+c^{2}_{2}c_{10} + c_{6}(c^{2}_{4}+ c^{4}_{2}), S_{8}c_{20}= 0, \\
&S_{4,4}c_{20}= c_{12}, S_{2,2,2,2}c_{20}= c^{2}_{6}+c^{2}_{2}c^{2}_{4}, S_{10}c_{20}= c_{10}+ c^{2}_{5}+ c^{2}_{2}c_{6}, \\
&S_{5,5}c_{20}= c_{10}+ c^{2}_{5}+ c^{2}_{2}c_{6}, S_{6,6}c_{20}= 0, S_{2,2,2,2,2}c_{20}= 0, S_{12}c_{20}= c^{2}_{4}+ c^{4}_{2},\\
&S_{4,4,4}c_{20}= c^{2}_{4}, S_{3,3,3,3}c_{20}= c^{2}_{4}, S_{2,2,2,2,2,2}c_{20}= 0, S_{14}c_{20}= c_{6},S_{7,7}c_{20}=c_{6},\\
&S_{2,2,2,2,2,2,2}c_{20}= 0, S_{16}c_{20}= 0, S_{8,8}c_{20}= 0, S_{4,4,4,4}c_{20}=0, S_{2,2,2,2,2,2,2,2}c_{20}= 0, \\
&S_{20}c_{20}= 1, S_{10,10}c_{20}= 0,\\
&S_{1}c_{21}= c_{20}+ c_{4}c_{16}, \\
&S_{2}c_{21}= c_{19}+c_{2}c_{4}c_{13}+c_{6}c_{13}+c_{2}c_{4}c_{5}(c^{4}_{2}+c^{2}_{4})+c_{5}(c_{14}+c^{2}_{2}c_{10}),\\
&S_{1,1}c_{21}=c_{2}c_{4}c_{13}+ c_{6}c_{13}+ c_{2}c_{4}c_{5}(c^{4}_{2}+ c^{2}_{4})+ c_{5}(c_{14}+ c^{2}_{2}c_{10}), \\
&S_{3}c_{21}= c_{2}c_{16}+ c_{18}+ c^{2}_{9}+ c^{2}_{2}(c_{14}+c^{2}_{4}c_{6})+ c^{4}_{2}(c_{10}+ c^{2}_{5}+ c^{2}_{2}c_{6})+ c_{6}(c_{12}+ c^{2}_{6}),\\
&S_{4}c_{21}= c_{17}+ c_{2}c_{4}c_{11}+c^{2}_{4}c_{9}+ c_{5}c_{12}, \\
&S_{2,2}c_{21}= c_{17}+ c^{2}_{4}c_{9}+ c_{6}c_{11}+ c_{4}c_{13}+ c_{4}c_{5}(c^{4}_{2}+ c^{2}_{4})+ c_{5}(c_{12}+ c^{2}_{6}+ c^{2}_{2}c^{2}_{4}),\\
&S_{5}c_{21}= c^{4}_{4}, S_{6}c_{21}= c_{2}c_{4}c_{9}+ c^{2}_{2}c_{11}+ c_{5}(c_{10}+ c^{2}_{5}+ c^{2}_{2}c_{6}), \\
&S_{3,3}c_{21}= c_{2}c_{4}c_{9}+ c^{2}_{2}c_{11}+ c_{5}(c_{10}+ c^{2}_{5}+ c^{2}_{2}c_{6}), \\
&S_{2,2,2}c_{21}= c_{2}c_{13}+ c^{2}_{2}c_{11}+ c_{6}c_{9}+ c_{5}(c_{10}+ c^{2}_{5}+ c^{2}_{2}c_{6})+ c_{2}c_{5}(c^{4}_{2}+ c^{2}_{4}), \\
&S_{8}c_{21}= c_{13}+ c_{5}c^{2}_{4}, S_{4,4}c_{21}= c_{13}+ c_{2}c_{11}+ c_{5}(c^{4}_{2}+ c^{2}_{4}), S_{2,2,2,2}c_{21}= 0,\\
\end{align*}

Table 10. (continuation)
\begin{align*}
&S_{7}c_{21}= c_{14}+ c^{2}_{2}c_{10}+ c_{6}(c^{4}_{2}+c^{2}_{4}), S_{9}c_{21}= c_{12}, \\
&S_{3,3,3}c_{21}= c_{2}(c_{10}+ c^{2}_{5}+ c^{2}_{2}c_{6})+ c^{6}_{2}+ c_{2}c_{4}c_{6}+ c^{2}_{2}c_{8}, S_{10}c_{21}= c_{2}c_{4}c_{5}+ c_{5}c_{6}, \\
&S_{5,5}c_{21}= c_{2}c_{4}c_{5}+c_{5}c_{6},S_{2,2,2,2,2}c_{21}=0, S_{12}c_{21}= 0, S_{6,6}c_{21}= 0, S_{4,4,4}c_{21}= 0, \\
&S_{3,3,3,3}c_{21}= c_{5}c^{2}_{2}, S_{2,2,2,2,2,2}c_{21}= 0, S_{11}c_{21}= c_{10}+ c^{2}_{5}+ c^{2}_{2}c_{6}, S_{13}c_{21}= c^{4}_{2},\\
&S_{14}c_{21}= 0, S_{7,7}c_{21}= 0, S_{2,2,2,2,2,2,2}c_{21}= 0,S_{15}c_{21}=c_{2}c_{4}+c_{6},\\
&S_{5,5,5}c_{21}= c_{2}c_{4}+ c_{6}, S_{3,3,3,3,3}c_{21}= c^{3}_{2}, S_{16}c_{21}= c_{5}, S_{8,8}c_{21}= c_{5}, \\
&S_{4,4,4,4}c_{21}= 0, S_{2,2,2,2,2,2,2,2}c_{21}= 0.\\
&S_{2}c_{22}= c_{12}(c^{2}_{4}+ c^{4}_{2})+ c_{10}(c_{10}+ c^{2}_{5}+ c^{2}_{2}c_{6})+ c^{4}_{5}, S_{4}c_{22}= c^{4}_{2}c_{10},\\
&S_{1,1}c_{22}= c_{12}(c^{2}_{4}+ c^{4}_{2})+ c_{10}(c_{10}+ c^{2}_{5}+ c^{2}_{2}c_{6})+ c^{4}_{5}, S_{6}c_{22}= c^{2}_{8}+ c^{8}_{2}+ c_{6}c_{10}, \\
&S_{2,2}c_{22}=c_{18}+ c^{2}_{9}+ c^{2}_{2}c_{14}+c_{6}(c_{12}+ c^{2}_{6}+ c^{2}_{2}c^{2}_{4})+c^{4}_{2}(c_{10}+ c^{2}_{5}+ c^{2}_{2}c_{6}),\\
&S_{3,3}c_{22}= c^{2}_{8}+ c^{8}_{2}+ c_{6}c_{10}, S_{8}c_{22}= 0, S_{2,2,2}c_{22}= c^{2}_{2}c_{12}+ c_{6}(c_{10}+ c^{2}_{5}+ c^{2}_{2}c_{6}),\\
&S_{4,4}c_{22}= 0, S_{2,2,2,2}c_{22}= 0, S_{10}c_{22}= c_{12}, S_{5,5}c_{22}= c_{12}, S_{2,2,2,2,2}c_{22}=0,\\
&S_{12}c_{22}= c_{10}, S_{6,6}c_{22}= c_{10}+ c^{2}_{5}+ c^{2}_{2}c_{6}, S_{4,4,4}c_{22}= c_{10}, S_{3,3,3,3}c_{22}= c^{2}_{5}+ c^{2}_{2}c_{6}, \\
&S_{2,2,2,2,2,2}c_{22}= 0, S_{14}c_{22}= c^{2}_{4}, S_{7,7}c_{22}= c^{2}_{4}, S_{2,2,2,2,2,2,2}c_{22}= 0, S_{16}c_{22}= 0, \\
&S_{8,8}c_{22}= 0, S_{4,4,4,4}c_{22}= 0, S_{2,2,2,2,2,2,2,2}c_{22}= 0, S_{18}c_{22}=0, S_{9,9}c_{22}=0, \\
&S_{6,6,6}c_{22}= c^{2}_{2}, S_{3,3,3,3,3,3}c_{22}= c^{2}_{2},S_{2,2,2,2,2,2,2,2,2}c_{22}=0, S_{22}c_{22}= 1, \\
&S_{11,11}c_{22}= 1, S_{2,2,2,2,2,2,2,2,2,2,2}c_{22}= 0.\\
&S_{2}c_{23} = c_{21}+ c_{5}c_{16}+ c_{9}c_{12}+ c_{11}(c_{10}+ c^{2}_{5}+ c^{2}_{2}c_{6})+ c_{8}c_{13}+ c_{5}c^{2}_{8} + (c_{13}+ \\
& \qquad \;  +c_{5}c_{8})(c^{2}_{4}+ c^{4}_{2}), S_{1}c_{23}= c_{22}, S_{3}c_{23}= c_{20}+ c^{4}_{5},\\
&S_{1,1}c_{23} = c_{21}+ c_{5}c_{16}+ c_{9}c_{12}+ c_{11}(c_{10}+ c^{2}_{5}+ c^{2}_{2}c_{6})+c_{8}c_{13}+ c_{5}c^{2}_{8}+(c_{13}+\\
& \quad \quad \quad + c_{5}c_{8})(c^{2}_{4}+c^{4}_{2}), S_{4}c_{23}= c_{19}+ c_{8}c_{11}+ c_{9}c_{10}+ c_{11}c^{4}_{2}, S_{5}c_{23}=c_{18},\\
&S_{2,2}c_{23}= c_{2}c_{4}(c_{13}+ c_{5}c_{8})+ c_{2}c_{5}c_{12}+ c_{4}c_{5}(c_{10}+ c^{2}_{5}+ c^{2}_{2}c_{6})+ c_{5}c_{8}c_{6}, \\
&S_{6}c_{23}= c_{17}+ c_{8}c_{9}+ c_{11}c_{6},S_{3,3}c_{23}= c_{17}+ c_{8}c_{9}+c_{11}c_{6},S_{2,2,2}c_{23}= 0,\\
&S_{7}c_{23}= c^{8}_{2}, S_{8}c_{23}= 0, S_{4,4}c_{23}= c_{4}c_{11}+ c_{5}c_{10}, S_{2,2,2,2}c_{23}= 0, S_{9}c_{23}= 0, \\
&S_{3,3,3}c_{23}= c_{14}+ c^{4}_{2}c_{6}, S_{10}c_{23}= c_{13}+ c_{5}c_{8}, S_{5,5}c_{23}= c_{13}+ c_{5}c_{8}, S_{2,2,2,2,2}c_{23}= 0\\
&S_{11}c_{23}= 0, S_{12}c_{23}= c_{11}, S_{6,6}c_{23}= c_{2}c_{9}+ c_{5}c_{6}, S_{4,4,4}c_{23}= 0, S_{2,2,2,2,2,2}c_{23}= 0, \\
&S_{3,3,3,3}c_{23}= c_{2}c_{9}+ c_{5}c_{6},S_{13}c_{23}= 0,S_{14}c_{23}= c_{9},S_{7,7}c_{23}= c_{9},S_{15}c_{23}= c_{8}+ c^{2}_{4},\\
&S_{2,2,2,2,2,2,2}c_{23}= 0, S_{5,5,5}c_{23}= c_{8}+ c^{2}_{4}+ c^{4}_{2}, S_{3,3,3,3,3}c_{23}= c_{8}+ c^{4}_{2},S_{16}c_{23}= 0, \\
&S_{8,8}c_{23}= 0, S_{4,4,4,4}c_{23}=0,S_{2,2,2,2,2,2,2,2}c_{23}=0, S_{18}c_{23}= 0, S_{9,9}c_{23}= 0, \\
&S_{6,6,6}c_{23}= 0, S_{3,3,3,3,3,3}c_{23}= c_{5}.\\
&S_{2}c_{24}= c^{2}_{11}+ c^{2}_{2}c_{18}+ c_{6}(c^{2}_{8}+ c^{4}_{4})+ c^{2}_{4}c_{14}+ c_{12}(c^{2}_{5}+ c^{2}_{2}c_{6})+ c^{2}_{6}c_{10},\\
&S_{1,1}c_{24}=c^{2}_{11}+ c^{2}_{2}c_{18}+ c_{6}(c^{2}_{8}+ c^{4}_{4})+ c^{2}_{4}c_{14}+ c_{12}(c^{2}_{5}+ c^{2}_{2}c_{6})+c^{2}_{6}c_{10}+ c_{22}, \\
&S_{4}c_{24}= c^{4}_{2}c_{12}+ c^{4}_{5}, S_{2,2}c_{24}= c_{6}c_{14}+ c^{2}_{4}(c_{12}+ c^{6}_{2})+c^{2}_{5}c_{10}+ c^{4}_{2}(c_{12}+ c^{6}_{2})+ c^{10}_{2}, \\
&S_{6}c_{24}= c_{18}+ c^{2}_{9}+ c^{2}_{2}c_{14}+ c_{6}c_{12}+c^{2}_{2}c^{2}_{4}c_{6}+c^{3}_{6}, S_{4,4}c_{24}= c^{2}_{8}+ c^{4}_{4}+ c^{8}_{2},\\
&S_{3,3}c_{24}= c_{18}+ c^{2}_{9}+ c^{2}_{2}c_{14}+ c_{6}c_{12}+ c^{2}_{2}c^{2}_{4}c_{6}+ c^{3}_{6},S_{8}c_{24}= c^{8}_{2},S_{2,2,2,2,2}c_{24}= 0, \\
\end{align*}

Table 10. (continuation)
\begin{align*}
&S_{2,2,2}c_{24}= c_{18}+ c^{2}_{9}+ c_{6}(c_{12}+ c^{2}_{2}c^{2}_{4})+ c^{4}_{2}c^{2}_{5},S_{10}c_{24}= c_{14}+ c^{2}_{2}c_{10}+ c^{2}_{4}c_{6},  \\
&S_{2,2,2,2}c_{24}= c^{2}_{2}(c_{12}+ c^{2}_{2}c^{2}_{4}+ c^{2}_{6})+ c_{6}(c_{10}+ c^{2}_{5}+ c^{2}_{2}c_{6}),S_{12}c_{24}= c_{12}, \\
&S_{5,5}c_{24}= c_{14}+ c^{2}_{2}c_{10}+c^{2}_{4}c_{6}, S_{6,6}c_{24}= c^{2}_{6}, S_{4,4,4}c_{24}= c_{12}, S_{3,3,3,3}c_{24}= c_{12}+ c^{2}_{6},\\
&S_{2,2,2,2,2,2}c_{24}= 0, S_{14}c_{24}= c_{10}+ c^{2}_{5}+ c^{2}_{2}c_{6}, S_{7,7}c_{24}= c_{10}+ c^{2}_{5}+ c^{2}_{2}c_{6},\\
&S_{2,2,2,2,2,2,2}c_{24}= 0, S_{16}c_{24}= 0,S_{8,8}c_{24}= c^{2}_{4}, S_{2,2,2,2,2,2,2,2}c_{24}= 0,\\
&S_{4,4,4,4}c_{24}= c^{2}_{4}, S_{18}c_{24}= c_{6}, S_{9,9}c_{24}= c_{6}, S_{6,6,6}c_{24}= 0, S_{3,3,3,3,3,3}c_{24}= c_{6},\\
&S_{2,2,2,2,2,2,2,2,2}c_{24}= 0, S_{20}c_{24}= 0, S_{10,10}c_{24}= c^{2}_{2}, S_{2,2,2,2,2,2,2,2,2,2}c_{24}= 0,\\
&S_{5,5,5,5}c_{24}= c^{2}_{2}, S_{4,4,4,4,4}c_{24}= 0, S_{12,12}c_{24}= 1, S_{8,8,8}c_{24}= 1, S_{6,6,6,6}c_{24}= 0,\\
&S_{4,4,4,4,4,4}c_{24}= 0, S_{3,3,3,3,3,3,3,3}c_{24}= 0, S_{2,2,2,2,2,2,2,2,2,2,2,2}c_{24}= 0.\\
&S_{2}c_{25}= c_{23}+ c_{2}c_{5}c_{16}+c_{2}c_{8}c_{13}+c_{19}c^{2}_{2}+c_{6}c_{17}++ c_{2}c_{5}c_{8}(c^{2}_{4}+ c^{4}_{2})+ c_{13}(c_{10}+\\
& \qquad \; + c^{2}_{5}+ c^{2}_{2}c_{6})+ c_{11}(c^{2}_{2}c^{2}_{4}+ c^{6}_{2})+c_{9}(c_{14}+ c^{2}_{2}c_{10}),S_{1}c_{25}= c_{24}+ c_{8}c_{16},\\
&S_{1,1}c_{25}=c_{2}c_{5}c_{16}+c_{2}c_{8}c_{13}+c_{19}c^{2}_{2}+c_{6}c_{17}+c_{2}c_{5}c_{8}(c^{2}_{4}+c^{4}_{2})+ c_{13}(c_{10}+ c^{2}_{5}+\\
&\qquad \quad + c^{2}_{2}c_{6})+ c_{11}(c^{2}_{2}c^{2}_{4}+ c^{6}_{2})+ c_{9}(c_{14}+ c^{2}_{2}c_{10}),\\
&S_{3}c_{25}= c_{22}+ c^{2}_{11}+ c^{2}_{4}c_{14}+ c_{10}(c_{12}+ c^{2}_{6}+ c^{2}_{2}c^{2}_{4}+ c^{6}_{2})+ c_{6}c^{4}_{4}, \\
&S_{4}c_{25}=c_{21}+ c_{2}c_{8}c_{11}+ c_{13}c^{4}_{2}+ c_{9}c_{12}+ c_{5}c^{2}_{8},\\
&S_{2,2}c_{25}= c_{21}+ c_{5}c_{16}+ c_{8}c_{13}+ c^{2}_{2}(c_{17}+ c_{4}c_{13}+ c_{4}c_{5}c_{8})+ (c_{2}c_{13}+ c_{2}c_{5}c_{8})c_{6}+\\
& \qquad \quad + (c_{13}+c_{5}c_{8})(c^{2}_{4}+ c^{4}_{2})+ (c_{11}+ c_{2}c_{4}c_{5})(c_{10}+ c^{2}_{5}+ c^{2}_{2}c_{6})+ c_{9}c_{12}+\\
& \qquad \quad + c_{2}c_{5}(c_{14}+ c^{2}_{2}c_{10}+ c_{6}(c^{4}_{2}+ c^{2}_{4}))+ c_{5}(c_{6}(c_{10}+ c^{2}_{5}+ c^{2}_{2}c_{6})+ c^{2}_{8}),\\
&S_{5}c_{25}= c_{20}+ c^{4}_{5}, S_{6}c_{25}= c_{2}c_{8}c_{9}+ c_{13}c_{6}+ c_{9}(c_{10}+ c^{2}_{5}+ c^{2}_{2}c_{6}),\\
&S_{3,3}c_{25}= c_{2}c_{8}c_{9}+ c_{13}c_{6}+ c_{9}(c_{10}+ c^{2}_{5}+ c^{2}_{2}c_{6}),\\
&S_{2,2,2}c_{25}= c_{2}c_{4}c_{13}+ c_{2}c_{4}c_{5}c_{8}+ c_{2}c_{13}c^{2}_{2}+ (c_{13}+ c_{5}c_{8})c_{6}+ c_{11}c^{4}_{2}+ c_{9}c^{2}_{2}c_{6}+\\
& \qquad \qquad + c_{4}c_{5}(c_{10}+ c^{2}_{5}+ c^{2}_{2}c_{6})+ c_{5}c^{2}_{2}(c_{10}+ c^{2}_{5}+ c^{2}_{2}c_{6})+ c_{2}c_{5}(c_{12}+ c^{2}_{6}+ c^{6}_{2}),\\
&S_{7}c_{25}= c_{18}+c_{2}c_{16}+ c^{2}_{9}+ c^{2}_{2}c_{14}+ c_{6}(c^{2}_{2}c^{2}_{4}+ c^{2}_{6}), S_{8}c_{25}= c_{17}+ c_{9}c^{2}_{4},\\
&S_{4,4}c_{25}=c_{5}c_{12}+ c_{17}+ c_{2}c_{4}c_{11}, S_{2,2,2,2}c_{25}= 0, S_{9}c_{25}= c^{2}_{8}+ c^{8}_{2}, S_{3,3,3}c_{25}=0,\\
&S_{10}c_{25}= c_{2}c_{5}c_{8}+ c_{11}c^{2}_{2}+ c_{9}c_{6}, S_{5,5}c_{25}= c_{2}c_{5}c_{8}+ c_{11}c^{2}_{2}+ c_{9}c_{6}, S_{2,2,2,2,2}c_{25}= 0, \\
&S_{11}c_{25}= c_{14}+ c^{2}_{4}c_{6}, S_{12}c_{25}= c_{13}, S_{6,6}c_{25} = c_{9}c^{2}_{2},  S_{3,3,3,3}c_{25}=c_{13} + c_{9}c^{2}_{2},\\
&S_{4,4,4}c_{25}= c_{13}+ c_{2}c_{11}+ c_{5}(c^{2}_{4}+ c^{4}_{2}),S_{2,2,2,2,2,2}c_{25}= 0, S_{13}c_{25}= 0, S_{14}c_{25}= 0, \\
&S_{7,7}c_{25}=0, S_{2,2,2,2,2,2,2}c_{25}= 0, S_{15}c_{25}= c_{10}+ c^{2}_{5}+ c^{2}_{2}c_{6}, S_{5,5,5}c_{25}= 0,\\
&S_{3,3,3,3,3}c_{25}= c_{10}+ c^{2}_{5}+ c^{2}_{2}c_{6}, S_{16}c_{25}= c_{9}, S_{8,8}c_{25}= 0,S_{4,4,4,4}c_{25}= 0,\\
&S_{2,2,2,2,2,2,2,2}c_{25}= 0, S_{17}c_{25}= 0,S_{18}c_{25}=0, S_{9,9}c_{25}= 0, S_{6,6,6}c_{25}= 0,\\
&S_{3,3,3,3,3,3}c_{25}= 0, S_{19}c_{25}= c_{6}, S_{2,2,2,2,2,2,2,2,2}c_{25}= 0, S_{20}c_{25}= 0, S_{10,10}c_{25}= 0,\\
&S_{21}c_{25}=0, S_{5,5,5,5}c_{25}= 0,S_{4,4,4,4,4}c_{25}= 0, S_{2,2,2,2,2,2,2,2,2,2}c_{25}=0,\\
&S_{7,7,7}c_{25}= 0, S_{3,3,3,3,3,3,3}c_{25}= 0, S_{23}c_{25}= 0, S_{25}c_{25}= 1, S_{5,5,5,5,5}c_{25}= 0.\\
&S_{2}c_{26}= c^{2}_{4}c^{2}_{8}+ c^{2}_{2}c^{2}_{10}+ c^{2}_{12}+ c^{4}_{2}c^{4}_{4}+ c^{4}_{6}+ c^{12}_{2}, \\
&S_{1,1}c_{26}= c^{2}_{4}c^{2}_{8} + c^{2}_{2}c^{2}_{10}+ c^{2}_{12}+ c^{4}_{2}c^{4}_{4}+ c^{4}_{6}+ c^{12}_{2}, S_{4}c_{26}= 0,\\
\end{align*}

Table 10. (continuation)
\begin{align*}
&S_{2,2}c_{26}= c_{22}+ c^{2}_{11} + c_{6}c^{2}_{8}+ c_{14}(c^{2}_{4}+ c^{4}_{2})+ c^{2}_{2}c^{2}_{4}(c_{10}+ c^{2}_{2}c_{6})+ c^{6}_{2}c_{10}+ c_{10}c_{12},\\
&S_{6}c_{26}= 0, S_{3,3}c_{26}= 0, S_{2,2,2}c_{26}= c^{2}_{2}c^{2}_{8}+ c^{2}_{6}(c^{2}_{4}+c^{4}_{2}) + c^{2}_{10}+ c^{4}_{5}+ c^{10}_{2}, \\
&S_{8}c_{26}= 0, S_{4,4}c_{26}= c_{18}, S_{2,2,2,2}c_{26}= 0, S_{10}c_{26}= c^{4}_{4}+ c^{8}_{2}, S_{5,5}c_{26}= c^{4}_{4}+ c^{8}_{2},\\
&S_{2,2,2,2,2}c_{26}= 0,S_{12}c_{26}=0, S_{6,6}c_{26}= c_{14}+ c^{2}_{4}c_{6}+ c^{4}_{2}c_{6}+ c^{2}_{2}c_{10}, S_{4,4,4}c_{26}= 0,\\
&S_{3,3,3,3}c_{26} = c_{14}+ c^{2}_{4}c_{6}+ c^{4}_{2}c_{6}+ c^{2}_{2}c_{10}, S_{2,2,2,2,2,2}c_{26}= 0, S_{14}c_{26}= 0,S_{7,7}c_{26}= 0\\
&S_{2,2,2,2,2,2,2}c_{26}= 0, S_{16}c_{26}= 0, S_{8,8}c_{26}= c_{10}, S_{4,4,4,4}c_{26}= 0, S_{18}c_{26}= c^{4}_{2},\\
&S_{2,2,2,2,2,2,2,2}c_{26}= 0, S_{9,9}c_{26} = c^{4}_{2}, S_{6,6,6}c_{26}= 0, S_{3,3,3,3,3,3}c_{26}= 0, S_{10,10}c_{26}= c_{6}\\
&S_{2,2,2,2,2,2,2,2,2}c_{26}= 0, S_{20}c_{26}= 0, S_{5,5,5,5}c_{26}= c_{6},  S_{2,2,2,2,2,2,2,2,2,2}c_{26}= 0,\\
&S_{4,4,4,4,4}c_{26}= 0, S_{22}c_{26}= 0, S_{11,11}c_{26}= 0, S_{2,2,2,2,2,2,2,2,2,2,2}c_{26}= 0, \\
&S_{13,13}c_{26}= 1, S_{2,2,2,2,2,2,2,2,2,2,2,2,2}c_{26}= 0.\\
&S_{2}c_{27}= c_{25}+ c_{4}c_{5}c_{16}+ c_{4}c_{8}c_{13}+ c_{19}c_{6}+ c_{17}c^{2}_{4} + c_{4}c_{5}c_{8}(c^{2}_{4}+ c^{4}_{2})+c_{13}c_{12}+ \\
& \qquad \; +c_{11}(c^{2}_{2}c_{10}+ c_{14})+ c_{9}(c^{4}_{2}c^{2}_{4}+ c^{4}_{4}),S_{1}c_{27}= c_{26}, \\
&S_{1,1}c_{27}= c_{25}+ c_{4}c_{5}c_{16}+ c_{4}c_{8}c_{13}+ c_{19}c_{6}+ c_{17}c^{2}_{4}+c_{4}c_{5}c_{8}(c^{2}_{4}+c^{4}_{2}) +c_{13}c_{12}+\\
&\qquad \quad +c_{11}(c^{2}_{2}c_{10}+c_{14})+ c_{9}(c^{4}_{2}c^{2}_{4}+c^{4}_{4}), \\
&S_{3}c_{27}=c_{8}c_{16}+c_{24}+c_{6}c_{18} + c_{10}c_{14}+ c^{4}_{6}+ c^{4}_{2}c^{4}_{4}+ c^{12}_{2},\\
&S_{4}c_{27}= c_{23}+ c_{4}c_{8}c_{11}+c_{13}c_{10}+c_{11}c_{12},\\
&S_{2,2}c_{27}= c_{2}c_{5}c_{16}+ c_{2}c_{8}c_{13}+ c_{6}(c_{17}+ c_{4}c_{13}+ c_{4}c_{5}c_{8})+ c_{2}c_{13}c^{2}_{4} + c_{13}(c_{10}+ \\
&\qquad \quad +c^{2}_{5} + c^{2}_{2}c_{6})+ c_{11}(c^{2}_{2}c^{2}_{4}+ c^{6}_{2})+ c_{2}c_{4}c_{5}c_{12}+ c_{5}c_{6}c_{12}+ c_{9}(c_{14}+c^{2}_{2}c_{10})+\\
&\qquad \quad + c_{4}c_{5}(c^{2}_{2}c_{10}+ c_{6}(c^{2}_{4}+ c^{4}_{2})), S_{5}c_{27}= c_{22},\\
&S_{6}c_{27}= c_{21}+ c_{4}c_{8}c_{9}+ c_{11}(c_{10}+ c^{2}_{5}+ c^{2}_{2}c_{6})+ c_{5}c^{2}_{8},\\
&S_{3,3}c_{27}= c_{21}+ c_{4}c_{8}c_{9}+ c_{11}(c_{10}+ c^{2}_{5}+ c^{2}_{2}c_{6})+ c_{5}c^{2}_{8},\\
&S_{2,2,2}c_{27}= c_{11}c^{2}_{2}c_{6}+ c^{2}_{2}(c_{17}+ c_{4}c_{5}c_{8})+ c_{6}(c_{2}c_{13}+ c_{2}c_{5}c_{8})+ c_{2}c_{4}c_{5}(c_{10}+c^{2}_{5}+ \\
&\qquad \quad +c^{2}_{2}c_{6})+ c^{2}_{2}c^{2}_{4}c_{9}+ c_{4}c_{5}(c^{2}_{6}+ c^{6}_{2})+ c_{2}c_{5}(c_{14}+ c^{2}_{2}c_{10}+ c_{6}(c^{2}_{4}+ c^{4}_{2}))+ \\
&\qquad \quad +c_{5}(c^{2}_{2}c_{12}+ c_{6}(c_{10}+ c^{2}_{5}+ c^{2}_{2}c_{6})), S_{4,4}c_{27}= c_{8}c_{11}+c_{11}(c^{2}_{4}+c^{4}_{2})+c_{9}c_{10},\\
&S_{7}c_{27}= c_{4}c_{16}+ c_{20}+ c_{10}(c_{10}+c^{2}_{5}+ c^{2}_{2}c_{6}), S_{8}c_{27}= c_{19}+c_{11}c^{2}_{4},\\
&S_{2,2,2,2}c_{27}= c^{3}_{2}c_{13}+ c^{4}_{2}c_{11}+ c^{2}_{2}c_{6}c_{9}+ c_{2}c_{5}(c^{2}_{6}+ c^{6}_{2})+ c^{2}_{2}c_{5}(c_{10} + c^{2}_{5}+ c^{2}_{2}c_{6}), \\
&S_{3,3,3}c_{27}= c_{2}c_{16}+ c_{4}c_{6}c_{8}+ c_{8}(c_{10}+ c^{2}_{5}+ c^{2}_{2}c_{6})+ c_{18}+ c^{2}_{9}+ c^{2}_{2}c_{14}+c_{6}(c_{12}+\\
&\qquad \qquad +c^{2}_{2}c^{2}_{4}+ c^{2}_{6}),S_{9}c_{27}= c_{18}+ c^{2}_{4}c_{10}, \\
&S_{10}c_{27}= c_{4}c_{5}c_{8}+ c_{6}c_{11}+ c^{2}_{4}c_{9}, S_{5,5}c_{27}= c_{4}c_{5}c_{8}+ c_{6}c_{11}+ c^{2}_{4}c_{9},S_{2,2,2,2,2}c_{27}= 0, \\
&S_{11}c_{27}= c_{6}c_{10}+ c^{8}_{2}+ c^{4}_{4}, S_{12}c_{27}= 0,S_{6,6}c_{27}= c_{2}c_{4}c_{9}+ c_{5}(c_{10}+ c^{2}_{5}+ c^{2}_{2}c_{6}), \\
&S_{4,4,4}c_{27}= 0, S_{3,3,3,3}c_{27}= c_{2}c_{4}c_{9}+ c_{5}(c_{10}+ c^{2}_{5}+ c^{2}_{2}c_{6}), S_{2,2,2,2,2,2}c_{27}= 0,\\
&S_{13}c_{27}= 0, S_{14}c_{27}= c_{13}, S_{7,7}c_{27}= c_{13}, S_{2,2,2,2,2,2,2}c_{27}= 0, S_{15}c_{27}= c_{4}c_{8}+ c_{12},\\
&S_{5,5,5}c_{27}= c_{4}c_{8}, S_{3,3,3,3,3}c_{27}= c^{2}_{2}c_{8}+ c_{2}c_{4}c_{6}+ c_{2}(c_{10}+ c^{2}_{5}+ c^{2}_{2}c_{6}), S_{16}c_{27}=c_{11},\\
&S_{8,8}c_{27}= 0, S_{4,4,4,4}c_{27}=0,S_{2,2,2,2,2,2,2,2}c_{27}= 0, S_{17}c_{27}= c_{10}, S_{18}c_{27}= 0, \\
&S_{9,9}c_{27}= 0,S_{6,6,6}c_{27}= c^{2}_{2}c_{5}, S_{3,3,3,3,3,3}c_{27}= c^{2}_{2}c_{5}, S_{19}c_{27}= 0,S_{20}c_{27}=0,\\
&S_{2,2,2,2,2,2,2,2,2}c_{27}= 0, S_{10,10}c_{27}= 0, S_{5,5,5,5}c_{27}= 0, S_{2,2,2,2,2,2,2,2,2,2}c_{27}= 0, \\
&S_{4,4,4,4,4}c_{27}= 0, S_{21}c_{27}= 0, S_{7,7,7}c_{27}= c_{2}c_{4}+ c_{6}, S_{3,3,3,3,3,3,3}c_{27}= c^{3}_{2},\\
&S_{22}c_{27}= 0,S_{11,11}c_{27}= 0, S_{2,2,2,2,2,2,2,2,2,2,2}c_{27}= 0, S_{23}c_{27}= 0,S_{25}c_{27}= 0,\\
&S_{5,5,5,5,5}c_{27}= 0.
\end{align*}
\clearpage

Table 11. The cell $E^{0,1,t}_{2}$ of the MASS for $t < 108$ (generators).
\begin{align*}
&106 \;\; \omega _{4} (= u_{3}c_{23}+ u_{4}c_{19}+ u_{5}c_{11}),\hbox{\qquad\qquad\qquad\qquad\qquad\qquad\qquad\qquad\qquad\quad}\\
&\qquad \psi _{10}(= u_{1}c_{9}c_{17}+ u_{2}(c_{25}+c_{8}c_{17}+ c_{9}c_{16})+ u_{4}c_{2}c_{17}+ u_{5}c_{2}c_{9}),\\
&\qquad \psi _{9}(= u_{1}c_{5}c_{21}+ u_{2}c_{4}(c_{21}+ c_{4}c_{17}+ c_{5}c_{16})+ u_{3}c_{2}(c_{21}+ c_{2}c_{19}+c_{5}c_{16})+\\
&\qquad\quad + u_{5}c_{2}c_{4}c_{5}).\\
&102 \;\; \tilde{\varphi }_{13}(= u_{1}c_{25}+ u_{2}c_{8}c_{16}+ u_{4}c_{2}c_{16}+ u_{5}c_{2}c_{8}).\\
&\;\;98  \;\; \omega _{3} (= u_{2}c_{23}+ u_{4}c_{17}+ u_{5}c_{9}),\\
&\qquad \psi _{8}(= u_{1}c_{5}c_{19}+ u_{2}c_{4}c_{19}+ u_{3}(c_{21}+ c_{2}c_{19}+ c_{5}c_{16})+ u_{5}c_{4}c_{5}),\\
&\qquad \psi _{7}(= u_{1}c_{11}c_{13}+ u_{2}c_{4}c_{8}c_{11}+ u_{3}c_{8}(c_{13}+ c_{2}c_{11}+ c_{5}c_{8})+\\
&\qquad\quad + u_{4}c_{4}(c_{13}+ c_{2}c_{11}+ c_{4}c_{9})).\\
&\;\;94 \;\; \tilde{\varphi }_{12}(= u_{1}c_{23}+ u_{4}c_{16}+ u_{5}c_{8}).\\
&\;\;90 \;\; \psi _{6} (= u_{1}c_{5}c_{17}+ u_{2}(c_{21}+ c_{4}c_{17}+ c_{5}c_{16})+ u_{3}c_{2}c_{17}+ u_{5}c_{2}c_{5}).\\
&\qquad \psi _{5} (= u_{1}c_{9}c_{13}+ u_{2}c_{8}(c_{13}+ c_{5}c_{8}+ c_{4}c_{9})+ u_{3}c_{2}c_{8}c_{9}+\\
&\qquad\quad\; + u_{4}c_{2}(c_{13}+ c_{2}c_{11}+ c_{4}c_{9})).\\
&\;\;86 \;\; \tilde{\varphi }_{11}(= u_{1}c_{21}+ u_{2}c_{4}c_{16}+ u_{3}c_{2}c_{16}+ u_{5}c_{2}c_{4}).\\
&\;\;82 \;\; \omega _{2} (= u_{2}c_{19}+ u_{3}c_{17}+ u_{5}c_{5}),\\
&\qquad \psi _{4} (= u_{1}c_{9}c_{11}+ u_{2}c_{8}c_{11}+ u_{3}c_{8}c_{9}+ u_{4}(c_{13}+ c_{2}c_{11}+ c_{4}c_{9}).\\
&\;\;78 \;\; \tilde{\varphi }_{10}(= u_{1}c_{19}+ u_{3}c_{16}+ u_{5}c_{4}).\\
&\;\;74 \;\; \psi _{3}(=u_{1}c_{5}c_{13}+ u_{2}c_{4}(c_{13}+ c_{4}c_{9}+ c_{5}c_{8})+u_{3}c_{2}(c_{13}+ c_{5}c_{8}+ c_{2}c_{11})+\\
&\qquad\quad\; +u_{4}c_{2}c_{4}c_{5}\\
&\;\;70 \;\; \tilde{\varphi }_{9} (= u_{1}c_{17}+ u_{2}c_{16}+ u_{5}c_{2}).\\
&\;\;66 \;\; \psi _{2} (= u_{1}c_{5}c_{11}+ u_{2}c_{4}c_{11}+ u_{3}(c_{13}+ c_{2}c_{11}+ c_{5}c_{8})+ u_{4}c_{4}c_{5}).\\
&\;\;62 \;\; u_{5}.\\
&\;\;58 \;\; \psi _{1} (= u_{1}c_{5}c_{9}+ u_{2}(c_{13}+ c_{4}c_{9}+ c_{5}c_{8} )+ u_{3}c_{2}c_{9}+ u_{4}c_{2}c_{5}).\\
&\;\;54 \;\; \tilde{\varphi }_{7} (= u_{1}c_{13}+ u_{2}c_{4}c_{8}+ u_{3}c_{2}c_{8}+ u_{4}c_{2}c_{4}).\\
&\;\;50 \;\; \omega _{1} (= u_{2}c_{11}+ u_{3}c_{9}+ u_{4}c_{5}).\\
&\;\;46 \;\; \tilde{\varphi }_{6} (= u_{1}c_{11}+ u_{3}c_{8}+ u_{4}c_{4}).\\
&\;\;38 \;\; \tilde{\varphi }_{5} (= u_{1}c_{9}+ u_{2}c_{8}+ u_{4}c_{2}).\\
&\;\;30 \;\; u_{4}.\\
&\;\;22 \;\; \varphi _{3} (= u_{1}c_{5}+ u_{2}c_{4}+ u_{3}c_{2}).\\
&\;\;14 \;\; u_{3}.\\
&\;\;\;\;6  \;\; u_{2}.\\
&\;\;\;\;2  \;\; u_{1}.\\
&\;\;\;\uparrow \\
&\;\;\; \; t
\end{align*}
\clearpage

Table 12. The cell $E^{0,0,t}_{2}$ of the MASS for $t<108$ (generators).
\begin{align*}
&104\;\;c_{26},\;\;e_{26}(= c^{2}_{13}).\hbox{\qquad\qquad\qquad\qquad\qquad\qquad\qquad\qquad\qquad\qquad\qquad\qquad\quad}\\
&\;\;96\;\;c_{24}.\\
&\;\;88\;\;c_{22},\;\;e_{22}(= c^{2}_{11}).\\
&\;\;80\;\;c_{20}.\\
&\;\;72\;\;c_{18},\;\;e_{18}(= c^{2}_{9}).\\
&\;\;64\;\;e_{16}( = c^{2}_{8}).\\
&\;\;56\;\;c_{14}.\\
&\;\;48\;\;c_{12}.\\
&\;\;40\;\;c_{10},\;\;e_{10}(= c^{2}_{5}).\\
&\;\;32\;\;e_{8}(= c^{2}_{4}).\\
&\;\;24\;\;c_{6}.\\
&\;\;16\;\;e_{4}(= c^{2}_{2}).\\
&\;\;\;\;0\;\;1.\\
&\;\;\;\uparrow \;t
\end{align*}

Table 13. The generators of the cell $E^{2,0,t}_{\infty }$ of the MASS
for $t < 54$.
\begin{align*}
&52\;\; a_{13}(=h_{0}c_{13}+ (h_{0}c_{2}+ h_{1}h_{2})c_{11}+ (h_{0}c_{4}+ h_{1}h_{3})c_{9}+ (h_{0}c_{8}+ h_{1}h_{4})c_{5})\hbox{\qquad}\\
&\;\;\;\;\;\;b_{13}(=h_{0}c_{13}+ (h_{0}c_{8}+ h_{1}h_{4})c_{5}+ h_{2}h_{4}c_{4} + h_{3}h_{4}c_{2})\\
&\;\;\;\;\;\;f_{13}(=h_{0}c_{13}+ (h_{0}c_{2}+ h_{1}h_{2})c_{11}+ h_{2}h_{3}c_{8} + h_{2}h_{4}c_{4})\\
&48\;\; b_{12}(=h_{0}c_{4}c_{8} + h_{1}h_{3}c_{8} + h_{1}h_{4}c_{4} + h^{2}_{1}c_{11})\\
&44\;\; a_{11}(=h_{0}c_{11}+ h_{3}h_{4})\\
&\;\;\;\;\;\;b_{11}(=h_{0}c_{2}c_{9} + h_{1}h_{2}c_{9} + h_{2}h_{4}c_{2} + h^{2}_{2}c_{8})\\
&40\;\; b_{10}(=h_{0}c_{2}c_{8} + h_{1}h_{2}c_{8} + h_{1}h_{4}c_{2} + h^{2}_{1}c_{9})\\
&36\;\; a_{9}(=h_{0}c_{9} + h_{2}h_{4})\\
&\;\;\;\;\;\;b_{9}(=h_{0}c_{4}c_{5} + h_{1}h_{3}c_{5} + h_{2}h_{3}c_{4} + h^{2}_{3}c_{2})\\
&32\;\; a_{8}(=h_{0}c_{8} + h_{1}h_{4})\\
&28\;\; a_{7}(=h^{2}_{3})\\
&\;\;\;\;\;\;b_{7}(=h_{0}c_{2}c_{5} + h_{1}h_{2}c_{5} + h_{2}h_{3}c_{2} + h^{2}_{2}c_{4})\\
&24\;\; b_{6}(=h_{0}c_{2}c_{4} + h_{1}h_{2}c_{4} + h_{1}h_{3}c_{2} + h^{2}_{1}c_{5})\\
&20\;\; a_{5}(=h_{0}c_{5} + h_{2}h_{3})\\
&16\;\; a_{4}(=h_{0}c_{4}+ h_{1}h_{3})\\
&12\;\; a_{3}(=h^{2}_{2})\\
&\;\;8\;\; a_{2}(=h_{0}c_{2} + h_{1}h_{2})\\
&\;\;4\;\; a_{1}(=h^{2}_{1})\\
&\;\;0\;\; h_{0}\\
&\;\uparrow \;t
\end{align*}
\clearpage

Table 14. Relations in the term $E^{*,*,t}_{\infty }$ for $t < 54$.
\begin{align*}
&h^{2}_{i}u^{2}_{j} = h^{2}_{j}u^{2}_{i}, (i,j = 1,2,3,4).\\
&h_{0}x = 0, \text{where } x \text{ is an arbitrary element with the second grading} > 0.\\
&u_{i}(h_{0}c_{[i,k]} + h_{i}h_{k}) = u_{k}h^{2}_{i},\text{ where }[i,k] = 2^{i-1}+ 2^{k-1}- 1; i,k \in  \{1,2,3,4\};\\
& \  i\neq k. \\
&u_{i}(h_{0}c_{[j,k]} + h_{j}h_{k}) = u_{j}(h_{0}c_{[i,k]} + h_{i}h_{k}) = u_{k}(h_{0}c_{[i,j]} + h_{i}h_{j}),\text{where}\\
&i\neq j\neq k\neq i\in \{1,2,3,4\};[i,k]= 2^{i-1}+ 2^{k-1}- 1,[i,j]= 2^{i-1}+ 2^{j-1}- 1;\\
&[j,k] = 2^{j-1}+ 2^{k-1}- 1.\\
&h^{2}_{0}e_{2k} = a^{2}_{k} + h^{2}_{i}h^{2}_{j},\text{where }k = 2^{i-1}+ 2^{j-1}- 1, i\neq j\in \{1,2,3\}.\\
&h_{0}(h_{0}c_{k}c_{j} + h_{i}h_{j}c_{l} + h_{s}h_{j}c_{k} + h^{2}_{j}c_{r})\!=\! (h_{0}c_{k} + h_{i}h_{j})(h_{0}c_{l} + h_{s}h_{j})(h_{0}c_{r} + h_{s}h_{i})\\
&\text{where } k = 2^{i-1} + 2^{j-1} - 1, l = 2^{s-1} + 2^{j-1} - 1, r = 2^{s-1} + 2^{i-1}- 1; \\
&i\neq j\neq k\neq i \in  \{1,2,3\}.\\
&h_{0}(a_{13} + b_{13}) = a_{2}a_{11} + a_{4}a_{9};\\
&u_{1}b_{6} = \varphi _{3}a_{1};\\
&h_{0}(a_{13} + f_{13}) = a_{5}a_{8} + a_{4}a_{9};\\
&u_{1}b_{7} = \varphi _{3}a_{2} = u_{2}b_{6};\\
&b^{2}_{6} = a^{2}_{2}e_{8} + a_{1}a_{7}e_{4} + a^{2}_{1}e_{10} = a^{2}_{4}e_{4} + a_{1}a_{3}e_{8} + a^{2}_{1}e_{10};\\
&u_{2}b_{7} = \varphi _{3}a_{3};\\
&u_{1}b_{9} = \varphi _{3}a_{4} = u_{3}b_{6};\\
&h_{0}(e_{4}a_{7}+ e_{8}a_{3}) = a_{2}b_{9}+ a_{4}b_{7};\\
&u_{1}b_{10}= \varphi _{5}a_{1} + u_{3}c_{6}a_{1};\\
&h_{0}(e_{8}a_{3}+ a_{1}e_{10})= a_{4}b_{7}+ a_{5}b_{6};\\
&u_{2}b_{9} = \varphi _{3}a_{5} = u_{3}b_{7};\\
&u_{1}\omega _{1} = u_{2}(\varphi _{6}+ u_{2}c_{10}+ u_{3}c^{4}_{2}) + u_{3}(\varphi _{5} + u_{3}c_{6}) + u_{4}\varphi _{3};\\
&u_{1}b_{11}= u_{2}b_{10}= \varphi _{5}a_{2} + u_{3}c_{6}a_{2};\\
&u_{2}b_{11}= \varphi _{5}a_{3} + u_{3}c_{6}a_{3};\\
&u_{1}(a_{1}e_{10}+ a_{3}e_{8}+ a_{7}e_{4})= b_{6}\varphi _{3};\\
&u_{3}b_{9} = \varphi _{3}a_{7};\\
&u_{2}(a_{1}e_{10}+ a_{3}e_{8}+ a_{7}e_{4})= b_{7}\varphi _{3};\\
&u_{1}b_{12}= \varphi _{6}a_{1} + u_{2}c_{10}a_{1} + u_{3}c^{4}_{2}a_{1};\\
&\omega _{1}a_{1} = u_{2}b_{12} + u_{3}b_{10} + u_{4}b_{6}.
\end{align*}

\clearpage

{ Table 15. The term $E_2=Ext_{A}(BP^*(MSp),BP^*)$ }
\bigskip



\clearpage

Table 16. The relations in ${\rm Ext}_{A}(BP^{*}(MSp),BP^{*})$.
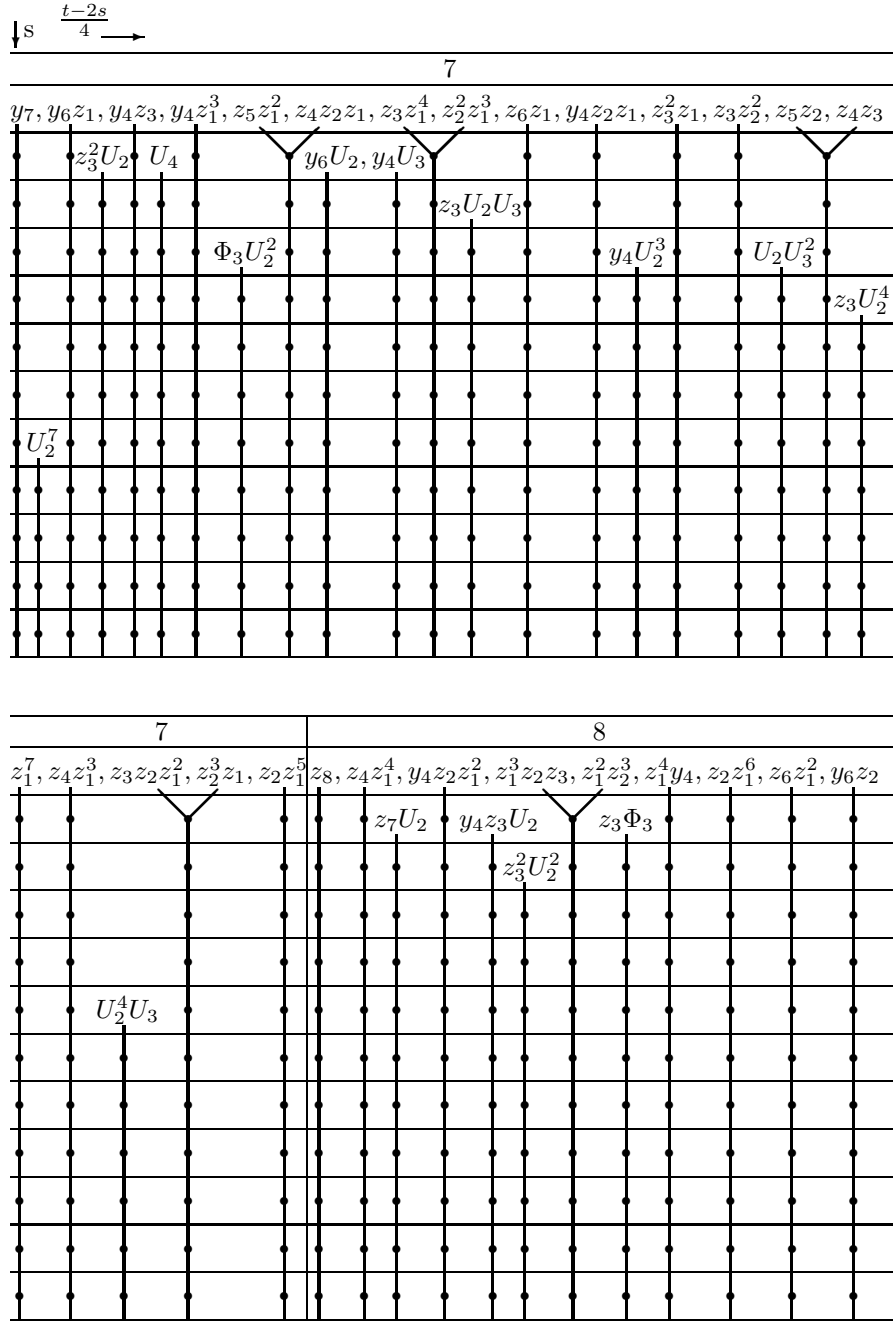
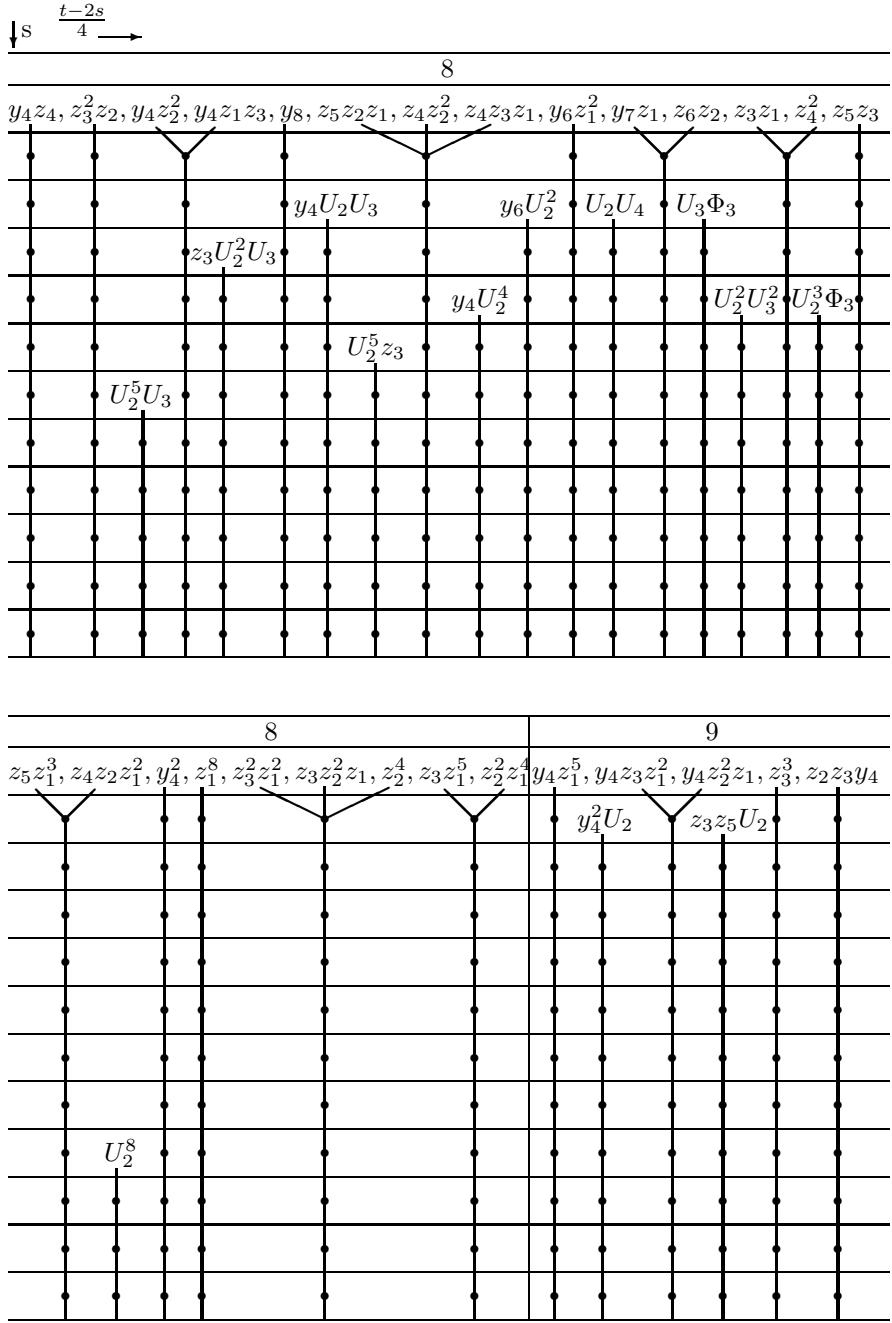
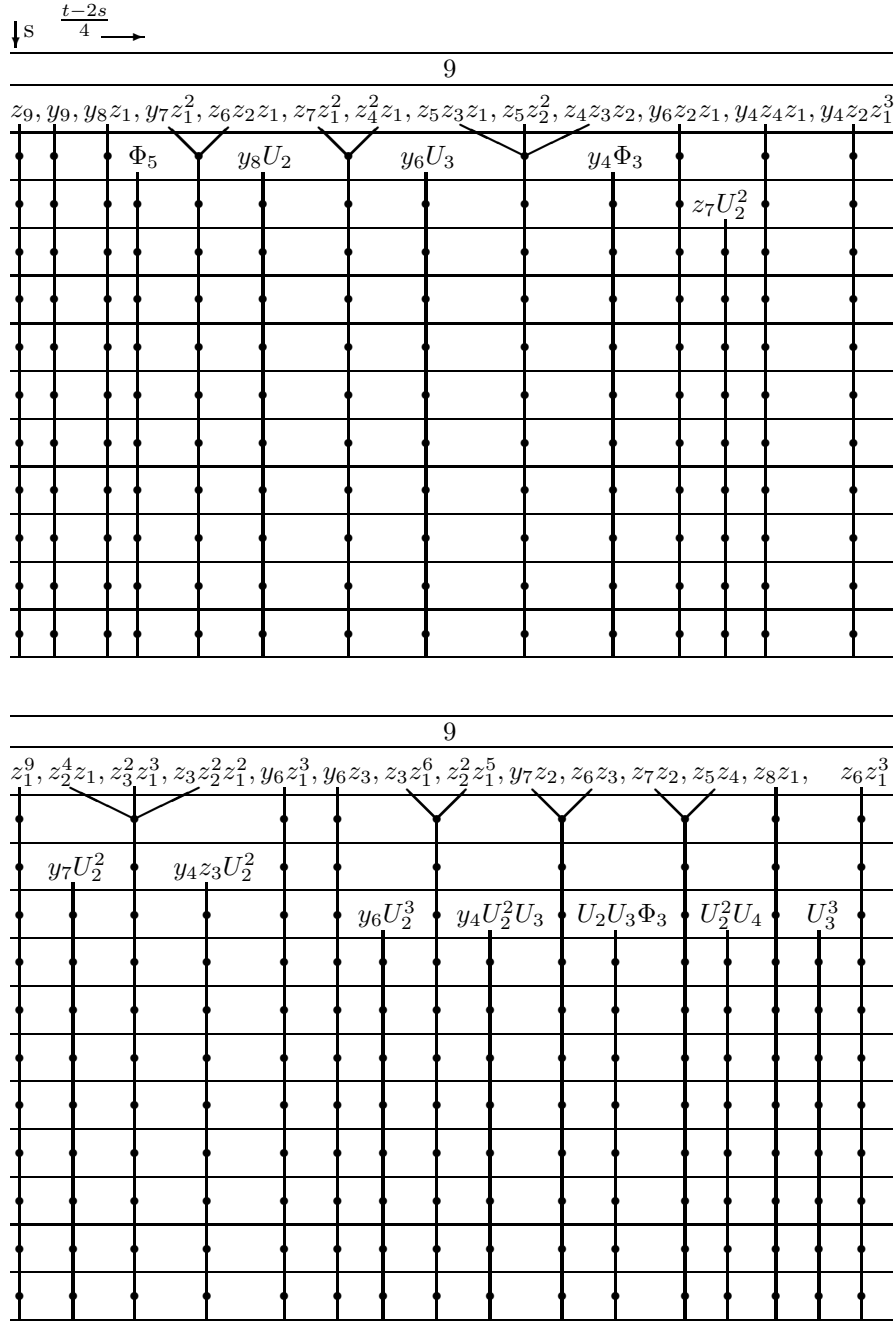
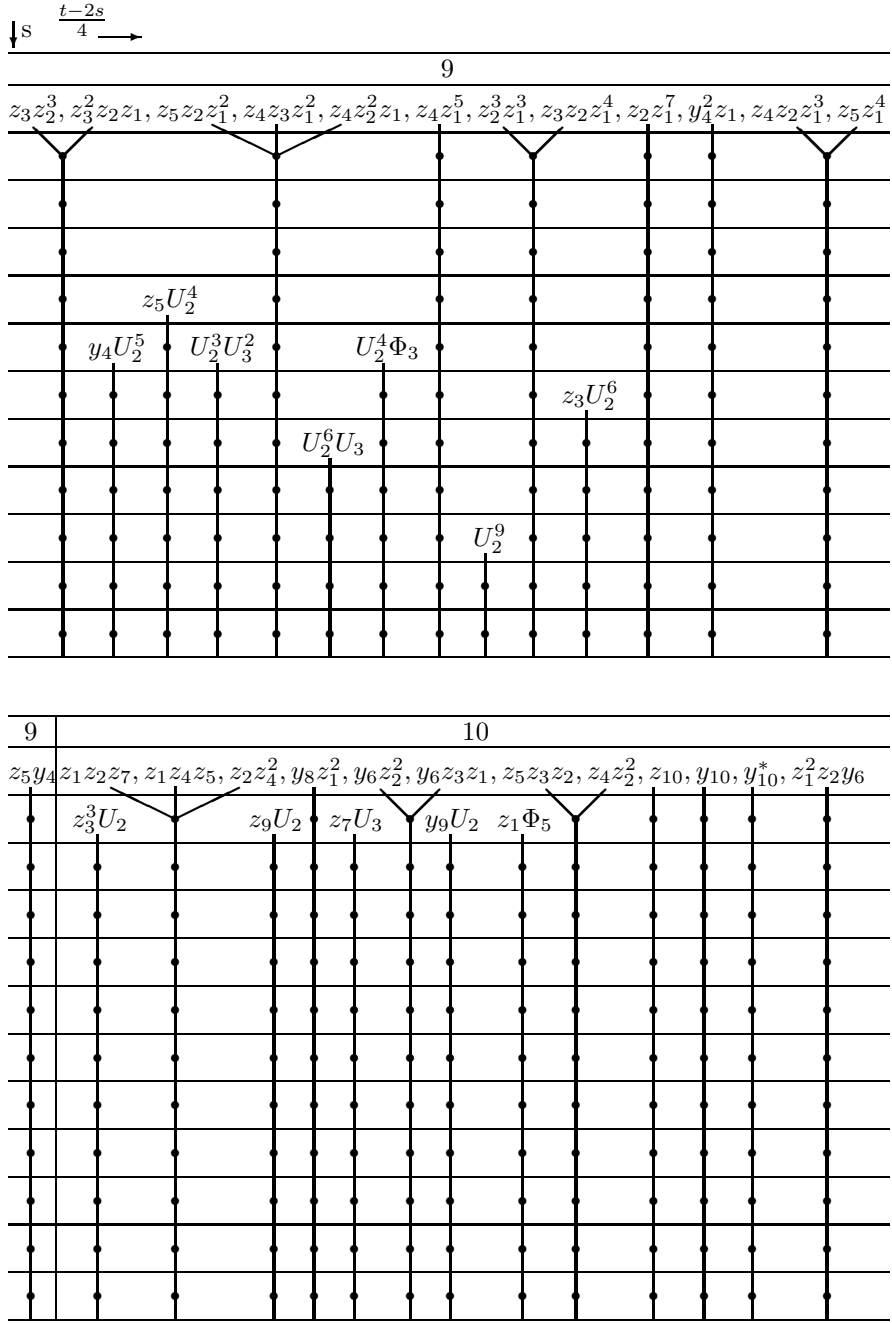
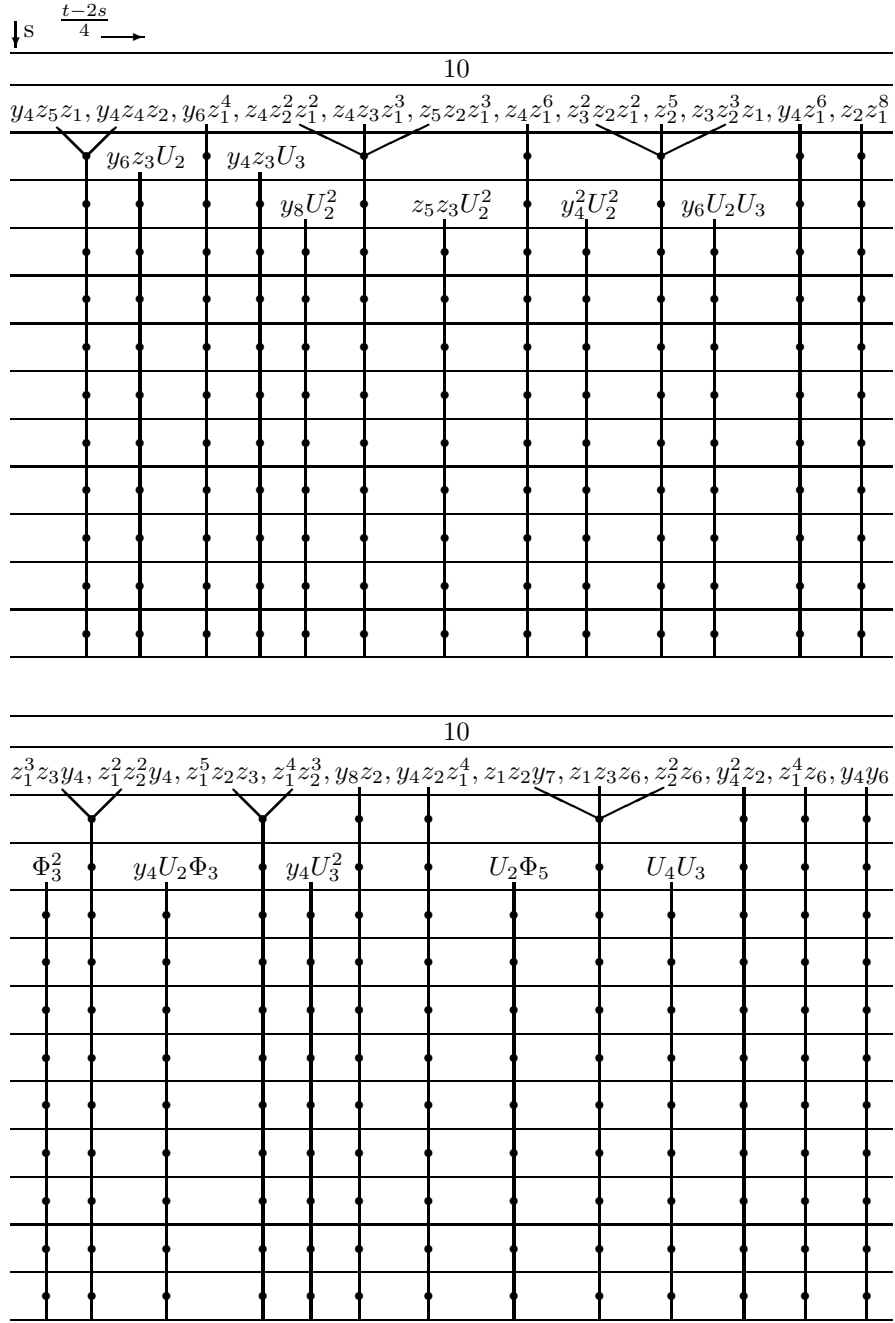
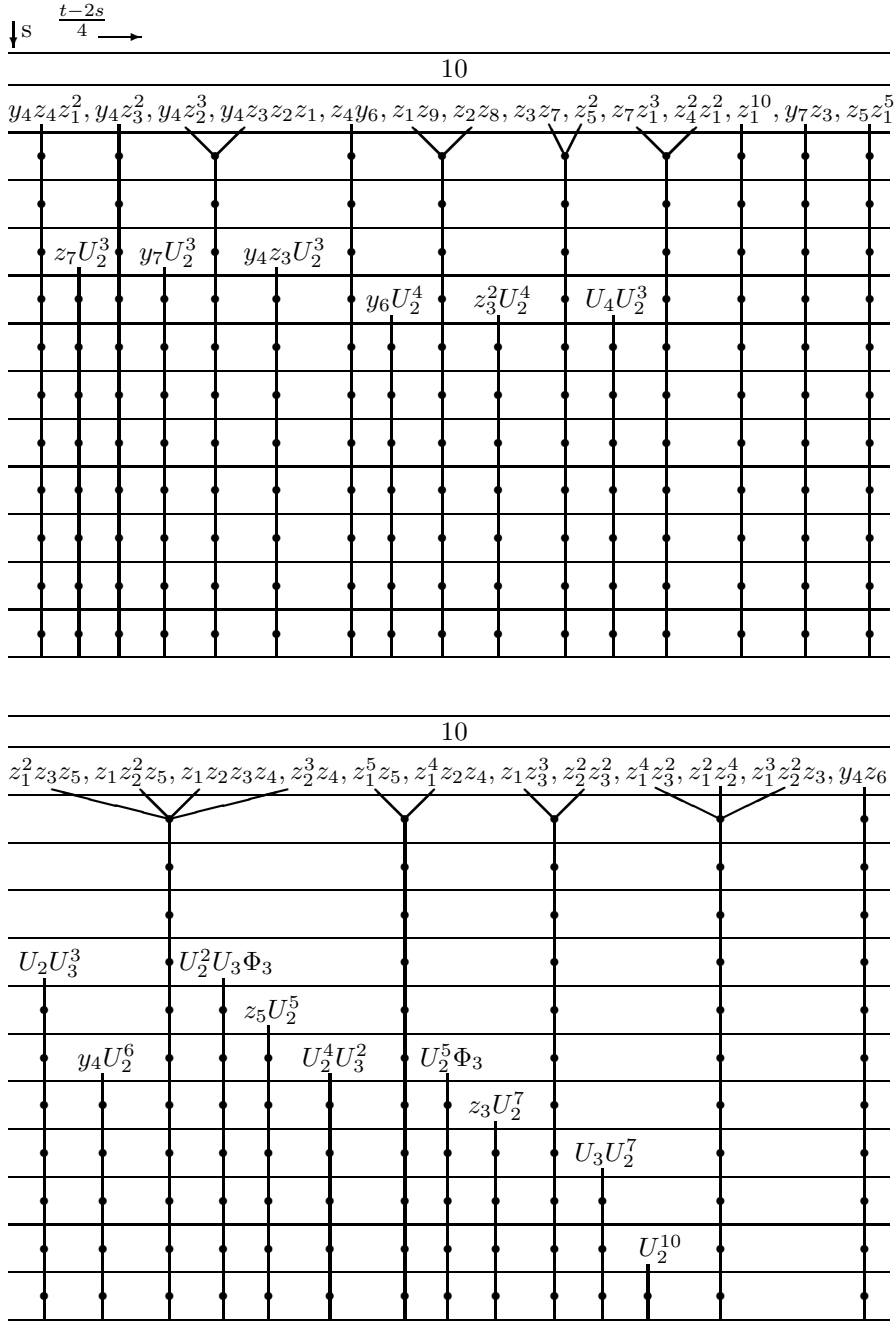
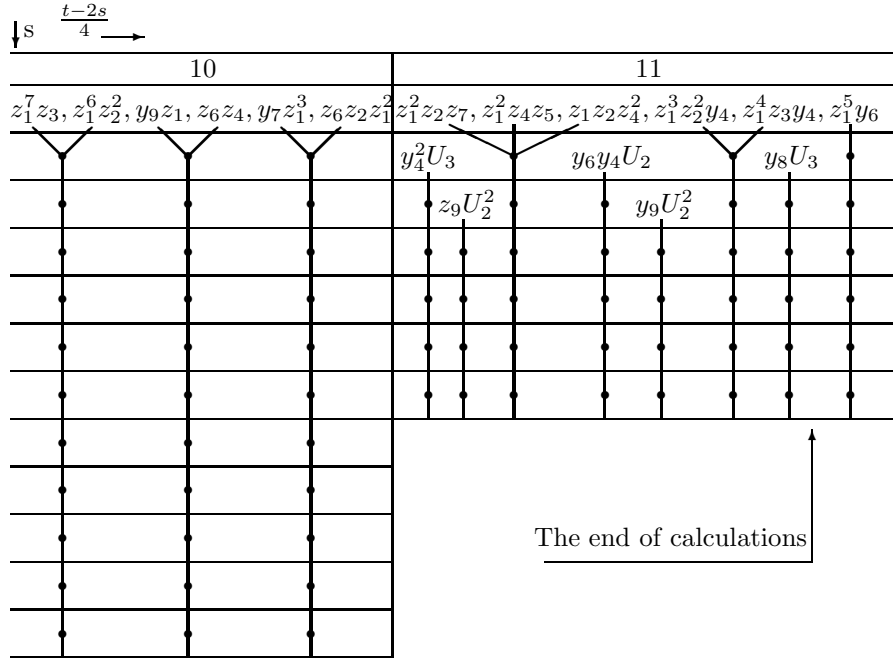
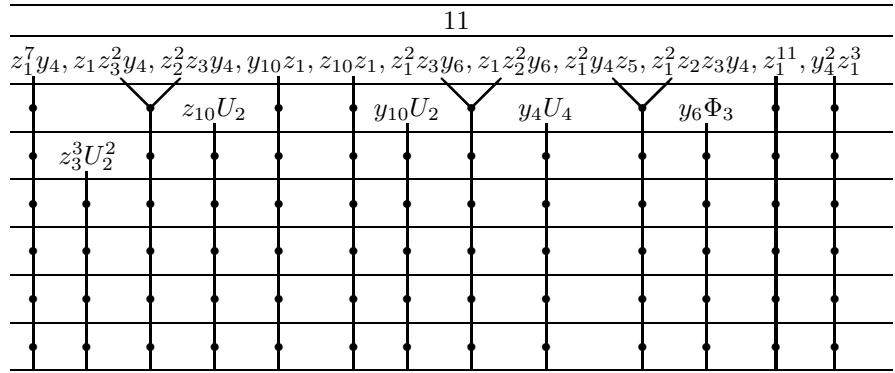
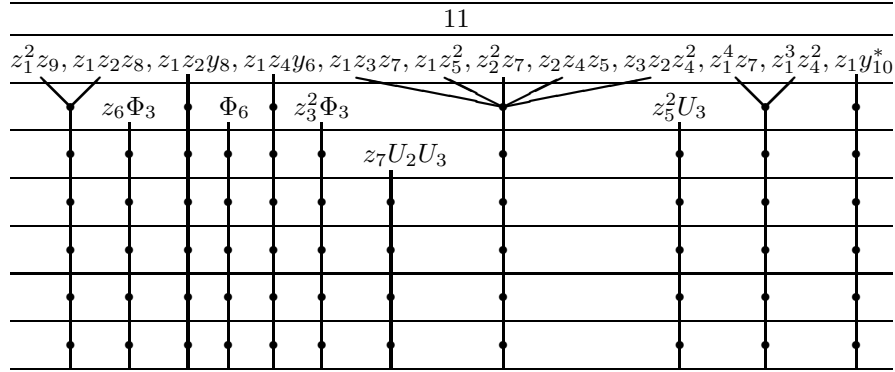
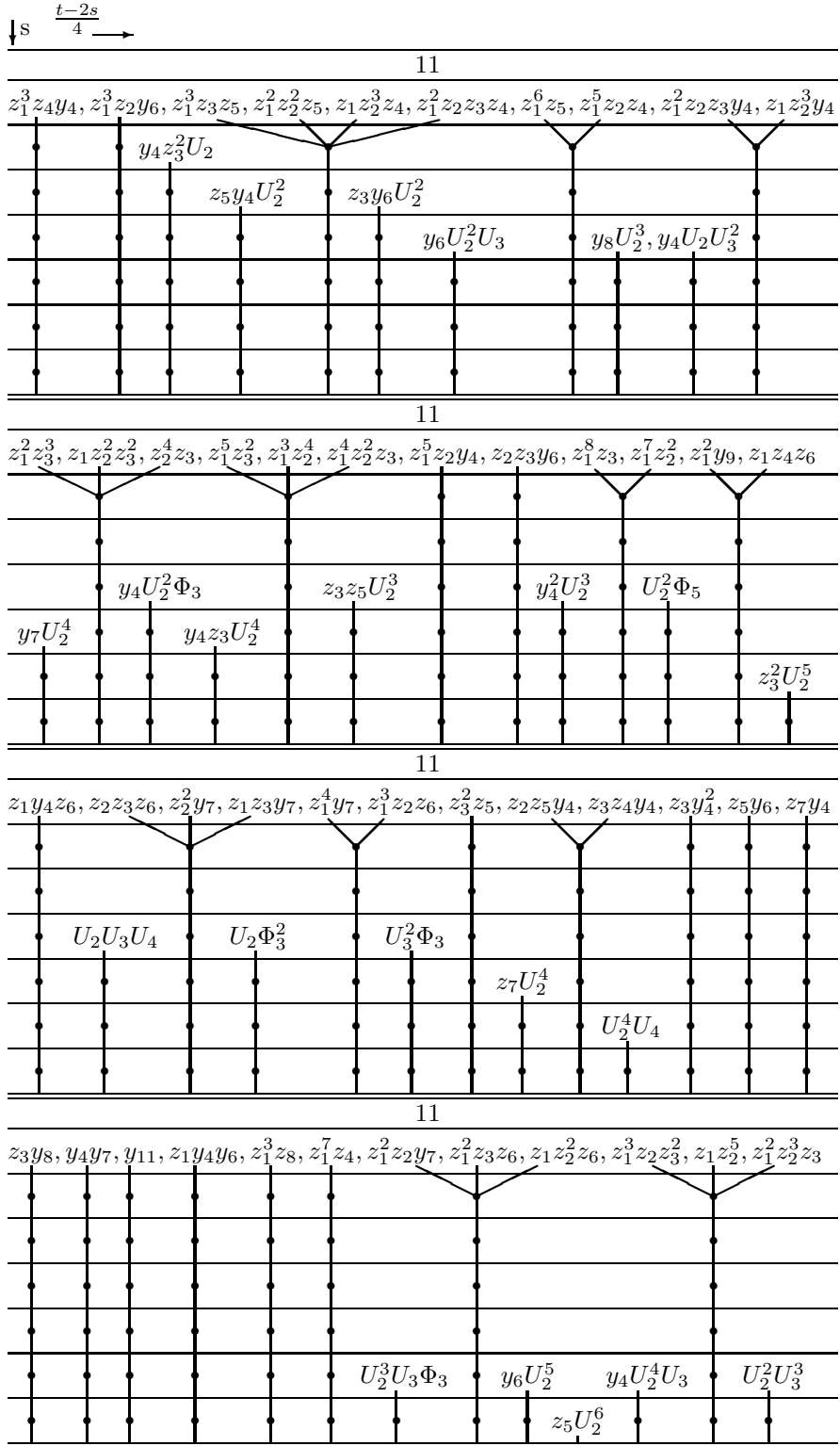
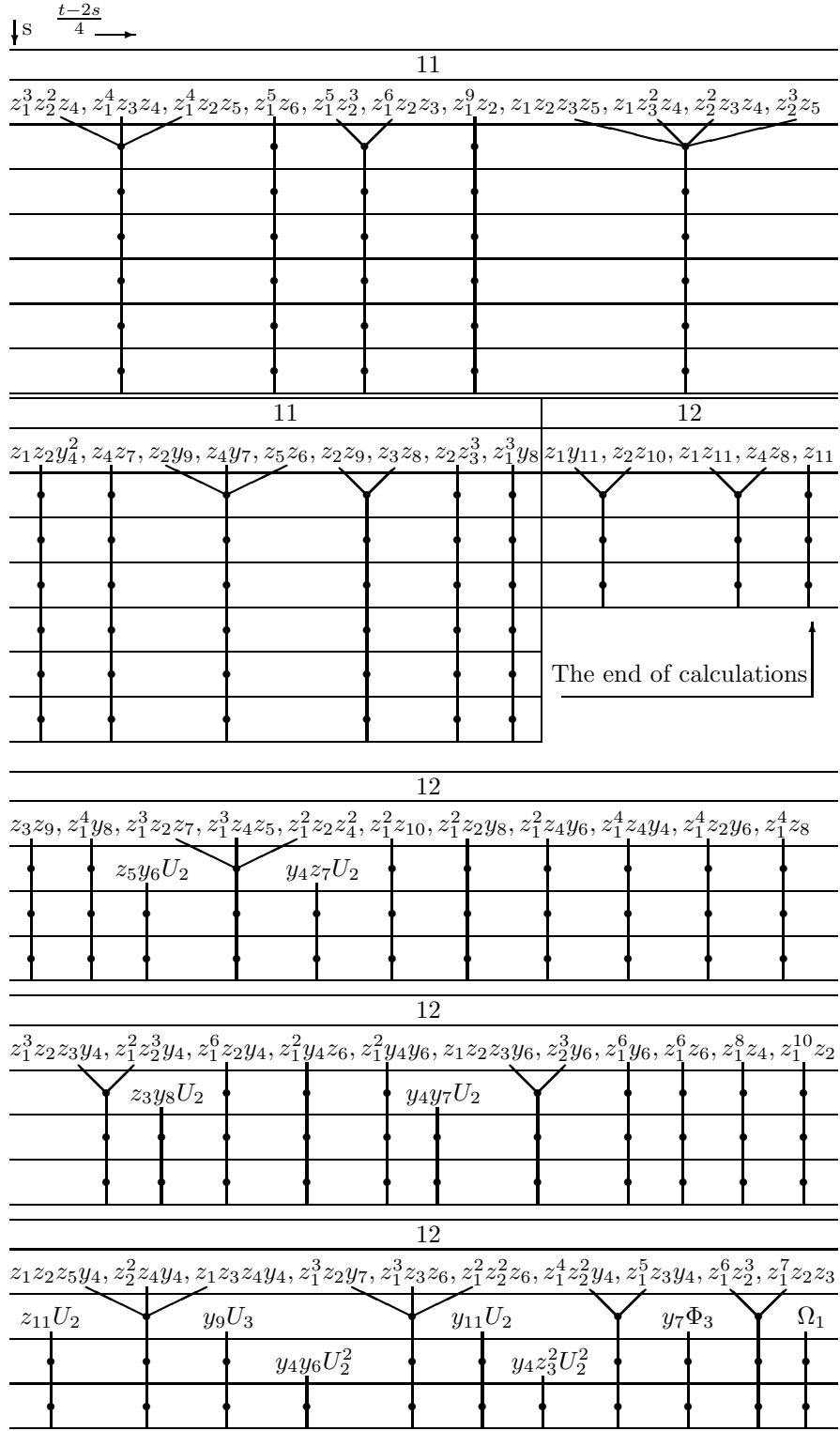
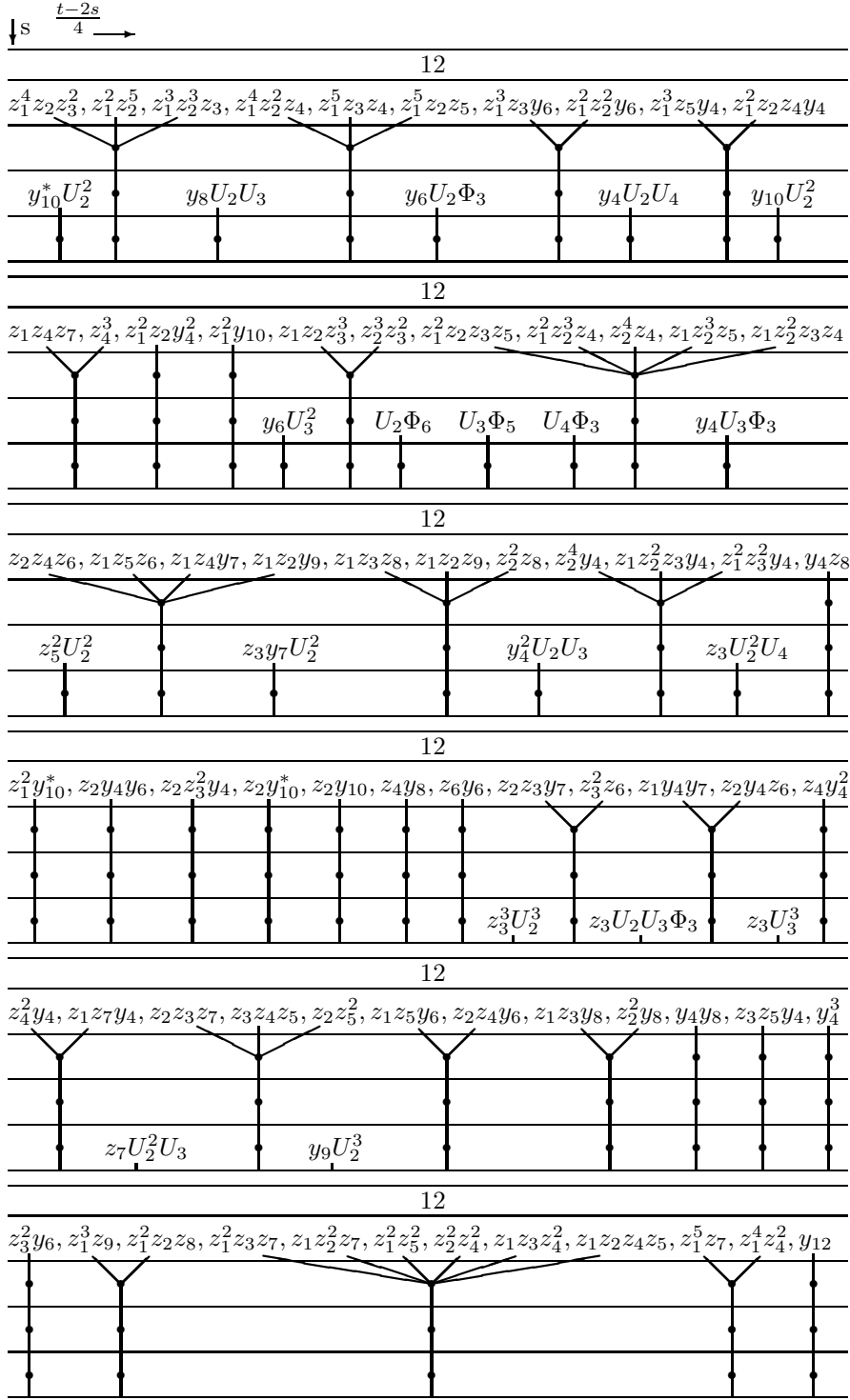
\begin{align*}
&1.\; 4(y_{8} + z_{8}+ z_{2}y_{6} + y_{4}z_{4}) = z_{1}z_{7}+ z^{2}_{4}.\\
&2.\; U_{1}z_{8} = U_{4}z_{1}.\\
&3.\; U_{2}y_{7} = \Phi _{3}z_{3}.\\
&4.\; U_{2}z_{7} = U_{3}z_{5}.\\
&5.\; 2y_{9} = z_{2}z_{7} + 3z_{4}z_{5}\\
&6.\; U_{1}y_{9} = U_{3}z_{6} = \Phi _{3}z_{4}.\\
&7.\; U_{1}z_{9} = U_{2}z_{8} = U_{4}z_{2}.\\
&8.\; 4(y^{*}_{10}+ y_{4}y_{6}) = z_{3}z_{7} + z^{2}_{5} + X^{1}_{40}(\hat{y^{*}_{10}},\hat{y_{4}y_{6}},\hat{z_{3}z_{7}},\hat{z^{2}_{5}}).\\
&9.\; 2(z_{10} + z_{4}y_{6}) = z_{2}z_{8} + 3z_{1}z_{9} + X^{2}_{40}(\hat{z_{10}},\hat{z_{4}y_{6}},\hat{z_{2}z_{8}},\hat{z_{1}z_{9}}).\\
&10.\; U_{1}z_{10} = \Phi _{5}z_{1}\\
&11.\; U_{2}z_{9} = U_{4}z_{3}.\\
&12.\; U_{2}y_{9} = U_{3}y_{7} = \Phi _{3}z_{5}.\\
&13.\; \Phi ^{2}_{3} = U_{1}[U_{1}y^{*}_{10} + U_{2}(y_{9} + z_{9}) + U_{3}z_{7}] + U^{2}_{2}y_{8} + U^{2}_{3}y_{4} + U_{1}[U_{1}y_{4}y_{6} + \\
&\qquad \;\; + U_{2}(y_{4}z_{5} + y_{6}z_{3})].\\
&14.\; 2y_{4}z_{7} + 2y_{8}z_{3} = z_{2}y_{9} + z_{4}y_{7} + X^{1}_{44}(\hat{y_{4}z_{7}},\hat{y_{8}z_{3}},\hat{z_{2}y_{9}},\hat{z_{4}y_{7}}).\\
&15.\; 2y_{8}z_{3} + 2(y^{*}_{10} + y_{6}y_{4})z_{1} = z_{4}y_{7} + z_{5}z_{6} + X^{2}_{44}(\hat{y_{6}z_{5}},\hat{y_{8}z_{3}},\hat{z_{1}(y_{4}y_{6} + y^{*}_{10})},\\
&\qquad \;\; \hat{z_{4}y_{7}}).\\
&16.\; 2y_{11} = z_{2}z_{9} + 3z_{3}z_{8} + X^{3}_{44}(\hat{y_{11}},\hat{z_{2}z_{9}},\hat{z_{3}z_{8}}).\\
&17.\; U_{1}(z_{1}y^{*}_{10} + z_{1}y_{4}y_{6} + z_{3}y_{8} + z_{7}y_{4}) + U_{1}z_{4}z_{7} + U_{1}z_{2}z_{9} + U_{1}z_{2}y_{9} +\\
&\qquad \;\; + U_{1}z_{3}(z_{2}y_{6} + z_{4}y_{4}) = \Phi _{3}y_{6}.\\
&18.\; U_{1}y_{11} = U_{2}z_{10} = \Phi _{5}z_{2}.\\
&19.\; U_{1}z_{11} = U_{3}z_{8} = U_{4}z_{4}.\\
&20.\; z^{2}_{6} = (1 + 2\beta _{1})z^{2}_{2}y_{8} + (1 + 2\beta _{2})z_{1}z_{7}y_{4} + (1 + 2\beta _{3})z^{2}_{1}(y^{*}_{10} + y_{4}y_{6}) + X^{1}_{48}.\\
&21.\; 2z_{12} = z_{1}z_{11} + 3z_{4}z_{8} + X^{2}_{48}(\hat{y_{12}},\hat{z_{1}z_{11}},\hat{z_{4}z_{8}}),\\
&22.\; U_{1}z_{12} = \Phi _{6}z_{1}.\\
&23.\; U_{2}y_{11} = \Phi _{5}z_{3}.\\
&24.\; U_{2}z_{11} = U_{3}z_{9} = U_{4}z_{5}.\\
&25.\; U_{3}y_{9} = \Phi _{3}z_{7}.\\
&26.\; \Phi _{3}y_{7} = U_{2}(z_{1}(y^{*}_{10} + z_{10} + y_{4}y_{6}) + z_{2}(y_{9} + z_{9}) + z_{3}(y_{8} + y_{6}z_{2} + y_{4}z_{4}) +\\
&\qquad \;\; z_{7}y_{4})\\
&27.\; U_{1}\Omega _{1} = U_{2}\Phi _{6} + U_{3}\Phi _{5} + U_{4}\Phi _{3} + U^{2}_{2}y_{10} + U^{2}_{3}y_{6} + U_{2}U_{3}y^{2}_{4}.
\end{align*}

{\bf Remark.} The summand of the form $X_{n}(\ldots ,\hat{x} ,\ldots )$
has the following properties:
(i) $X_{n} \in  F^{6}(E^{0,n}_{2})$ , (ii) $d_{3}(X_{n}) = 0$,
(iii) the summand $x$ do not appear in the expression for $X_{n}$.

\clearpage

Table 17. The action of the differential $d_{3}$ of the Adams-Novikov
spectral sequence (continuation of Table~4 of the work \cite{V2}).
\begin{align*}
7) \ &d_{3}(U^{n}_{1}U^{4}_{2}z_{3}) = U^{n+1}_{1}U^{6}_{2}, d_{3}(U^{n}_{1}U^{4}_{2}U_{3}) = 0, d_{3}(U^{n}_{1}U^{7}_{2}) = 0.\\
8) \ &d_{3}(U^{n}_{1}U_{2}z_{7}) = d_{3}(U^{n}_{1}U_{3}z_{5}) = U^{n+1}_{1}U_{2}U^{2}_{3}, d_{3}(U^{n}_{1}\Phi _{3}z_{3}) = U^{n+1}_{1}U^{2}_{2}\Phi _{3},\\
&d_{3}(U^{n}_{1}U_{2}z_{3}y_{4}) = U^{n}_{1}U_{2}(U^{3}_{2}z_{3}+ U_{1}U^{2}_{2}y_{4}),d_{3}(U^{n}_{1}U^{2}_{2}U_{3}z_{3}) = U^{n+1}_{1}U^{4}_{2}U_{3},\\
&d_{3}(U^{n}_{1}U_{2}U_{3}y_{4}) = d_{3}(U^{n}_{1}U^{2}_{2}y_{6}) = U^{n}_{1}U^{4}_{2}U_{3}, d_{3}(U^{n}_{1}U^{4}_{2}y_{4}) = U^{n}_{1}U^{7}_{2},d_{3}(U^{n}_{1}\times \\
&\times U^{8}_{2}) = 0, d_{3}(U^{n}_{1}U^{5}_{2}z_{3}) = U^{n+1}_{1}U^{7}_{2}, d_{3}(U^{n}_{1}U_{3}\Phi _{3}) = 0, d_{3}(U^{n}_{1}U_{2}U_{4}) = 0,\\
&d_{3}(U^{n}_{1}U^{2}_{2}z^{2}_{3}) = 0, d_{3}(U^{n}_{1}U^{2}_{2}U^{2}_{3}) = 0, d_{3}(U^{n}_{1}U^{3}_{2}\Phi _{3}) = 0,d_{3}(U^{n}_{1}U^{5}_{2}U_{3}) = 0.\\
9) \ &d_{3}(U^{n}_{1}z_{9}) = U^{n+1}_{1}U_{2}U_{4}, d_{3}(U^{n}_{1}y_{9}) = U^{n+1}_{1}U_{3}\Phi _{3}, d_{3}(U^{n}_{1}z^{5}_{1}y_{4}) = U^{n+3}_{1}z^{2}_{1}\times\\
&\times (z^{2}_{1}y_{4} + z^{3}_{2}), d_{3}(U^{n}_{1}z^{3}_{3}) = d_{3}(U^{n}_{1}z_{3}z_{3}y_{4}) = U^{n+1}_{1}U^{2}_{2}z^{2}_{3}, d_{3}(U^{n}_{1}z^{2}_{1}z_{3}y_{4}) =\\
&= d_{3}(U^{n}_{1}z_{1}z^{2}_{2}) = U^{n+3}_{1}z_{2}(z_{2}y_{4}+ z^{2}_{3}),d_{3}(U^{n}_{1}z^{2}_{1}y_{7}) = d_{3}(U^{n}_{1}z_{1}z_{2}z_{6}) = \\
&= U^{n+3}_{1}z_{1}y_{7}, d_{3}(U^{n}_{1}z^{2}_{1}z_{7}) = d_{3}(U^{n}_{1}z_{1}z^{2}_{4}) = U^{n+3}_{1}z_{1}z_{7}, d_{3}(U^{n}_{1}z_{1}z_{3}z_{5}) = \\
&= d_{3}(U^{n}_{1}z_{1}z_{2}y_{6}) = d_{3}(U^{n}_{1}z_{2}z_{3}z_{4}) = d_{3}(U^{n}_{1}z_{1}z_{4}y_{4}) = d_{3}(U^{n}_{1}z^{2}_{2}z_{5}) = U^{n+3}_{1}\times\\
&\times z_{3}z_{5}, d_{3}(U^{n}_{1}z_{4}z_{2}z^{3}_{1}) = d_{3}(U^{n}_{1}z_{5}z^{4}_{1}) = U^{n+3}_{1}z^{3}_{1}z_{5}, d_{3}(U^{n}_{1}z_{1}y^{2}_{4}) = U^{n+3}_{1}y^{2}_{4},\\
& d_{3}(U^{n}_{1}z_{1}z^{4}_{2})= d_{3}(U^{n}_{1}z^{3}_{1}z^{2}_{3}) = d_{3}(U^{n}_{1}z^{3}_{1}z_{2}y_{4}) = d_{3}(U^{n}_{1}z^{2}_{1}z^{2}_{2}z_{3}) = U^{n+3}_{1}z^{4}_{2},\\
& d_{3}(U^{n}_{1}z_{1}y_{8}) = U^{n+2}_{1}(U_{1}y_{8} + U_{2}z_{7}),d_{3}(U^{n}_{1}z^{3}_{1}y_{6}) = U^{n+3}_{1}z_{1}(z_{1}y_{6}+z_{3}z_{4}),\\
&d_{3}(U^{n}_{1}z^{9}_{1}) = U^{n+3}_{1}z^{8}_{1}, d_{3}(U^{n}_{1}z_{5}y_{4}) = U^{n}_{1}U_{2}U_{3}(U_{1}y_{4}+ U_{2}z_{3}), d_{3}(U^{n}_{1}z_{3}y_{6}) =\\
& = U^{n}_{1}U^{2}_{2}(U_{6}y_{6} + U_{3}z_{3}), d_{3}(U^{n}_{1}z^{6}_{1}z_{3})= d_{3}(U^{n}_{1}z^{5}_{1}z^{2}_{2}) = U^{n+3}_{1}z^{5}_{1}z_{3}, d_{3}(U^{n}_{1}\times\\
&\times\!U_{2}y_{8}) = d_{3}(U^{n}_{1}U_{3}y_{6}) =U^{n}_{1}U^{2}_{2}U^{2}_{3},d_{3}(U^{n}_{1}\Phi _{3}y_{4}) = U^{n}_{1}\Phi _{3}U^{3}_{2}, d_{3}(U^{n}_{1}U^{2}_{2}z_{7})\! = \\
&= U^{n+1}_{1}U^{2}_{2}U^{2}_{3}, d_{3}(U^{n}_{1}U^{2}_{2}y_{7}) = U^{n+2}_{1}U^{3}_{2}\Phi _{3}, d_{3}(U^{n}_{1}U^{2}_{2}z_{3}y_{4}) = U^{n}_{1}U^{2}_{2}(U_{1}y_{4} + \\
&+ U_{2}z_{3}), d_{3}(U^{n}_{1}U^{2}_{2}U_{3}y_{4}) = d_{3}(U^{n}_{1}U^{3}_{2}y_{6}) = U^{n}_{1}U^{5}_{2}U_{3}, d_{3}(U^{n}_{1}U^{4}_{2}z_{5}) = U^{n+1}_{1}\times\\
&\times U^{5}_{2}U_{3}, d_{3}(U^{n}_{1}U^{5}_{2}y_{4})= U^{n}_{1}U^{8}_{2}, d_{3}(U^{n}_{1}U^{6}_{2}z_{3}) = U^{n+2}_{1}U^{7}_{2}, d_{3}(U^{n}_{1}z_{2}y_{7}) = 0,\\
&d_{3}(U^{n}_{1}z_{2}z_{7}) = 0, d_{3}(U^{n}_{1}z_{4}z_{5}) = 0, d_{3}(U^{n}_{1}z^{3}_{2}z_{3}) = 0, d_{3}(U^{n}_{1}z_{1}z_{2}z^{2}_{3}) = 0,\\
&d_{3}(U^{n}_{1}z^{2}_{1}z_{2}z_{5}) = 0, d_{3}(U^{n}_{1}z^{2}_{1}z_{3}z_{4}) = 0, d_{3}(U^{n}_{1}z_{1}z^{2}_{2}z_{4}) =0, d_{3}(U^{n}_{1}z_{1}z_{8}) = 0,\\
&d_{3}(U^{n}_{1}z^{3}_{1}z_{6})=0, d_{3}(U^{n}_{1}z^{5}_{1}z_{4}) = 0, d_{3}(U^{n}_{1}z^{3}_{1}z^{3}_{2}) = 0, d_{3}(U^{n}_{1}z^{4}_{1}z_{2}z_{3}) = 0,\\
&d_{3}(U^{n}_{1}z^{7}_{1}z_{2}) = 0, d_{3}(U^{n}_{1}U_{2}y^{2}_{4}) = 0, d_{3}(U^{n}_{1}z_{3}z_{6}) = 0, d_{3}(U^{n}_{1}U_{2}z_{3}z_{5}) = 0,\\
&d_{3}(U^{n}_{1}U^{3}_{3}) = 0, d_{3}(U^{n}_{1}U^{3}_{2}z^{2}_{3}) = 0, d_{3}(U^{n}_{1}U_{2}U_{3}\Phi _{3}) = 0, d_{3}(U^{n}_{1}U^{2}_{2}U_{4}) = 0,\\
&d_{3}(U^{n}_{1}U^{3}_{2}U^{2}_{3}) =0, d_{3}(U^{n}_{1}U^{4}_{2}\Phi _{3}) =0, d_{3}(U^{n}_{1}U^{6}_{2}U_{3}) =0, d_{3}(U^{n}_{1}U^{9}_{2}) =0,\\
&d_{3}(U^{n}_{1}\Phi _{5}) = 0.\\
10)& \ d_{3}(U^{n}_{1}z_{1}z_{2}z_{7}) = d_{3}(U^{n}_{1}z_{1}z_{4}z_{5}) = d_{3}(U^{n}_{1}z_{2}z^{2}_{4}) = d_{3}(U^{n}_{1}y_{8}z^{2}_{1}) = U^{n+3}_{1}z_{2}z_{7},\\
&d_{3}(U^{n}_{1}y_{6}z^{2}_{2}) = d_{3}(U^{n}_{1}y_{6}z_{3}z_{1}) = d_{3}(U^{n}_{1}z_{5}z_{3}z_{2}) = d_{3}(U^{n}_{1}z_{4}z^{2}_{2}) = d_{3}(U^{n}_{1}y_{4}\times\\
&\times z_{5}z_{1}) = d_{3}(U^{n}_{1}y_{4}z_{4}z_{2}) = U^{n+2}_{1}U_{3}z^{2}_{3}, d_{3}(U^{n}_{1}z_{10}) = U^{n+3}_{1}\Phi _{5}, d_{3}(U^{n}_{1}y^{*}_{10}) =\\
&= U^{n}_{1}(U^{3}_{3} + U^{2}_{2}U_{4} + U_{2}U_{3}\Phi _{3}), d_{3}(U^{n}_{1}y_{10}) = U^{n}_{1}U^{3}_{3}, d_{3}(U^{n}_{1}z^{2}_{1}z_{2}y_{6}) = U^{n+1}_{1}\!\times\\
&\times U_{2}z^{2}_{1}(U_{1}y_{6} + U_{2}z_{5}), d_{3}(U^{n}_{1}y_{6}z^{4}_{1}) = d_{3}(U^{n}_{1}z^{2}_{1}z^{2}_{2}z_{4}) = d_{3}(U^{n}_{1}z^{3}_{1}z_{3}z_{4}) =\\
&= d_{3}(U^{n}_{1}z^{3}_{1}z_{2}z_{5}) = U^{n+3}_{1}z_{1}z^{2}_{2}z_{4}, d_{3}(U^{n}_{1}z_{4}z^{6}_{1}) = U^{n+3}_{1}z_{4}z^{5}_{1}, d_{3}(U^{n}_{1}z^{2}_{1}z_{2}z^{2}_{3})=\\
&=d_{3}(U^{n}_{1}z^{5}_{2})=d_{3}(U^{n}_{1}z_{1}z^{3}_{2}z_{3})=d_{3}(U^{n}_{1}z^{3}_{1}z_{3}y_{4})=\!d_{3}(U^{n}_{1}z^{2}_{1}z^{2}_{2}y_{4})=\! U^{n+2}_{1}U_{2}z^{4}_{2},
\end{align*}
Table 17. (continuation).
\begin{align*}
&d_{3}(U^{n}_{1}z^{5}_{1}z_{2}z_{3}) = d_{3}(U^{n}_{1}z^{4}_{1}z^{3}_{2}) = d_{3}(U^{n}_{1}z^{6}_{1}y_{4}) = U^{n+2}_{1}U_{2}z_{3}z^{5}_{1},d_{3}(U^{n}_{1}z^{8}_{1}z_{2}) =\\
&= U^{n}_{1}U_{2}z^{8}_{1},d_{3}(U^{n}_{1}z^{4}_{1}z_{2}y_{4}) = U^{n+1}_{1}U_{2}z^{4}_{1}(U_{1}y_{4} + U_{2}z_{3}), d_{3}(U^{n}_{1}z_{1}z_{2}y_{7}) =\\
&= d_{3}(U^{n}_{1}z_{1}z_{3}z_{6}) = d_{3}(U^{n}_{1}z^{2}_{2}z_{6}) = U^{n+3}_{1}z_{2}y_{7}, d_{3}(U^{n}_{1}z_{2}y^{2}_{4}) = U^{n+2}_{1}U_{2}y^{2}_{4},\\
&d_{3}(U^{n}_{1}z_{2}y_{8}) = U^{n+1}_{1}U_{2}(U_{1}y_{8} + U_{2}z_{7}), d_{3}(U^{n}_{1}y_{4}y_{6}) = U^{n}_{1}U^{2}_{2}(U_{2}y_{6} + U_{3}y_{4}),\\
& d_{3}(U^{n}_{1}z^{4}_{1}z_{6}) = U^{n+2}_{1}\Phi _{3}z^{4}_{1}, d_{3}(U^{n}_{1}z^{2}_{1}z_{4}y_{4})= U^{n+1}_{1}U_{3}z^{2}_{1}(U_{1}y_{4}+U_{2}z_{3}), \\
&d_{3}(U^{n}_{1}z^{2}_{3}y_{4}) = U^{n}_{1}U^{3}_{2}z^{2}_{3}, d_{3}(U^{n}_{1}z_{4}y_{6})=U^{n+1}_{1}U_{3}(U_{1}y_{6}\!+\!U_{2}z_{5}),  d_{3}(U^{n}_{1}y_{4}z_{6}) =\\
&= U^{n+1}_{1}\Phi _{3}(U_{1}y_{3} + U_{2}z_{3}),d_{3}(U^{n}_{1}z^{3}_{2}y_{4})=d_{3}(U^{n}_{1}z_{1}z_{2}z_{3}y_{4}) = U^{n+1}_{1}U_{2}z^{2}_{2}\times\\
&\times (U_{1}y_{4} + U_{2}z_{3}), d_{3}(U^{n}_{1}z_{1}z_{9})=0, d_{3}(U^{n}_{1}z^{2}_{5})=0, d_{3}(U^{n}_{1}z_{1}z_{2}z_{3}z_{4}) = 0,\\
&d_{3}(U^{n}_{1}z_{2}z_{8}) =0, d_{3}(U^{n}_{1}z_{3}z_{7}) =0, d_{3}(U^{n}_{1}z_{7}z^{3}_{1}) =0, d_{3}(U^{n}_{1}z^{3}_{1}z^{2}_{2}z_{3}) = 0,\\
&d_{3}(U^{n}_{1}z^{2}_{4}z^{2}_{1}) = 0, d_{3}(U^{n}_{1}z^{2}_{1}z_{3}z_{5}) = 0, d_{3}(U^{n}_{1}z_{1}z^{2}_{2}z_{5}) = 0, d_{3}(U^{n}_{1}z^{4}_{1}z_{2}z_{4}) = 0,\\
&d_{3}(U^{n}_{1}y_{7}z_{3}) = 0, d_{3}(U^{n}_{1}z^{5}_{1}z_{5}) =0, d_{3}(U^{n}_{1}z^{3}_{2}z_{4}) =0,d_{3}(U^{n}_{1}z_{1}z^{3}_{3}) =0,\\
&d_{3}(U^{n}_{1}z^{2}_{2}z^{2}_{3}) =0, d_{3}(U^{n}_{1}z^{4}_{1}z^{2}_{3})=0, d_{3}(U^{n}_{1}z^{2}_{1}z^{4}_{2})=0, d_{3}(U^{n}_{1}z^{7}_{1}z_{3}) =0,\\
&d_{3}(U^{n}_{1}z^{6}_{1}z^{2}_{2})=0, d_{3}(U^{n}_{1}y_{9}z_{1})=0, d_{3}(U^{n}_{1}z_{6}z_{4}) =0, d_{3}(U^{n}_{1}y^{2}_{4}z^{2}_{1}) =0,\\
&d_{3}(U^{n}_{1}z^{10}_{1}) =0, d_{3}(U^{n}_{1}y_{7}z^{3}_{1}) = 0, d_{3}(U^{n}_{1}z_{6}z_{2}z^{2}_{1}) = 0.\\
11) \ &d_{3}(U^{n}_{1}z^{3}_{1}z^{2}_{2}y_{4})=d_{3}(U^{n}_{1}z^{4}_{1}z_{3}y_{4})=U^{n+2}_{1}z^{2}_{1}z^{2}_{2}(U_{1}y_{4}+U_{2}z_{3}),d_{3}(U^{n}_{1}z^{5}_{1}y_{6})=\\
&= U^{n+1}_{1}z^{4}_{1}(U_{1}y_{6} + U_{2}z_{5}), d_{3}(U^{n}_{1}z^{7}_{1}y_{4}) = U^{n+1}_{1}z^{6}_{1}(U_{1}y_{4} + U_{2}z_{3}),d_{3}(U^{n}_{1}z_{1}\times\\
&\times z^{2}_{3}y_{4})=d_{3}(U^{n}_{1}z^{2}_{2}z_{3}y_{4})=U^{n+1}_{1}z^{2}_{3}(U_{1}y_{4}+U_{2}z_{3}), d_{3}(U^{n}_{1}y_{10}z_{1})=U^{n+2}_{1}\times\\
&\times (U_{1}y_{10}+U_{3}z_{7}), d_{3}(U^{n}_{1}z^{2}_{1}z_{3}y_{6}) =d_{3}(U^{n}_{1}z_{1}z^{2}_{2}y_{6}) = U^{n+2}_{1}z_{1}z_{3}(U_{1}y_{6} +\\
&+ U_{2}z_{5}), d_{3}(U^{n}_{1}z^{2}_{1}z_{9}) = d_{3}(U^{n}_{1}z_{1}z_{2}z_{8}) = d_{3}(U^{n}_{1}z^{2}_{1}y_{9}) = d_{3}(U^{n}_{1}z_{1}z_{4}z_{6}) =\\
&= U^{n+3}_{1}z_{1}z_{9}, d_{3}(U^{n}_{1}z^{2}_{1}y_{4}z_{5}) = U^{n+2}_{1}z_{1}z_{5}(U_{1}y_{4} + U_{2}z_{3}), d_{3}(U^{n}_{1}z_{7}z^{4}_{1}) =\\
&= d_{3}(U^{n}_{1}z^{2}_{4}z^{3}_{1}) = U^{n+3}_{1}z^{3}_{1}z_{7}, d_{3}(U^{n}_{1}z^{11}_{1}) = U^{n+2}_{1}z^{10}_{1}, d_{3}(U^{n}_{1}z^{2}_{1}z_{2}z_{3}y_{4}) =\\
&= d_{3}(U^{n}_{1}z_{1}z^{3}_{2}y_{4}) = d_{3}(U^{n}_{1}z_{1}z^{2}_{2}z^{2}_{3}) = d_{3}(U^{n}_{1}z^{2}_{1}z^{3}_{3})= d_{3}(U^{n}_{1}z^{4}_{2}z_{3})= U^{n+3}_{1}\times\\
&\times z^{2}_{2}z^{2}_{3}, d_{3}(U^{n}_{1}y^{2}_{4}z^{3}_{1})= U^{n+3}_{1}z^{2}_{1}y^{2}_{4}, d_{3}(U^{n}_{1}y_{8}z_{2}z_{1}) = d_{3}(U^{n}_{1}y_{6}z_{4}z_{1}) = d_{3}(U^{n}_{1}\!\times\\
&\times z_{7}z_{3}z_{1}) = d_{3}(U^{n}_{1}z^{2}_{5}z_{1}) = d_{3}(U^{n}_{1}z_{5}z_{4}z_{2})= d_{3}(U^{n}_{1}z_{7}z^{2}_{2})= d_{3}(U^{n}_{1}z^{2}_{4}z_{3}) = \\
&= U^{n+3}_{1}z_{3}z_{7},d_{3}(U^{n}_{1}z_{1}z^{3}_{2}z_{4})= d_{3}(U^{n}_{1}z^{3}_{1}z_{4}y_{4})= d_{3}(U^{n}_{1}z^{2}_{1}z_{2}z_{3}z_{4}) = d_{3}(U^{n}_{1}\times\\
&\times z^{2}_{1}z^{2}_{2}z_{5})= d_{3}(U^{n}_{1}z^{3}_{1}z_{2}y_{6}) = d_{3}(U^{n}_{1}z^{3}_{1}z_{3}z_{5}) = U^{n+3}_{1}z^{3}_{2}z_{4}, d_{3}(U^{n}_{1}z^{5}_{1}z_{2}z_{4}) =\\
&= d_{3}(U^{n}_{1}z^{6}_{1}z_{5}) = U^{n+3}_{1}z^{5}_{1}z_{5}, d_{3}(U^{n}_{1}y^{*}_{10}z_{1}) = U^{n+2}_{1}(U_{1}y^{*}_{10} + U_{2}(z_{9} + y_{9}) +\\
&+ U_{3}z_{7}), d_{3}(U^{n}_{1}z^{5}_{1}z^{2}_{3}) = d_{3}(U^{n}_{1}z^{3}_{1}z^{4}_{2}) = d_{3}(U^{n}_{1}z^{4}_{1}z^{2}_{2}z_{3}) = d_{3}(U^{n}_{1}z^{5}_{1}z_{2}y_{4}) =\\
&= U^{n+3}_{1}z^{4}_{1}z^{2}_{3}, d_{3}(U^{n}_{1}z^{8}_{1}z_{3}) = d_{3}(U^{n}_{1}z^{7}_{1}z^{2}_{2}) = U^{n+1}_{1}U^{2}_{2}z^{6}_{1}z^{2}_{2}, d_{3}(U^{n}_{1}z_{6}y_{4}z_{1}) =\\
&= d_{3}(U^{n}_{1}z_{6}z_{3}z_{2}) = d_{3}(U^{n}_{1}y_{7}z^{2}_{2}) = d_{3}(U^{n}_{1}y_{7}z_{3}z_{1})= U^{n+3}_{1}z_{3}y_{7}, d_{3}(U^{n}_{1}y_{7}z^{4}_{1})\! =\\
&= d_{3}(U^{n}_{1}z_{6}z_{2}z^{3}_{1}) = U^{n+3}_{1}z_{6}z_{2}z^{2}_{1}, d_{3}(U^{n}_{1}y_{4}y_{6}z_{1}) = U^{n+2}_{1}[U_{1}y_{4}y_{6}+U_{2}(z_{3}y_{6}+\\
&+ z_{5}y_{4})], d_{3}(U^{n}_{1}z^{2}_{3}z_{5}) = d_{3}(U^{n}_{1}z_{2}z_{3}y_{6}) = d_{3}(U^{n}_{1}z_{2}z_{5}y_{4}) = d_{3}(U^{n}_{1}z_{3}z_{4}y_{4}) =\\
&= U^{n+1}_{1}U^{2}_{2}z_{3}z_{5}, d_{3}(U^{n}_{1}y^{2}_{4}z_{3}) = U^{n+1}_{1}U^{2}_{2}y^{2}_{4}, d_{3}(U^{n}_{1}y_{11})= U^{n+1}_{1}U_{2}\Phi _{5}, \\
&d_{3}(U^{n}_{1}z_{7}y_{4})=U^{n}_{1}U^{2}_{3}(U_{1}y_{4}+U_{2}z_{3}), d_{3}(U^{n}_{1}y_{7}y_{4})=U^{n}_{1}U_{2}\Phi _{3}(U_{1}y_{4}+U_{2}z_{3}),
\end{align*}
Table 17. (continuation).
\begin{align*}
&d_{3}(U^{n}_{1}z_{3}y_{8})=U^{n}_{1}U^{2}_{2}(U_{1}y_{8}+U_{2}z_{7}),d_{3}(U^{n}_{1}y_{6}z_{5})=U^{n}_{1}U_{2}U_{3}(U_{1}y_{6}+U_{2}z_{5}),\\
&d_{3}(U^{n}_{1}z_{11})=U^{n+1}_{1}U_{3}U_{4}, d_{3}(U^{n}_{1}\Phi _{3}y_{6}) = U^{n}_{1}U^{2}_{2}U_{3}\Phi _{3},d_{3}(U^{n}_{1}U_{2}y_{6}y_{4})=U^{n}_{1}\times\\
&=\times U^{3}_{2}(U_{2}y_{6}+U_{3}y_{4}),d_{3}(U^{n}_{1}U_{3}y_{8})=d_{3}(U^{n}_{1}U_{2}y_{10})=U^{n}_{1}U_{2}U^{3}_{3},d_{3}(U^{n}_{1}U_{2}\times\\
&\times\! y^{*}_{10})= U^{n}_{1}U_{2}(U^{2}_{2}U_{4}+U_{2}U_{3}\Phi _{3}+U^{3}_{3}), d_{3}(U^{n}_{1}U_{4}y_{4}) = U^{n}_{1}U^{3}_{2}U_{4},\! d_{3}(U^{n}_{1}\Phi _{3}\!\times\\
&\times z_{6}) = U^{n+3}_{1}\Phi ^{2}_{3},d_{3}(U^{n}_{1}z^{3}_{1}y_{8})=U^{n}_{1}z^{2}_{1}(U_{1}y_{8}+U_{2}z_{7}),\!d_{3}(U^{n}_{1}U_{2}z^{2}_{3}y_{4})= U^{n}_{1}\!\times\\
&\times U^{4}_{2}z^{2}_{3},d_{3}(U^{n}_{1}U^{2}_{2}z_{3}y_{6})=U^{n}_{1}U^{4}_{2}(U_{1}y_{6}+U_{2}z_{5}), d_{3}(U^{n}_{1}U^{2}_{2}z_{9})= U^{n+1}_{1}U^{3}_{2}U_{4},\\
& d_{3}(U^{n}_{1}U^{2}_{2}z^{3}_{3})= U^{n+1}_{1}U^{4}_{2}z^{2}_{3}, d_{3}(U^{n}_{1}U_{2}U_{3}z_{7})= U^{n+1}_{1}U_{2}U^{3}_{3},d_{3}(U^{n}_{1}U^{2}_{2}z_{5}y_{4}) =\\
&=U^{n}_{1}U^{3}_{2}U_{3}(U_{1}y_{4}+U_{2}z_{3}), d_{3}(U^{n}_{1}U^{2}_{2}y_{9})= U^{n+1}_{1}U^{2}_{2}U_{3}\Phi _{3},d_{3}(U^{n}_{1}U^{2}_{2}U_{3}y_{6})=\\
&=d_{3}(U^{n}_{1}U^{3}_{2}y_{8})=d_{3}(U^{n}_{1}U_{2}U^{2}_{3}y_{4})=U^{n}_{1}U^{4}_{2}U^{2}_{3},d_{3}(U^{n}_{1}U^{2}_{2}\Phi _{3}y_{4})=U^{n}_{1}U^{5}_{2}\Phi_{3},\\
&d_{3}(U^{n}_{1}z^{3}_{1}z^{2}_{2}z_{4}) = 0,d_{3}(U^{n}_{1}U^{4}_{2}z_{3}y_{4})=U^{n}_{1}U^{6}_{2}(U_{1}y_{4}+U_{2}z_{3}),d_{3}(U^{n}_{1}U^{4}_{2}z_{7})=\\
&= U^{n+1}_{1}U^{4}_{2}U^{2}_{3}, d_{3}(U^{n}_{1}z_{1}z^{2}_{2}z_{6}) = 0, d_{3}(U^{n}_{1}z^{3}_{1}z_{2}z^{2}_{3}) = 0,\! d_{3}(U^{n}_{1}U^{4}_{2}y_{7})=U^{n+1}_{1}\!\times\\
&\times U^{5}_{2}\Phi _{3},d_{3}(U^{n}_{1}U^{5}_{2}y_{6}) = d_{3}(U^{n}_{1}U^{4}_{2}U_{3}y_{4})= U^{n}_{1}U^{7}_{2}U_{3}, d_{3}(U^{n}_{1}U^{6}_{2}z_{5}) = U^{n+1}_{1}\times\\
&\times U^{7}_{2}U_{3}, d_{3}(U^{n}_{1}z^{7}_{1}z_{4})=0, d_{3}(U^{n}_{1}z^{3}_{1}z_{8})=0,d_{3}(U^{n}_{1}z_{6}z^{5}_{1})= 0, d_{3}(U^{n}_{1}z^{5}_{1}z^{3}_{2}) = \\
&= d_{3}(U^{n}_{1}z^{2}_{1}z^{3}_{2}z_{3})= 0, d_{3}(U^{n}_{1}z^{4}_{1}z_{3}z_{4})= 0, d_{3}(U^{n}_{1}z^{4}_{1}z_{2}z_{5})= 0, d_{3}(U^{n}_{1}z_{1}z^{5}_{2})= \\
&= d_{3}(U^{n}_{1}z^{2}_{1}z_{2}y_{7})= d_{3}(U^{n}_{1}z^{2}_{1}z_{3}z_{6})=0, d_{3}(U^{n}_{1}z^{9}_{1}z_{2}) = 0, d_{3}(U^{n}_{1}z^{6}_{1}z_{2}z_{3}) = 0,\\
&d_{3}(U^{n}_{1}z_{1}z_{2}z_{3}z_{5}) = 0, d_{3}(U^{n}_{1}z_{1}z^{2}_{3}z_{4}) = 0, d_{3}(U^{n}_{1}z^{3}_{2}z_{5})= 0, d_{3}(U^{n}_{1}z^{2}_{2}z_{3}z_{4})= 0,\\
&d_{3}(U^{n}_{1}z_{1}z_{2}y^{2}_{4})= 0, d_{3}(U^{n}_{1}z_{1}z_{2}z^{2}_{4})= 0, d_{3}(U^{n}_{1}z_{4}z_{7})= 0, d_{3}(U^{n}_{1}z_{2}y_{9}) = 0,\\
&d_{3}(U^{n}_{1}z_{2}z^{3}_{3})=d_{3}(U^{n}_{1}z_{2}z_{9})=d_{3}(U^{n}_{1}z_{3}z_{8})=d_{3}(U^{n}_{1}z_{4}y_{7})=d_{3}(U^{n}_{1}z_{5}z_{6})=0,\\
&d_{3}(U^{n}_{1}z^{2}_{1}z_{2}z_{7})= 0, d_{3}(U^{n}_{1}z^{2}_{1}z_{4}z_{5})= 0, d_{3}(U^{n}_{1}z_{10}z_{1}) = 0, d_{3}(U^{n}_{1}U_{3}y^{2}_{2}) = 0,\\
&d_{3}(U^{n}_{1}U^{3}_{2}y^{2}_{4})=0, d_{3}(U^{n}_{1}U_{2}z^{2}_{5}) = 0, d_{3}(U^{n}_{1}\Phi _{3}z^{2}_{3}) = 0,d_{3}(U^{n}_{1}U^{3}_{2}z_{3}z_{5})=0,\\
&d_{3}(U^{n}_{1}\Phi _{6}) = 0, d_{3}(U^{n}_{1}U^{2}_{2}\Phi _{5})=0, d_{3}(U^{n}_{1}U_{2}\Phi ^{2}_{3})=0,d_{3}(U^{n}_{1}U_{2}U_{3}U_{4})= 0,\\
&d_{3}(U^{n}_{1}U^{2}_{3}\Phi _{3}) = 0, d_{3}(U^{n}_{1}U^{5}_{2}z^{2}_{3}) = 0, d_{3}(U^{n}_{1}U^{4}_{2}U_{4}) = 0, d_{3}(U^{n}_{1}U^{3}_{2}U_{3}\Phi _{3}) = 0,\\
&d_{3}(U^{n}_{1}U^{2}_{2}U^{3}_{3}) = 0.\\
12) \ &d_{3}(U^{n}_{1}z^{4}_{1}y_{8})=d_{3}(U^{n}_{1}z^{3}_{1}z_{2}z_{7})=d_{3}(U^{n}_{1}z^{3}_{1}z_{4}z_{5})=d_{3}(U^{n}_{1}z^{2}_{1}z_{2}z^{2}_{4})=U^{n+3}_{1}z_{1}\times\\
&\times\! z_{2}z^{2}_{4},d_{3}(U^{n}_{1}z_{10}z^{2}_{1}) = U^{n+3}_{1}z_{10}z_{1}, d_{3}(U^{n}_{1}z^{2}_{1}z_{2}y_{8}) = U^{n+2}_{1}z_{1}z_{2}(U_{1}y_{8}+U_{2}z_{7})\\
&d_{3}(U^{n}_{1}z^{2}_{1}z_{4}y_{6})=U^{n+2}_{1}z_{1}z_{4}(U_{1}y_{6}+U_{2}z_{5}),d_{3}(U^{n}_{1}z^{4}_{1}z_{4}y_{4})\!=\!U^{n+2}_{1}z^{3}_{1}z_{4}(U_{1}y_{4}+\\
&+ U_{2}z_{3}), d_{3}(U^{n}_{1}z^{4}_{1}z_{2}y_{6}) = U^{n+2}_{1}z^{3}_{1}z_{2}(U_{1}y_{6}+ U_{2}z_{5}), d_{3}(U^{n}_{1}z^{4}_{1}z_{8}) = U^{n+3}_{1}\times\\
&\times z^{4}_{1}z_{8},\, d_{3}(U^{n}_{1}z^{3}_{1}z_{2}z_{3}y_{4})\,=\, d_{3}(U^{n}_{1}z^{2}_{1}z^{3}_{2}y_{4})\,=\, U^{n+2}_{1}z_{1}z^{3}_{2}(U_{1}y_{4}\,+\,U_{2}z_{3})\,,\\
&d_{3}(U^{n}_{1}z^{6}_{1}z_{2}y_{4}) = U^{n+2}_{1}z^{5}_{1}z_{2}(U_{1}y_{4} + U_{2}z_{3}),d_{3}(U^{n}_{1}z^{6}_{1}z_{6}) = U^{n+3}_{1}z^{5}_{1}z_{6}, d_{3}(U^{n}_{1}\!\times\\
&\times\! z^{2}_{1}y_{4}z_{6})\! = \!U^{n+2}_{1}z^{2}_{1}z_{6}(U_{1}y_{4}\! +\! U_{2}z_{3}),d_{3}(U^{n}_{1}z^{2}_{1}y_{4}y_{6})\! =\! U^{n+2}_{1}z^{2}_{2}(U_{2}y_{6} + U_{3}y_{4}),\\
& d_{3}(U^{n}_{1}z_{1}z_{2}z_{3}y_{6})\, =\, d_{3}(U^{n}_{1}z^{3}_{2}y_{6})\, =\, U^{n+2}_{1}z_{2}z_{3}(U_{1}y_{6}\,+\,U_{2}z_{5}),\, d_{3}(U^{n}_{1}z_{12})\, =\\
&= U^{n+2}_{1}\Phi _{6}, d_{3}(U^{n}_{1}z_{1}z_{2}z_{5}y_{4}) = d_{3}(U^{n}_{1}z_{1}z_{3}z_{4}y_{4}) = d_{3}(U^{n}_{1}z^{2}_{2}z_{4}y_{4}) =U^{n+2}_{1}\times\\
&\times z_{2}z_{5}(U_{1}y_{4}\, +\, U_{2}z_{3}),\, d_{3}(U^{n}_{1}z^{3}_{1}z_{2}y_{7})\, =\, d_{3}(U^{n}_{1}z^{3}_{1}z_{3}z_{6}) = d_{3}(U^{n}_{1}z^{2}_{1}z^{2}_{2}z_{6}) =\\
&= U^{n+3}_{1}z^{2}_{1}z_{2}y_{7}, d_{3}(U^{n}_{1}z^{4}_{1}z^{2}_{2}y_{4})= d_{3}(U^{n}_{1}z^{5}_{1}z_{3}y_{4}) = U^{n+3}_{1}z_{1}z^{5}_{2}, d_{3}(U^{n}_{1}z^{2}_{1}z^{5}_{2})=
\end{align*}
Table 17. (continuation).
\begin{align*}
& = d_{3}(U^{n}_{1}z^{3}_{1}z^{3}_{2}z_{3}) = d_{3}(U^{n}_{1}z^{4}_{1}z_{2}z^{2}_{3}) = U^{n+3}_{1}z^{3}_{1}z_{2}z^{2}_{3}, d_{3}(U^{n}_{1}z^{4}_{1}z^{2}_{2}z_{4}) = d_{3}(U^{n}_{1}z^{5}_{1}z_{3}\times\\
&\times z_{4}) = d_{3}(U^{n}_{1}z^{5}_{1}z_{2}z_{5}) = d_{3}(U^{n}_{1}z^{6}_{1}y_{6}) = U^{n+3}_{1}z^{3}_{1}z^{2}_{2}z_{4},\, d_{3}(U^{n}_{1}z^{8}_{1}z_{4})\, = U^{n+3}_{1}z^{7}_{1}z_{4},\\
&d_{3}(U^{n}_{1}z_{1}z_{4}z_{7}) = d_{3}(U^{n}_{1}z^{2}_{1}y_{10}) = d_{3}(U^{n}_{1}z^{3}_{4}) = U^{n+3}_{1}z_{4}z_{7},\, d_{3}(U^{n}_{1}z^{10}_{1}z_{2})\, = U^{n+3}_{1}\times\\
&\times\! z^{9}_{1}z_{2}, d_{3}(U^{n}_{1}z^{7}_{1}z_{2}z_{3}) = d_{3}(U^{n}_{1}z^{6}_{1}z^{3}_{2}) = U^{n+3}_{1}z^{5}_{1}z^{3}_{2},d_{3}(U^{n}_{1}z^{4}_{2}z_{4}) =d_{3}(U^{n}_{1}z^{2}_{1}z_{2}z_{4}\!\times\\
&\times y_{4})\,= d_{3}(U^{n}_{1}z^{2}_{1}z_{2}z_{3}z_{5})\, = d_{3}(U^{n}_{1}z^{3}_{1}z_{5}y_{4})\, = d_{3}(U^{n}_{1}z^{3}_{1}z_{3}y_{6})\,= d_{3}(U^{n}_{1}z^{2}_{1}z^{2}_{2}y_{6})=\\
&=  d_{3}(U^{n}_{1}z^{2}_{1}z^{3}_{2}z_{4}) = d_{3}(U^{n}_{1}z_{1}z^{3}_{2}z_{5})= d_{3}(U^{n}_{1}z_{1}z^{2}_{2}z_{3}z_{4}) = U^{n+3}_{1}z_{5}z^{3}_{2},d_{3}(U^{n}_{1}z^{2}_{1}z_{2}\times\\
&\times y^{2}_{4}) = U^{n+3}_{1}z_{1}z_{2}y^{2}_{4}, d_{3}(U^{n}_{1}z_{1}z_{2}z^{3}_{3}) = d_{3}(U^{n}_{1}z^{4}_{2}y_{4}) = d_{3}(U^{n}_{1}z^{2}_{1}z^{2}_{3}y_{4}) = d_{3}(U^{n}_{1}\times\\
&\times \, z^{3}_{2}z^{2}_{3})\, =\,\, d_{3}(U^{n}_{1}z_{1}z^{2}_{2}z_{3}y_{4})\, =\,\, U^{n+3}_{1}z_{2}z^{3}_{3},\,\, d_{3}(U^{n}_{1}z_{1}z_{3}z_{8})\, = \,\,d_{3}(U^{n}_{1}z_{1}z_{2}z_{9})\, =\\
&= d_{3}(U^{n}_{1}z^{2}_{2}z_{8}) = U^{n+3}_{1}z_{3}z_{8},\,\,d_{3}(U^{n}_{1}z_{2}z_{4}z_{6})= d_{3}(U^{n}_{1}z_{1}z_{5}z_{6}) = d_{3}(U^{n}_{1}z_{1}z_{4}y_{7})=\\
&= d_{3}(U^{n}_{1}z_{1}z_{2}y_{9})\! = U^{n+3}_{1}z_{5}z_{6}, d_{3}(U^{n}_{1}z^{2}_{1}y^{*}_{10}) = U^{n+3}_{1}(z_{4}z_{7}\! +\! z_{4}y_{7}\! +\! z_{1}z_{9}),d_{3}(U^{n}_{1}\!\times\\
&\times y_{4}z_{8})=U^{n+1}_{1}U_{4}(U_{1}y_{4}+U_{2}z_{3}),d_{3}(U^{n}_{1}z_{2}y_{4}y_{6})=U^{n+1}_{1}U_{3}(U_{1}y_{4}y_{6}+U_{2}z_{3}y_{6}+\\
&+U_{2}z_{5}y_{4}),d_{3}(U^{n}_{1}z_{2}y_{10})=U^{n+1}_{1}U_{2}(U_{1}y_{10}+U_{3}z_{7}),d_{3}(U^{n}_{1}z_{4}y_{8})=U^{n+1}_{1}U_{3}(U_{1}\!\times\\
&\times y_{8}+U_{2}z_{7}),d_{3}(U^{n}_{1}z_{2}z^{2}_{3}y_{4})=U^{n+1}_{1}U_{2}z^{2}_{3}(U_{1}y_{4}+U_{2}z_{3}),d_{3}(U^{n}_{1}z_{6}y_{6})=U^{n+1}_{1}\times\\
&\times \Phi _{3}(U_{1}y_{6}\,+\,U_{2}z_{5}),\,d_{3}(U^{n}_{1}z_{2}y^{*}_{10})\,=\, U^{n+1}_{1}U_{2}(U_{1}y^{*}_{10}\,+\,U_{2}z_{9}\,+\,U_{2}y_{9}\,+\,U_{3}z_{7}),\\
&d_{3}(U^{n}_{1}z_{2}z_{3}y_{7})\,=d_{3}(U^{n}_{1}z^{2}_{3}z_{6})\,=d_{3}(U^{n}_{1}z_{1}y_{4}y_{7})\,=d_{3}(U^{n}_{1}z_{2}y_{4}z_{6})\,=U^{n+2}_{1}U_{2}z_{3}y_{7},\\
&d_{3}(U^{n}_{1}z_{4}y^{2}_{4})\,=\, U^{n+2}_{1}U_{4}y^{2}_{4},\, d_{3}(U^{n}_{1}z_{3}z_{4}z_{5})\, =\, d_{3}(U^{n}_{1}z_{1}z_{7}y_{4})\, =\, d_{3}(U^{n}_{1}z_{2}z_{3}z_{7})\, =\\
&= d_{3}(U^{n}_{1}z^{2}_{4}y_{4}) = d_{3}(U^{n}_{1}z_{2}z^{2}_{5})= d_{3}(U^{n}_{1}z_{1}z_{5}y_{6})= d_{3}(U^{n}_{1}z_{2}z_{4}y_{6})\! = d_{3}(U^{n}_{1}z^{2}_{2}y_{8}) =\\
&=\! d_{3}(U^{n}_{1}z_{1}z_{3}y_{8})\! = U^{n+2}_{1}U_{2}z^{2}_{5}, d_{3}(U^{n}_{1}y_{4}y_{8})\!= U^{n}_{1}U_{2}(U^{2}_{2}y_{8}\! + U^{2}_{3}y_{4}), d_{3}(U^{n}_{1}z^{2}_{3}y_{6})\! =\\
&= d_{3}(U^{n}_{1}z_{3}z_{5}y_{4})\, = U^{n}_{1}U^{3}_{2}z_{3}z_{5},\,d_{3}(U^{n}_{1}y^{3}_{4})\,= U^{n}_{1}U^{3}_{2}y^{2}_{4}, d_{3}(U^{n}_{1}y_{12})= U^{n}_{1}(U^{2}_{3}\Phi _{3}+\\
&+U_{2}U_{3}U_{4}), d_{3}(U^{n}_{1}U_{3}y_{9})\!=U^{n+1}_{1}U^{2}_{3}\Phi _{3}, d_{3}(U^{n}_{1}U_{2}z_{5}y_{6})\!= U^{n}_{1}U^{2}_{2}U_{3}(U_{1}y_{6}\!+U_{2}z_{5}),\\
& d_{3}(U^{n}_{1}U_{2}y_{4}z_{7})= U^{n}_{1}U_{2}U^{2}_{3}(U_{1}y_{4} + U_{2}z_{3}), d_{3}(U^{n}_{1}U_{2}z_{3}y_{8}) = U^{n}_{1}U^{3}_{2}(U_{1}y_{8} + U_{2}z_{7}),\\
&d_{3}(U^{n}_{1}U_{2}y_{4}y_{7})= U^{n}_{1}U_{2}\Phi _{3}(U_{1}y_{4}+U_{2}z_{3}), d_{3}(U^{n}_{1}U_{2}z_{11}) = U^{n+1}_{1}U_{2}U_{3}U_{4},d_{3}(U^{n}_{1}\!\times\\
&\times U_{2}y_{11})=U^{n+1}_{1}U^{2}_{2}\Phi _{5}, d_{3}(U^{n}_{1}\Phi _{3}y_{7})\,= U^{n+1}_{1}U_{2}\Phi ^{2}_{3},\, d_{3}(U^{n}_{1}U^{2}_{2}y_{4}z^{2}_{3})\,= U^{n}_{1}U^{5}_{2}z^{2}_{3},\\
& d_{3}(U^{n}_{1}U^{2}_{2}y_{4}y_{6}) = U^{n}_{1}U^{4}_{2}(U_{2}y_{6} + U_{3}y_{4}), d_{3}(U^{n}_{1}U^{2}_{2}y^{*}_{10})\! = U^{n}_{1}U^{2}_{2}(U^{2}_{2}U_{4}\! + U_{2}U_{3}\Phi _{3}\! +\\
& + U^{3}_{3}),\, d_{3}(U^{n}_{1}U_{2}U_{3}y_{8}) = d_{3}(U^{n}_{1}U^{2}_{2}y_{10})\, = d_{3}(U^{n}_{1}U^{2}_{3}y_{6})\, = U^{n}_{1}U^{2}_{2}U^{3}_{3},\, d_{3}(U^{n}_{1}U_{2}\times\\
&\times\! \Phi _{3}y_{6})\! = d_{3}(U^{n}_{1}U_{3}\Phi _{3}y_{4}) = U^{n}_{1}U^{3}_{2}U_{3}\Phi _{3}, d_{3}(U^{n}_{1}U_{2}U_{4}y_{4}) = U^{n}_{1}U^{4}_{2}U_{4}, d_{3}(U^{n}_{1}U^{3}_{2}\!\times\\
&\times z^{3}_{3}) = U^{n+1}_{1}U^{5}_{2}z^{2}_{3},d_{3}(U^{n}_{1}U^{2}_{2}U_{4}z_{3}) = U^{n+1}_{1}U^{4}_{2}U_{4}, d_{3}(U^{n}_{1}U_{2}U_{3}\Phi _{3}z_{3}) = d_{3}(U^{n}_{1}\times\\
&\times\! U^{3}_{2}y_{9})\! = U^{n+1}_{1}U^{3}_{2}U_{3}\Phi _{3}, d_{3}(U^{n}_{1}U^{3}_{3}z_{3}) = d_{3}(U^{n}_{1}U^{2}_{2}U_{3}z_{7}) = U^{n+1}_{1}U^{2}_{2}U^{3}_{3},d_{3}(U^{n}_{1}\!\times\\
&\times z^{2}_{1}z_{2}z_{8})\, =\, 0,\, d_{3}(U^{n}_{1}z^{2}_{1}z_{3}z_{7})\, =\, 0,\, d_{3}(U^{n}_{1}z_{1}z_{2}z_{4}z_{5})\, =\, 0,\,\,d_{3}(U^{n}_{1}z_{1}z^{2}_{2}z_{7})\,=\, 0,\\
&d_{3}(U^{n}_{1}z^{4}_{1}z_{3}z_{5}) = 0,\, d_{3}(U^{n}_{1}z^{3}_{1}z^{2}_{2}z_{5}) = 0,\, d_{3}(U^{n}_{1}z^{3}_{1}z_{2}z_{3}z_{4}) = 0,\, d_{3}(U^{n}_{1}z_{1}z_{3}z^{2}_{4}) = 0,\\
& d_{3}(U^{n}_{1}z^{3}_{1}z_{9})= 0, d_{3}(U^{n}_{1}z^{2}_{1}z^{2}_{5})=0,d_{3}(U^{n}_{1}z^{5}_{1}z_{7}) = 0, d_{3}(U^{n}_{1}z^{4}_{1}z^{2}_{4})= 0,  d_{3}(U^{n}_{1}\!\times\\
&\times z^{7}_{1}z_{5})= 0, d_{3}(U^{n}_{1}z^{6}_{1}z^{2}_{3})=0, d_{3}(U^{n}_{1}z^{2}_{2}z^{2}_{4})= 0, d_{3}(U^{n}_{1}z^{2}_{1}z^{3}_{2}z_{4}) = 0,\, d_{3}(U^{n}_{1}\times\\
&\times z^{6}_{1}z_{2}z_{4})\, = 0,\, d_{3}(U^{n}_{1}z^{5}_{1}z^{2}_{2}z_{3})\, = 0,\, d_{3}(U^{n}_{1}z^{2}_{1}z^{2}_{2}z^{2}_{3})\, =\, 0,\, d_{3}(U^{n}_{1}z_{1}z^{4}_{2}z_{3})\, =\, 0,\\
&d_{3}(U^{n}_{1}z^{4}_{1}z^{4}_{2}) = 0, d_{3}(U^{n}_{1}z^{6}_{2}) = 0, d_{3}(U^{n}_{1}z^{3}_{1}z^{3}_{3}) = 0,d_{3}(U^{n}_{1}z^{9}_{1}z_{3}) = 0, d_{3}(U^{n}_{1}z^{2}_{1}\!\times
\end{align*}
Table 17. (continuation).
\begin{align*}
&\times z_{4}z_{6})\, = 0,\, d_{3}(U^{n}_{1}z^{4}_{1}z_{2}z_{6})\, =\, 0,\, d_{3}(U^{n}_{1}z_{1}z_{2}z_{3}z_{6})\, =\, 0,\, d_{3}(U^{n}_{1}z_{1}z^{2}_{2}y_{7})\, =\, 0,\\
& d_{3}(U^{n}_{1}z^{2}_{1}z_{3}y_{7}) = 0, d_{3}(U^{n}_{1}z_{1}z^{2}_{3}z_{5}) = 0,d_{3}(U^{n}_{1}z^{2}_{2}z_{3}z_{5}) = 0, d_{3}(U^{n}_{1}z_{2}z^{2}_{3}z_{4}) = 0,\\
&d_{3}(U^{n}_{1}z^{8}_{1}z^{2}_{2})= 0, d_{3}(U^{n}_{1}z^{12}_{1})= 0, d_{3}(U^{n}_{1}z^{3}_{1}y_{9})= 0, d_{3}(U^{n}_{1}z^{4}_{1}y^{2}_{4})= 0, d_{3}(U^{n}_{1}\times\\
&\times z_{1}z_{3}y^{2}_{4}) = 0,d_{3}(U^{n}_{1}z^{5}_{1}y_{7}) = 0, d_{3}(U^{n}_{1}z^{3}_{2}z_{6}) = 0, d_{3}(U^{n}_{1}z^{2}_{2}y^{2}_{4}) = 0, d_{3}(U^{n}_{1}\times\\
&\times z^{4}_{3})= 0,\, d_{3}(U^{n}_{1}z_{5}y_{7})= 0,\, d_{3}(U^{n}_{1}z^{2}_{6})= 0,\, d_{3}(U^{n}_{1}y^{2}_{6})= 0,\, d_{3}(U^{n}_{1}z_{5}z_{7})=0,\\
& d_{3}(U^{n}_{1}U^{2}_{2}z_{3}y_{7}) = 0, d_{3}(U^{n}_{1}z_{1}y_{11}) = 0, d_{3}(U^{n}_{1}z_{2}z_{10}) = 0,d_{3}(U^{n}_{1}U_{2}U_{3}y^{2}_{4})=0, \\
&d_{3}(U^{n}_{1}z_{4}z_{8}) = 0, d_{3}(U^{n}_{1}z_{3}z_{9})= 0, d_{3}(U^{n}_{1}\Omega _{1})=0,d_{3}(U^{n}_{1}U_{2}\Phi _{6})= 0, d_{3}(U^{n}_{1}\times\\
&\times U_{3}\Phi _{5})= 0, d_{3}(U^{n}_{1}U_{4}\Phi _{3}) = 0, d_{3}(U^{n}_{1}U^{2}_{2}z^{2}_{5})= 0, d_{3}(U^{n}_{1}z_{3}y_{9}) = 0, d_{3}(U^{n}_{1}\times\\
&\times z_{1}z_{11}) = 0.\\
13) \ &d_{3}(U^{n}_{1}z_{13})\, =\, U^{n+2}_{1}\Omega _{1} ,\, d_{3}(U^{n}_{1}y_{13})\, =\, U^{n+1}_{1}\Phi _{3}U_{4} ,\, d_{3}(U^{n}_{1}y^{*}_{13})\, =\, U^{n+1}_{1}U_{2}\Phi _{6}, \\
&d_{3}(U^{n}_{1}z_{1}y_{4}z_{8}) = U^{n+1}_{1}z_{3}z_{9}, d_{3}(U^{n}_{1}z^{5}_{1}y_{8}) = U^{n+2}_{1}z^{4}_{1}(U_{1}y_{8}+ U_{2}z_{7}), d_{3}(U^{n}_{1}z^{3}_{1}\!\times\\
&\times z_{2}y_{8})\, =\, d_{3}(U^{n}_{1}z^{3}_{1}z_{4}y_{6})\, =\, U^{n+3}_{1}z^{2}_{2}z^{2}_{4},\,d_{3}(U^{n}_{1}z^{5}_{1}z_{4}y_{4})\, =\, d_{3}(U^{n}_{1}z^{5}_{1}z_{2}y_{6})\,=\\
& = U^{n+3}_{1}z^{2}_{1}z^{3}_{2}z_{4}, d_{3}(U^{n}_{1}z_{1}z^{6}_{2}) = d_{3}(U^{n}_{1}z^{4}_{1}z^{3}_{3}) = d_{3}(U^{n}_{1}z^{3}_{1}z^{3}_{2}y_{4}) = d_{3}(U^{n}_{1}z^{3}_{1}z^{2}_{2}\times\\
&\times z^{2}_{3}) = d_{3}(U^{n}_{1}z^{4}_{1}z_{2}z_{3}y_{4}) = d_{3}(U^{n}_{1}z^{2}_{1}z^{4}_{2}z_{3}) = U^{n+3}_{1}z^{6}_{2}, d_{3}(U^{n}_{1}z^{2}_{1}z_{2}z_{3}y_{6}) =\\
&= d_{3}(U^{n}_{1}z_{1}z^{3}_{2}y_{6})= d_{3}(U^{n}_{1}z^{2}_{1}z_{2}z_{5}y_{4})= d_{3}(U^{n}_{1}z^{2}_{1}z_{3}z_{4}y_{4})= U^{n+3}_{1}z^{2}_{2}z_{3}z_{5},\\
& d_{3}(U^{n}_{1}z^{7}_{1}y_{6})= U^{n+2}_{1}z^{6}_{1}(U_{1}y_{6}+U_{2}z_{5}), d_{3}(U^{n}_{1}z^{3}_{1}y_{4}z_{6})= U^{n+3}_{1}z^{3}_{2}z_{6},d_{3}(U^{n}_{1}z^{7}_{1}\!\times\\
&\times z_{2}y_{4})= U^{n+3}_{1}z^{4}_{1}z^{4}_{2}, d_{3}(U^{n}_{1}z^{3}_{1}y_{4}y_{6}) = U^{n+2}_{1}z^{2}_{1}(U_{1}y_{6}y_{4}+ U_{2}z_{3}y_{6}+ U_{2}z_{5}y_{4}),\\
&d_{3}(U^{n}_{1}z_{1}z^{2}_{2}z_{4}y_{4}) = U^{n+3}_{1}z_{2}z^{2}_{3}z_{4}, d_{3}(U^{n}_{1}z^{5}_{1}z^{2}_{2}y_{4}) = d_{3}(U^{n}_{1}z^{6}_{1}z_{3}y_{4})= U^{n+2}_{1}z^{4}_{1}\times\\
&\times z^{2}_{2}(U_{1}y_{4}+U_{2}z_{3}),\! d_{3}(U^{n}_{1}z^{4}_{1}z_{3}y_{6}) =\! d_{3}(U^{n}_{1}z^{3}_{1}z^{2}_{2}y_{6}) =\! U^{n+2}_{1}z^{3}_{1}z_{3}(U_{1}y_{6} + U_{2}\!\times\\
&\times z_{5}),\! d_{3}(U^{n}_{1}z^{4}_{1}z_{5}y_{4}) = d_{3}(U^{n}_{1}z^{3}_{1}z_{2}z_{4}y_{4}) = U^{n+2}_{1}z^{3}_{1}z_{5}(U_{1}y_{4}+ U_{2}z_{3}), d_{3}(U^{n}_{1}\!\times\\
&\times z^{3}_{1}y_{10}) = U^{n+2}_{1}z^{2}_{1}(U_{1}y_{10}+ U_{3}z_{7}),\, d_{3}(U^{n}_{1}z^{3}_{1}z^{3}_{2}z_{4})= U^{n+3}_{1}z^{2}_{1}z^{3}_{2}z_{4},\, d_{3}(U^{n}_{1}\times\\
&\times z_{1}z^{4}_{2}y_{4})\,=\, d_{3}(U^{n}_{1}z^{2}_{1}z^{2}_{2}z_{3}y_{4})\,=\, d_{3}(U^{n}_{1}z^{3}_{1}z^{2}_{3}y_{4})\,=\, U^{n+2}_{1}z^{4}_{2}(U_{1}y_{4}\,+\, U_{2}z_{3}),\\
& d_{3}(U^{n}_{1}z^{3}_{1}y^{*}_{10})\, = \,U^{n+2}_{1}z^{2}_{1}(U_{1}y^{*}_{10}+ U_{2}z_{9} + U_{2}y_{9} + U_{3}z_{7}),\, d_{3}(U^{n}_{1}z_{1}z_{2}y_{4}y_{6})\, =\\
&= U^{n+2}_{1}z_{2}z_{3}(U_{2}y_{6}+ U_{3}y_{4}),d_{3}(U^{n}_{1}z_{1}z_{2}z^{2}_{3}y_{4}) = d_{3}(U^{n}_{1}z^{3}_{2}z_{3}y_{4}) = d_{3}(U^{n}_{1}z^{2}_{2}\times\\
&\times z^{3}_{3}) = U^{n+3}_{1}z^{4}_{3},d_{3}(U^{n}_{1}z_{1}z_{6}y_{6}) = U^{n+3}_{1}z_{5}y_{7}, d_{3}(U^{n}_{1}z^{2}_{1}y_{4}y_{7}) = d_{3}(U^{n}_{1}z_{1}z_{2}\times\\
&\times y_{4}z_{6})\,\, =\,\, U^{n+2}_{1}z_{1}y_{7}(U_{1}y_{4}\,+\, U_{2}z_{3}),\,\, d_{3}(U^{n}_{1}z_{1}z^{2}_{4}y_{4})\,\, =\,\, d_{3}(U^{n}_{1}z^{2}_{1}z_{7}y_{4})\, =\\
&= U^{n+2}_{1}z^{2}_{4}(U_{1}y_{4}+ U_{2}z_{3}),\, d_{3}(U^{n}_{1}z_{1}y^{3}_{4}) = U^{n+2}_{1}y^{2}_{4}(U_{1}y_{4} + U_{2}z_{3}),\,d_{3}(U^{n}_{1}z_{1}\times\\
&\times z_{2}y^{*}_{10})\: =\: U^{n+3}_{1}(z_{3}z_{9} + z_{5}z_{7} + z_{5}y_{7}),\: d_{3}(U^{n}_{1}z_{1}z_{2}y_{10})\: =\:d_{3}(U^{n}_{1}z_{1}z_{4}y_{8})=\\
&=U^{n+3}_{1}z_{5}z_{7}, d_{3}(U^{n}_{1}z^{4}_{1}y_{9})=U^{n+3}_{1}z^{2}_{1}z_{4}z_{6},d_{3}(U^{n}_{1}z^{5}_{1}y^{2}_{4})= U^{n+3}_{1}z^{4}_{1}y^{2}_{4}, d_{3}(U^{n}_{1}\times\\
&\times z^{6}_{1}y_{7}) = d_{3}(U^{n}_{1}z^{5}_{1}z_{2}z_{6}) = U^{n+3}_{1}z^{4}_{1}z_{2}z_{6},d_{3}(U^{n}_{1}z^{2}_{1}z_{5}y_{6})\! = d_{3}(U^{n}_{1}z_{1}z_{2}z_{4}y_{6})=\\
& = U^{n+2}_{1}z_{1}z_{5}(U_{1}y_{6} + U_{2}z_{5}),\!d_{3}(U^{n}_{1}z^{2}_{1}z_{3}y_{8})\! = d_{3}(U^{n}_{1}z_{1}z^{2}_{2}y_{8})\! = U^{n+2}_{1}z_{1}z_{3}(U_{1}\!\times\\
&\times y_{8} + U_{2}z_{7}),\:d_{3}(U^{n}_{1}z_{1}y_{4}y_{8})= U^{n+2}_{1}(U_{1}y_{8}y_{4}+U_{2}z_{3}y_{8}+U_{2}z_{7}y_{4}),\: d_{3}(U^{n}_{1}\times\\
&\times z_{1}z_{3}z_{5}y_{4})\:=\: d_{3}(U^{n}_{1}z_{1}z^{2}_{3}y_{6})\: =\: U^{n+2}_{1}z^{2}_{3}(U_{1}y_{6} + U_{2}z_{5}),\: d_{3}(U^{n}_{1}z^{3}_{1}z_{2}z_{8})\: =\\
&= d_{3}(U^{n}_{1}z^{4}_{1}z_{9}) = U^{n+3}_{1}z^{3}_{1}z_{9}, d_{3}(U^{n}_{1}z^{3}_{1}z_{3}z_{7}) = d_{3}(U^{n}_{1}z^{2}_{1}z_{2}z_{4}z_{5}) = d_{3}(U^{n}_{1}z^{2}_{1}\times
\end{align*}
Table 17. (continuation).
\begin{align*}
&\times z^{2}_{2}z_{7})\,\; =\,\; d_{3}(U^{n}_{1}z^{3}_{1}z^{2}_{5})\,\;=\,\; d_{3}(U^{n}_{1}z^{2}_{1}z_{3}z^{2}_{4})\: =\: d_{3}(U^{n}_{1}z_{1}z^{2}_{2}z^{2}_{4})\;=\;U^{n+3}_{1}z^{2}_{1}z_{3}z_{7},\\
& d_{3}(U^{n}_{1}z^{6}_{1}z_{7})=d_{3}(U^{n}_{1}z^{5}_{1}z^{2}_{4})=U^{n+3}_{1}z^{5}_{1}z_{7}, d_{3}(U^{n}_{1}z_{1}y_{12})= U^{n+2}_{1}(U_{1}y_{12}+ U_{2}z_{11}+\\
&+ U_{3}y_{9}),\; d_{3}(U^{n}_{1}z^{5}_{1}z_{3}z_{5}) = d_{3}(U^{n}_{1}z^{4}_{1}z^{2}_{2}z_{5}) = d_{3}(U^{n}_{1}z^{4}_{1}z_{2}z_{3}z_{4}) = d_{3}(U^{n}_{1}z^{3}_{1}z^{3}_{2}z_{4}) =\\
&= U^{n+3}_{1}z^{4}_{1}z_{3}z_{5},\! d_{3}(U^{n}_{1}z^{8}_{1}z_{5})\!= d_{3}(U^{n}_{1}z^{7}_{1}z_{2}z_{4})\! = U^{n+3}_{1}z^{7}_{1}z_{5},\! d_{3}(U^{n}_{1}z^{7}_{1}z^{2}_{3})\! = d_{3}(U^{n}_{1}\!\times\\
&\times z^{6}_{1}z^{2}_{2}z_{3}) = d_{3}(U^{n}_{1}z^{5}_{1}z^{4}_{2})= U^{n+3}_{1}z^{6}_{1}z^{2}_{3}, d_{3}(U^{n}_{1}z^{10}_{1}z_{3})\!= d_{3}(U^{n}_{1}z^{9}_{1}z^{2}_{2})\!= U^{n+3}_{1}z^{9}_{1}z_{3},\\
&d_{3}(U^{n}_{1}z^{13}_{1})= U^{n+3}_{1}z^{12}_{1}, d_{3}(U^{n}_{1}z^{3}_{1}z_{4}z_{6})= U^{n+3}_{1}z^{2}_{1}z_{4}z_{6}, d_{3}(U^{n}_{1}z^{4}_{2}z_{5})\!= d_{3}(U^{n}_{1}z^{3}_{2}z_{3}\!\times\\
&\times z_{4})\, = U^{n+3}_{1}z^{2}_{2}z_{3}z_{5},\, d_{3}(U^{n}_{1}z^{2}_{2}z_{3}y_{6})\, = U^{n+2}_{1}z^{2}_{3}(U_{1}y_{6} +U_{2}z_{5}),\, d_{3}(U^{n}_{1}z_{2}z_{4}z_{7})=\\
&= U^{n+3}_{1}z_{5}z_{7},\; d_{3}(U^{n}_{1}z^{2}_{2}y_{9})\;\,=\; d_{3}(U^{n}_{1}z_{4}y_{7}z_{2})\;\,=\; U^{n+3}_{1}z_{3}y_{9},\;\, d_{3}(U^{n}_{1}z_{2}z_{3}z_{4}y_{4}) =\\
&= d_{3}(U^{n}_{1}z^{2}_{2}z_{5}y_{4})\,=\, U^{n+2}_{1}z_{3}z_{5}(U_{1}y_{6}+ U_{2}z_{5}),\, d_{3}(U^{n}_{1}z_{2}z_{5}y_{6})\, =\, d_{3}(U^{n}_{1}z_{2}z_{7}y_{4}) =\\
&= d_{3}(U^{n}_{1}z_{2}y_{8}z_{3})\, =\, d_{3}(U^{n}_{1}z_{3}z_{4}y_{6})\, =\, U^{n+1}_{1}U^{2}_{2}z^{2}_{5}, d_{3}(U^{n}_{1}z_{2}y_{7}y_{4}) = U^{n+1}_{1}U^{2}_{2}z_{3}y_{7},\\
&d_{3}(U^{n}_{1}z_{2}z_{5}z_{6})= U^{n+3}_{1}z_{5}y_{7}, d_{3}(U^{n}_{1}z^{2}_{2}z_{9})\!= d_{3}(U^{n}_{1}z_{1}z_{3}z_{9})=U^{n+3}_{1}z_{3}z_{9},d_{3}(U^{n}_{1}z_{3}\!\times\\
&\times y_{10})\, =\, U^{n}_{1}U^{2}_{2}(U_{1}y_{10}\,+\,U_{3}z_{7}), d_{3}(U^{n}_{1}z_{3}y^{*}_{10}) = U^{n}_{1}U^{2}_{2}(U_{1}y^{*}_{10}+ U_{2}z_{9}+ U_{2}y_{9}+ \\
&+U_{3}z_{7}),\;d_{3}(U^{n}_{1}z_{3}y_{4}y_{6})\; =\; U^{n}_{1}U^{2}_{2}(U_{1}y_{6}y_{4}\;+ U_{2}z_{5}y_{4}\; + U_{2}y_{6}z_{3}),\; d_{3}(U^{n}_{1}z^{2}_{3}z_{7}) =\\
&= d_{3}(U^{n}_{1}z^{2}_{5}z_{3})=\! U^{n}_{1}U^{2}_{2}z_{3}z_{7},\! d_{3}(U^{n}_{1}z^{3}_{3}y_{4})\!=\! U^{n}_{1}U^{2}_{2}z^{2}_{3}(U_{1}y_{4}+ U_{2}z_{3}),\! d_{3}(U^{n}_{1}y_{7}z^{2}_{3})=\\
&= U^{n}_{1}U^{2}_{2}y_{7}z_{3}, d_{3}(U^{n}_{1}z_{3}z_{6}z_{4})= U^{n}_{1}U^{2}_{2}z_{4}z_{6}, d_{3}(U^{n}_{1}y_{4}z_{9})\! = U^{n}_{1}U_{2}U_{4}(U_{1}y_{4}+ U_{2}z_{3}),\\
& d_{3}(U^{n}_{1}y_{9}y_{4})\! = U^{n}_{1}U_{3}\Phi _{3}(U_{1}y_{4}+U_{2}z_{3}),\! d_{3}(U^{n}_{1}z^{2}_{4}z_{5})\!= U^{n+2}_{1}U_{3}z_{4}z_{5}, d_{3}(U^{n}_{1}y^{2}_{4}z_{5})=\\
&= U^{n+1}_{1}U_{2}U_{3}y^{2}_{4},\; d_{3}(U^{n}_{1}y_{4}z_{4}z_{5})\;=\; U^{n+1}_{1}U^{2}_{2}z^{2}_{5},\; d_{3}(U^{n}_{1}z_{5}y_{8})\;= U^{n}_{1}U_{2}U_{3}(U_{1}y_{8}+\\
&+U_{2}z_{7}),\;d_{3}(U^{n}_{1}y_{6}y_{7})\; = U^{n}_{1}U_{2}\Phi _{3}(U_{1}y_{6}+U_{2}z_{5}),\; d_{3}(U^{n}_{1}y_{6}z_{7})\;=\; U^{n}_{1}U^{2}_{3}(U_{1}y_{6} + \\
&+U_{2}z_{5}),d_{3}(U^{n}_{1}z_{1}z^{2}_{2}z_{8})\! = 0, d_{3}(U^{n}_{1}z_{1}z^{4}_{2}z_{4})\! = 0,d_{3}(U^{n}_{1}z^{9}_{1}z_{4})\!=0, d_{3}(U^{n}_{1}z_{1}z^{3}_{4})=0,\\
&d_{3}(U^{n}_{1}z^{4}_{1}z_{2}z_{7})\;=\;0,\; d_{3}(U^{n}_{1}z^{4}_{1}z_{4}z_{5})\,=0,\;d_{3}(U^{n}_{1}z^{3}_{1}z_{2}z^{2}_{4})\;=0,\; d_{3}(U^{n}_{1}z^{4}_{1}z_{2}y_{7})=0,\\
&d_{3}(U^{n}_{1}z_{10}z^{3}_{1})=0,\, d_{3}(U^{n}_{1}z^{5}_{1}z_{8})=0,\, d_{3}(U^{n}_{1}z^{7}_{1}z_{6})=0,d_{3}(U^{n}_{1}z^{4}_{1}z_{3}z_{6})=0,d_{3}(U^{n}_{1}\times\\
&\times z^{3}_{1}z^{2}_{2}z_{6})=0,d_{3}(U^{n}_{1}z^{5}_{2}z_{3})=0,d_{3}(U^{n}_{1}z^{2}_{1}z_{4}z_{7}) = 0, d_{3}(U^{n}_{1}z^{2}_{1}z_{2}z_{9}) = 0, d_{3}(U^{n}_{1}\times\\
&\times z^{5}_{1}z_{2}z^{2}_{3}) = 0,d_{3}(U^{n}_{1}z_{1}z_{2}z_{4}z_{6}) = 0,d_{3}(U^{n}_{1}z^{3}_{1}z_{2}z_{3}z_{5}) = 0,d_{3}(U^{n}_{1}z^{2}_{1}z^{2}_{2}z_{3}z_{4}) = 0,\\
&d_{3}(U^{n}_{1}z^{2}_{1}z_{2}y_{9}) = 0,d_{3}(U^{n}_{1}z^{3}_{1}z_{2}y^{2}_{4}) = 0, d_{3}(U^{n}_{1}z^{2}_{1}z_{2}z^{3}_{3}) = 0, d_{3}(U^{n}_{1}z^{2}_{1}z_{3}z_{8}) = 0,\\
&d_{3}(U^{n}_{1}z^{4}_{1}z^{3}_{2}z_{3}) = 0, d_{3}(U^{n}_{1}z^{8}_{1}z_{2}z_{3}) = 0, d_{3}(U^{n}_{1}z_{2}y^{2}_{4}z_{3}) = 0,d_{3}(U^{n}_{1}z^{5}_{1}z^{2}_{2}z_{4}) = 0, \\
&d_{3}(U^{n}_{1}z_{1}z_{2}z_{3}y_{7}) = 0,d_{3}(U^{n}_{1}z^{6}_{1}z_{3}z_{4}) = 0, d_{3}(U^{n}_{1}z^{6}_{1}z_{2}z_{5}) = 0, d_{3}(U^{n}_{1}z^{2}_{1}z^{3}_{2}z_{5}) = 0,\\
&d_{3}(U^{n}_{1}z^{9}_{1}z_{4})=0, d_{3}(U^{n}_{1}z^{11}_{1}z_{2})=0, d_{3}(U^{n}_{1}z^{3}_{1}z^{5}_{2})=0, d_{3}(U^{n}_{1}z^{7}_{1}z^{3}_{2})=0, d_{3}(U^{n}_{1}z_{5}\times\\
&\times z_{8}) = 0, d_{3}(U^{n}_{1}z^{2}_{1}z_{5}z_{6}) = 0, d_{3}(U^{n}_{1}z^{2}_{1}z_{4}y_{7}) = 0, d_{3}(U^{n}_{1}z_{5}z_{4}z^{2}_{2}) = 0, d_{3}(U^{n}_{1}z^{2}_{4}\times\\
&\times z_{3}z_{2}) = 0, d_{3}(U^{n}_{1}z_{1}z^{2}_{3}z_{6}) = 0, d_{3}(U^{n}_{1}z_{1}z_{4}y^{2}_{4}) = 0,d_{3}(U^{n}_{1}z_{6}z_{7}) = 0, d_{3}(U^{n}_{1}z_{6}\times\\
&\times y_{7}) = 0,d_{3}(U^{n}_{1}z_{1}z_{2}z_{3}z_{7}) = 0, d_{3}(U^{n}_{1}z_{1}z_{3}z_{4}z_{5}) = 0,\! d_{3}(U^{n}_{1}z_{1}z_{2}z^{2}_{5}) = 0,\!d_{3}(U^{n}_{1}\!\times\\
&\times z_{3}z_{10})=0,d_{3}(U^{n}_{1}z_{2}z_{11})=0,\! d_{3}(U^{n}_{1}z_{7}z^{3}_{2})=0,\! d_{3}(U^{n}_{1}y_{7}z^{3}_{2})=0,\! d_{3}(U^{n}_{1}z_{2}y_{11})=\\
&=0,d_{3}(U^{n}_{1}z_{6}z_{3}z^{2}_{2}) = 0, d_{3}(U^{n}_{1}z^{2}_{3}z_{5}z_{2}) = 0, d_{3}(U^{n}_{1}z_{1}z^{3}_{2}z^{2}_{3}) = 0,d_{3}(U^{n}_{1}z_{4}z_{9})=0, \\
&d_{3}(U^{n}_{1}y_{9}z_{4})=0, d_{3}(U^{n}_{1}z_{1}z_{12})=0.
\end{align*}

\clearpage

Table 18. The ring $\pi _{*}(MSp)$ in dimensions from 32 to 52.
\begin{align*}
&\qquad\; n\makebox[50mm]{} 32\\
&\;\:\pi _{n}(MSp)\makebox[42mm]{} 22{\Bbb Z}\\
&\mbox{Generators }z^{8}_{1},\, 2z_{2}z^{6}_{1},\, z^{4}_{1}z^{2}_{2},\, 2z^{4}_{1}z_{4},\, z^{4}_{1}y_{4} + z^{2}_{1}z^{3}_{2},\, 2z^{3}_{1}z_{2}z_{3}, z^{3}_{1}z_{5},2z^{2}_{1}z_{2}y_{4}, 2z^{2}_{1}y_{6},\\
&\qquad\; n\makebox[50 mm]{} 32\\
&\;\:\pi _{n}(MSp)\makebox[42 mm]{} 22{\Bbb Z}\\
&\mbox{Generators } z^{2}_{1}z_{2}z_{4},\, z^{2}_{1}y_{6}+ z_{1}z_{3}z_{4},\, 2z_{4}y_{4},\, 2z_{2}y_{6},\,z^{2}_{2}y_{4}+ z_{2}z^{2}_{3},\, z_{1}y_{7}, z_{2}z_{6}, z_{3}z_{5}, \\
&\qquad\; n\makebox[25 mm]{} 32 \makebox[47 mm]{} 33\\
&\;\:\pi _{n}(MSp)\makebox[17 mm]{} 22{\Bbb Z} \makebox[45 mm]{} 12{\Bbb Z_{2}}\\
&\mbox{Generators }z^{2}_{4},\! y^{2}_{4}, z^{2}_{1}z^{2}_{3}, 2z_{8}, 2y_{8}\qquad\; \theta _{1}z^{4}_{1}z^{2}_{2}, \theta _{1}(z^{4}_{1}y_{4}+ z^{2}_{1}z^{3}_{2}), \theta _{1}(z^{2}_{2}y_{4} + z_{2}z^{2}_{3}), \\
&\qquad\; n\makebox[50 mm]{} 33\\
&\;\:\pi _{n}(MSp)\makebox[42 mm]{} 12{\Bbb Z}_{2}\\
&\mbox{Generators }\theta _{1}(z^{2}_{1}y_{6} + z_{1}z_{3}z_{4}),\tau_{4},\theta _{1}z_{1}y_{7} = \theta _{1}z_{2}z_{6} = \Phi _{3}z_{1}z_{2}, \theta _{1}z^{2}_{4} = \theta _{1}z_{1}z_{7}=\\
&\qquad\; n\makebox[50 mm]{} 33\\
&\;\:\pi _{n}(MSp)\makebox[42 mm]{} 12{\Bbb Z}_{2}\\
&\mbox{Generators } = \Phi _{2}z_{1}z_{4},\,\theta _{1}z_{3}z_{5}\, = \Phi _{2}z_{2}z_{3}\, = \Phi _{1}z^{2}_{3},\, \theta _{1}z^{3}_{1}z_{5},\, \theta _{1}z^{2}_{1}z^{2}_{3},\, \theta _{1}y^{2}_{4},\, \theta _{1}z^{8}_{1}\\
&\qquad\; n\makebox[50 mm]{} 34\\
&\;\:\pi _{n}(MSp)\makebox[42 mm]{} 16{\Bbb Z}_{2}\\
&\mbox{Generators } \theta ^{2}_{1}z^{4}_{1}z^{2}_{2}, \theta _{1}\tau_{4}, \Phi ^{2}_{1}z^{2}_{3}, \Phi _{1}\Phi _{4}, \Phi _{2}\Phi _{3}, \Phi _{1}\tau_{3},\theta ^{2}_{1}z_{1}y_{7}, \theta ^{2}_{1}z^{2}_{4}, \theta ^{2}_{1}z_{3}z_{5}, \theta ^{2}_{1}z^{3}_{1}z_{5}, \\
&\qquad\; n\makebox[50 mm]{} 34\\
&\;\:\pi _{n}(MSp)\makebox[42 mm]{} 16{\Bbb Z}_{2}\\
&\mbox{Generators }\theta ^{2}_{1}(z^{4}_{1}y_{4}\,+\, z^{2}_{1}z^{3}_{2}),\, \theta ^{2}_{1}(z^{2}_{2}y_{4} + z_{2}z^{2}_{3}),\, \theta ^{2}_{1}(z^{2}_{1}y_{6} + z_{1}z_{3}z_{4}),\, \theta ^{2}_{1}y^{2}_{4},\, \theta ^{2}_{1}z^{8}_{1},\\
&\qquad\; n\makebox[14 mm]{} 34\makebox[9 mm]{} 35\makebox[39 mm]{} 36\\
&\;\:\pi _{n}(MSp)\makebox[6 mm]{} 16{\Bbb Z}_{2}\makebox[8 mm]{} 0\makebox[38 mm]{} 30{\Bbb Z}\\
&\mbox{Generators }\theta ^{2}_{1}z^{2}_{1}z^{2}_{3}\qquad \qquad 2z^{9}_{1}, z^{7}_{1}z_{2}, 2z^{5}_{1}z^{2}_{2}, z^{5}_{1}z_{4}, 2z^{5}_{1}y_{4}, 2z_{1}z^{2}_{2}y_{4}, 2z^{3}_{3},z^{2}_{1}z_{3}z_{4}, \\
&\qquad\; n\makebox[50 mm]{} 36\\
&\;\:\pi _{n}(MSp)\makebox[42 mm]{} 30{\Bbb Z}\\
&\mbox{Generators }z^{3}_{3} + z_{2}z_{3}y_{4}, 2z_{9}, z_{2}z_{7}, 2z^{2}_{1}y_{7}, 2z_{1}z^{2}_{4}, z_{3}z_{6}, 2y^{2}_{4}z_{1},2y_{6}z_{3}, z_{5}z_{4}, z^{3}_{2}z_{3},\\
&\qquad\; n\makebox[50 mm]{} 36\\
&\;\:\pi _{n}(MSp)\makebox[42 mm]{} 30{\Bbb Z}\\
&\mbox{Generators }z_{2}z_{3}z_{4} + z_{1}z_{2}y_{6},\! z_{2}z_{3}z_{4} + z_{1}z_{4}y_{4},\! 2z_{5}z^{4}_{1},\! 2y_{8}z_{1}, y_{7}z_{2},\!z^{3}_{1}z^{2}_{3}+ z^{3}_{1}z_{2}y_{4}, \\
&\qquad\; n\makebox[33 mm]{} 36\makebox[46 mm]{} 37\\
&\;\:\pi _{n}(MSp)\makebox[26 mm]{} 30{\Bbb Z}\makebox[43 mm]{} 17{\Bbb Z}_{2}\\
&\mbox{Generators }2z^{3}_{1}y_{6}, 2z_{5}y_{4}, 2z^{3}_{1}z^{2}_{3}, z_{1}z_{8}, z^{3}_{1}z_{6}, z^{3}_{1}z^{3}_{2}\quad \theta _{1}z^{7}_{1}z_{2}, \theta _{1}z^{5}_{1}z_{4}, \theta _{1}z^{3}_{2}z_{3}, \tau_{5}, \Phi _{5}\\
&\qquad\; n\makebox[50 mm]{} 37\\
&\;\:\pi _{n}(MSp)\makebox[42 mm]{} 17{\Bbb Z}_{2}\\
&\mbox{Generators }\theta _{1}(z^{3}_{3}\, + \,z_{2}z_{3}y_{4}),\, \theta _{1}z_{2}z_{7}\, =\, \theta _{1}z_{5}z_{4}\, = \Phi _{2}z_{2}z_{4}\, = \Phi _{2}z_{5}z_{1},\, \theta _{1}z^{2}_{1}z_{3}z_{4},
\end{align*}
Table 18. (continuation).
\begin{align*}
&\qquad\; n\makebox[50 mm]{} 37\\
&\;\:\pi _{n}(MSp)\makebox[42 mm]{} 17{\Bbb Z}_{2}\\
&\mbox{Generators }\theta _{1}(z_{2}z_{3}z_{4} + z_{1}z_{2}y_{6}), \theta _{1}(z_{2}z_{3}z_{4} + z_{1}z_{4}y_{4}),\theta _{1}(z^{3}_{1}z^{2}_{3}+ z^{3}_{1}z_{2}y_{4}),\!\Phi _{1}y^{2}_{4},\\
&\qquad\; n\makebox[50 mm]{} 37\\
&\;\:\pi _{n}(MSp)\makebox[42 mm]{} 17{\Bbb Z}_{2}\\
&\mbox{Generators }\theta _{1}z_{1}z_{8} = \Phi _{4}z^{2}_{1}, \theta _{1}z^{3}_{1}z_{6}, \theta _{1}z^{3}_{1}z^{3}_{2},\theta _{1}y_{7}z_{2}= \theta _{1}z_{6}z_{3}= \Phi _{3}z_{1}z_{3}= \Phi _{3}z^{2}_{2}, \\
&\qquad\; n\makebox[20 mm]{} 37\makebox[46 mm]{} 38\\
&\;\:\pi _{n}(MSp)\makebox[12 mm]{} 17{\Bbb Z}_{2}\makebox[42 mm]{} 19{\Bbb Z}_{2}\\
&\mbox{Generators }\Phi _{1}z_{3}z_{5}= \Phi _{2}z^{2}_{3}\qquad\; \theta ^{2}_{1}z^{7}_{1}z_{2}, \theta ^{2}_{1}z^{5}_{1}z_{4}, \theta ^{2}_{1}z^{3}_{2}z_{3}, \theta ^{2}_{1}(z^{3}_{3}+ z_{2}z_{3}y_{4}),\Phi _{3}\tau_{1},\\
&\qquad\; n\makebox[50 mm]{} 38\\
&\;\:\pi _{n}(MSp)\makebox[42 mm]{} 19{\Bbb Z}_{2}\\
&\mbox{Generators }\theta _{1}\Phi _{5}, \theta ^{2}_{1}z_{5}z_{4},\theta ^{2}_{1}(z^{3}_{1}z^{2}_{3}+ z^{3}_{1}z_{2}y_{4}), \theta ^{2}_{1}z_{1}z_{8},\, \theta ^{2}_{1}z^{3}_{1}z_{6}, \theta ^{2}_{1}z^{3}_{1}z^{3}_{2}, \theta _{1}\Phi _{1}z_{3}z_{5}, \\
&\qquad\; n\makebox[50 mm]{} 38\\
&\;\:\pi _{n}(MSp)\makebox[42 mm]{} 19{\Bbb Z}_{2}\\
&\mbox{Generators }\Phi _{1}\tau_{4},\, \theta ^{2}_{1}y_{7}z_{2},\, \theta ^{2}_{1}(z_{2}z_{3}z_{4} + z_{1}z_{2}y_{6}),\, \theta ^{2}_{1}(z_{2}z_{3}z_{4} + z_{1}z_{4}y_{4}), \theta ^{2}_{1}z^{2}_{1}z_{3}z_{4},\\
&\qquad\; n\makebox[20 mm]{} 38\makebox[27 mm]{} 39\makebox[34 mm]{} 40\\
&\;\:\pi _{n}(MSp)\makebox[12 mm]{} 19{\Bbb Z}_{2}\makebox[25 mm]{} {\Bbb Z}_{2}\makebox[32 mm]{} 42{\Bbb Z}\\
&\mbox{Generators }\theta _{1}\Phi _{1}y^{2}_{4}, \Phi _{2}\tau_{2}\qquad\; \Phi ^{2}_{1}\Phi _{4} = \Phi _{1}\Phi _{2}\Phi _{3}\qquad\;\; 2z^{2}_{1}z_{2}y_{6}, 2z^{2}_{1}z_{4}y_{4},2z^{3}_{2}y_{4},  \\
&\qquad\; n\makebox[50 mm]{} 40\\
&\;\:\pi _{n}(MSp)\makebox[42 mm]{} 42{\Bbb Z}\\
&\mbox{Generators }z^{2}_{1}z_{8} + z_{1}z_{4}z_{5}, 2z^{5}_{2}, z_{1}z^{3}_{2}z_{3} + z^{2}_{1}z^{2}_{2}y_{4}, 2z^{2}_{1}z_{8}, z_{6}z_{4},2z^{4}_{1}z_{2}y_{4},2z^{2}_{1}z^{2}_{2}z_{4}, \\
&\qquad\; n\makebox[50 mm]{} 40\\
&\;\:\pi _{n}(MSp)\makebox[42 mm]{} 42{\Bbb Z}\\
&\mbox{Generators }2y^{*}_{10}, 2z_{2}y_{8}, 2z^{2}_{2}y_{6}, z^{3}_{1}z_{3}z_{4} + z^{4}_{1}y_{6}, 2z^{6}_{1}z_{4}, 2z^{4}_{1}z_{6}, 2z^{4}_{1}z^{3}_{2}, z^{4}_{1}z^{3}_{2} + z^{6}_{1}y_{4},\\
&\qquad\; n\makebox[50 mm]{} 40\\
&\;\:\pi _{n}(MSp)\makebox[42 mm]{} 42{\Bbb Z}\\
&\mbox{Generators }z^{2}_{5}, 2z^{8}_{1}z_{2}, 2z_{4}y_{6}, 2y_{10},\, z_{1}z_{3}y_{6} + z^{2}_{3}z_{4},\, z_{1}y_{4}z_{5} + z^{2}_{3}z_{4}, 2z^{2}_{2}y_{6}, 2z_{2}y^{2}_{4},\\
&\qquad\; n\makebox[50 mm]{} 40\\
&\;\:\pi _{n}(MSp)\makebox[42 mm]{} 42{\Bbb Z}\\
&\mbox{Generators }z_{1}y_{9},\, 2y_{4}y_{6},\, 2y_{4}z_{6},\, 2z^{2}_{3}y_{4},\, z_{1}z_{9},\, z^{2}_{1}z^{2}_{4},\, z^{3}_{2}z_{4}, z^{5}_{1}z_{5}, z^{2}_{2}z^{2}_{3}, z^{2}_{1}z^{4}_{2}, z^{6}_{1}z^{2}_{2},\\
&\qquad\; n\makebox[20 mm]{} 40\makebox[54 mm]{} 41\\
&\;\:\pi _{n}(MSp)\makebox[12 mm]{} 42{\Bbb Z}\makebox[52 mm]{} 23{\Bbb Z}_{2}\\
&\mbox{Generators }  z^{10}_{1},z_{3}y_{7}, z^{3}_{1}y_{7}, z^{2}_{1}y^{2}_{4}, z_{2}z_{8}\qquad\;\;\;\; \theta _{1}(z^{2}_{1}z_{8} + z_{1}z_{4}z_{5}),\theta _{1}z^{2}_{1}z^{2}_{4}, \theta _{1}z^{10}_{1},\\
&\qquad\; n\makebox[50 mm]{} 41\\
&\;\:\pi _{n}(MSp)\makebox[42 mm]{} 23{\Bbb Z}_{2}\\
&\mbox{Generators }\theta _{1}(z_{1}z^{3}_{2}z_{3}+ z^{2}_{1}z^{2}_{2}y_{4}),\, \theta _{1}(z^{3}_{1}z_{3}z_{4}+ z^{4}_{1}y_{6}),\, \theta _{1}(z^{4}_{1}z^{3}_{2} + z^{6}_{1}y_{4}),\,  \theta _{1}z^{6}_{1}z^{2}_{2},
\end{align*}
Table 18. (continuation).
\begin{align*}
&\qquad\; n\makebox[50 mm]{} 41\\
&\;\:\pi _{n}(MSp)\makebox[42 mm]{} 23{\Bbb Z}_{2}\\
&\mbox{Generators }\theta _{1}(z_{1}y_{4}z_{5} + z^{2}_{3}z_{4}), \theta _{1}(z_{1}z_{3}y_{6} + z^{2}_{3}z_{4}), \theta _{1}z_{3}y_{7} = \Phi _{3}z_{2}z_{3},\Phi _{2}z_{3}z_{4} = \\
&\qquad\; n\makebox[50 mm]{} 41\\
&\;\:\pi _{n}(MSp)\makebox[42 mm]{} 23{\Bbb Z}_{2}\\
&\mbox{Generators }=\theta _{1}z^{2}_{5}\, =  \Phi _{2}z_{2}z_{5},\,\theta _{1}z^{2}_{2}z^{2}_{3},\, \theta _{1}z^{2}_{1}z^{4}_{2}, \theta _{1}z_{1}z_{9} = \theta _{1}z_{2}z_{8} = \Phi _{4}z_{1}z_{2},\tau^{*}_{6}, \\
&\qquad\; n\makebox[50 mm]{} 41\\
&\;\:\pi _{n}(MSp)\makebox[42 mm]{} 23{\Bbb Z}_{2}\\
&\mbox{Generators }\theta _{1}z^{5}_{1}z_{5},\,\theta _{1}z^{3}_{2}z_{4},\, \theta _{1}z_{1}y_{9},\, \theta _{1}z^{3}_{1}y_{7},\, \theta _{1}z^{2}_{1}y^{2}_{4},\, \tau_{6},\, \Phi _{1}(z^{3}_{3}\, + \, z_{2}z_{3}y_{4}),\,\kappa _{1} \\
&\qquad\; n\makebox[50 mm]{} 42\\
&\;\:\pi _{n}(MSp)\makebox[42 mm]{} 29{\Bbb Z}_{2}\\
&\mbox{Generators }\theta ^{2}_{1}(z^{3}_{1}z_{3}z_{4}+ z^{4}_{1}y_{6}),\! \theta ^{2}_{1}z^{2}_{1}z^{2}_{4},\theta ^{2}_{1}(z^{2}_{1}z_{8} + z_{1}z_{4}z_{5}), \theta ^{2}_{1}(z_{1}z^{3}_{2}z_{3} + z^{2}_{1}z^{2}_{2}y_{4}),\\
&\qquad\; n\makebox[50 mm]{} 42\\
&\;\:\pi _{n}(MSp)\makebox[42 mm]{} 29{\Bbb Z}_{2}\\
&\mbox{Generators }\theta ^{2}_{1}(z_{1}y_{4}z_{5} + z^{2}_{3}z_{4}),\theta ^{2}_{1}(z_{1}z_{3}y_{6} + z^{2}_{3}z_{4}),\theta ^{2}_{1}z^{10}_{1},\theta ^{2}_{1}z^{2}_{5}, \theta ^{2}_{1}(z^{4}_{1}z^{3}_{2} + z^{6}_{1}y_{4}),\\
&\qquad\; n\makebox[50 mm]{} 42\\
&\;\:\pi _{n}(MSp)\makebox[42 mm]{} 29{\Bbb Z}_{2}\\
&\mbox{Generators }\theta ^{2}_{1}z^{3}_{2}z_{4},\! \theta ^{2}_{1}z_{3}y_{7},\theta ^{2}_{1}z^{5}_{1}z_{5}, \theta ^{2}_{1}z^{2}_{2}z^{2}_{3}, \theta ^{2}_{1}z^{2}_{1}z^{4}_{2}, \theta ^{2}_{1}z_{1}z_{9}, \theta ^{2}_{1}z^{6}_{1}z^{2}_{2}, \Phi _{2}\Phi _{4}, \Phi _{1}\Phi _{5},\\
&\qquad\; n\makebox[50 mm]{} 42\\
&\;\:\pi _{n}(MSp)\makebox[42 mm]{} 29{\Bbb Z}_{2}\\
&\mbox{Generators }\Phi ^{2}_{1}y^{2}_{4}, \,\theta _{1}\tau^{*}_{6},\, \theta _{1}\kappa _{1},\, \theta _{1}\tau_{6}, \theta ^{2}_{1}z_{1}y_{9}, \theta ^{2}_{1}z^{3}_{1}y_{7}, \theta ^{2}_{1}z^{2}_{1}y^{2}_{4}, \theta _{1}\Phi _{1}(z^{3}_{3} + z_{2}z_{3}y_{4}),\\
&\qquad\; n\makebox[22 mm]{} 42\makebox[19 mm]{} 43\makebox[32 mm]{} 44\\
&\;\:\pi _{n}(MSp)\makebox[14 mm]{} 29{\Bbb Z}_{2}\makebox[18 mm]{} 0\makebox[32 mm]{} 56{\Bbb Z}\\
&\mbox{Generators }\Phi ^{2}_{1}z_{3}z_{5}, \Phi _{1}\tau_{5}, \Phi _{2}\tau_{3}\qquad \qquad \;\; 2z^{3}_{1}y_{8}, 2z^{3}_{1}z^{2}_{2}y_{4}, 2z^{5}_{1}y_{6}, 2z^{7}_{1}y_{4}, 2z_{1}y_{10}, \\
&\qquad\; n\makebox[50 mm]{} 44\\
&\;\:\pi _{n}(MSp)\makebox[42 mm]{} 56{\Bbb Z}\\
&\mbox{Generators }2z^{2}_{2}z_{3}y_{4},\, z_{1}z_{10},\, 2z_{1}z^{2}_{2}y_{6},\, 2z^{2}_{1}z_{5}y_{4}, 2z^{2}_{1}z_{9}, z_{1}z_{2}y_{8} + z_{1}z_{4}y_{6}, 2z^{3}_{1}z_{3}z_{5},\\
&\qquad\; n\makebox[50 mm]{} 44\\
&\;\:\pi _{n}(MSp)\makebox[42 mm]{} 56{\Bbb Z}\\
&\mbox{Generators }z_{1}z_{4}y_{6} + z_{1}z^{2}_{5}, z_{5}z_{6}, 2z^{3}_{1}z^{2}_{4},\, z^{3}_{1}z_{4}y_{4} + z^{3}_{1}z_{2}y_{6},\, 2z^{6}_{1}z_{5},\, 2z^{3}_{1}z^{4}_{2},\, 2z^{7}_{1}z^{2}_{2},\\
&\qquad\; n\makebox[50 mm]{} 44\\
&\;\:\pi _{n}(MSp)\makebox[42 mm]{} 56{\Bbb Z}\\
&\mbox{Generators }2z^{2}_{1}y_{9},\, z^{4}_{1}z^{2}_{2}z_{3} + z^{5}_{1}z_{2}y_{4},\, z^{3}_{1}z_{2}y_{6}+ z^{3}_{1}z_{3}z_{5},\, z_{1}z^{3}_{2}y_{4}+ z_{1}z^{2}_{2}z^{2}_{3}, 2z^{4}_{2}z_{3}, \\
&\qquad\; n\makebox[50 mm]{} 44\\
&\;\:\pi _{n}(MSp)\makebox[42 mm]{} 56{\Bbb Z}\\
&\mbox{Generators }2z^{11}_{1}, z_{1}y_{4}z_{6}+ z_{2}z_{3}z_{6}, z_{4}y_{7},z_{3}z_{4}y_{4} + z_{2}z_{3}y_{6}, z_{2}z_{3}y_{6} + z^{2}_{3}z_{5}, 2z^{2}_{3}z_{5},
\end{align*}
Table 18. (continuation).
\begin{align*}
&\qquad\; n\makebox[50 mm]{} 44\\
&\;\:\pi _{n}(MSp)\makebox[42 mm]{} 56{\Bbb Z}\\
&\mbox{Generators }2z_{11}, 2z^{4}_{1}y_{7}, 2z^{3}_{1}y^{2}_{4}, 2z_{3}y^{2}_{4}, 2z_{5}y_{6}, 2y_{4}z_{7}, 2z_{3}y_{8}, 2y_{4}y_{7}, z_{3}z_{8}, 2z_{1}y_{6}y_{4},\\
&\qquad\; n\makebox[50 mm]{} 44\\
&\;\:\pi _{n}(MSp)\makebox[42 mm]{} 56{\Bbb Z}\\
&\mbox{Generators }2z_{1}y^{*}_{10}, z^{2}_{1}z_{2}y_{7}, z_{1}z^{5}_{2}, z^{3}_{1}z_{8}, z^{3}_{1}z^{2}_{2}z_{4}, z^{7}_{1}z_{4}, z^{5}_{1}z_{6}, z^{5}_{1}z^{3}_{2}, z^{2}_{2}z_{3}z_{4}, z_{1}z_{2}y^{2}_{4}, \\
&\qquad\; n\makebox[32 mm]{} 44\makebox[50 mm]{} 45\\
&\;\:\pi _{n}(MSp)\makebox[24 mm]{} 56{\Bbb Z}\makebox[48 mm]{} 31{\Bbb Z}_{2}\\
&\mbox{Generators }z^{9}_{1}z_{2},\! z_{4}z_{7},\! z_{2}y_{9},\! z_{2}z_{9}, z^{2}_{1}z_{4}z_{5},z_{2}z^{3}_{3}\qquad \theta _{1}(z_{1}z_{2}y_{8} + z_{1}z_{4}y_{6}),\theta _{1}z_{1}z_{10}, \\
&\qquad\; n\makebox[50 mm]{} 45\\
&\;\:\pi _{n}(MSp)\makebox[42 mm]{} 31{\Bbb Z}_{2}\\
&\mbox{Generators }\theta _{1}(z_{1}z_{4}y_{6} + z_{1}z^{2}_{5}),\theta _{1}(z^{3}_{1}z_{4}y_{4} + z^{3}_{1}z_{2}y_{6}), \theta _{1}(z^{3}_{1}z_{2}y_{6} + z^{3}_{1}z_{3}z_{5}),\! \tau _{7},\!\tau^{*}_{7}\\
&\qquad\; n\makebox[50 mm]{} 45\\
&\;\:\pi _{n}(MSp)\makebox[42 mm]{} 31{\Bbb Z}_{2}\\
&\mbox{Generators }\theta _{1}(z_{1}z^{3}_{2}y_{4}+ z_{1}z^{2}_{2}z^{2}_{3}), \theta _{1}z_{3}z_{8} = \theta _{1}z_{2}z_{9} = \Phi _{4}z_{1}z_{3} = \Phi _{4}z^{2}_{2},\theta _{1}z^{2}_{1}z_{4}z_{5},\\
&\qquad\; n\makebox[50 mm]{} 45\\
&\;\:\pi _{n}(MSp)\makebox[42 mm]{} 31{\Bbb Z}_{2}\\
&\mbox{Generators }\theta _{1}(z_{3}z_{4}y_{4}\,+\, z_{2}z_{3}y_{6}),\; \theta _{1}(z^{4}_{1}z^{2}_{2}z_{3}\, +\, z^{5}_{1}z_{2}y_{4}),\; \theta _{1}(z_{1}y_{4}z_{6}\, +\, z_{2}z_{3}z_{6}), \\
&\qquad\; n\makebox[50 mm]{} 45\\
&\;\:\pi _{n}(MSp)\makebox[42 mm]{} 31{\Bbb Z}_{2}\\
&\mbox{Generators }\theta _{1}(z_{2}z_{3}y_{6}+ z^{2}_{3}z_{5}), \theta _{1}z^{7}_{1}z_{4},\theta _{1}z^{5}_{1}z_{6},\theta _{1}z^{3}_{1}z^{2}_{2}z_{4}, \theta _{1}z^{5}_{1}z^{3}_{2}, \theta _{1}z^{9}_{1}z_{2},\!\theta _{1}z_{1}z^{5}_{2},\\
&\qquad\; n\makebox[50 mm]{} 45\\
&\;\:\pi _{n}(MSp)\makebox[42 mm]{} 31{\Bbb Z}_{2}\\
&\mbox{Generators }\theta _{1}z^{2}_{1}z_{2}y_{7}, \theta _{1}z^{2}_{2}z_{3}z_{4}, \theta _{1}z_{1}z_{2}y^{2}_{4}, \theta _{1}z_{4}z_{7} = \Phi _{2}z^{2}_{4} = \Phi _{2}z_{1}z_{7}, \theta _{1}z_{2}z^{3}_{3},\Phi_{6},\\
&\qquad\; n\makebox[50 mm]{} 45\\
&\;\:\pi _{n}(MSp)\makebox[42 mm]{} 31{\Bbb Z}_{2}\\
&\mbox{Generators }\theta _{1}z_{2}y_{9} = \theta _{1}z_{4}y_{7} = \theta _{1}z_{5}z_{6} = \Phi _{3}z_{2}z_{4} =\! \Phi _{1}z_{1}y_{9} =\! \Phi _{2}z_{1}y_{7}=\!\Phi _{2}z_{2}z_{6}, \\
&\qquad\; n\makebox[46 mm]{} 45\makebox[47 mm]{} 46\\
&\;\:\pi _{n}(MSp)\makebox[38 mm]{} 31{\Bbb Z}_{2}\makebox[43 mm]{} 36{\Bbb Z}_{2}\\
&\mbox{Generators }\theta _{1}z^{3}_{1}z_{8},\Phi _{2}y^{2}_{4}, \Phi _{1}z_{3}z_{7} = \Phi _{2}z_{3}z_{5} = \Phi _{1}z^{2}_{5}, \Phi _{2}z_{3}y_{7} \qquad \theta ^{2}_{1}z_{1}z_{10},\Phi _{1}\tau_{6},\\
&\qquad\; n\makebox[50 mm]{} 46\\
&\;\:\pi _{n}(MSp)\makebox[42 mm]{} 36{\Bbb Z}_{2}\\
&\mbox{Generators }\theta ^{2}_{1}(z^{3}_{1}z_{4}y_{4}\; +\; z^{3}_{1}z_{2}y_{6}),\; \theta ^{2}_{1}(z_{1}z_{2}y_{8}\; +\; z_{1}z_{4}y_{6}),\; \theta ^{2}_{1}(z_{1}z_{4}y_{6}\; +\; z_{1}z^{2}_{5}),\\
&\qquad\; n\makebox[50 mm]{} 46\\
&\;\:\pi _{n}(MSp)\makebox[42 mm]{} 36{\Bbb Z}_{2}\\
&\mbox{Generators }\theta ^{2}_{1}(z^{3}_{1}z_{2}y_{6}\; +\; z^{3}_{1}z_{3}z_{5}),\; \theta ^{2}_{1}(z_{1}z^{3}_{2}y_{4}\;+\; z_{1}z^{2}_{2}z^{2}_{3}),\; \theta ^{2}_{1}z^{2}_{2}z_{3}z_{4},\; \theta ^{2}_{1}z_{1}z_{2}y^{2}_{4},
\end{align*}
Table 18. (continuation).
\begin{align*}
&\qquad\; n\makebox[50 mm]{} 46\\
&\;\:\pi _{n}(MSp)\makebox[42 mm]{} 36{\Bbb Z}_{2}\\
&\mbox{Generators }\theta ^{2}_{1}(z^{4}_{1}z^{2}_{2}z_{3} + z^{5}_{1}z_{2}y_{4}),\, \theta ^{2}_{1}(z_{1}y_{4}z_{6} + z_{2}z_{3}z_{6}),\, \theta ^{2}_{1}z_{3}z_{8},\, \theta ^{2}_{1}z^{7}_{1}z_{4},\, \theta ^{2}_{1}z^{5}_{1}z_{6},\\
&\qquad\; n\makebox[50 mm]{} 46\\
&\;\:\pi _{n}(MSp)\makebox[42 mm]{} 36{\Bbb Z}_{2}\\
&\mbox{Generators }\theta ^{2}_{1}z^{2}_{1}z_{2}y_{7},\! \theta ^{2}_{1}z_{1}z^{5}_{2}, \theta ^{2}_{1}(z_{3}z_{4}y_{4}+ z_{2}z_{3}y_{6}), \theta ^{2}_{1}(z_{2}z_{3}y_{6}+ z^{2}_{3}z_{5}), \theta ^{2}_{1}z^{2}_{1}z_{4}z_{5}\\
&\qquad\; n\makebox[50 mm]{} 46\\
&\;\:\pi _{n}(MSp)\makebox[42 mm]{} 36{\Bbb Z}_{2}\\
&\mbox{Generators }\theta ^{2}_{1}z_{4}z_{7},\, \theta _{1}\Phi _{2}z_{3}y_{7},\, \theta _{1}\Phi _{2}y^{2}_{4},\, \theta _{1}\Phi _{6}, \theta ^{2}_{1}z^{3}_{1}z_{8}, \theta ^{2}_{1}z^{3}_{1}z^{2}_{2}z_{4}, \theta ^{2}_{1}z^{5}_{1}z^{3}_{2}, \theta ^{2}_{1}z^{9}_{1}z_{2},\\
&\qquad\; n\makebox[50 mm]{} 46\\
&\;\:\pi _{n}(MSp)\makebox[42 mm]{} 36{\Bbb Z}_{2}\\
&\mbox{Generators }\theta _{1}\Phi _{1}z_{3}z_{7},\theta ^{2}_{1}z_{2}z^{3}_{3},\theta ^{2}_{1}z_{2}y_{9}, \Phi _{2}\tau_{4}, \Phi _{1}\tau^{*}_{6}, \Phi _{3}\tau_{2}, \Phi _{4}\tau_{1}, \Phi _{1}\kappa _{1}, \Phi _{1}\tau_{1}z^{2}_{3}\\
&\qquad\; n\makebox[36 mm]{} 47\makebox[52 mm]{} 48\\
&\;\:\pi _{n}(MSp)\makebox[29 mm]{} 3{\Bbb Z}_{2}\makebox[49 mm]{} 77{\Bbb Z}\\
&\mbox{Generators }\Phi ^{2}_{1}\Phi _{5}, \Phi _{1}\Phi _{2}\Phi _{4} = \Phi ^{2}_{2}\Phi _{3}, \Phi ^{2}_{1}\tau_{5} = \Phi _{1}\Phi _{2}\tau_{3} \qquad z^{3}_{1}y_{8} + z^{2}_{1}z_{2}y_{7},2z^{2}_{1}z^{5}_{2}, \\
&\qquad\; n\makebox[50 mm]{} 48\\
&\;\:\pi _{n}(MSp)\makebox[42 mm]{} 77{\Bbb Z}\\
&\mbox{Generators }2z^{2}_{1}z_{2}z^{2}_{4}, 2z^{2}_{1}z_{10}, 2z^{2}_{1}z_{2}y_{8}, 2z^{2}_{1}z_{4}y_{6}, 2z^{4}_{1}z_{4}y_{4}, 2z^{4}_{1}z_{2}y_{6},\!2z^{3}_{2}y_{6},\! 2z^{2}_{2}z_{4}y_{4},\\
&\qquad\; n\makebox[50 mm]{} 48\\
&\;\:\pi _{n}(MSp)\makebox[42 mm]{} 77{\Bbb Z}\\
&\mbox{Generators }2z^{6}_{1}z_{2}y_{4}, 2z^{2}_{1}y_{4}y_{6}, 2z^{3}_{1}z_{2}y_{7}, z^{4}_{1}z_{2}z^{2}_{3}+ z^{5}_{1}z_{3}y_{4}, 2z^{4}_{1}z_{8}, z^{6}_{1}z_{6}+ z^{4}_{1}z^{2}_{2}z_{4},\\
&\qquad\; n\makebox[50 mm]{} 48\\
&\;\:\pi _{n}(MSp)\makebox[42 mm]{} 77{\Bbb Z}\\
&\mbox{Generators }2z^{4}_{1}z^{2}_{2}z_{4}, 2z^{8}_{1}z_{4}, 2z^{6}_{1}z_{6}, z^{8}_{1}y_{4} + z^{7}_{1}z_{2}z_{3}, 2z^{6}_{1}z^{3}_{2}, 2z^{10}_{1}z_{2},2z^{2}_{1}y_{4}z_{6},\! 2z^{2}_{1}z^{5}_{2}, \\
&\qquad\; n\makebox[50 mm]{} 48\\
&\;\:\pi _{n}(MSp)\makebox[42 mm]{} 77{\Bbb Z}\\
&\mbox{Generators }z_{1}z_{4}z_{7} + z^{2}_{1}y_{10}, z^{3}_{1}z_{3}y_{6} + z^{4}_{2}z_{4}, 2z^{4}_{2}z_{4}, z^{4}_{2}z_{4} + z^{2}_{1}z_{2}z_{4}y_{4}, 2z^{3}_{4},2z^{3}_{2}z^{2}_{3},\\
&\qquad\; n\makebox[50 mm]{} 48\\
&\;\:\pi _{n}(MSp)\makebox[42 mm]{} 77{\Bbb Z}\\
&\mbox{Generators }z^{2}_{1}y_{10} + z^{2}_{1}y^{*}_{10} + z_{1}z_{2}z_{9} + z_{1}z_{2}y_{9}, 2z_{2}y_{4}y_{6}, 2z_{2}z^{2}_{3}y_{4},2z_{1}z_{2}y_{9}, z_{2}z_{10}, \\
&\qquad\; n\makebox[50 mm]{} 48\\
&\;\:\pi _{n}(MSp)\makebox[42 mm]{} 77{\Bbb Z}\\
&\mbox{Generators }2z_{2}y^{*}_{10}, 2z_{2}y_{10}, 2z_{4}y_{8}, 2z_{6}y_{6}, 2z^{2}_{4}y_{4}, 2z_{4}y^{2}_{4},z^{3}_{2}z^{2}_{3} + z^{4}_{2}y_{4}, 2z^{2}_{1}z_{2}y^{2}_{4}, z^{6}_{2}, \\
&\qquad\; n\makebox[50 mm]{} 48\\
&\;\:\pi _{n}(MSp)\makebox[42 mm]{} 77{\Bbb Z}\\
&\mbox{Generators }2y_{4}z_{8}, z_{1}y_{4}z_{7}+ z_{3}z_{4}z_{5}, z_{3}z_{4}z_{5}+ z_{1}z_{5}y_{6}, z_{2}z_{4}y_{6}+ z^{2}_{2}y_{8}, 2z^{2}_{3}z_{6},2y^{3}_{4},
\end{align*}
Table 18. (continuation).
\begin{align*}
&\qquad\; n\makebox[50 mm]{} 48\\
&\;\:\pi _{n}(MSp)\makebox[42 mm]{} 77{\Bbb Z}\\
&\mbox{Generators }2y_{4}y_{8}, z_{3}z_{5}y_{4} + z^{2}_{3}y_{6}, 2z^{2}_{3}y_{6}, z_{2}z_{3}y_{7} + z_{1}y_{4}y_{7},2z_{12},\! z^{3}_{1}z_{9},\! z^{2}_{2}z^{2}_{4},\! z^{4}_{1}z^{2}_{4}, \\
&\qquad\; n\makebox[50 mm]{} 48\\
&\;\:\pi _{n}(MSp)\makebox[42 mm]{} 77{\Bbb Z}\\
&\mbox{Generators } z^{8}_{1}z^{2}_{2},z^{7}_{1}z_{5}, z^{4}_{1}z^{4}_{2}, z^{3}_{1}y_{9}, z^{3}_{2}z_{6}, z^{5}_{1}y_{7}, z^{12}_{1}, z^{4}_{1}y^{2}_{4}, z_{2}z^{2}_{3}z_{4}, z^{2}_{2}y^{2}_{4}, z_{1}y_{11}, z^{4}_{3},\\
&\qquad\; n\makebox[35 mm]{} 48\makebox[50 mm]{} 49\\
&\;\:\pi _{n}(MSp)\makebox[27 mm]{} 77{\Bbb Z}\makebox[48 mm]{} 41{\Bbb Z}_{2}\\
&\mbox{Generators }z^{2}_{1}z^{3}_{2}z_{4}, z_{1}z_{11}, z_{4}z_{8}, z_{3}z_{9}, z_{3}y_{9}, y^{2}_{6}, z_{5}z_{7}\qquad\; \kappa _{2},\theta _{1}(z_{2}z_{4}y_{6} + z^{2}_{2}y_{8}), \\
&\qquad\; n\makebox[50 mm]{} 49\\
&\;\:\pi _{n}(MSp)\makebox[42 mm]{} 41{\Bbb Z}_{2}\\
&\mbox{Generators }\theta _{1}z^{4}_{1}z^{4}_{2},\, \theta _{1}z^{8}_{1}z^{2}_{2},\, y^{2}_{4}\tau_{1},\, \Omega _{1},\, \theta _{1}y^{2}_{6},\, \theta _{1}(z^{6}_{1}z_{6}+ z^{4}_{1}z^{2}_{2}z_{4}), \theta _{1}z^{3}_{1}z_{9},\theta _{1}z^{2}_{2}z^{2}_{4}, \\
&\qquad\; n\makebox[50 mm]{} 49\\
&\;\:\pi _{n}(MSp)\makebox[42 mm]{} 41{\Bbb Z}_{2}\\
&\mbox{Generators }\theta _{1}(z^{4}_{1}z_{2}z^{2}_{3}+ z^{5}_{1}z_{3}y_{4}),\theta _{1}(z^{8}_{1}y_{4} + z^{7}_{1}z_{2}z_{3}), \theta _{1}z^{2}_{2}y^{2}_{4}, \theta _{1}z^{4}_{1}z^{2}_{4},\theta _{1}z^{2}_{1}z^{3}_{2}z_{4},\\
&\qquad\; n\makebox[50 mm]{} 49\\
&\;\:\pi _{n}(MSp)\makebox[42 mm]{} 41{\Bbb Z}_{2}\\
&\mbox{Generators }\tau_{8}, \theta _{1}z^{4}_{3}, \theta _{1}(z^{3}_{1}z_{3}y_{6} + z^{4}_{2}z_{4}), \theta _{1}(z^{4}_{2}z_{4} + z^{2}_{1}z_{2}z_{4}y_{4}), \theta _{1}z_{2}z^{2}_{3}z_{4},\theta _{1}z^{5}_{1}y_{7},\\
&\qquad\; n\makebox[50 mm]{} 49\\
&\;\:\pi _{n}(MSp)\makebox[42 mm]{} 41{\Bbb Z}_{2}\\
&\mbox{Generators }\theta _{1}(z_{1}z_{4}z_{7} + z^{2}_{1}y_{10}), \theta _{1}(z^{3}_{2}z^{2}_{3} + z^{4}_{2}y_{4}), \theta _{1}z^{7}_{1}z_{5}, \theta _{1}z^{6}_{2}, \theta _{1}z^{4}_{1}y^{2}_{4}, \theta _{1}z^{3}_{1}y_{9}, \\
&\qquad\; n\makebox[50 mm]{} 49\\
&\;\:\pi _{n}(MSp)\makebox[42 mm]{} 41{\Bbb Z}_{2}\\
&\mbox{Generators }\theta _{1}(z^{2}_{1}y_{10} + z^{2}_{1}y^{*}_{10} + z_{1}z_{2}z_{9} + z_{1}z_{2}y_{9}),\theta _{1}(z_{1}y_{4}z_{7} + z_{3}z_{4}z_{5}), \theta _{1}z^{3}_{2}z_{6}, \\
&\qquad\; n\makebox[50 mm]{} 49\\
&\;\:\pi _{n}(MSp)\makebox[42 mm]{} 41{\Bbb Z}_{2}\\
&\mbox{Generators }\theta _{1}(z_{3}z_{4}z_{5} + z_{1}z_{5}y_{6}),\theta _{1}(z_{3}z_{5}y_{4}+ z^{2}_{3}y_{6}), \theta _{1}(z_{2}z_{3}y_{7}+ z_{1}y_{4}y_{7}),\theta _{1}z^{12}_{1}\\
&\qquad\; n\makebox[50 mm]{} 49\\
&\;\:\pi _{n}(MSp)\makebox[42 mm]{} 41{\Bbb Z}_{2}\\
&\mbox{Generators } \theta _{1}z_{4}z_{8} = \theta _{1}z_{1}z_{11} = \Phi _{2}z^{2}_{4} = \Phi _{4}z_{1}z_{4}, \Phi _{1}(z_{2}z_{3}y_{6} + z^{2}_{3}z_{5}),\theta _{1}z_{1}y_{11}=\\
&\qquad\; n\makebox[50 mm]{} 49\\
&\;\:\pi _{n}(MSp)\makebox[42 mm]{} 41{\Bbb Z}_{2}\\
&\mbox{Generators }= \theta _{1}z_{2}z_{10}\! =\! \Phi _{5}z_{1}z_{2},\! \theta _{1}z_{3}z_{9}\! =\! \Phi _{1}z_{3}z_{8}\! =\! \Phi _{4}z_{2}z_{3},\!\theta _{1}z_{3}y_{9}\! =\! \Phi _{2}z_{3}z_{6}\! =\\
&\qquad\; n\makebox[50 mm]{} 49\\
&\;\:\pi _{n}(MSp)\makebox[42 mm]{} 41{\Bbb Z}_{2}\\
&\mbox{Generators }= \Phi _{1}z_{4}y_{7} = \Phi _{3}z_{3}z_{4},\quad \Phi _{1}(z_{3}z_{4}y_{4} + z^{2}_{3}z_{5}),\quad\theta _{1}z_{5}z_{7} = \Phi _{1}z_{4}z_{7} =
\end{align*}
Table 18. (continuation).
\begin{align*}
&\qquad\; n\makebox[24 mm]{} 49\makebox[50 mm]{} 50\\
&\;\:\pi _{n}(MSp)\makebox[16 mm]{} 41{\Bbb Z}_{2}\makebox[46 mm]{} 48{\Bbb Z}_{2}\\
&\mbox{Generators }= \Phi _{2}z_{2}z_{7} = \Phi _{2}z_{4}z_{5}\qquad\; \theta _{1}\kappa _{2},\theta ^{2}_{1}(z_{2}z_{4}y_{6}+z^{2}_{2}y_{8}), \theta ^{2}_{1}z^{4}_{1}z^{4}_{2},\theta ^{2}_{1}z^{8}_{1}z^{2}_{2},  \\
&\qquad\; n\makebox[50 mm]{} 50\\
&\;\:\pi _{n}(MSp)\makebox[42 mm]{} 48{\Bbb Z}_{2}\\
&\mbox{Generators }\theta _{1}\tau_{8}, \theta _{1}y^{2}_{4}\tau_{1}, \theta _{1}\Omega _{1}, \theta ^{2}_{1}y^{2}_{6},\theta ^{2}_{1}(z^{4}_{1}z_{2}z^{2}_{3}+ z^{5}_{1}z_{3}y_{4}), \theta ^{2}_{1}(z^{6}_{1}z_{6}+ z^{4}_{1}z^{2}_{2}z_{4}),\\
&\qquad\; n\makebox[50 mm]{} 50\\
&\;\:\pi _{n}(MSp)\makebox[42 mm]{} 48{\Bbb Z}_{2}\\
&\mbox{Generators }\theta ^{2}_{1}z^{3}_{1}z_{9}, \theta ^{2}_{1}z^{2}_{2}z^{2}_{4}\theta ^{2}_{1}(z^{8}_{1}y_{4} + z^{7}_{1}z_{2}z_{3}), \theta ^{2}_{1}z^{2}_{2}y^{2}_{4}, \theta ^{2}_{1}z^{4}_{3}, \theta ^{2}_{1}z^{4}_{1}z^{2}_{4}, \theta ^{2}_{1}z^{2}_{1}z^{3}_{2}z_{4},\\
&\qquad\; n\makebox[50 mm]{} 50\\
&\;\:\pi _{n}(MSp)\makebox[42 mm]{} 48{\Bbb Z}_{2}\\
&\mbox{Generators }\theta ^{2}_{1}(z^{3}_{1}z_{3}y_{6} + z^{4}_{2}z_{4}), \theta ^{2}_{1}(z^{4}_{2}z_{4} + z^{2}_{1}z_{2}z_{4}y_{4}), \theta ^{2}_{1}z_{2}z^{2}_{3}z_{4}, \theta ^{2}_{1}z^{7}_{1}z_{5}, \theta ^{2}_{1}z^{5}_{1}y_{7},\\
&\qquad\; n\makebox[50 mm]{} 50\\
&\;\:\pi _{n}(MSp)\makebox[42 mm]{} 48{\Bbb Z}_{2}\\
&\mbox{Generators }\theta ^{2}_{1}(z_{1}z_{4}z_{7} + z^{2}_{1}y_{10}), \theta ^{2}_{1}(z^{3}_{2}z^{2}_{3} + z^{4}_{2}y_{4}),\theta ^{2}_{1}(z_{1}y_{4}z_{7} + z_{3}z_{4}z_{5}),\theta ^{2}_{1}z^{4}_{1}y^{2}_{4}, \\
&\qquad\; n\makebox[50 mm]{} 50\\
&\;\:\pi _{n}(MSp)\makebox[42 mm]{} 48{\Bbb Z}_{2}\\
&\mbox{Generators }\theta ^{2}_{1}(z^{2}_{1}y_{10} + z^{2}_{1}y^{*}_{10} + z_{1}z_{2}z_{9} + z_{1}z_{2}y_{9}),\theta ^{2}_{1}(z_{3}z_{4}z_{5} + z_{1}z_{5}y_{6}),\theta ^{2}_{1}z^{3}_{1}y_{9}, \\
&\qquad\; n\makebox[50 mm]{} 50\\
&\;\:\pi _{n}(MSp)\makebox[42 mm]{} 48{\Bbb Z}_{2}\\
&\mbox{Generators }\Phi _{2}\tau_{5},\theta ^{2}_{1}z^{6}_{2},\theta ^{2}_{1}(z_{3}z_{5}y_{4}+ z^{2}_{3}y_{6}), \theta ^{2}_{1}(z_{2}z_{3}y_{7}+ z_{1}y_{4}y_{7}), \theta ^{2}_{1}z^{3}_{2}z_{6}, \theta ^{2}_{1}z^{12}_{1},\\
&\qquad\; n\makebox[50 mm]{} 50\\
&\;\:\pi _{n}(MSp)\makebox[42 mm]{} 48{\Bbb Z}_{2}\\
&\mbox{Generators }\theta ^{2}_{1}z_{1}z_{11}, \theta _{1}\Phi _{1}(z_{2}z_{3}y_{6} + z^{2}_{3}z_{5}), \theta _{1}\Phi _{1}(z_{3}z_{4}y_{4} + z^{2}_{3}z_{5}),\theta ^{2}_{1}z_{3}y_{9}, \Phi _{1}\tau_{7}, \\
&\qquad\; n\makebox[50 mm]{} 50\makebox[50 mm]{} 51\\
&\;\:\pi _{n}(MSp)\makebox[42 mm]{} 48{\Bbb Z}_{2}\makebox[49 mm]{} 0\\
&\mbox{Generators }\theta ^{2}_{1}z_{1}y_{11}, \theta ^{2}_{1}z_{3}z_{9}, \theta ^{2}_{1}z_{5}z_{7}, \Phi _{1}\tau^{*}_{7}, \Phi _{3}\tau_{3}, \Phi _{1}\Phi _{6}, \Phi _{2}\Phi _{5}, \Phi _{1}\Phi _{2}y^{2}_{4}
\end{align*}
\medskip
Remark. In Table 18 we have used the following notations:
\begin{align*}
&\tau_{1}\ = U_{1}y_{4} + U_{2}z_{3}, \tau_{2} = U_{1}y_{6} + U_{3}z_{3}, \tau_{3} = U_{2}y_{6} + U_{3}y_{4}, \tau_{4} = U_{1}y_{8}+ U_{2}z_{7},\tau_{5}\! =\\
& = U_{2}y_{8} + U_{3}y_{6}, \tau_{6} = U_{1}y_{10}+ U_{3}z_{7}, \tau^{*}_{6} = U_{1}(y^{*}_{10} + y_{10}) + U_{2}(z_{9} + y_{9}),\chi _{1} = \\
&=U_{1}y_{6}y_{4} + U_{2}(y_{6}z_{3} + y_{4}z_{5}), \tau_{7} = U_{2}y_{10}+ U_{3}y_{8},\tau_{8} = U_{1}y_{12}+ U_{3}(z_{9} + y_{9}),\\
&\tau^{*}_{7} = U_{2}(y^{*}_{10} + y_{10}) + \Phi _{3}y_{6} + U_{4}y_{4}, \chi _{2} = U_{1}y_{8}y_{4} + U_{2}(y_{8}z_{3} + y_{4}z_{7}).
\end{align*}

\addcontentsline{toc}{chapter}{References}
\end{document}